\documentclass[hidelinks]{article}
\usepackage[left=2.1cm,right=2.1cm,top=2cm,bottom=2cm]{geometry} 
\usepackage{amsmath}
\usepackage{amsthm}
\usepackage{amssymb}
\usepackage{dsfont}
\usepackage{stackengine}
\usepackage{graphicx}
\usepackage[title]{appendix}
\usepackage{bm}
\stackMath
\usepackage{hyperref}
\usepackage[svgnames]{xcolor}
\usepackage[style=numeric,backend=bibtex,natbib=true,maxnames=5,giveninits=true]{biblatex}
\numberwithin{equation}{section}

\usepackage{mathrsfs} 
\usepackage{orcidlink}

\makeatletter
\newsavebox{\@brx}
\newcommand{\llangle}[1][]{\savebox{\@brx}{\(\m@th{#1\langle}\)}%
	\mathopen{\copy\@brx\kern-0.5\wd\@brx\usebox{\@brx}}}
\newcommand{\rrangle}[1][]{\savebox{\@brx}{\(\m@th{#1\rangle}\)}%
	\mathclose{\copy\@brx\kern-0.5\wd\@brx\usebox{\@brx}}}
\makeatother

\hypersetup{
	colorlinks = true,
	citecolor = teal, 
	linkcolor=red,
	urlcolor=blue}

\DeclareMathOperator{\Tr}{Tr}

\newcommand{\norm}[1]{\left\lVert#1\right\rVert}

\newcommand{\m}{\vect{m}}
\newcommand{\supp}[1]{\operatorname{supp}(#1)}
\DeclareMathOperator{\im}{Im}
\DeclareMathOperator{\re}{Re}

\DeclareMathOperator{\Expv}{\mathbb{E}}

\DeclareMathOperator{\sign}{sign}

\newcommand{\Gap}[1]{\operatorname{Gap}\left(#1\right)}
\newcommand{\diag}[1]{\operatorname{diag}\left(#1\right)}
\newcommand{\other}[1]{\widetilde{#1}}

\newcommand{\vv}{\vect{v}}

\newcommand{\eigB}{\beta}

\newcommand{\vect}[1]{\mathbf{#1}}
\newcommand{\I}{i}
\newcommand{\stab}{\mathcal{B}}

\DeclareFontFamily{OT1}{pzc}{}
\DeclareFontShape{OT1}{pzc}{m}{it}{ <-> s*[1.1] pzcmi7t }{}
\DeclareMathAlphabet{\mathpzc}{OT1}{pzc}{m}{it}

\newcommand{\M}[1]{%
	M(#1)
}

\newcommand{\Brwn}{\mathfrak{B}}
\newcommand{\bulk}{\mathbb{D}^{\mathrm{bulk}}}
\newcommand{\n}{\mathrm{n}}
\newcommand{\chain}{\mathcal{X}}
\newcommand{\size}{\mathfrak{s}}

\newtheorem{theorem}{Theorem} %
\numberwithin{theorem}{section} 
\newtheorem{lemma}[theorem]{Lemma} %
\newtheorem{Def}[theorem]{Definition} %
\newtheorem{prop}[theorem]{Proposition} %
\newtheorem{remark}[theorem]{Remark} %
\DeclareFieldFormat*{title}{#1}
\AtEveryBibitem{\clearfield{month}}
\AtEveryBibitem{\clearfield{number}}
\AtEveryBibitem{\clearfield{publisher}}
\AtEveryBibitem{\clearfield{url}}
\DeclareFieldFormat*{eprint}{\href{http://arxiv.org/abs/#1}{#1}}
\begin{document}
	
	\begin{minipage}{0.85\textwidth}
		\vspace{2.5cm}
	\end{minipage}
	\begin{center}
		\large\bf Eigenstate Thermalization Hypothesis for Wigner-type Matrices
	\end{center}
	\vspace{0.75cm}
	
	\renewcommand{\thefootnote}{\fnsymbol{footnote}}

	\noindent
	\mbox{}%
	\hfill% 
	\begin{minipage}{0.21\textwidth}
		\centering
		{L\'aszl\'o Erd\H{o}s}\footnotemark[1]~\orcidlink{0000-0001-5366-9603}\\
		\footnotesize{\textit{lerdos@ist.ac.at}}
	\end{minipage}
	\hfill%
	\begin{minipage}{0.21\textwidth}
		\centering
		{Volodymyr Riabov}\footnotemark[1]~\orcidlink{0009-0007-4989-7524}\\
		\footnotesize{\textit{vriabov@ist.ac.at}}
	\end{minipage}
	\hfill%
	\mbox{}%
	\footnotetext[1]{Institute of Science and Technology Austria, Am Campus 1, 3400 Klosterneuburg, Austria.  Supported by the ERC Advanced Grant "RMTBeyond" No.~101020331.
	}
	
	\renewcommand*{\thefootnote}{\arabic{footnote}}
	\vspace{0.75cm}
	
	\begin{center}
		\begin{minipage}{0.91\textwidth}\footnotesize{
				{\bf Abstract.}}
				 	We prove the Eigenstate Thermalization Hypothesis for general Wigner-type matrices in the bulk of the self-consistent spectrum, with optimal control on the fluctuations for observables of arbitrary rank. 
				 	As the main technical ingredient, we prove rank-uniform optimal local laws for one and two resolvents of a Wigner-type matrix with regular observables.  
				 	Our results hold under very general conditions on the variance profile, even allowing many vanishing entries, demonstrating that Eigenstate Thermalization occurs robustly across a diverse class of random matrix ensembles, for which the underlying quantum system has a non-trivial spatial structure.  
		\end{minipage}
	\end{center}
	
	\vspace{0.8cm}
	
	{\small
		\footnotesize{\noindent\textit{Date}: \today}
	}
	
	\vspace{10mm}
	
	\thispagestyle{headings}

%\vspace{15pt}
%
%\textbf{Acknowledgments.}
%
%	
%\vspace{15pt}
%
%\textbf{Funding.}
% The author was supported by the ERC Advanced Grant "RMTBeyond" No.~101020331.

\section{Introduction}
The \textit{Quantum Unique Ergodicity}, also known as the \textit{Eigenstate Thermalization Hypothesis (ETH)} in the physics literature, asserts that the eigenfunctions of a sufficiently disordered quantum system are uniformly distributed in the phase space. This concept was first formalized by Deutsch \cite{Deutsch91} and Srednicki \cite{Srednicki94}  for general interacting quantum systems. 
For closed chaotic systems, Quantum (Unique) Ergodicity goes back to the more general Bohigas-Giannoni-Schmit \cite{BGS} conjecture in the physics literature and to Shnirelman's theorem \cite{Shn74} in the mathematics literature, see also \cite{ Sarnak, CdV, Zelditch}
for extensions.  In more recent physics literature, the ETH ansatz \cite{dabelow2022thermalization, reimann2021refining} and its refinement for higher-order correlation functions of the eigenvector overlaps \cite{foini2019eigenstate, pappalardi2022eigenstate, pappalardi2023general} were set as a prime framework for understanding thermalization and other approach-to-equilibrium phenomena (such as the growth of out-of-time-order correlators \cite{murthy2019bounds}) for a very general class of chaotic quantum systems.

In the context of random matrix theory, ETH for Wigner matrices was first conjectured by Deutsch \cite{Deutsch91}.
 In particular, this version of ETH states that the orthonormal eigenvectors $\vect{u}_j$ of a standard Wigner\footnote{
	A standard  $N\times N$ Wigner matrix $W = W^*$ has independent identically distributed entries (up to the symmetry constraint) with zero mean and variance $N^{-1}$.
} matrix $W$ tested against a deterministic \textit{observable} (matrix) $A$ converge to the statistical average of the bilinear form $A$, given (in this case) by the normalized trace $\langle A\rangle := \frac{1}{N}\Tr [A]$, that is
\begin{equation} \label{eq:WigETH_norm}
	\max_{j,k}\bigl\lvert	\langle \vect{u}_j, A \vect{u}_k \rangle - \langle A \rangle \delta_{jk}\bigr\rvert \lesssim N^{-1/2+\varepsilon}\norm{A}.
\end{equation}
In this form, ETH was recently proven in \cite{Cipolloni21ETH}. The result was improved upon to capture the optimal rate of convergence  first for the special case when $A$ is a projection in \cite{Benigni_2022, Benigni2023}, then for observables $A$ of arbitrary rank  in \cite{Cipolloni2022RankUnif} (for indices $j,k$ in the \textit{bulk}) and finally in \cite{Cipolloni2023Edge} for general observables and uniformly in the spectrum. In this stronger form, ETH asserts that, with very high probability,
\begin{equation}\label{eq:WigETH_hs}
	\max_{j,k}\bigl\lvert	\langle \vect{u}_j, A \vect{u}_k \rangle - \langle A \rangle \delta_{jk}\bigr\rvert \lesssim N^{-1/2+\varepsilon}\bigl\langle |A-\langle A\rangle|^2 \bigr\rangle^{1/2}.
\end{equation}
Note that instead of the operator norm $\norm{A}$ as in \eqref{eq:WigETH_norm}, the error term in \eqref{eq:WigETH_hs} is controlled in terms of the Hilbert-Schmidt norm of the traceless part of $A$, the latter being substantially smaller for observables $A$ of low rank. Furthermore, $N^{-1/2}$  is the optimal convergence rate; in fact  the limiting distribution of rescaled fluctuations  $N^{1/2} [ \langle \vect{u}_j, A \vect{u}_k \rangle - \langle A \rangle \delta_{jk}] $ is Gaussian for each $j, k$ with variance proportional to $\langle |A - \langle A \rangle|^2 \rangle$ as was shown in \cite{Cipolloni2022RankUnif}.
Beyond standard Wigner matrices, eigenstate thermalization in the form \eqref{eq:WigETH_norm} was also proved for a restricted class of generalized\footnote{
	The entries of the generalized Wigner matrix $W=W^*$ are centered but need not be identically distributed, as in the standard Wigner case. Instead, the matrix of variances $S_{jk} := \mathrm{Var}[W_{jk}]$ is assumed to be flat ($S_{jk}\sim N^{-1}$) and stochastic ($\sum_{k}S_{jk} = 1$).
} Wigner matrices \cite{Adhikari23}, and for deformed random matrix ensembles \cite{Cipolloni2023Overlap, Cipolloni2023EquiPart}. For a more detailed overview of the Eigenstate Thermalization Hypothesis, we direct the reader to the introduction of \cite{Cipolloni21ETH}. 

In the main Theorem \ref{th:ETH} of the present paper, we prove the ETH in the bulk of general \textit{Wigner-type} matrices with optimal error term controlled by the natural Hilbert-Schmidt norm of the observable.
Introduced in \cite{Ajanki2016Univ}, a Wigner-type matrix $H = H^*$ has independent but not necessarily identically 
distributed  entries with possibly nonzero expected values $\Expv[H_{jk}] = \delta_{jk}\mathfrak{a}_j$  on the
diagonal. The matrix of  variances, $S_{jk}:=\mathrm{Var}[H_{jk}]$, satisfies the natural \textit{flatness} and \textit{uniform primitivity} assumptions (see Definition \ref{def:WT} below). Wigner-type matrices serve as a natural generalization of both standard Wigner (with $S_{jk} = N^{-1}$) and generalized Wigner matrices (with $S$ being stochastic) but they go well beyond them, even allowing many zero matrix elements.
Compared to standard Wigner ensemble, random matrices of general Wigner type encompass non-trivial spatial variations in the underlying quantum model. 
%Therefore, our results show that the ETH phenomenon is not restricted to the simpler homogeneous situations. 
Therefore our results rigorously show that the ETH phenomenon is not restricted to the simplest homogeneous situation; it is robustly present
in a quite general class of quantum Hamiltonians with a spatial structure. 
%In particular, our result allows for zero matrix elements  thus it covers some band matrices albeit only with a bandwidth comparable with the system size (see  \cite{stone2023random} for the Anderson transition in this model, where a special case of ETH was also proven in Theorem 2.2 of \cite{stone2023random}, while the general case was left as an open problem).

%
Unlike in the aforementioned standard and generalized Wigner ensembles, the limiting density of the eigenvalues of a Wigner-type random matrix $H$ is no longer given by the semicircular law. Instead, the self-consistent spectral density $\rho$ is recovered from the unique solution $\m := (m_j)_{j=1}^N$ of the corresponding \textit{vector Dyson equation} (formerly referred to as the \textit{quadratic vector equation} in \cite{Ajanki2019QVE}) 
with a data pair $(\bm{\mathfrak{a}},S)$ and a spectral parameter $z\in \mathbb{C}\setminus \mathbb{R}$,
\begin{equation} \label{eq:VDEz} 
	-\frac{1}{m_j(z)} = z - \mathfrak{a}_{j} + \sum_{k=1}^N S_{jk} m_k(z), \quad (\im z) \im m_j(z) >0,
\end{equation}
via the Stieltjes inversion formula applied to the function  $z\mapsto N^{-1}\sum_{j=1}^N \im m_j(z)$.

As is the case for many mean-field random matrices, the resolvent $G(z) := (H-z)^{-1}$ concentrates around a deterministic matrix called the \textit{self-consistent resolvent} $M(z)$ which can be computed from  the  general (matrix) Dyson equation.
In the Wigner-type setup, the matrix Dyson equation reduces to the vector equation~\eqref{eq:VDEz} and $\M{z} := \diag{\m(z)}$ is given by a diagonal matrix with the entries of the solution vector $\m(z)$ on the main diagonal. 
In general, the self-consistent resolvent for Wigner-type matrices is not proportional to the identity matrix, as is the case for the standard and generalized Wigner matrices that correspond to the scalar Dyson equation $-1/m_{\mathrm{sc}}(z) = z + m_{\mathrm{sc}}(z)$. The main consequence of the non-trivial spatial structure for ETH is that the normalized trace $\langle A \rangle$  in \eqref{eq:WigETH_norm} is replaced with a considerably more complicated energy-dependent quantity, namely
\begin{equation}
	\langle \vect{u}_j, A\vect{u}_k \rangle \approx \delta_{jk}\frac{\bigl\langle \im\M{\lambda_j} A\bigr\rangle}{\pi\rho(\lambda_j)},
\end{equation}
where $\lambda_j$ is the eigenvalue that corresponds to the eigenvector $\vect{u}_j$ of the Wigner-type matrix $H$. 
In fact, this effect was already observed in \cite{Cipolloni2023Overlap, Cipolloni2023EquiPart}, where the corresponding matrix Dyson equations also produced non-trivial self-consistent resolvents, but in these papers the spatial inhomogeneity solely stemmed from the non-trivial matrix of expectations while the variance profile was still fully homogeneous (so-called \emph{deformed Wigner matrices}).

Likewise, energy dependence enters the concept of observable \textit{regularity}, which is crucial for the proof of ETH.
%that replaces tracelessness in determining the size of both the deterministic approximation and the fluctuating part of the quantity $\langle G(z_1) A_1 G(z_2) A_2 \rangle$, which plays a crucial role in the proof of ETH for Wigner matrices \cite{Cipolloni21ETH}. 
To see this, note that by spectral decomposition of $H$ we have the identity
\begin{equation*}
   \frac{1}{N}\sum_{j,k}  \frac{\im z_1}{|\lambda_j-z_1|^2}\frac{\im z_2}{|\lambda_k-z_1|^2}
    |\langle \vect{u}_j, A\vect{u}_k \rangle|^2 =  \langle \im G(z_1)A\im G(z_2)A\rangle,
\end{equation*}
thus, along the proof, we detect the size of $\langle \vect{u}_j, A\vect{u}_k \rangle$ via a good upper bound on the right-hand side. It turns out that both the deterministic approximation and the fluctuating part of   $\langle G(z_1) A_1 G(z_2) A_2 \rangle$ are much smaller if the observables $A_1, A_2$ belong to a special one-dimensional subspace; these
are called \textit{regular observables}.
Roughly speaking, regular matrices are orthogonal to the principal eigenprojector of the \textit{two-body stability operator} $1 - \M{z_1}\M{z_2}S $ corresponding to the vector Dyson equation \eqref{eq:VDEz}, see Definition \ref{def:reg_A} below for more details. 
In the standard and in the generalized Wigner case, traceless matrices are the regular ones, hence the concept of regularity is independent of the energy, which greatly simplifies their analysis.

Now we summarize the new ideas in our proof.
Our main achievement is the optimal local laws for two resolvents of a Wigner-type matrix, interlaced with regular observables, with spectral parameters in the bulk of the self-consistent spectrum, Theorem \ref{th:locallaws}.
The major conceptual difficulty that arises in the setting of Wigner-type matrices,
 in contrast to all previous works, is associated with the general non-factorizability of the matrix of variances $S$. To explain this, consider the \textit{self-energy operator} $\mathcal{S}$ defined by its action on deterministic $N\times N$ matrices $B$,
\begin{equation} \label{eq:self-energy}
	\mathcal{S}[B] := \Expv\bigl[(H -\Expv[H])B(H -\Expv[H])\bigr], % \quad \text{or, equivalently,} \quad \mathscr{S}_{jk}^{ab} = \delta_{jk}\delta_{ab}S_{ja},
\end{equation}
which appears naturally in the second-order renormalization of resolvent chains \cite{Cipolloni21ETH}. Note that for Wigner-type matrices $\mathcal{S}_{jk}^{ab} = \delta_{jk}\delta_{ab}S_{ja}$ and for standard and deformed Wigner matrices, $S_{jk} = N^{-1}$, i.e.  $\mathcal{S}$ 
%\new [Some mismatch of $\mathcal{S}$ and $\mathscr{S}$, I thought we entirely discarded $\mathcal{S}$, search for it in the whole document] \old 
is given simply by $\langle \cdot \rangle I$, where $I$ is the identity matrix\footnote{
	Here, for simplicity, we consider only the \textit{diagonal part} of the self-energy operator $\mathcal{S}$. In the sequel, we denote the diagonal part of $\mathcal{S}$ by $\mathscr{S}$ (see \eqref{eq:superS_def} below). The off-diagonal part of $\mathcal{S}$ (denoted by $\mathscr{T}$ in \eqref{eq:superT_def} below) vanishes identically, for example, for the complex Hermitian random matrices with $\Expv [(H_{jk})^2] = 0$, $j \neq k$. In the proof, however, we treat both the real symmetric and complex Hermitian cases with no restriction on $\Expv[(H_{jk})^2]$, see Section \ref{sec:realT} below.%For the discussion of the real symmetric case and general assumpti, see Section \ref{sec:real_sym_case} below.
}.
The proof of the local laws is naturally  reduced to estimating the action of the quadratic form\footnote{
	We equip the space of $N\times N$ matrices with the Hilbert-Schmidt inner product $\langle X,Y\rangle := \langle X^*Y \rangle = N^{-1}\Tr[X^*Y]$.
}
of $\mathcal{S}$ on two resolvent chains, that emerges   
from the cumulant expansion formula.
For example, along the analysis of a two-resolvent chain,  
one encounters quantities of the form
\begin{equation} \label{ex:quad_term}
	\bigl\langle G A \mathcal{S}[G A G] \bigr\rangle = \frac{1}{N} \sum_{j,k} S_{jk} (GAG)_{kk} (GA)_{jj}, %\,\, \text{or} \,\,  \bigl\langle G A \mathscr{S}[G A G] \bigr\rangle = \frac{1}{N} \sum_{j,k} S_{jk} (GA)_{jj} (GAG)_{kk}
\end{equation}
that  involve \emph{three} resolvents. This potentially leads to an infinite hierarchy of equations
for longer and longer resolvent chains that would be impossible to close.  If $S_{jk} = N^{-1}$, then the  sum  in \eqref{ex:quad_term}  factorizes into a product of two averaged traces of resolvent chains of length at most two  and thus the hierarchy can be closed and one can 
still benefit from the fluctuation averaging in both summation indices simultaneously.  
In the general Wigner-type setting, however, the quadratic form of $\mathcal{S}$ does not factorize,
thus closing the hierarchy requires estimating
%meaning that one can not benefit from the fluctuation averaging in both summation indices simultaneously, i.e.,
one of the chains in such  quadratic term by using an \emph{isotropic law}, giving up a
gain from \emph{fluctuation averaging}. 
This effectively results in a loss of $\sqrt{N|\im z|}$ factor.  
%\new [the text leaves it open whether it is affordable or not (we had a commented out discussion on it -- see tex file)? I hope not, then we should just assert this: This effectively results in a loss of $\sqrt{N|\im z|}$ factor, which is not affordable. ] \old
%[Just for ourselves: when we use reduction ineq. we also lose a factor $\sqrt{N|\im z|}$ and then we say that it is affordable there since we anyway use it when $N|\im z|$ is almost one. Here we say it is not. Probably because this three-G term does not come with extra small prefactor like the quadr. var. term? We should see it very clearly for ourselves.]
%\color{teal} [I think reductions don't carry a loss of $\sqrt{N|\im z|}$ in the H-S norm. When we do a reduction, say, for a 4 chain, the bound we get (ignoring the $\psi$'s) is $\langle (A\im G)^4 \rangle \prec N \langle |A|^2 \rangle^{1/2 \times 4}$. It has two features: $1)$ it's sharp, i.e., it matches the bound on the deterministic approximation $\langle \M{z,A,z,A,z,A,z}A \rangle$, and $2)$ it gives no information about the fluctuations at all, contrary to the $(N\eta)^{-1/2}$ loss above which pertains to a $(G-M)$-type quantity. When we eventually use the reduction to estimate the fluctuations of a chain, the "correct" size is recovered from the extra $(N^2\eta^2)^{-1}$ prefactor. So to me it feels like there's no logical contradiction.] \color{black}
For example, compare the averaged and the isotropic laws for a single resolvent:
\begin{equation}
	\bigl\lvert\langle G(z) - M(z) \rangle\bigr\rvert \lesssim  N^\varepsilon(N|\im z|)^{-1} \quad \text{ and }  \quad \bigl\lvert \bigl(G(z)-M(z)\bigr)_{jj} \bigr\rvert \lesssim N^\varepsilon(N|\im z|)^{-1/2}.
\end{equation}
This loss is not affordable. 
The non-trivial structure of $\mathcal{S}$ also prevents the use of any other algebraic relations, such as
the cyclicity of trace and various resolvent identities, to effectively reduce the length of the emergent chains.
%and leads to a substantially more difficult hierarchy of estimates.
In fact, the same phenomena were encountered in the setting of generalized Wigner matrices in \cite{Adhikari23}, which is the only prior work that has proved ETH for a matrix $S$ that is not the trivial $S_{jk}=N^{-1}$.
However, the authors of \cite{Adhikari23} avoided the  
 non-factorization issue by explicitly assuming that the matrix of variances factorizes, $S=\widetilde S\widetilde S$, with a matrix $\widetilde S$ that is also flat. Note that $S$ is entry-wise positive, but in general, it is not positive-definite as a matrix, and even if it is, there is no natural reason to expect that $S_{jk}\sim N^{-1}$ should imply $\widetilde S_{jk}\sim N^{-1}$.  

Our approach  resolves this key difficulty in full generality without placing any additional conditions on $S$,
by performing a two-stage bootstrap inside the framework of the \textit{characteristic flow method} that we generalize to account for the features of the Wigner-type ensemble.
First, we prove the local laws for one and two resolvents with general observables, i.e., without any improvement from regularity, (Lemma \ref{lemma:1Flaw} and \ref{lemma:2G_weak}). Afterwards, we use these weaker estimates as an input to construct a closed hierarchy of master inequalities for resolvent chains with regular observables (Proposition \ref{prop:masters}). We close the hierarchy by employing two auxiliary control quantities designed to take advantage of the smoothing properties of $\mathcal{S}$ and account for chains containing a mixture of regular and general observables. For further details, see the proof strategy laid out in Section \ref{sec:proof_strat} below.
%
%
%
%We obtain this result using a version of the \textit{characteristic flow method} that we generalize to account for the features of the Wigner-type ensemble. 

The characteristic flow method has been used to great effect to obtain local laws for a single resolvent \cite{Huang2019Rigidity, AdhikariHuang, Landon2021Wignertype, Adhikari2023locallaw, Landon2022, Bourgade21, Lee2015}
as well as multi-resolvent local laws in \cite{Bourgade22, Cipolloni23nonHerm, Cipolloni2023Edge, Cipolloni23Gumbel, stone2023random}
for various models, %but not yet for Wigner-type matrices.
but it has not yet been utilized for Wigner-type matrices. 
  % but in all of these settings the corresponding spectral parameters remained scalar. 
%
It is a dynamical approach that combines a diffusion process in the space of matrices (the Ornstein-Uhlenbeck process) with a conjugate differential equation (\textit{characteristic flow}) that drives the spectral parameters from the global regime, where the local laws are easier to prove, into the local regime of interest. The characteristic flow is carefully chosen so that the simultaneous effect of the two flows results in a crucial algebraic cancellation.
In particular, we construct the Ornstein-Uhlenbeck process (see \eqref{eq:OUflow} below) in such a way that the first two moments of the matrix entries remain invariant. %This guarantees that  the concept of observable regularity can be preserved along the trajectories. 

This allows us to study general Wigner-type ensembles with non-trivial spacial structure that satisfies the \textit{uniform primitivity} assumption (see \eqref{as:S_flat} below). Compared to the usual mean-field assumption of $S_{jk} \sim N^{-1}$, uniform primitivity allows large blocks of matrix elements to vanish
and thus it can mimic more realistic physical models.
%Moreover, multi-resolvent local laws for any random matrix model satisfying $S_{jk} \sim N^{-1}$ can be studied using the characteristic flow method with the standard Brownian motion as the diffusion component in the Ornstein-Uhlenbeck process. %TODO:: finish this
%
The trade-off, however, is that the corresponding conjugate differential equation \eqref{eq:z_evol} produces \emph{vector-valued spectral parameters}, therefore along the flow one has to work with \textit{generalized resolvents} $G(\vect{z}) := (H - \diag{\vect{z}})^{-1}$  (see the discourse in Section \ref{sec:locallaws_proof} below).
In all previous applications of the characteristic flow,  the spectral parameters remained scalar; 
developing the theory for the vector-valued case is one of our two main 
methodological novelties.

%The characteristic flow method was used to great effect to obtain local laws for a single resolvent \cite{Huang2019Rigidity, AdhikariHuang, Landon2021Wignertype, Adhikari2023locallaw, Landon2022, Bourgade21, Lee2015}, and longer resolvent chains \cite{Bourgade22, Cipolloni23nonHerm, Cipolloni2023Edge, Cipolloni23Gumbel} for various models, but in all of these settings the corresponding spectral parameters remained scalar. 
%
To close the system of master inequalities, one inevitably needs to bound longer resolvent chains in terms of shorter ones. 
In all prior works (e.g., \cite{Cipolloni2022Optimal, Cipolloni2022RankUnif, Cipolloni2023Edge, Cipolloni23Gumbel}), 
such estimates, coined as \textit{reduction inequalities} in \cite{Cipolloni2022Optimal}, 
were proved using the spectral decomposition of the resolvent $G$. 
However, this approach is not applicable in the present setting because we work with generalized resolvents $G(\vect{z})$ whose eigenbases depend on the vector $\vect{z}$. We provide a more robust
proof of the reduction inequalities (Lemma \ref{lemma:reds}) using the submultiplicativity of the trace functional for positive-definite matrices. To express an arbitrary resolvent chain as a product of positive-definite matrices, we derive a new versatile integral representation for the resolvent $G$ in terms of its sign-definite imaginary part (Lemma \ref{lemma:imF_rep}). Crucially, our representation allows us to keep the spectral parameter in the resolvent localized and hence keep our analysis restricted to the bulk of the self-consistent spectrum (see Remark \ref{rem:im_g_int} for further detail). Proving reduction inequalities using this new integral
representation is our other key methodological novelty.
%

%[Altogether each paragraph reads very well. The non-factorizability is well explained. Two remarks
%for possible improvements (but I am not sure it is possible or a good idea at all):
%
%1) Char. flow appears twice: first it comes up on top of p3   ("Our main technical achievement..."),
%then on top of p4. Would it be possible first talk about non-factorizability as a general major conceptual
%issue (one does not need flow for that, longer chains and how to split them 
%appear in any  cumulant expansion method) and after that say  that our solution
%is char +OU flow, but it requires (for good reasons) at least two major complications:
%
%a) vector valued z
%
%b) reduction ineq needs a new proof.
%
%I.e practically swap the two main paragraphs of  p3?
%
%
%2)  Can one better highlight the main novelty/idea? Suppose somebody  outsider reads it. 
%It is clear that there are major complications and novelties. But what is the ranking among these?
%Which one we consider THE main new idea?  Maybe there is no such and you don't
%want to make distinction. 
%
%]

\subsection{Notations}
%\color{red} [This is still Work In Progress, I will spruce it up after the dust has otherwise settled.] \color{black}
We denote the complex upper and lower half-planes by $\mathbb{H} := \{z\in\mathbb{C}: \im z >0\}$, and $\mathbb{H}^* := \{z\in\mathbb{C}: \im z < 0\}$.
%For a matrix $X \in \mathbb{C}^{N\times N}$, we denote the operator norm of $X$ induced from $\ell^2(\mathbb{C}^N)$ by $\norm{X}$, and the averaged  trace of $X$ by $\langle X\rangle := N^{-1}\Tr[X]$.
For vectors $\vect{x} := (x_j)_{j=1}^N, \vect{y} := (y_j)_{j=1}^N \in \mathbb{C}^N$, we denote their entry-wise product by $\vect{xy} := (x_jy_j)_{j=1}^N$, and write $1/\vect{x} := (1/x_j)_{j=1}^N$ for the entry-wise multiplicative inverse of $\vect{x}$. 
 Moreover, we denote the diagonal matrix with the entries of the vectors $\vect{x}$ on the main diagonal by $\diag{\vect{x}} := (\delta_{jk} x_j)_{j,k=1}^N$. We use the following conventions for the scalar product and the $\ell^p$-norms with $p \in \{1, 2,\infty\}$,
\begin{equation*}
	\langle \vect{x}, \vect{y} \rangle := \sum_j \overline{x_j}y_j, \quad \norm{\vect{x}}_1 := \sum_j |x_j|, \quad \norm{\vect{x}}_2 := \langle \vect{x}, \vect{x} \rangle^{1/2}, \quad \norm{\vect{x}}_\infty := \max_{j}|x_j|.
\end{equation*}
Here and in the sequel, all unrestricted summations run over the set $\{1,\dots,N\}$.  

\section{Main Results}
We work in the setting of Wigner-type matrices originally introduced in \cite{Ajanki2016Univ}.
\begin{Def}[Wigner-type Matrices] \label{def:WT} %TODO:: add reference
	Let $H = (H_{jk})_{j,k=1}^N$ be an $N\times N$ random matrix with independent entries up to the symmetry constraint $H=H^*$, that satisfy
	\begin{equation} \label{eq:vect_a}
		\Expv [H_{jk}] = \delta_{jk} \mathfrak{a}_j, \quad \bm{\mathfrak{a}} := \bigl(\mathfrak{a}_j\bigr)_{j=1}^N\in\mathbb{R}^N, \quad \norm{\bm{\mathfrak{a}}}_\infty \le C_{\mathfrak{a}},
	\end{equation}
	for some positive constant $C_{\mathfrak{a}}$.
	We consider both real symmetric and complex Hermitian Wigner-type matrices. % In case the matrix $H$ is complex, we assume additionally that the $\{\re H_{jk},\im H_{jk}\}_{j\le k}$ are jointly independent and that $\Expv [(H_{jk})^2] = 0$ for $j\neq k$.

	\textbf{Assumption (A).}
	Let $S$ denote the matrix of variances $S:=(S_{jk})_{j,k=1}^N$,  $S_{jk} := \Expv |H_{jk}-\delta_{jk}\mathfrak{a}_j|^2$. We assume that $S$ satisfies the \textit{uniform primitivity}\footnote{
		Uniform primitivity, i.e., the first bound in \eqref{as:S_flat}, was used in \cite{Ajanki2019QVE} to derive various properties of the vector Dyson equation \eqref{eq:VDEz} and its solution $\m$ that we use throughout our proof, otherwise only the upper bound in \eqref{as:S_flat},  the flatness property, 
		 is used directly in the proof of the local laws.
	}
	and \textit{flatness}  conditions, i.e., there exists an integer $L$ such that  
	\begin{equation} \tag{A} \label{as:S_flat}
		 \bigl(S^L\bigr)_{jk} \ge \frac{c_{\inf}}{N},  \quad  S_{jk} \le \frac{C_{\mathrm{sup}}}{N},
	\end{equation}
	for some $N$-independent strictly positive constants $c_{\inf}, C_{\mathrm{sup}}$, and all $j,k\in \{1,\dots,N\}$. 
	
	\textbf{Assumption (B).} Furthermore, we assume that all higher centered moments of $\sqrt{N}H_{jk}$ are uniformly bounded in $N$, that is, for all $p\in\mathbb{N}$, there exists a positive constant $C_p$ such that for all $j,k \in \{1,\dots, N\}$,
	\begin{equation} \tag{B} \label{as:moments}
		\Expv\bigl[ |H_{jk}-\delta_{jk}\mathfrak{a}_j|^p \bigr] \le \frac{C_p}{N^{p/2}}.
	\end{equation}
	
	\textbf{Assumption (C).} We assume\footnote{
		General condition on the deformation $\bm{\mathfrak{a}}$ and the matrix of variances $S$ for the Assumption \eqref{as:m_bound} were obtained in Chapter 6 of \cite{Ajanki2019QVE}.   For example, given \eqref{as:S_flat}, the uniform bound in Assumption \eqref{as:m_bound} holds if the entries of the expectation %bare 
		vector $\mathfrak{a}_j = \mathfrak{a}(j/N)$ and the elements of the matrix of variances $S_{jk} = \tfrac{1}{N} S(j/N, k/N)$ are sampled from a pair of fixed piece-wise $1/2$-H\"{o}lder regular functions $\mathfrak{a} : [0,1] \to \mathbb{R}$ and $S:[0,1]^2 \to \mathbb{R}_+$ (see Remark 6.5 in \cite{Ajanki2019QVE}).
	} that the unique (Theorem 2.1 in \cite{Ajanki2019QVE}) solution $\m \equiv \m_N := (m_j)_{j=1}^N$ of the vector Dyson equation \eqref{eq:VDEz} with data pair $(\bm{\mathfrak{a}},S) \in \mathbb{R}^N \times \mathbb{R}_{\ge0}^{N\times N}$ satisfies the bound
	\begin{equation} \tag{C} \label{as:m_bound}
		\norm{\m(z)}_\infty \le C_\mathrm{m}, \quad  z \in \mathbb{C}, %\quad \im z > 0,
	\end{equation}
	uniformly in $N$ for some positive constant $C_\mathrm{m}$. 
\end{Def}

We remark that the uniform primitivity Assumption \eqref{as:S_flat} allows large blocks of the entries of $H$ to vanish. For example, our model encompasses random band matrices, albeit with band width comparable to the size of the matrix. A special version of Quantum Unique Ergodicity (for quite
specific observables) for such band matrices was proved in \cite{bourgade2016universality} under the assumption that $S$ is stochastic and $S_{jk} \sim N^{-1}$ inside the band.

Let $\M{z}$ denote the self-consistent resolvent given by
\begin{equation} \label{eq:Mz_def}
	M(z) := \diag{\m(z)}.
\end{equation}
By Theorem 4.1 in \cite{Ajanki2016Univ}, the solution vector $\m(z)$ admits a uniformly $1/3$-H\"{o}lder regular extension to the closed upper half-plane $\overline{\mathbb{H}}$, and the \textit{self-consistent density of states} $\rho \equiv \rho_N$ is defined using the Stieltjes inversion formula
\begin{equation} \label{eq:rho_def}
	\rho(E) := \frac{1}{\pi} \lim\limits_{\eta \to +0} \im m(E+\I\eta), \quad \text{with}\quad m(z) := \bigl\langle \M{z} \bigr\rangle = \frac{1}{N}\sum\limits_{j} m_j(z).
\end{equation}
We denote the $j/N$-quantiles of the density $\rho$ by $\gamma_j$, that is
\begin{equation} \label{eq:quantile_gamma}
	\int_{-\infty}^{\gamma_j} \rho(x)\mathrm{d}x = \frac{j}{N}.
\end{equation}

The estimate in our main result is obtained in the sense of \textit{stochastic domination}.
\begin{Def}[Stochastic domination] \label{def:prec}
	Let $X := X^{(N)}(u)$ and $Y := Y^{(N)}(u)$ be two families of random variables depending on a parameter $u \in U^{(N)}$. We say that $Y$ stochastically dominates $X$ uniformly in $u$ if for any $\varepsilon >0$ and $D>0$, there exists $N_0(\varepsilon,D)$ such that for all integers $N \ge N_0(\varepsilon,D)$,
	\begin{equation}
		\sup_{u\in U^{(N)}} \mathbb{P}\biggl[X^{(N)}(u) \ge N^{\varepsilon}Y^{(N)}(u)\biggr] < N^{-D}.
	\end{equation}
	We denote this relation by $X\prec Y$. For complex valued $X$ satisfying $|X| \prec Y$, we write $X = \mathcal{O}_\prec(Y)$.
\end{Def} 
We can now state our main result.
\begin{theorem}(Eigenstate Thermalization Hypothesis for Wigner-type Matrices ) \label{th:ETH}
	Let $H$ be an $N\times N$ random matrix of Wigner-type as in Definition \ref{def:WT}, and let $\lambda_1 \le \dots \le \lambda_N$ and $\vect{u}_1,\dots, \vect{u}_N$ denote its ordered eigenvalues and the corresponding orthonormal eigenvectors, respectively. Let $\rho_{\min} > 0$ be a positive $N$-independent constant, then for any deterministic matrix $B$, the estimate
	\begin{equation} \label{eq:ETH}
		\biggl \lvert \langle \vect{u}_j, B\,\vect{u}_k \rangle - \delta_{jk}\frac{\bigl\langle \im\M{\gamma_j}B\bigr\rangle}{\pi\rho(\gamma_j)} \biggr \rvert \prec \frac{\langle |B|^2\rangle^{1/2}}{\sqrt{N}},
	\end{equation}
	holds uniformly in indices $j,k$ satisfying $\rho(\gamma_j), \rho(\gamma_k) \ge \rho_{\min}$, where $\gamma_j$ are defined by \eqref{eq:quantile_gamma}.
\end{theorem}
\begin{remark}	
	In the setting of standard Wigner matrices, the analog of the estimate \eqref{eq:ETH} was previously obtained in \cite{Cipolloni2022RankUnif, Cipolloni2023Edge} (in the bulk and uniformly in the spectrum, respectively). In an earlier work \cite{Cipolloni21ETH}, the result was proved uniformly in the spectrum but with the Hilbert-Schmidt norm $\langle |B|^2\rangle^{1/2}$ on the right-hand side of \eqref{eq:ETH} replaced by the operator norm $\norm{B - \langle B\rangle}$.
	
	Eigenstate Thermalization \eqref{eq:ETH} with fluctuations controlled only by $\norm{B}$ was also proved in the setting of deformed Wigner (Theorem 2.7 in \cite{Cipolloni2023EquiPart}, bulk spectrum), non-Hermitian random matrices (Theorem 2.2 in \cite{Cipolloni2023Overlap}, bulk spectrum), and a restricted class of generalized Wigner matrices (Theorem 2.3 in \cite{Adhikari23}, uniformly in the spectrum).  %TODO:: talk about the bulk assumption

\end{remark}

\subsection{Two-Resolvent Local Laws}
The key input for proving the Eigenstate Thermalization in Theorem \ref{th:ETH} and our main technical result are the local laws for two resolvents of a Wigner-type matrix.
Our analysis is restricted to the bulk of the self consistent spectrum, which we define by
\begin{equation}\label{eq:Dbulk}
	\bulk \equiv \bulk_{\rho_*,\eta_*} := \{z\in\mathbb{C}\backslash\mathbb{R} : \rho(\re z) \ge \rho_*, |\im z|\le \eta_* \},
\end{equation}
for positive constants $\rho_*,\eta_*>0$.

We refer to the $N$-independent constants $C_{\mathfrak{a}}$, $L$, $c_{\inf}$, $C_{\mathrm{sup}}$, $\{C_p\}_{p\in\mathbb{N}}$, $C_\mathrm{m}$ in \eqref{eq:vect_a}, Assumptions \eqref{as:S_flat}--\eqref{as:m_bound}, and $\rho_*, \eta_*$ in \eqref{eq:Dbulk} as \textit{model parameters}.
For two positive quantities $X$ and $Y$, we write $X \lesssim Y$ if $X \le C Y$ for some constant $C>0$ that depends only on the model parameters. We write $X \sim Y$ if both $X\lesssim Y$ and $Y\lesssim X$ hold.

Let $\mathscr{S}$ denote the \textit{diagonal component} of the self-energy operator $\mathcal{S}$ associated with the random matrix $H$, defined in \eqref{eq:self-energy}. The action of $\mathscr{S}$ on deterministic matrices $B\in\mathbb{C}^{N\times N}$ is defined by
\begin{equation} \label{eq:superS_def}
	\mathscr{S}[B] := \Expv\bigl[ (H-\Expv[H]) \diag{\vect{b}^{\mathrm{diag}}} (H-\Expv[H]) \bigr] 
	=  \diag{S[\vect{b}^{\mathrm{diag}}]}, \quad \vect{b}^{\mathrm{diag}} := (B_{jj})_{j=1}^N.
\end{equation}
We denote the \textit{off-diagonal component}
 of the self-energy operator $\mathcal{S}$ by $\mathscr{T}$, with its action given by
\begin{equation} \label{eq:superT_def}
	\mathscr{T}[B] := \Expv\bigl[ (H-\Expv[H]) B^{\mathrm{od}} (H-\Expv[H]) \bigr] 
	=  \mathcal{T}\odot B^\mathfrak{t}, \quad \mathcal{T}_{jk} := \delta_{j\neq k}\Expv\bigl[(H_{jk})^2\bigr],
\end{equation}
where $B^{\mathrm{od}} := B - \diag{\vect{b}^{\mathrm{diag}}}$ is the off-diagonal part of a deterministic matrix $B$, $\odot$ denotes the  entry-wise Hadamard product, and $(\cdot)^\mathfrak{t}$ denotes transpose. In this notation, the self-energy operator $\mathcal{S}$ defined in \eqref{eq:self-energy} admits the decomposition $\mathcal{S} = \mathscr{S} + \mathscr{T}$.
Following \cite{Erdos2018CuspUF}, we place no further assumptions on the matrix $\mathcal{T}$ (see Eq. (3.2) and Remark 2.9 in \cite{Erdos2018CuspUF}). In particular, the matrix $\mathcal{T}$ allows us to interpolate between the real symmetric and complex Hermitian settings. 

For a deterministic matrix $B$ (often called an \textit{observable} in this context) and a pair of spectral parameters $z_1,z_2$ in the bulk of the spectrum $\bulk$, the deterministic approximation to the resolvent chain $G(z_1)BG(z_2)$ is given by
\begin{equation} \label{eq:M_def}
	\M{z_1,B,z_2} := \bigl(1- \M{z_1}\M{z_2}\mathscr{S}\bigr)^{-1}\bigl[\M{z_1}B\M{z_2}\bigr],
\end{equation}
where $\mathscr{S}$ is defined in \eqref{eq:superS_def}.
%Here, and in the sequel, we identify the solution vector $\m$ with the diagonal matrix $\diag{\m}$, e.g., we use the shorthand notation $\m(z_1)B\m(z_2)$ for the matrix product $\diag{\m(z_1)}B\diag{\m(z_2)}$. 
In particular, the deterministic approximation $\M{z_1,B,z_2}$ admits the expression
\begin{equation} \label{eq:M_offdiag}
	\M{z_1,B,z_2} = \M{z_1}B^{\mathrm{od}}\M{z_2} + \diag{\stab_{z_1,z_2}^{-1}[\m(z_1)\vect{b}^{\mathrm{diag}}\m(z_2)]},
\end{equation} 
where $B^{\mathrm{od}} := B - \diag{\vect{b}^{\mathrm{diag}}}$, and  $\stab_{z_1,z_2} : \mathbb{C}^N \to \mathbb{C}^N$ is the non-Hermitian \textit{two-body stability operator},
\begin{equation} \label{eq:stab_def}
	\stab_{z_1,z_2} := 1 - \M{z_1}\M{z_2}S.
\end{equation}
In general, we have the identity,
\begin{equation} \label{eq:d-od_stab_identity}
	\bigl(1- \M{z_1}\M{z_2}\mathscr{S}\bigr)^{-1}[B] =  B^{\mathrm{od}} + \diag{\stab_{z_1,z_2}^{-1}[\vect{b}^{\mathrm{diag}}]}.
\end{equation}

The key properties of the two-body stability operator are collected in the following lemma.
\begin{lemma}[Stability Operator\footnote{
	In parts of Lemma \ref{lemma:stab_lemma}, we refer to the results from \cite{Landon2021Wignertype} and \cite{R2023bulk}, that were proved under the stronger \textit{flatness} assumption $S_{jk} \ge c N^{-1}$. However, a careful examination the proofs in \cite{Landon2021Wignertype}, \cite{R2023bulk}, together with Proposition 5.4 in \cite{Ajanki2019QVE}, reveal that the same stability operator analysis holds under the weaker Assumption \eqref{as:S_flat}. }] \label{lemma:stab_lemma}
	For all $z_1, z_2 \in \mathbb{H}$ satisfying $\rho(\re z_j) \ge \tfrac{1}{2}\rho_*$ and $\im z_j \le \eta_*$, the stability operators $\mathcal{B}_{z_1,z_2}$ and $\mathcal{B}_{\bar{z}_1,z_2}$ satisfy the uniform bounds (Proposition 4.6 and Lemma 4.7 in \cite{Landon2021Wignertype})
	\begin{equation} \label{eq:scalar_stab_bounds}
		\norm{\mathcal{B}_{z_1,z_2}^{-1}}_* \lesssim 1, \quad \norm{\mathcal{B}_{\bar{z}_1,z_2}^{-1}}_* \lesssim \bigl(|\im z_1| + |\im z_2|\bigr)^{-1},
	\end{equation}
	where $\norm{\cdot}_*$ denotes the operator norm induced by either $\ell^2$ or $\ell^\infty$ vector norm. The second estimate in \eqref{eq:scalar_stab_bounds} can be improved to
	a uniform bound in $z_1,z_2 \in \bulk\cap\mathbb{H}$,
	\begin{equation} \label{eq:far_stability}
		\norm{\mathcal{B}_{\bar{z}_1,z_2}^{-1}}_* \lesssim |\bar{z}_1-z_2|^{-1}.
	\end{equation}

	Furthermore, there exists a threshold $\delta \sim 1$ and a small constant $c \sim 1$, such that for all $z_1, z_2 \in \bulk\cap \mathbb{H}$, satisfying $|z_1-z_2| \le \delta$, the operator  $\mathcal{B}_{\bar{z}_1,z_2}$ has a single isolated eigenvalue $\eigB_{\bar{z}_1,z_2}$ in the disk $\{\zeta \in \mathbb{C} : |\zeta| < c/2\}$, and (Claims 6.4 -- 6.6 in \cite{R2023bulk})
	\begin{equation} \label{eq:stab_gap}
		\norm{\stab_{\bar{z}_1,z_2}^{-1}\bigl(1-\Pi_{\bar{z}_1,z_2}\bigr)}_* \lesssim 1, \quad  \norm{(\zeta - \stab_{\bar{z}_1, z_2})^{-1}}_* \lesssim 1, \quad c/2 \le |\zeta|\le 2c.
	\end{equation}
	 Here, $\Pi_{\bar{z}_1,z_2}$ is the rank one eigenprojector corresponding to the eigenvalue of $\stab_{\bar{z}_1,z_2}$ with the smallest modulus defined via the contour integral
	\begin{equation} \label{eq:Pi_int}
		\Pi_{\bar{z}_1,z_2} := \frac{1}{2\pi\I}\oint_{|\zeta| = c} \bigl(\zeta - \stab_{\bar{z}_1,z_2}\bigr)^{-1} \mathrm{d}\zeta.
	\end{equation}
	The projector $\Pi_{\bar{z}_1,z_2}$ satisfies, uniformly in $z_1, z_2\in\bulk\cap\mathbb{H}$ with $|z_1 - z_2| \le \delta$, (Claim 6.7 in \cite{R2023bulk})
	\begin{equation} \label{eq:Pi_separation}
		\bigl\lvert \Pi_{\bar z_1, z_2}[\m(\bar z_1)\m(z_2)] \bigr\rvert \sim \bigl\lvert \Pi_{\bar z_1, z_2}^\mathfrak{t}[\vect{1}] \bigr\rvert \sim \vect{1},
	\end{equation}
	where $(\cdot)^\mathfrak{t}$ denotes the transpose of a matrix.
	Finally, the projector $\Pi_{\bar{z}_1,z_2}$ satisfies the norm bound
	\begin{equation} \label{eq:Pi_bound}
		\norm{\Pi_{\bar z_1, z_2}}_{\ell^1\to\ell^\infty} + \norm{\Pi_{\bar z_1, z_2}^\mathfrak{t}}_{\ell^1\to\ell^\infty} \lesssim N^{-1},
	\end{equation}
	and Lipschitz-continuity property, 
	\begin{equation} \label{eq:Pi_continuity}
		\norm{\Pi_{\bar z_1, z_2} - \Pi_{\bar z_1, z_3}}_{\ell^1\to\ell^\infty} + \norm{\Pi_{\bar z_1, z_2}^\mathfrak{t}- \Pi_{\bar z_1, z_3}^\mathfrak{t}}_{\ell^1\to\ell^\infty} \lesssim N^{-1}|z_2-z_3|,
	\end{equation}
	uniformly in $z_1,z_2, z_3 \in \bulk\cap\mathbb{H}$ with $|z_1-z_2|, |z_1-z_3| \le \delta$.
	Moreover, in the special case $z_2 = z_1 = E + \I0 \in \overline{\bulk}$, the eigenprojector $\Pi_{\bar{z}_1,z_2}$ can be computed explicitly (Eq. (5.52) in \cite{Ajanki2019QVE}),
	\begin{equation} \label{eq:Pi_barzz}
		\Pi_{E-\I0, E+\I0}\bigl[\cdot\bigr] = \im\m(E ) \frac{\bigl\langle |\m(E )|^{-2} \im\m(E ), \, \cdot \,\bigr\rangle}{\bigl\lVert|\m(E )|^{-1} \im\m(E )\bigr\rVert_2^2},
	\end{equation}
	where $\m(E) := \m(E+\I0)$.
\end{lemma}
\begin{remark}
	We note that in prior works \cite{Landon2021Wignertype, KN23}, the bound \eqref{eq:far_stability} on the stability operator $\stab_{\bar{z}_1, z_2}$ was obtained only in the perturbative regime $|z_1 - z_2| \le \delta \sim 1$ for $z_1,z_2$ in the bulk of the spectrum. For $z_1,z_2$  in the vicinity of the small local minima of the self-consistent density, but still in the perturbative regime,  $|z_1 - z_2| \le \delta_{*} \sim 1$, the two-body stability operators $\stab_{z_1,z_2}$ and $\stab_{\bar{z}_1, z_2}$ were fully analyzed in \cite{R2023cusp}.
 	
	Here we prove the bound in the complementary long range regime $|z_1-z_2| \gtrsim 1$. In fact, our proof holds uniformly in the spectrum, i.e., for all $z_1,z_2 \in \mathbb{C}\backslash\mathbb{R}$ satisfying $|z_j| \le C$ and $|z_1 - z_2| \gtrsim 1$, thus completing the analysis of the two-body stability operator for Wigner-type matrices. 
\end{remark}

We prove the long-range stability estimate \eqref{eq:far_stability} and the bounds \eqref{eq:Pi_bound}, \eqref{eq:Pi_continuity} of Lemma \ref{lemma:stab_lemma} in Appendix \ref{sec:stab_section}. 
In particular, Lemma \ref{lemma:stab_lemma} implies that for a fixed pair $(z_1,z_2)$ of spectral parameters in the bulk domain $\bulk$, there exists a codimension one subspace of observables $A$ in $\mathbb{C}^{N\times N}$ such that the size of the corresponding deterministic approximation $\M{z_1,A,z_2}$ is smaller than typical (see \eqref{eq:ex_regA} and \eqref{eq:ex_genB} below). We call such observables \textit{regular}, and they play a key role in our analysis.
\begin{Def}[Regular Observables]  \label{def:reg_A}
	Let $(z_1,z_2)$ be an ordered pair of spectral parameters in $\bulk$, and let $\delta$ be the threshold introduced in Lemma \ref{lemma:stab_lemma}.
	In the regime $\min\{|z_1-z_2|, |\bar{z}_1-z_2|\} \le \tfrac{1}{2}\delta$,
	 we say that a deterministic matrix (observable) $A$ is regular with respect to $(z_1,z_2)$  (or $(z_1,z_2)$-regular for short) if and only if
	\begin{equation}
		\Pi_{z_1^-,z_2^+}[\m(z_1^-)\vect{a}^{\mathrm{diag}}\,\m(z_2^+)] = 0, \quad  \vect{a}^\mathrm{diag} := \bigl(A_{jj}\bigr)_{j=1}^N,
	\end{equation}
	where $z_j^\pm := \re z_j \pm \I|\im z_j|$. 
	In the complementary regime $\min\{|z_1-z_2|, |\bar{z}_1-z_2|\} > \tfrac{1}{2}\delta$, we consider all observables in $\mathbb{C}^{N\times N}$ regular.
\end{Def}  
\begin{remark}[Some remarks about regularity] 
	First note that for standard and generalized Wigner matrices, i.e., if the expectation vector $\bm{\mathfrak{a}} = 0$ and the matrix of variances $S$ is stochastic, the concept of regularity reduces to tracelessness $\langle A \rangle = 0$, and does not depend on the spectral parameters $z_1, z_2$.
	
	Second, we remark that 
	the Definition \ref{def:reg_A} implies that if $A$ is $(z_1, z_2)$-regular, then $A^*$ is $(z_2,z_1)$-regular.
	Moreover, the concepts of regularity with respect to $(z_1, z_2)$, $(\bar z_1, z_2)$, $(z_1, \bar z_2)$, and $(\bar z_1, \bar z_2)$ are mutually equivalent.
\end{remark}

In particular, \eqref{eq:M_offdiag}, Assumption \eqref{as:m_bound} and \eqref{eq:stab_gap} of Lemma \ref{lemma:stab_lemma} imply that for all $z_1,z_2 \in \bulk\cap\mathbb{H}$,
\begin{equation} \label{eq:ex_regA}
	\langle |\M{\bar{z}_1,A,z_2}|^2 \rangle^{1/2} \lesssim \langle |A|^2\rangle^{1/2},
\end{equation}
for any observable $A$ regular with respect to $(z_1,z_2)$ in the sense of Definition \ref{def:reg_A}, while for a general observable $B$, we have only a weaker estimate from \eqref{eq:scalar_stab_bounds},
\begin{equation} \label{eq:ex_genB}
	\langle |\M{\bar{z}_1,B,z_2}|^2 \rangle^{1/2} \lesssim  (|\im z_1| + |\im z_2|)^{-1}\langle|B|^2\rangle^{1/2},
\end{equation}
with both bounds potentially saturating when $z_2$  is close to $\bar{z}_1$.
In the sequel, we adhere to the convention that $A$'s denote observables that are regular with respect to the spectral parameters of the adjacent resolvents, and $B$'s denote general observables.

As it was previously observed in \cite{Cipolloni2022Optimal}, \cite{Cipolloni2023EquiPart} in the setting of standard and deformed Wigner matrices, the regularity of the observables impacts not only the size of the deterministic approximation itself, but also the size of the fluctuations of the corresponding resolvent chain. This constitutes the main technical result of the present paper, contained in the following theorem.
\begin{theorem}[One- and Two-Resolvent Local Laws with Regular Observables] \label{th:locallaws}
	Let $H$ be an $N\times N$ random matrix of Wigner-type as in Definition \ref{def:WT}. Consider spectral parameters $z_1, z_2 \in \mathbb{C}\backslash\mathbb{R}$, and the corresponding resolvent $G_j = G(z_j) := (H-z_j)^{-1}$, for $j \in \{1,2\}$. Let $A_1$ be a $(z_1,z_2)$-regular, and $A_2$ be a $(z_2, z_1)$-regular observable in the sense of Definition \ref{def:reg_A}, let $\eta := \min_j|\im z_j|$, then the averaged law
	\begin{equation} \label{eq:2G_av_reg}
		\bigl\lvert \bigl\langle \bigl(G_1A_1G_2 - \M{z_1,A_1,z_2}\bigr)A_2 \bigr\rangle \bigr\rvert \prec \frac{\langle |A_1|^2\rangle^{1/2}\langle |A_2|^2\rangle^{1/2}}{\sqrt{N\eta}},
	\end{equation}
	holds uniformly in observables $A_1, A_2$, spectral parameters $z_1,z_2\in\bulk$ satisfying $\min_j|\im z_j| \ge N^{-1+\varepsilon}$, for some constants $\rho_*,\eta_*,\varepsilon>0$.
	
	Under the same assumptions, the isotropic law
	\begin{equation} \label{eq:2G_iso_reg}
		\bigl\lvert \bigl\langle \vect{x}, \bigl(G_1A_1G_2 - \M{z_1,A_1,z_2}\bigr)\vect{y} \bigr\rangle \bigr\rvert \prec \frac{\langle |A_1|^2\rangle^{1/2}\norm{\vect{x}}_2\norm{\vect{y}}_2}{\sqrt{\eta}},
	\end{equation}
	holds uniformly in $A_1$, deterministic vectors $\vect{x},\vect{y}$ and in $z_1,z_2\in\bulk$ satisfying $\min_j|\im z_j| \ge N^{-1+\varepsilon}$, for some constants $\rho_*,\eta_*,\varepsilon  > 0$.
	
	Furthermore, we obtain the following optimal averaged local law for the resolvent
	\begin{equation} \label{eq:1G_av_reg}
		\bigl\lvert \bigl\langle \bigl(G(z) - \m(z)\bigr)A \bigr\rangle \bigr\rvert \prec \frac{\langle |A|^2\rangle^{1/2}}{N\sqrt{|\im z|}},
	\end{equation}
	uniformly in $(\bar z, z)$-regular observables $A$, and in $z \in \bulk$ satisfying $|\im z| \ge N^{-1+\varepsilon}$.
\end{theorem}

\begin{remark}[General Observables]
	Along the proof we also establish the analogs of \eqref{eq:2G_av_reg} and \eqref{eq:2G_iso_reg} for non-regular observables. More precisely, under the assumptions and notation of Theorem \ref{th:locallaws}, let $\eta_j := |\im z_j|$, then the two-resolvent averaged and isotropic local laws
	\begin{equation}
		\begin{split}
			\bigl\lvert \bigl\langle \bigl(G_1B_1G_2 - \M{z_1,B_1,z_2}\bigr)B_2 \bigr\rangle \bigr\rvert &\prec \frac{\langle |B_1|^2\rangle^{1/2}\langle |B_2|^2\rangle^{1/2}}{\sqrt{N\eta_1\eta_2\eta}},\\
			\bigl\lvert \bigl\langle \vect{x}, \bigl(G_1B_1G_2 - \M{z_1,B_1,z_2}\bigr)\vect{y} \bigr\rangle \bigr\rvert &\prec \frac{\norm{\vect{x}}_2\norm{\vect{y}}_2\langle |B_1|^2\rangle^{1/2}}{\sqrt{\eta_1\eta_2}},
		\end{split}		
	\end{equation}
	hold uniformly in all deterministic matrices $B_1, B_2$, deterministic vectors $\vect{x}, \vect{y}$, and in spectral parameters $z_1,z_2\in\bulk$ satisfying $\min_j|\im z_j| \ge N^{-1+\varepsilon}$ for some constants $\varepsilon > 0$.
\end{remark}

\begin{remark}[Optimality]
	The local laws of Theorem \ref{th:locallaws} are optimal in both $N$ and $\eta$ as long as one uses exclusively the Hilbert-Schmidt norm of the observables $\langle |A_j|^2 \rangle^{1/2}$ to control the error terms. However, if one consider higher Schatten norms, the $N\eta$ dependence can be improved, see \cite{Cipolloni2024out}.
\end{remark}
\begin{remark}[Small Local Minima]
	One can consider spectral parameters $z_1, z_2$ lying in the vicinity of the small local minima of the self-consistent density of states, e.g., spectral edges (as was done in \cite{Cipolloni2023Edge}) and cusps. Since our main objective is to deal with a non-trivial self-energy operator $\mathcal{S}$, we do not pursue the optimal $\rho$ dependence in the present paper, but explain the difficulties in Remark \ref{rem:small_rho}.
\end{remark}

\subsection{Proof of Eigenstate Thermalization}
\begin{proof}[Proof of Theorem \ref{th:ETH}]
	Fix $\varepsilon > 0$, and let $\eta := N^{-1 + \varepsilon}$. Let $B$ be a deterministic observable, and let $j,k$ be two bulk indices, i.e., $\rho(\gamma_j), \rho(\gamma_k) \ge \rho_{\min}$.
	From the Definition \ref{def:reg_A} of observable regularity and the properties of the stability operator in Lemma \ref{lemma:stab_lemma}, we conclude that any deterministic matrix $B$ admits the decomposition (for a detailed derivation, see the proof of the more general Lemma \ref{lemma:B_renorm} with $t = T$ in Section \ref{sec:A_reg_proofs} below)
	\begin{equation} \label{eq:ETH_decomp}
		B = A(z_j, z_k) + b(z_j,z_k)\,I,
	\end{equation}
	where $z_j := \gamma_j + \I\eta$, $z_k := \gamma_k + \I\eta$, and the observable $A$ is $(z_j,z_k)$-regular. Moreover, using the Lipschitz continuity of $\Pi_{\cdot,\cdot}$ from \eqref{eq:Pi_continuity} and the expression for $\Pi_{\gamma_j,\gamma_j}$ in \eqref{eq:Pi_barzz}, we deduce that
	\begin{equation} \label{eq:ETH_continuity}
		b(z_j,z_k) = \bigl(\pi\rho(\gamma_j)\bigr)^{-1}\bigl\langle \im\M{\gamma_j}B \bigr\rangle + 
		\langle |B|^2 \rangle^{1/2} \mathcal{O}(|\gamma_j - \gamma_k| + \eta).
	\end{equation}
	In particular, we obtain
	\begin{equation}
		\langle \vect{u}_j, B\, \vect{u}_k\rangle = \bigl\langle\vect{u}_j, A(z_j,z_k)\, \vect{u}_k\bigr\rangle + \delta_{jk}\bigl(\pi\rho(\gamma_j)\bigr)^{-1}\bigl\langle \im\M{\gamma_j}B \bigr\rangle + 
		\langle |B|^2 \rangle^{1/2}\mathcal{O}(\eta).
	\end{equation}
	Next, using the spectral decomposition of $\im G$ together with the eigenvalue rigidity $|\lambda_i - \gamma_i| \prec N^{-1}$ (see Corollary 1.11 in \cite{Ajanki2016Univ}), we obtain (c.f. Lemma 3.5 in \cite{Cipolloni2023Overlap})
	\begin{equation}
		\bigl\lvert\bigl\langle\vect{u}_j, A(z_j,z_k)\, \vect{u}_k\bigr\rangle\bigr\rvert^2 \prec	N\eta^2\bigl\langle \im G(z_j) A(z_j,z_k) \im G(z_k) \bigl(A(z_j, z_k)\bigr)^* \bigr\rangle \prec N^{-1 +2\varepsilon}\langle |A(z_j,z_k)|^2 \rangle,
	\end{equation}
	where we used the estimate \eqref{eq:ex_regA} and averaged local law \eqref{eq:2G_av_reg} of Theorem \ref{th:locallaws} to conclude the last estimate. Since $\varepsilon > 0$ was arbitrary small, the $N^{2\varepsilon}$ factor can be absorbed into $\prec$ by Definition \ref{def:prec}. A simple estimate $\langle |A(z_j,z_k)|^2\rangle = \mathcal{O}(
	\langle |B|^2 \rangle)$, that follows immediately from Assumption \eqref{as:m_bound}, \eqref{eq:ETH_decomp}, and \eqref{eq:ETH_continuity}, thus concludes the proof of Theorem \ref{th:ETH}.
\end{proof}

\section{Proof of The Local Laws} \label{sec:locallaws_proof}
\subsection{Characteristic Flow for Wigner-type Matrices} \label{sec:char_flow}
We prove Theorem \ref{th:locallaws} using a generalized version of the \textit{characteristic flow} method with vector-valued spectral parameters.
In particular, for a given scalar spectral parameter $z\in \mathbb{C}\backslash\mathbb{R}$, we consider the time evolution in time governed by the system of differential equations (flow)
\begin{equation} \label{eq:z_evol}
	\partial_t \vect{z}_t = - S[\m(\vect{z}_t)] -\frac{1}{2}(\vect{z}_t-\bm{\mathfrak{a}}),  \quad \vect{z}_t : t\in [0,T] \to \mathbb{H}^N\cup(\mathbb{H}^*)^N,
\end{equation}
with the final condition $\vect{z}_T = z\vect{1}$ at some fixed terminal time $T \sim 1$. Here, $\vect{1} := (1,\dots, 1)\in\mathbb{C}^N$, and $\m(\vect{z})$ denotes the solution to the (generalized) vector Dyson equation 
with a vector-valued spectral variable $\vect{z} \in \mathbb{H}^N$,
\begin{equation} \label{eq:VDE}
	-\frac{1}{\m(\vect{z})} = \vect{z} - \bm{\mathfrak{a}} + S[\m(\vect{z})],
\end{equation}
that is  uniquely defined for all $\vect{z} \in \mathbb{H}^N$ under the constraint $\m(\vect{z}) \in \mathbb{H}^N$ by a simple fixed point argument (see Lemma 4.2 in \cite{Ajanki2019QVE}), and can be extended to $\vect{z} \in (\mathbb{H}^*)^N$ by $\m(\overline{\vect{z}}) := \overline{\m(\vect{z})}$. The characteristic flow satisfies the following properties that we prove in Appendix \ref{sec:eta_sec}.
\begin{lemma}[Properties of the Characteristic Flow] \label{lemma:flow_exists}
	For any terminal time $T > 0$ and any $z\in \mathbb{C}\backslash\mathbb{R}$, the flow \eqref{eq:z_evol} admits a unique solution $\vect{z}_t := \vect{z}_t(z)$ to the terminal value problem $\vect{z}_T = z\vect{1}$. Moreover, the solution $\vect{z}_t$ satisfies
	\begin{equation} \label{eta_t_lower}
		\sign \im \vect{z}_t = \sign(\im z) \vect{1}, \quad 
		\text{and} \quad |\im \vect{z}_t| \ge |\im z|\, \vect{1}, \quad t \in [0,T].
	\end{equation}
	Furthermore, along the trajectory $\vect{z}_t$, the solution to the vector Dyson equation \eqref{eq:VDE} satisfies
	\begin{equation} \label{eq:m_evol}
		\partial_t \m(\vect{z}_t) = \frac{1}{2}\m(\vect{z}_t), \quad t \in [0,T].
	\end{equation}
\end{lemma}

Note that for a general data pairs $(\bm{\mathfrak{a}},S)$, the vector $S[\m]$ is typically not proportional to the vectors of ones $\vect{1}$. Therefore, the flow \eqref{eq:z_evol} can produce genuine vector-valued spectral parameters, and one has to consider generalized resolvents $G(X,\vect{z})$ defined for $X=X^* \in \mathbb{C}^{N\times N}$ and $\vect{z} \in \mathbb{H}^N \cup (\mathbb{H}^*)^N$ by
\begin{equation} \label{eq:Gvect}
	G(X,\vect{z}) := \bigl(X - \diag{\vect{z}}\bigr)^{-1}.
\end{equation}

Since the main difficulty of the proof lies in dealing with the non-trivial structure of $\mathscr{S}$, we first present the proof in the complex Hermitian setting under the assumption that the off-diagonal part of the self-energy operator defined in \eqref{eq:superT_def} vanishes identically, that is $\mathscr{T} = 0$. We explain how to lift this constraint in Section \ref{sec:realT} below. 

We run the evolution of the spectral parameters governed by \eqref{eq:z_evol} simultaneously with the evolution of the Wigner-type matrix $H$ along the Ornstein-Uhlenbeck flow,
\begin{equation} \label{eq:OUflow}
	\mathrm{d}H_t = -\frac{1}{2}\bigl(H_t-\diag{\bm{\mathfrak{a}}}\bigr)\mathrm{d}t + \widehat{S}\odot
	\mathrm{d}\Brwn_t, \quad H_0 = H.
\end{equation}
Here, $\widehat{S}$ denotes the entry-wise square root of $S$, i.e., $\widehat{S}_{j k} := \sqrt{S_{jk}}$, $\odot$ denotes the Hadamard product, and $\Brwn_t$ is the standard matrix-valued Brownian motion in the same symmetry class as $H$. 
Note that the first two moments of the entries of $H_t$ are preserved along the flow \eqref{eq:OUflow}. %Consequently, the self-consistent density of states $\rho$ of $H_t$, and hence the concept of the bulk spectrum $\bulk$ defined in \eqref{eq:Dbulk}, remain time-independent.

We then study the evolution of traces of alternating chains $G_{t}B$ and $G_{1,t}B_1G_{2,t}B_2$ of deterministic matrices sandwiched between generalized time-dependent resolvents $G_{t}$ and $G_{j,t}$ that are defined, as in \eqref{eq:Gvect}, by
\begin{equation} \label{eq:Gt}
	G_{t} := G(H_t,\vect{z}_{t}) = \bigl(H_t - \diag{\vect{z}_{t}}\bigr)^{-1}, \quad G_{j,t} := G(H_t,\vect{z}_{j,t}).
\end{equation}
 Following the convention of \eqref{eq:Mz_def}, we denote the deterministic approximations the generalized resolvents $G_t$, $G_{j,t}$, respectively, by 
\begin{equation} \label{eq:Mt_def}
	M_t \equiv \M{\vect{z}_t} := \diag{\m(\vect{z}_t)}, \quad M_{j,t} \equiv \M{\vect{z}_{j,t}} := \diag{\m(\vect{z}_{j,t})}. 
\end{equation} 
Using It\^{o}'s formula together with the definition \eqref{eq:Gt}, and denoting the complex derivative in the direction of matrix element $H_{jk,t}$ by $\partial_{jk}$, we obtain
\begin{equation} \label{eq:GBevol}
	\begin{split}
		\mathrm{d}\langle G_{t}B\rangle =&~ \frac{1}{2}\sum_{j,k}\partial_{jk}\bigl\langle G_{t}B\bigr\rangle\sqrt{S_{jk}}\mathrm{d}\Brwn_{jk,t}
		+ \bigl\langle G_{t} B G_{t}  \diag{\partial_t\vect{z}_{t} +\tfrac{1}{2}\vect{z}_{t}-\tfrac{1}{2}\bm{\mathfrak{a}}+S[\m_{t}]}\bigr\rangle \mathrm{d}t\\ 
		&+	\frac{1}{2} \bigl\langle G_{t}B\bigr\rangle\mathrm{d}t
		+\bigl\langle \bigl(G_{t}-M_t\bigr) \mathscr{S}[ G_{t}BG_{t}]\bigr\rangle\mathrm{d}t.
	\end{split}
\end{equation}
The term $\bigl\langle G_{t} B G_{t}\diag{\partial_t\vect{z}_{t} +\tfrac{1}{2}\vect{z}_{t}-\tfrac{1}{2}\bm{\mathfrak{a}}+S[\m_{t}] } \bigr\rangle$ in the first line on the right-hand side of \eqref{eq:GBevol} vanishes identically by \eqref{eq:z_evol}. Similarly, using \eqref{eq:z_evol}, we obtain
\begin{equation} \label{eq:GBGBevol}
	\begin{split}
		\mathrm{d}\bigl\langle G_{1,t}B_1 G_{2,t}B_2 \bigr\rangle =&~ 
		\frac{1}{2}\sum_{j,k}\partial_{jk}\bigl\langle G_{1,t}B_1 G_{2,t} B_2 \bigr\rangle \sqrt{S_{jk}}\mathrm{d}\Brwn_{jk,t} + 
		%\bigl\langle G_{1,t} B_1 G_{2,t} B_2 	G_{1,t}\,\bigl(\partial_t\vect{z}_{1,t}+\tfrac{1}{2}\vect{z}_{1,t} + \mathscr{S}[\m_{1,t}]\bigr)  \bigr\rangle \\
		%&+ \bigl\langle G_{2,t} B_2 G_{1,t} B_1 G_{2,t}\,\bigl(\partial_t\vect{z}_{2,t}+\tfrac{1}{2}\vect{z}_{2,t}+\mathscr{S}[\m_{2,t}]\bigr)\bigr\rangle\\ %&+	\bigl\langle  (F A_1\other{F}-M_1)\mathscr{S}[M_2]\bigr\rangle  
		%+ \bigl\langle (\other{F}A_2F-M_2)\mathscr{S}[M_1]\bigr\rangle\\		&+ 
		\bigl\langle G_{1,t}B_1 G_{2,t} B_2 \bigr\rangle\mathrm{d}t \\
		&+ \bigl\langle G_{1,t} B_1 G_{2,t}\, \mathscr{S}[G_{2,t} B_2 G_{1,t}]\bigr\rangle\mathrm{d}t +\bigl\langle\mathscr{S}[G_{1,t}-M_{1,t}] G_{1,t} B_1 G_{2,t} B_2 G_{1,t} \bigr\rangle\mathrm{d}t \\
		&+\bigl\langle\mathscr{S}[G_{2,t}-M_{2,t}] G_{2,t} B_2 G_{1,t} B_1 G_{2,t} \bigr\rangle\mathrm{d}t,
	\end{split}
\end{equation}
where we recall $M_{j,t} := \diag{\m(\vect{z}_{j,t})}$ from \eqref{eq:Mt_def}.
The algebraic cancellation in the time-differentials $\mathrm{d}\langle G_{t}B \rangle$ and $\mathrm{d}\langle G_{1,t}B_1G_{2,t}B_2 \rangle$ resulting from the combined effects of the evolutions \eqref{eq:z_evol} and \eqref{eq:OUflow} (the analog of which was first observed in \cite{Cipolloni23nonHerm}, in the setting of non-Hermitian random matrices) is the key insight of the characteristic flow method.

Recall the set $\bulk$ defined in \eqref{eq:Dbulk}, and let $\mathcal{D} \subset \mathbb{C}$ denote the domain
\begin{equation} \label{eq:calD_def}
	\mathcal{D} \equiv \mathcal{D}_{\varepsilon,\rho_*,\eta_*} := \bulk_{\rho_*,\eta_*}\cap \{z\in\mathbb{C} : |\im z| \ge N^{-1+\varepsilon}\}, \quad \varepsilon > 0.
\end{equation}
We now state the propositions containing the three main steps of the proof. 
\begin{prop}[Global Laws] \label{prop:global_laws}
	Let $H$ be a Wigner-type matrix as in Definition \ref{def:WT}, and let $H_t$ satisfy \eqref{eq:OUflow}.
	There exists a terminal time $T\sim 1$ which depends only on the model parameters and the constants $\rho_*$, $\eta_*$, such that the following is true. 
	
	For a pair of spectral parameters $z_1,z_2 \in \bulk$, let $\vect{z}_{j,t}$, $j \in \{1,2\}$ denote the solutions to the flow \eqref{eq:z_evol} satisfying $\vect{z}_{j,T} = z_j\vect{1}$ at the terminal time $T$. Denote $G_{j,t} := G(\vect{z}_{j,t})$ as in \eqref{eq:Gt}, $\eta_{j,t} := |\langle \im \vect{z}_{j,t} \rangle|$ for $j\in \{1,2\}$, and $\eta_t := \min_j \eta_{j,t}$. Then the two-resolvent averaged and isotropic global laws
	\begin{equation} \label{eq:Global2av}
		\bigl\lvert \bigl\langle \bigl(G_{1,0} B_1 G_{2,0} - \M{\vect{z}_{1,0},B_1,\vect{z}_{2,0}}\bigr) B_2\bigr\rangle \bigr\rvert \prec \frac{\norm{B_1}\langle |B_2|^2\rangle^{1/2}}{N},
	\end{equation}
	\begin{equation} \label{eq:Global1iso}
		\bigl\lvert \bigl\langle\vect{x},\bigl(G_{1,0} B_1 G_{2,0} - \M{\vect{z}_{1,0},B_1,\vect{z}_{2,0}}\bigr) \vect{y} \bigr\rangle\bigr\rvert \prec  \frac{\norm{B_1}\norm{\vect{x}}_2\norm{\vect{y}}_2}{\sqrt{N}},
	\end{equation}
	hold uniformly in deterministic matrices $B_1,B_2$, deterministic vectors $\vect{x},  \vect{y}$, and in spectral parameters $z_1,z_2 \in \mathcal{D}_{\varepsilon,\rho_*,\eta_*}$.
	Here, similarly to \eqref{eq:M_def}, the matrix $\M{\vect{z}_{1,t},B_1,\vect{z}_{2,t}}$ defined by
	\begin{equation} \label{eq:M_t_def}
		\M{\vect{z}_{1,t},B_1,\vect{z}_{2,t}} := (1-M_{1,t}M_{2,t}\mathscr{S})^{-1}\bigl[M_{1,t} B_1M_{2,t}\bigr],
	\end{equation}
	denotes the deterministic approximation to the chain $G_{1,t}B_1G_{2,t}$, and we recall $M_{j,t} := \diag{ \m(\vect{z}_{j,t})}$.
	
	Furthermore, the averaged singe-resolvent global law
	\begin{equation}\label{eq:Global1av}
		\bigl\lvert \bigl\langle \bigl(G_{1,0}-M_{1,0}\bigr)B \bigr\rangle \bigr\rvert \prec\frac{\langle|B|^2\rangle^{1/2}}{N},
	\end{equation}
	holds uniformly in $B \in \mathbb{C}^{N\times N}$ and $z_1 \in \mathcal{D}$.
\end{prop}
%We refer to Proposition \ref{prop:global_laws} as the Global Law because, as we shown below, $\eta_0 \sim 1$. 
We prove Proposition \ref{prop:global_laws} in Appendix \ref{sec:global_laws}.   
In view of the inequalities $\langle |B|^2 \rangle^{1/2} \le \norm{B} \le \sqrt{N}\langle |B|^2 \rangle^{1/2}$, the mixed-norm global laws \eqref{eq:Global2av}, \eqref{eq:Global1iso} imply the estimates in terms of pure Hilbert-Schmidt norms,
\begin{equation} \label{eq:global_laws-HS}
	\begin{split}
	\bigl\lvert \bigl\langle \bigl(G_{1,0} B_1 G_{2,0} - \M{\vect{z}_{1,0},B_1,\vect{z}_{2,0}}\bigr) B_2\bigr\rangle \bigr\rvert &\prec \frac{\langle |B_1|^2\rangle^{1/2}\langle |B_2|^2\rangle^{1/2}}{\sqrt{N}},\\	
	\bigl\lvert \bigl\langle\vect{x},\bigl(G_{1,0} B_1 G_{2,0} - \M{\vect{z}_{1,0},B_1,\vect{z}_{2,0}}\bigr) \vect{y} \bigr\rangle\bigr\rvert &\prec  \langle |B_1|^2\rangle^{1/2}\norm{\vect{x}}_2\norm{\vect{y}}_2.
	\end{split}
\end{equation}
Nevertheless, we state \eqref{eq:Global2av} and \eqref{eq:Global1iso} in the mixed-norm sense, because their counterparts time-propagated into the local regime (see Lemma \ref{lemma:2G_weak} below), with the additional $\sqrt{N\eta_t}$ smallness coming from the use of operator norm, are an essential stepping stone in the proof of Theorem \ref{th:locallaws}. We remark that the  global laws of Proposition \ref{prop:global_laws} are formulated for arbitrary $B_1, B_2$ because, as we show below, $\eta_0 \sim 1$ and hence the regularity of the observables does not improve the error bound.

In the second step, we show that global laws \eqref{eq:global_laws-HS} with $B_1$, $B_2$ replaced by a $(z_1,z_2)$-regular $A_1$  and a $(z_2,z_1)$-regular $A_2$, respectively, can be propagated along the combined flow \eqref{eq:z_evol} and \eqref{eq:OUflow} into the local regime $\eta_T \sim N^{-1+\varepsilon}$.
\begin{prop}[Local Laws with Regular Observables along the Flow] \label{prop:flow_local_laws}
	Fix $\varepsilon>0$, then under the assumptions and notation of Proposition \ref{prop:global_laws}, the averaged and isotropic two-resolvent local laws
	\begin{equation} \label{eq:2G_av_t}
		\bigl\lvert \bigl\langle \bigl(G_{1,t} A_1 G_{2,t} - \M{\vect{z}_{1,t},A_1,\vect{z}_{2,t}}\bigr) A_2\bigr\rangle \bigr\rvert \prec \frac{\langle |A_1|^2\rangle^{1/2}\langle |A_2|^2\rangle^{1/2}}{\sqrt{N\eta_t}},
	\end{equation}
	\begin{equation} \label{eq:2G_iso_t}
		\bigl\lvert \bigl\langle\vect{x},\bigl(G_{1,t} A_1 G_{2,t} - \M{\vect{z}_{1,t},A_1,\vect{z}_{2,t}}\bigr) \vect{y} \bigr\rangle\bigr\rvert \prec  \frac{\langle |A_1|^2\rangle^{1/2}\norm{\vect{x}}_2\norm{\vect{y}}_2}{\sqrt{\eta_t}},
	\end{equation}
	hold uniformly in $t\in[0,T]$, in deterministic observables $A_1,A_2$ regular with respect to $(z_1,z_2)$ and $(z_2,z_1)$, respectively, as in Definition \ref{def:reg_A}, in deterministic vectors $\vect{x}, \vect{y}$, and in spectral parameters $z_1,z_2 \in \mathcal{D}_{\varepsilon,\rho_*,\eta_*}$.
	
	Furthermore, the averaged single-resolvent local law
	\begin{equation} \label{eq:1G_av_t}
		\bigl\lvert \bigl\langle \bigl(G_{1,t} - M_{1,t}\bigr) A \bigr\rangle\bigr\rvert \prec  \frac{\langle |A|^2\rangle^{1/2}}{N\sqrt{\eta_{1,t}}},
	\end{equation}
	holds uniformly in $t\in[0,T]$, deterministic observables $A$ regular with respect to $\{\bar z_1,z_1\}$, and in spectral parameter $z_1\in \mathcal{D}_{\varepsilon,\rho_*,\eta_*}$.
\end{prop}
We prove Proposition \ref{prop:flow_local_laws} in Section \ref{sec:flow_laws_proof}.
Observe that $\vect{z}_{j,T} = z_j \vect{1}$ by construction, hence the estimates \eqref{eq:2G_av_t}, \eqref{eq:2G_iso_t} and \eqref{eq:1G_av_t} at time $t=T$ imply the local laws of Theorem \ref{th:locallaws} for the Wigner-type matrix $H_T$, which differs from the initial matrix $H$ by an order one Gaussian component. In the third and final step, we remove this Gaussian component using Green function comparison argument while preserving the local laws \eqref{eq:2G_av_reg}, \eqref{eq:2G_iso_reg} and \eqref{eq:1G_av_reg}.
\begin{prop}[Green function comparison] \label{prop:GFT}
	Let $H$ and $W$ be two $N\times N$ Wigner-type matrices as in Definition \ref{def:WT} with matching moments up to the third order, that is, for all $j,k \in \{1,\dots,N\}$,
	\begin{equation}
		\Expv[H_{jk}^p H_{kj}^{q-p}] = \Expv[W_{jk}^p W_{kj}^{q-p}], \quad p,q \in \{0,1,2, 3\}, \quad p\le q.
	\end{equation}
	Assume that the local laws of  Theorem \ref{th:locallaws} hold for resolvents of the Wigner-type matrix $H$, i.e., with $G(\zeta) := (H-\zeta)^{-1}$. Then Theorem \ref{th:locallaws} also holds for the Wigner-type matrix $W$, i.e., with $G(\zeta) := (W-\zeta)^{-1}$.
\end{prop}
We defer the proof of Proposition \ref{prop:GFT} to Section \ref{sec:GFT}. 
Armed with Propositions \ref{prop:global_laws}--\ref{prop:GFT}, we are ready to prove Theorem \ref{th:locallaws}.
\begin{proof}[Proof of Theorem \ref{th:locallaws}]
	First, given a Wigner-type matrix $H$, using the complex moment-matching lemma (Lemma A.2 in \cite{Cipolloni2023Edge}), we construct the initial condition matrix $\widehat{H}$ of Wigner-type, such that the first three moments of the entries $\widehat{H}_{jk,T}$ match those of $H_{jk}$. Here, $\widehat{H}_T$ denotes the result of running the Ornstein-Uhlenbeck flow \eqref{eq:OUflow} with the initial condition $\widehat{H}_0 = \widehat{H}$ up to the terminal time $T$ given by Proposition \ref{prop:global_laws}.
	
	Next, using Propositions \ref{prop:global_laws} and \ref{prop:flow_local_laws}, we show that the statement of Theorem \ref{th:locallaws} holds for the Wigner-type matrix $\widehat{H}_T$. Finally, we remove the Gaussian component added to $\widehat{H}_T$ by the Ornstein-Uhlenbeck flow using Proposition \ref{prop:GFT}. Therefore, we conclude the proof of Theorem \ref{th:locallaws} for the desired Wigner-type ensemble $H$.
\end{proof}
The remainder of this section is dedicated to the proof of Proposition \ref{prop:flow_local_laws}.
\subsection{Local Laws along the Flow. Proof of Proposition \ref{prop:flow_local_laws}} \label{sec:flow_laws_proof}
Fix $T\sim 1$ to be the terminal time provided by Proposition \ref{prop:global_laws}.
Recall that for a spectral parameter $z_j \in \bulk$, the vector  $\vect{z}_{j,t}$ denotes the solution to the characteristic flow equation \eqref{eq:z_evol} with the terminal condition $\vect{z}_{j,T} = z_j\vect{1}$, and let $\eta_{j,t} := |\langle \im \vect{z}_t \rangle|$. Finally, we adhere to the notation $G_{j,t} := (H_t -\diag{\vect{z}_{j,t}})^{-1}$ for the generalized resolvents as in \eqref{eq:Gt}, and $M_{j,t} := M_{j,t}$ for the self-consistent resolvents as in \eqref{eq:Mt_def}.

For all $0\le t \le T$ and all $z_1,z_2\in\mathcal{D}$, defined in \eqref{eq:calD_def}, we consider the following control quantities,
\begin{equation} \label{eq:Phi_def}
	\begin{split}
		\Phi_{1}(t) \equiv \Phi_{1}(z_1,t, A) &:=
		\frac{N \sqrt{\eta_{1,t}}}{\langle|A|^2 \rangle^{1/2}} \bigl\lvert \bigl\langle (G_{1,t}-M_{1,t}) A \bigr\rangle \bigr\rvert,\\
		\Phi_{2}^{\mathrm{hs}}(t) \equiv \Phi_{2}^{\mathrm{hs}}(z_1,z_2,A_1,A_2,t) &:= 		\frac{\sqrt{N\eta_{t}}}{\langle|A_1|^2 \rangle^{1/2}\langle|A_2|^2 \rangle^{1/2}} \bigl\lvert \bigl\langle (G_{1,t}A_1G_{2,t}-\M{\vect{z}_{1,t},A_1,\vect{z}_{2,t}}) A_2 \bigr\rangle \bigr\rvert,\\
		%
		%\Phi_{2}^{\mathrm{op}}(t) \equiv \Phi_{2}^{\mathrm{op}}(z_1,z_2,A_1,A_2,t) &:= 		\frac{\sqrt{N\eta_{t}}}{\langle|A_1|^2 \rangle^{1/2}\norm{A_2}} \bigl\lvert \bigl\langle (G_{1,t}A_1G_{2,t}-\M{\vect{z}_{1,t},A_1,\vect{z}_{2,t}}) A_2 \bigr\rangle \bigr\rvert,%\\
		%
		%\Phi_{2}^{\n}(t) \equiv \Phi_{2}^{\n}(z_1,z_2,A_1,A_2,t) &:=
		%\frac{(N\eta_{t})^{\alpha(\n)}}{\langle|A_1|^2 \rangle^{1/2}\norm{A_2}_{\n}} \bigl\lvert \bigl\langle (G_{1,t}A_1G_{2,t}-\M{\vect{z}_{1,t},A_1,\vect{z}_{2,t}}) A_2 \bigr\rangle \bigr\rvert,
	\end{split}	
\end{equation}
and two auxiliary control quantities $\Phi_{2}^{\mathrm{op}}(t)$ and $\Phi_{(1,1)}(t)$, defined as
\begin{equation} \label{eq:Phi(1,1)_def}
	\begin{split}
		\Phi_{2}^{\mathrm{op}}(t) \equiv \Phi_{2}^{\mathrm{op}}(z_1,z_2,A_1,A_2,t) &:= 		\frac{\sqrt{N\eta_{t}}}{\langle|A_1|^2 \rangle^{1/2}\norm{A_2}} \bigl\lvert \bigl\langle (G_{1,t}A_1G_{2,t}-\M{\vect{z}_{1,t},A_1,\vect{z}_{2,t}}) A_2 \bigr\rangle \bigr\rvert,\\
		\Phi_{(1,1)}(t) \equiv \Phi_{(1,1)}(z_1,z_2,A_1,B,t) &:= %\sup\limits_{z_1,z_2\in\mathcal{D}}	\sup\limits_{A\in \chain^{1,\mathrm{av}}_{z_1,z_2}}\sup_{B\in\mathbb{C}^{N\times N}} 
		\frac{\sqrt{N\eta_t} \sqrt{\eta_t}}{\langle|A_1|^2 \rangle^{1/2}\norm{B}} \bigl\lvert \bigl\langle (G_{1,t}A_1G_{2,t}-\M{\vect{z}_{1,t},A_1,\vect{z}_{2,t}}) B \bigr\rangle \bigr\rvert,
	\end{split}
\end{equation}
where $\eta_t := \min\{\eta_{1,t}, \eta_{2,t}\}$. Here, the superscripts $\mathrm{hs}$ and $\mathrm{op}$ signal the use of the Hilbert-Schmidt or the operator norm of the observable $A_2$, respectively, and the subscript $(1,1)$ is to denote the presence of one regular and one arbitrary observable in \eqref{eq:Phi(1,1)_def}. The prefactors in \eqref{eq:Phi_def}--\eqref{eq:Phi(1,1)_def} are the reciprocals of the target bounds for the fluctuations of the traces of the respective resolvent chains. In particular, our goal is to show that $\Phi$'s are stochastically dominated by $1$\footnote{
	The bounds $\Phi_{1}, \Phi_2^\mathrm{hs} \prec 1$ are optimal, but the true size of $\Phi_2^\mathrm{op}$ and $\Phi_{(1,1)}$ is $(N\eta_t)^{-1/2}$ (c.f. Theorem 2.5 in \cite{Cipolloni2022Optimal}), as signaled by \eqref{eq:Weak2av} below, since the general heuristic is that each regular observable improves the error bound by a factor of $\sqrt{\eta_t}$. The bound on $\Phi_2^\mathrm{op}$ can be amended by considering resolvents chains of length up to $4$. The optimal bound on $\Phi_{(1,1)}$ can be obtained by using the stochastic Gronwall estimate (see Lemma \ref{lemma:Gronwall} below).  However, we do not pursue the sharper bounds on $\Phi_2^\mathrm{op}, \Phi_{(1,1)}$, since they are not necessary for the proof of Proposition \ref{prop:flow_local_laws}.}.

We only introduce control quantities tailored for proving the averaged local law \eqref{eq:2G_av_t}. Indeed this is sufficient as the corresponding isotropic local law can be deduced from the averaged laws since the error terms are controlled by the Hilbert-Schmidt norm of the observables. This is formulated in the following lemma, that we prove in Section \ref{sec:A_reg_proofs}.
\begin{lemma}[Isotropic Lemma] \label{lemma:isotropic}
	Under the assumptions and notation of Proposition \ref{prop:global_laws}, for any $z_1,z_2\in\mathcal{D}$ and any pair of deterministic vectors $\vect{x},\vect{y}\in \mathbb{C}^N$, there exists a matrix $A_2 := A_2(z_1,z_2,\vect{x},\vect{y})$ regular with respect to $(z_2,z_1)$ in the sense of Definition \ref{def:reg_A}, and a complex number $a:= a(z_1,z_2,\vect{x},\vect{y})$, such that for all $0\le t \le T$,
	\begin{equation}
		\bigl\langle\vect{x},\bigl(G_{1,t} A_1 G_{2,t} - M_{[1,2],t}\bigr) \vect{y} \bigr\rangle = \sqrt{N} \bigl\langle\bigl(G_{1,t} A_1 G_{2,t} - M_{[1,2],t}\bigr) A_2 \bigr\rangle + a \, \bigl\langle G_{1,t} A_1 G_{2,t} - M_{[1,2],t} \bigr\rangle,
	\end{equation}
	where $M_{[1,2],t} := \M{\vect{z}_{1,t},A_1,\vect{z}_{2,t}}$. Moreover, for all  $z_1,z_2\in \mathcal{D}$, the matrix $A_2$ and the number $a$ satisfy
	\begin{equation}
		\langle |A_2|^2 \rangle^{1/2}+|a| \lesssim \norm{\vect{x}}_2\norm{\vect{y}}_2.
	\end{equation}
\end{lemma}
In particular Lemma \ref{lemma:isotropic} implies the uniform bound 
\begin{equation} \label{eq:1iso_identity}
	\frac{\sqrt{\eta_t}\bigl\lvert\bigl\langle\vect{x},\bigl(G_{1,t} A_1 G_{2,t} - M_{[1,2],t}\bigr) \vect{y} \bigr\rangle \bigr\rvert}{\langle |A_1|^2 \rangle^{1/2}\norm{\vect{x}}_2\norm{\vect{y}}_2}  \lesssim \Phi_{2}^{\mathrm{hs}}(z_1,z_2,A_1,A_2,t) + \frac{\Phi_{(1,1)}(z_1,z_2,A_1,I,t)}{\sqrt{N\eta_t}},
\end{equation}
where the matrix $A_2 := A_2(z_1,z_2,\vect{x},\vect{y})$ is given in Lemma \ref{lemma:isotropic}.
Therefore, to establish Proposition \ref{prop:flow_local_laws}, it suffices to show that
\begin{equation} \label{eq:Phi_goal}
	\Phi_{1} (t)\prec 1,\quad  \Phi_{2}^{\mathrm{hs}}(t) \prec 1,\quad  \Phi_{(1,1)}(t) \prec 1,
\end{equation}
uniformly in time $0\le t\le T$, in $z_1,z_2\in\mathcal{D}$, and \textit{uniformly in observables}, that is, in deterministic matrices $A$ regular with respect to $(\bar z_1,z_1)$, in $(z_1,z_2)$-regular $A_1$, $(z_2,z_1)$-regular $A_2$, and in deterministic matrices $B\in\mathbb{C}^{N\times N}$, respectively (in the sequel we write "uniformly in observables" implying the precise notion described above).

To this end, we first assume that the estimates
\begin{equation} \label{eq:phi_assume}
		\Phi_{1} (t)\prec \phi_{1},\quad  \Phi_{2}^{\mathrm{op}}(t) \prec \phi_{2}^{\mathrm{op}},\quad \Phi_{2}^{\mathrm{hs}}(t) \prec \phi_{2}^{\mathrm{hs}},\quad  \Phi_{(1,1)}(t) \prec \phi_{(1,1)},
\end{equation}
hold uniformly in $t\in [0,T]$, in observables, and in $z_1,z_2\in\mathcal{D}$, for some deterministic $t$-independent control parameters $\phi_1, \phi_2^{\mathrm{op}},\phi_2^{\mathrm{hs}}, \phi_{(1,1)}$. Given this input, we show that the quantities $\Phi_{\dots}(t)$ satisfy an improved system of inequalities in terms of the parameters $\phi_{\dots}$.
\begin{prop}[Master Inequalities] \label{prop:masters}
	Assume that the estimates  \eqref{eq:phi_assume} hold uniformly in $t\in [0,T]$, in observables, and in $z_1,z_2 \in \mathcal{D}$. Then the following upper bounds
	\begin{equation} \label{eq:Master1av}
		\Phi_{1}(t) \prec 1+\frac{\sqrt{\phi_{2}^{\mathrm{hs}}}}{(N\eta_{1,t})^{1/4}}+ \phi_{(1,1)},
	\end{equation}
	\begin{equation} \label{eq:Master2av_HS}
		\Phi_{2}^{\mathrm{hs}}(t) \prec 1+ \frac{\phi_{2}^{\mathrm{hs}}}{\sqrt{N\eta_t}} + \phi_{2}^{\mathrm{op}} + \frac{\phi_{(1,1)}^2}{\sqrt{N\eta_t}},
	\end{equation}
	\begin{equation} \label{eq:Master2av}
		\Phi_{2}^{\mathrm{op}}(t) \prec 1+ \frac{\phi_{2}^{\mathrm{hs}}}{\sqrt{N\eta_t}} + \frac{\phi_{2}^{\mathrm{op}}}{\sqrt{N\eta_t}} + \frac{\phi_{(1,1)}^2}{\sqrt{N\eta_t}},
	\end{equation}
	\begin{equation} \label{eq:Master(1,1)av}
		\Phi_{(1,1)}(t) \prec 1+ \frac{\sqrt{\phi_{2}^{\mathrm{hs}}}}{(N\eta_t)^{1/4}} + \frac{\phi_1}{\sqrt{N\eta_t}}  + \frac{\phi_{(1,1)}}{\sqrt{N\eta_t} },
	\end{equation}
	hold uniformly in $t\in [0,T]$, in observables, and $z_1,z_2\in\mathcal{D}$.
\end{prop}
We defer the proof of Proposition \ref{prop:masters} to Section \ref{sec:master_proofs}.
To obtain the improved bounds on $\Phi_{\dots}(t)$ using the master inequalities \eqref{eq:Master1av}--\eqref{eq:Master(1,1)av}, we apply the following iteration lemma.
\begin{lemma}[Iteration (Lemma 4.5 in \cite{Cipolloni2023Edge})] \label{lemma:iter}
	Let $X \equiv X_N(\vect{u})$ be an $N$-dependent random variable that also depends on a parameter $\vect{u} \in \mathcal{U} \equiv \mathcal{U}_N$. Suppose that $X \prec N^D$ uniformly in $\vect{u} \in \mathcal{U}$, and that for any $x>0$ the fact that $X \prec x$ uniformly in $\vect{u}\in\mathcal{U}$ implies 
	\begin{equation} \label{eq:iter}
		X \prec a + \frac{x}{b} + x^{1-\alpha}d^\alpha,
	\end{equation} 
	uniformly in $\vect{u}\in\mathcal{U}$ for some potentially $N$ and $\vect{u}$-dependent quantities $a,b,d$ satisfying $a,d>0$, $b \ge N^\delta$, and constants $D,\alpha,\delta>0$ independent of $N$ and $\vect{u}$. Then, uniformly in $\vect{u} \in \mathcal{U}$,
	\begin{equation} \label{eq:iter_result}
		X \prec a + d.
	\end{equation}
\end{lemma}

\begin{proof}[Proof of Proposition \ref{prop:flow_local_laws}]
Recall that the goal is to establish \eqref{eq:Phi_goal}.
Master inequality \eqref{eq:Master1av} and the lower bound $N\eta_t \ge N^{\varepsilon}$ by \eqref{eta_t_lower} imply that the assumed bounds \eqref{eq:phi_assume} also hold with the time-independent parameter $\phi_1$ replaced by $1 + \sqrt{\phi_2^\mathrm{hs}}(N\eta_T)^{-1/4} + \phi_{(1,1)}$. Substituting the new $\phi_1$ into \eqref{eq:Master(1,1)av} and applying Lemma \ref{lemma:iter} yields
\begin{equation} \label{eq:Master(1,1)_reduced}
	\Phi_{(1,1)}(t) \prec 1+ \frac{\sqrt{\phi_{2}^{\mathrm{hs}}}}{(N\eta_T)^{1/4}} \quad \text{uniformly in }  t\in[0,T], \text{ observables, and } z_1,z_2\in\mathcal{D}.
\end{equation}
Therefore, we can choose the new $t$-independent control parameter $\phi_{(1,1)} := 1+ \sqrt{\phi_{2}^{\mathrm{hs}}}(N\eta_T)^{-1/4}$. Next, we substitute the new $\phi_{(1,1)}$ into \eqref{eq:Master2av_HS} and apply the iteration Lemma \ref{lemma:iter}, to obtain 
\begin{equation} \label{eq:Master2_HS_reduced}
	\Phi_{2}^{\mathrm{hs}}(t) \prec 1+  \phi_{2}^{\mathrm{op}}, \quad \text{ uniformly in } t\in[0,T], \text{ observables, and } z_1,z_2\in\mathcal{D}.
\end{equation}
Hence, we can plug $\phi_2^\mathrm{hs} := 1 + \phi_2^\mathrm{op}$ and, using \eqref{eq:Master(1,1)_reduced}, $\phi_{(1,1)} := 1+ \sqrt{\phi_{2}^{\mathrm{op}}}(N\eta_T)^{-3/4}$ into \eqref{eq:Master2av} to deduce via Lemma \ref{lemma:iter} that $\Phi_{2}^{\mathrm{op}}(t) \prec 1$. Therefore, it follows from \eqref{eq:Master(1,1)_reduced} and \eqref{eq:Master2_HS_reduced} that
\begin{equation} \label{eq:Phi_final}
	\Phi_{(1,1)}(t) \prec 1, \quad \Phi_{2}^{\mathrm{hs}}(t) \prec 1 \quad \text{ uniformly in } 0\le t \le T.
\end{equation}
Finally, it follows from \eqref{eq:Master1av}, \eqref{eq:Phi_final} that $\Phi_{1}(t) \prec 1$.
This concludes the proof of Proposition \ref{prop:flow_local_laws}.
\end{proof}

\section{Master Inequalities. Proof of Proposition \ref{prop:masters}} \label{sec:master_proofs}
In this section, we present the proof the master inequalities \eqref{eq:Master1av}--\eqref{eq:Master(1,1)av} that constitute the main technical ingredient in the proof of the local laws of Theorem \ref{th:locallaws}.

\subsection{Proof Strategy}  \label{sec:proof_strat}
First, we present a brief overview of our proof method and explain the importance of the auxiliary control quantities $\Phi_{(1,1)}$, $\Phi_{2}^{\mathrm{op}}$ that we introduced in \eqref{eq:Phi(1,1)_def}. 

The starting point in the proof of the master inequalities is the evolution equations \eqref{eq:GBevol} and \eqref{eq:GBGBevol} that we use to estimate the change of the control quantities $\Phi_{\dots}(t)$ along the combination of the characteristic flow \eqref{eq:z_evol} and the Ornstein-Uhlenbeck flow \eqref{eq:OUflow}, and exploit the algebraic cancellation resulting from their simultaneous effect, as discussed in Section \ref{sec:char_flow}. We proceed in a two steps.

The first step in our approach is to prove local laws for one and two-resolvents with general observables (see Lemmas \ref{lemma:1Flaw} and \ref{lemma:2G_weak} below). Without any gain from regularity, these results can be obtained in their optimal form using a combination of the characteristic flow and a simple stopping time argument. Nevertheless, they highlight an important distinction between the error terms controlled in terms of Hilbert-Schmidt and the operator norms, respectively. In particular, the denominator on the right-hand side of \eqref{eq:Weak1iso} carries and additional small $\sqrt{N\eta_t}$ factor compared to \eqref{eq:Weak1iso_HS}. This extra $\sqrt{N\eta_t}$ is a crucial ingredient in the following step two. 

In the second step, we use the local laws with general observables obtained above as an input to close the system of master inequalities for resolvent chains with regular observables. Recall that the non-trivial structure of (diagonal component of) the self-energy operator $\mathscr{S}$ results in two key difficulties: 
\begin{itemize}
	\item[(1)] First, when computing the action of $\mathscr{S}$ on two resolvent chains, e.g., 
	\begin{equation} \label{ex:2G_quad}
		\bigl\langle G_{1,t} A_1 G_{2,t} \mathscr{S}[G_{2,t}A_2G_{1,t}] \bigr\rangle = \frac{1}{N}\sum_{j,k} S_{jk} (G_{1,t} A_1 G_{2,t})_{jj} (G_{2,t}A_2G_{1,t})_{kk},
	\end{equation}
	one can only benefit from fluctuation averaging in a single summation index (either $j$ or $k$), while the remaining chain has to be estimated in the isotropic sense; 
	
	\item[(2)] Secondly, the rows of the matrix of variances $S$ act as new diagonal observables, that is
	\begin{equation}
		\sum_{j,k} S_{jk} (G_{1,t} A_1 G_{2,t})_{jj} (G_{2,t}A_2G_{1,t})_{kk} = \frac{1}{N}\sum_{k} \bigl\langle A_1 G_{2,t} S^{(k)} G_{1,t} \bigr\rangle(G_{1,t} A_1 G_{2,t})_{kk},
	\end{equation}
	where $S^{(k)}_{ja} := NS_{kj}\delta_{ja}$. Contrary to the standard Wigner case $S_{jk} = N^{-1}$, yielding $S^{(k)}=I$ for all $k$, the product $G_{2,t} S^{(k)} G_{1,t}$ can not be linearized using the resolvent identity, and hence the length of the averaged chain cannot be reduced algebraically.
\end{itemize}
In particular, ignoring the non-regularity of $S^{(k)}$ for now, naively using \eqref{eq:1iso_identity} to bound the isotropic factors in \eqref{ex:2G_quad}, and integrating the evolution equation \eqref{eq:GBGBevol} in time, yields the estimate
\begin{equation} \label{ex:Phi_bad1}
	\frac{\Phi_2^{\mathrm{hs}}(t)}{\sqrt{N\eta_t}} \prec 1 + \int_{0}^t\biggl(\frac{(\phi_{2}^{\mathrm{hs}})^2}{\sqrt{N}\eta_s} + \dots \biggr)\mathrm{d}s \prec 1 + \frac{(\phi_{2}^{\mathrm{hs}})^2}{\sqrt{N}} + \bigl(\text{positive terms}\bigr),
\end{equation}
where we used the integration rule $\int_0^t \eta_s^{-1}\mathrm{d}s \lesssim 1$ from \eqref{eq:int_rules} below. This estimate cannot be iterated using Lemma \ref{lemma:iter} since the right-hand side is quadratic in $\phi_{2}^{\mathrm{hs}}$.
However, if we use the suboptimal bound \eqref{eq:Weak1iso_HS} that ignores the regularity of $A_2$ and was proved in the first step, the estimate \eqref{ex:Phi_bad1} will then be improved to
\begin{equation}
	\frac{\Phi_2^{\mathrm{hs}}(t)}{\sqrt{N\eta_t}} \prec \frac{\phi_{2}^{\mathrm{hs}}}{\sqrt{N\eta_t}} + \bigl(\text{positive terms}\bigr),
\end{equation}
which is almost suitable for iteration, up to a missing small $(N\eta_t)^{-c}$ factor on the right-hand side. The key to recovering this factor is the following observation:  
the trivial bound $\langle |S^{(k)}|^2\rangle^{1/2} \lesssim \lVert S^{(k)}\rVert \lesssim 1$ is in fact sharp (see \eqref{eq:S_norms} below for proof). Therefore, we can control the averaged chains in \eqref{ex:2G_quad} using the auxiliary quantity $\Phi_{2}^{\mathrm{op}}$ with no loss. In turn, when we derive the master inequality for $\Phi_{2}^{\mathrm{op}}$, we can use the isotropic bound \eqref{eq:Weak1iso} to benefit from the additional $(N\eta_t)^{-1/2}$ factor in the denominator. Following this strategy, we deduce the system of inequalities,
\begin{equation}
	\frac{\Phi_2^{\mathrm{hs}}(t)}{\sqrt{N\eta_t}} \prec \frac{\phi_{2}^{\mathrm{op}}}{\sqrt{N\eta_t}} + \bigl(\text{positive terms}\bigr), \quad 
	\frac{\Phi_2^{\mathrm{op}}(t)}{\sqrt{N\eta_t}} \prec \frac{\phi_{2}^{\mathrm{op}}}{N\eta_t} + \bigl(\text{positive terms}\bigr),
\end{equation}
which can now be solved using the iteration from Lemma \ref{lemma:iter}.

Finally, to account for the non-regularity of the observables $S^{(k)}$, we decompose them according to the projector $\Pi$ from Definition \ref{def:reg_A}, and control the irregular part using the auxiliary quantity $\Phi_{(1,1)}$.

\begin{remark} [Small $\rho$ Regime] \label{rem:small_rho}
	The complete analysis of the stability operator $\stab_{z_1,z_2}$ near the small local minima of $\rho$ was carried out in \cite{R2023cusp}. Contrary to the standard Wigner setting, the destabilizing eigenvectors of $\stab_{z_1,z_2}$ and $\stab_{\bar{z}_1,z_2}$ are different, which introduces an additional order $\eta^{-1}\rho^2$ error in the analysis of local laws with regular observables. To offset this error, one has to carefully track the \emph{effective} distance between the spectral parameter and the self-consistent spectrum, as the trajectories of the characteristic flow become almost horizontal when $\rho$ is small. This becomes much more intricate for Wigner-type matrices, since the spectral parameters are vector-valued. However, a straightforward computation reveals that this effective distance behaves quadratically in time near a regular edge and as $(T-t)^{3/2}$ near a sharp cusp, as compared to the linear behavior of the imaginary part in the bulk (see \eqref{eq:eta_asymp} below).
\end{remark}

\begin{remark} [The Fullness Condition]
	Our method presented in this paper works under the uniform primitivity condition, the lower bound in \eqref{as:S_flat},
	in particular we can have  large zero blocks in $H$.  We briefly mention an alternative approach
	under the more restrictive \emph{fullness} condition requiring a lower bound  $\langle X\mathcal{S}[X] \rangle \ge c \langle |X|^2 \rangle$ for all $X \in \mathbb{C}^{N\times N}$.
	
	Local laws for such mean-field random matrices can also be studied by using the Ornstein-Uhlenbeck process driven by the standard matrix-valued
	Brownian motion, i.e. replacing all entries of $\widehat{S}$ by $1/\sqrt{N}$ in~\eqref{eq:OUflow}.  
	This replaces the self-energy operator $\mathcal{S}$ of the original ensemble by the averaged trace $\langle \cdot \rangle$ in all evolution equations. 
	The self-consistent density of states changes along this new flow, but this minor complication can be handled easily.  
	More importantly, the characteristic flow conjugate to the standard Ornstein-Uhlenbeck process preserves the scalar nature of spectral parameters, meaning that the resolvents along the flow commute and resolvent identities can be used to reduce the length of chain as in the standard Wigner case. 
	This greatly simplifies the resulting hierarchy of master inequalities, provided sufficient knowledge about the two-body stability operator of the Dyson equation corresponding to the random matrix ensemble $H$. 
	
	For Wigner-type matrices, for example, the necessary properties of $\stab_{z_1,z_2}$ have been established in the bulk.
	As another example, since for generalized Wigner matrices $M(z) = m_{\mathrm{sc}}(z)$ and the only destabilizing eigenvector of $1 - M(z)^{2}S$ is the vector of ones $\vect{1}$, multi-resolvent local laws can be obtained uniformly in the spectrum using the same hierarchy as in \cite{Cipolloni2023Edge} assuming $S_{jk} \sim N^{-1}$.

	However, we stress that this simplification is not available in the more general setup of uniformly primitive matrix of variances $S$. 
\end{remark}

\subsection{Inputs for the Proof of Master Inequalities}
Before proceeding to the proof of Proposition \ref{prop:masters}, we collect several necessary inputs, the proof of which is deferred to further sections.
One key input used to estimate the occurring $(G_t - M_t)$ quantities is the following local laws for a single generalized resolvent.
\begin{lemma}[Local Laws for one Generalized Resolvent] \label{lemma:1Flaw}
	Let $H$ be a random matrix of Wigner-type as in Definition \ref{def:WT}, then %. Fix constants $\rho_*,\eta_*,\varepsilon>0$, then there exists an order-one threshold $T_*\sim 1$ such that for any fixed terminal time $T\in[0,T_*]$, 
	the following averaged and isotropic local laws
	\begin{equation} \label{eq:1G_laws}
		\bigl\lvert \bigl \langle \bigl(G(H_t,\vect{z}_t) - M_t\bigr)B\bigr \rangle \bigr\rvert  \prec \frac{
		\norm{B}}{N\eta_t}, \quad 
		\bigl\lvert \bigl \langle \vect{x}, \bigl(G(H_t,\vect{z}_t) - M_t\bigr)\vect{y}\bigr \rangle \bigr\rvert\prec \frac{\norm{\vect{x}}_2\norm{\vect{y}}_2}{\sqrt{N\eta_t}},
	\end{equation}
	hold uniformly in deterministic matrices $B$, deterministic vectors $\vect{x}, \vect{y}$, in time $0\le t\le T$, and in spectral parameters $z \in \mathcal{D}$. %$ satisfying $\rho(\re z) \ge \tfrac{1}{2}\rho_*$, $N^{-\varepsilon/2}\le |\im z| \le \eta_*$. 
	Here, for a fixed $z\in\mathbb{C}\backslash\mathbb{R}$, the vector $\vect{z}_t$ is a solution to \eqref{eq:z_evol} with the terminal condition $\vect{z}_T = z\,\vect{1}$, $\eta_t := |\langle \im \vect{z}_t \rangle|$. %, $M := \diag{\m}$,  $\m$ denotes the solution to the Dyson equation \eqref{eq:VDE}, and the generalized resolvent $G(H,\vect{z}_t)$ is defined in \eqref{eq:Gt}.
\end{lemma}
The estimates in \eqref{eq:1G_laws} generalize the usual local laws for the resolvent of a Wigner-type matrix at a scalar-valued spectral parameter $z$ (Theorem 2.5 in \cite{Erdos2018CuspUF}). However, since $\vect{z}_t$ is not proportional to the vector of ones $\vect{1}$, they require a separate proof that we present in Section \ref{seq:prior_laws}.

Next, we state the local laws for two generalized resolvents with arbitrary observables. These statements are much easier to prove than their counterparts with regular observables in Proposition \ref{prop:flow_local_laws}, so we defer their proof to Section \ref{seq:prior_laws}.
\begin{lemma}[Two-Resolvent Local Laws for General Observables] \label{lemma:2G_weak}
	Fix $\rho_*,\eta_*,\varepsilon>0$, then under the assumptions and notation of Proposition \ref{prop:global_laws}, the local laws
%	\begin{equation} \label{eq:Weak2av_HS}
%		\bigl \langle \bigl(G_{1,t}B_1G_{2,t} - \M{\vect{z}_{1,t},B_1,\vect{z}_{2,t}}\bigr)B_2\bigr \rangle \prec \frac{\langle |B_1|^2\rangle^{1/2}\langle |B_2|^2\rangle^{1/2}}{\sqrt{N\eta_{1,t}\eta_{2,t} \eta_t}},
%	\end{equation}
	\begin{equation} \label{eq:Weak2av}
		\bigl\lvert\bigl \langle \bigl(G_{1,t}B_1G_{2,t} - \M{\vect{z}_{1,t},B_1,\vect{z}_{2,t}}\bigr)B_2\bigr \rangle\bigr\rvert \prec \frac{\langle |B_1|^2\rangle^{1/2}\norm{B_2}}{N\eta_{1,t}\eta_{2,t}},
	\end{equation}
	\begin{equation} \label{eq:Weak1iso_HS}
		\bigl\lvert\bigl \langle \vect{x},\bigl(G_{1,t}B_1 G_{2,t} - \M{\vect{z}_{1,t},B_1,\vect{z}_{2,t}}\bigr)\vect{y}\bigr \rangle \bigr\rvert \prec \frac{\langle |B_1|^2\rangle^{1/2}\norm{\vect{x}}_2\norm{\vect{y}}_2}{\sqrt{\eta_{1,t}\eta_{2,t}}},
	\end{equation}
	\begin{equation} \label{eq:Weak1iso}
		\bigl\lvert\bigl \langle \vect{x},\bigl(G_{1,t}B_1 G_{2,t} - \M{\vect{z}_{1,t},B_1,\vect{z}_{2,t}}\bigr)\vect{y}\bigr \rangle \bigr\rvert \prec \frac{\norm{B_1}\norm{\vect{x}}_2\norm{\vect{y}}_2}{\sqrt{N\eta_{1,t}\eta_{2,t} \eta_t}},
	\end{equation}
	holds uniformly in $t\in [0,T]$, deterministic matrices $B_1,B_2$, deterministic vectors $\vect{x}, \vect{y}$, 
	and in spectral parameters $z_1,z_2 \in \mathcal{D}$, defined in \eqref{eq:calD_def}.
\end{lemma}

To shorten the presentation, in the present paper we focus on controlling resolvent chains of length one and two, as it is sufficient for establishing our main result, Theorem \ref{th:ETH}. However, the time evolution equation \eqref{eq:GBGBevol} and, eventually, the quadratic variation of the martingale term therein (c.f. \eqref{eq:2av_mart}) contain chains involving three and four generalized resolvent. To estimate the contribution of these terms, we use the following \textit{reduction inequalities}, the proof of which we defer to Section \ref{sec:reds_proof}.
\begin{lemma}[Reduction Inequalities] \label{lemma:reds}
	Under the assumptions and notations of Proposition \ref{prop:global_laws}, assume additionally that the uniform bounds \eqref{eq:phi_assume} hold, then the estimates	
	\begin{equation} \label{eq:Red4av_HS}
		\bigl\lvert \bigl\langle (\im G_{1,t}) A_1 G_{2,t} A_2 (\im G_{1,t}) A_2^* G_{2,t}^* A_1^* \bigr\rangle \bigr\rvert \prec N\biggl(1 + \frac{\phi_{2}^{\mathrm{hs}}}{\sqrt{N\eta_t}}\biggr)^2\langle |A_1|^2\rangle\langle |A_2|^2\rangle,
	\end{equation}
	%
%	\begin{equation} \label{eq:Red4av}
%		\bigl\lvert \bigl\langle (\im G_{1,s}) A_1 G_{2,s} A_2 (\im G_{1,s}) A_2^* G_{2,s}^* A_1^* \bigr\rangle \bigr\rvert \prec N\biggl(1 + \frac{\phi_{2}^{\mathrm{hs}}}{\sqrt{N\eta_s}}\biggr)\biggl(1 + \frac{\phi_{2}^{\mathrm{op}}}{\sqrt{N\eta_s}}\biggr)\langle |A_1|^2\rangle\norm{A_2},
%	\end{equation}
	%	
	\begin{equation} \label{eq:Red(2,1)av_HS}
		\bigl\lvert \bigl\langle G_{1,t} A_1 G_{2,t} A_2 G_{1,t} B  \bigr\rangle \bigr\rvert \prec \frac{1}{\eta_t}\biggl(1 + \frac{\phi_{2}^{\mathrm{hs}}}{\sqrt{N\eta_t}}\biggr)\langle |A_1|^2\rangle^{1/2}\langle |A_2|^2\rangle^{1/2}\norm{B},
	\end{equation}
	%
%	\begin{equation} \label{eq:Red(2,1)av}
%		\bigl\lvert \bigl\langle G_{1,s} A_1 G_{2,s} A_2 G_{1,s} B  \bigr\rangle \bigr\rvert \prec \frac{1}{\eta_s}\biggl(1 + \frac{\phi_{2}^{\mathrm{hs}}}{\sqrt{N\eta_s}}\biggr)^{1/2}\biggl(1 + \frac{\phi_{2}^{\mathrm{op}}}{\sqrt{N\eta_s}}\biggr)^{1/2}\langle |A_1|^2\rangle^{1/2}\norm{A_2}\norm{B},
%	\end{equation}
	%
%	\begin{equation} \label{eq:Red2iso}
%		\bigl\lvert \bigl\langle \vect{x}, G_{2,s} A (\im G_{1,s}) A^* G_{2,s}^* \vect{x} \bigr\rangle \bigr\rvert \prec N \biggl(1 + \frac{\phi_{2}}{ N\eta_s}\biggr)\langle |A|^2\rangle\norm{\vect{x}}_2^2,
%	\end{equation}
	hold uniformly in $t\in[0,T]$, observables, and in $z_1,z_2\in\mathcal{D}$. Here, $\im G := \tfrac{1}{2\I}(G-G^*)$ denotes the imaginary part of the generalized resolvent $G$.
\end{lemma}

Besides longer resolvent chains, the evolution equations \eqref{eq:GBevol} and \eqref{eq:GBGBevol} involve terms where the super-operator $\mathscr{S}$, defined in \eqref{eq:superS_def}, acts as a quadratic form on two (random) resolvent chains,  e.g., $\bigl\langle (G_{t}-M_{t}) \mathscr{S}[ G_{t}BG_{t}]\bigr\rangle$ in \eqref{eq:GBevol}. We refer to terms of the form $\langle G\dots\, \mathscr{S}[G\dots] \rangle$ as \textit{quadratic terms}.  Here $(G\dots)$ denotes a resolvent chain or only its fluctuating part after subtracting the corresponding deterministic approximations. In contrast, if the quadratic form of $\mathscr{S}$ is acting on one resolvent chain (or its fluctuating part) and one deterministic matrix, e.g., $\langle G_t A G_t \mathscr{S}[\M{\vect{z}_t, B, \vect{z}_t}]\rangle$, we refer to it as a \textit{linear term}. Thus our terminology for linear and quadratic reflects the number of \textit{random} inputs (one or two) in the quadratic form of $\mathscr{S}$, and not the number of $G$ factors. Since $\mathscr{S}$ does not factorize, the linear and quadratic terms require different treatment. 

To bound the quadratic terms effectively, we employ the following decomposition of the matrix $S$ into a "regular" part $\mathring{S}$ and a rank one matrix. The proof is given in Section \ref{sec:A_reg_proofs}.
\begin{lemma}[$S$-Decomposition] \label{lemma:S_decomp}
	Let $z_1,z_2$ be two spectral parameters in the domain $\mathcal{D}$, defined in \eqref{eq:calD_def}. Then there exists a matrix $\mathring{S}(z_1,z_2)$ and a vector $\vect{s}(z_1,z_2)$, such that
	\begin{equation} \label{eq:S_decomp}
		S = N^{-1}\bigl(\mathring{S}(z_1,z_2) + \vect{s}(z_1,z_2)\,\vect{1}^*\bigr),
	\end{equation}
	and for all $p\in \{1,\dots,N\}$, the diagonal matrices $\mathring{S}^{(p)}(z_1,z_2)$, defined as
	\begin{equation} \label{eq:Sp_def}
		\mathring{S}^{(p)}(z_1,z_2):= \diag{\bigl(\mathring{S}_{pj}(z_1,z_2)\bigr)_{j=1}^N},
	\end{equation}
	are regular with respect to $(z_2,z_1)$ in the sense of Definition \ref{def:reg_A}. Moreover, for all $z_1,z_2\in \mathcal{D}$,
	\begin{equation} \label{eq:S_decomp_bounds}
		\max_{p}\bigl\lVert\mathring{S}^{(p)}(z_1,z_2)\bigr\rVert\lesssim 1,\quad \norm{\vect{s}(z_1,z_2)}_\infty \lesssim 1.
	\end{equation}
\end{lemma}

Next, we collect the properties of the deterministic approximations $\m(\vect{z}_t)$ and $\M{\vect{z}_{1,t},B,\vect{z}_{2,t}}$ that are essential for controlling the fluctuations of the corresponding generalized resolvent chains, that we prove in Appendix \ref{sec:stab_section}.
\begin{lemma}[$M$-Bounds] \label{lemma:Mbounds}
	Let $(\bm{\mathfrak{a}},S)$ be a data pair satisfying the Assumptions \eqref{as:S_flat} and \eqref{as:m_bound} of Definition \ref{def:WT}, then
	the following properties hold. 
	\begin{itemize}
%		\item[(i)] If the vector-valued spectral parameter $\vect{z}_t$ satisfies the characteristic flow equation \eqref{eq:z_evol}, then the solution $\m$ to the generalized vector Dyson equation \eqref{eq:VDE} with the data pair $(\bm{\mathfrak{a}},S)$ satisfies
%		\begin{equation} \label{eq:m_evol}
%			\partial_t \m(\vect{z}_t) = \frac{1}{2}\m(\vect{z}_t).
%		\end{equation}	
		\item[(i)] If the vector-valued spectral parameters $\vect{z}_{j,t}$ satisfy the characteristic flow equation \eqref{eq:z_evol}, then the solution vector $\m(\vect{z}_{j,t})$ and the self-consistent resolvent $M_{j,t}$ satisfy
		\begin{equation} \label{eq:m_upper}
			\norm{\m(\vect{z}_{j,t})}_\infty \lesssim 1, \quad \norm{M_{j,t}} \lesssim 1.
		\end{equation}
		\item[(ii)] If the vector-valued spectral parameters $\vect{z}_{j,t}$ satisfy the characteristic flow equation \eqref{eq:z_evol}, then the deterministic approximation matrix $\M{\vect{z}_{1,t},B,\vect{z}_{2,t}}$ defined in \eqref{eq:M_t_def} satisfies for all $B_1,B_2 \in \mathbb{C}^{N\times N}$,
		\begin{equation} \label{eq:M_evol}
			\partial_t \langle \M{\vect{z}_{1,t},B_1,\vect{z}_{2,t}} B_2 \rangle = \bigl\langle \M{\vect{z}_{1,t},B_1,\vect{z}_{2,t}} \mathscr{S}[\M{\vect{z}_{2,t},B_2,\vect{z}_{1,t}}] \bigr\rangle + \langle \M{\vect{z}_{1,t},B_1,\vect{z}_{2,t}} B_2 \rangle,
		\end{equation} 
		where the super-operator $\mathscr{S}$ is defined in \eqref{eq:superS_def}.
		
		\item[(iii)] Moreover, there exists a positive threshold $T_*\sim 1$, such that for all terminal times $0\le T\le T_*$, if $\vect{z}_{j,t}$ satisfy the characteristic flow equation \eqref{eq:z_evol} for all $t\in[0,T]$ with the terminal condition $\vect{z}_{j,T} = z_j \in \mathcal{D}$, then
		\begin{equation} \label{eq:regM_bound}
			\norm{\M{\vect{z}_{1,t},A,\vect{z}_{2,t}}} \lesssim \norm{A}, \quad  \bigl\langle |\M{\vect{z}_{1,t},A,\vect{z}_{2,t}}|^2\bigr\rangle^{1/2} \lesssim \langle |A|^2\rangle^{1/2}.
		\end{equation}
		for all deterministic matrices $A$ regular with respect to $(z_1,z_2)$ as in Definition \ref{def:reg_A}. 
		Under the same set of assumptions, for all deterministic matrices $B$,
		\begin{equation} \label{eq:M_bound}
				\frac{\norm{\M{\vect{z}_{1,t},B,\vect{z}_{2,t}}}}{\norm{B}} + \frac{\bigl\langle |\M{\vect{z}_{1,t}, B, \vect{z}_{2,t}}|^2\bigr\rangle^{1/2}}{\langle |B|^2\rangle^{1/2}} \lesssim \frac{1}{\eta_{1,t} + \eta_{2,t} + |z_1-z_2|+\mathds{1}_{(\im z_1)(\im z_2)>0}},
		\end{equation}
		where we recall that $\eta_{j,t} := |\langle \im \vect{z}_{j,t} \rangle|$ for $j\in\{1,2\}$.
	\end{itemize}
\end{lemma}

Finally, we record the following properties of the characteristic flow \eqref{eq:z_evol} and the integration rules for the parameters $\eta_{j,t}$ and $\eta_t$, that we prove in Appendix  \ref{sec:eta_sec}.
\begin{lemma}[$\eta$-Lemma] (c.f. Equation (4.3) in \cite{Cipolloni2023Edge}) \label{lemma:imz}
	Let $\mathcal{D}$ be the domain defined in \eqref{eq:calD_def}, and let $T$ be the terminal time provided by Proposition \ref{prop:global_laws}. For any $z_1,z_2\in\mathcal{D}$, define $\eta_{j,t} := |\langle \im \vect{z}_{j,t}\rangle |$ where $\vect{z}_{j,t}$ is the solution to the characteristic flow equation with the terminal condition $\vect{z}_{j,T} = z_j\vect{1}$. Then for all spectral parameters $z_j \in \mathcal{D}$, the comparison 
	\begin{equation} \label{eq:imz_flat}
		(\sign z_j)\im \vect{z}_{j,t} \sim  |\langle \im \vect{z}_{j,t} \rangle|=\eta_{j,t},
	\end{equation}
	holds entry-wise. Moreover, the quantities $\eta_{j,t}$ satisfy
	\begin{equation}\label{eq:eta_asymp}
		\eta_{j,t} \sim |\im z_j| + (T-t),
	\end{equation}
	and we have the following integration rules for all $\alpha > 1$,
	\begin{equation} \label{eq:int_rules}
		\int_0^t \frac{\mathrm{d}s}{\eta_{j,s}^{\alpha}} \lesssim \eta_{j,t}^{1-\alpha}, \quad \int_0^t \frac{\mathrm{d}s}{\eta_{s}^{\alpha}} \lesssim \eta_{t}^{1-\alpha}, \quad \int_0^t \frac{\mathrm{d}s}{\eta_{j,s}} \lesssim \log N, \quad \int_0^t \frac{\mathrm{d}s}{\eta_{s}} \lesssim \log N,
	\end{equation}
	where $\eta_t := \min\{\eta_{1,t},\eta_{2,t}\}$.
\end{lemma}

\subsection{Proof of Master Inequalities}
Equipped with Lemmas \ref{lemma:1Flaw}--\ref{lemma:imz}, we are ready to prove the master inequalities of Proposition \ref{prop:masters}. In the sequel, unless explicitly stated otherwise, all the stochastic domination bounds hold uniformly in observables and in spectral parameters $z_1, z_2\in\mathcal{D}$.

Note that at time $t=0$ the global laws of Proposition \ref{prop:global_laws} and \eqref{eq:global_laws-HS} imply the bound
\begin{equation} \label{eq:Phi_goal_at0}
	\Phi_{1}(0) +  \Phi_{2}^{\mathrm{op}}(0) + \Phi_{2}^{\mathrm{hs}}(0) + \Phi_{(1,1)}(0) \prec 1.
\end{equation}
\begin{proof}[Proof of Master Inequality \eqref{eq:Master1av}] 
To condense the notation, we drop the lower index $1$, abbreviating $z := z_1$, $\vect{z}_t := \vect{z}_{1,t}$, $G_t := G_{1,t}$, $M_t := M_{1,t}$, and $\eta_t := \eta_{1,t}$.
Combining \eqref{eq:GBevol} with $B:=A$ and \eqref{eq:m_evol}, we obtain
\begin{equation}\label{eq:(G-m)A_evol}
	\begin{split}
		\mathrm{e}^{t/2}\mathrm{d}\bigl\{\mathrm{e}^{-t/2}\bigl\langle (G_{t}-M_t)A\bigr\rangle\} =&~ \frac{1}{2}\sum_{j,k}\partial_{jk}\bigl\langle G_{t}A\bigr\rangle\sqrt{S_{jk}}\mathrm{d}\Brwn_{jk,t}
		%\\&
		%+	\frac{1}{2} \bigl\langle (G_{t}-\m_t)A\bigr\rangle\mathrm{d}t
		+\bigl\langle (G_{t}-M_{t}) \mathscr{S}[ G_{t}AG_{t}]\bigr\rangle\mathrm{d}t.
	\end{split}
\end{equation}	
First, we bound the stochastic term in \eqref{eq:(G-m)A_evol} using the inequality (see \cite{martingale}, Appendix B.6, Eq. (18)) for continuous martingales $\mathcal{M}_t$, finite stopping time $\tau$, and any fixed $x,y > 0$,
\begin{equation} \label{eq:mart_ineq}
	\mathbb{P}\biggl(\sup_{0\le s \le \tau} \biggl\lvert \int_0^{s} \mathrm{d}\mathcal{M}_r \biggr\rvert \ge x, \quad  \biggl[ \int_0^{\cdot} \mathrm{d}\mathcal{M}_r \biggr]_{\tau} \le y \biggr) \le 2\mathrm{e}^{-x^2/(2y)},
\end{equation}
where $[\cdot]_t$ denotes the quadratic variation process.
Therefore, using the bound $\lvert \partial_{jk}\langle G_{t}A\rangle \rvert \prec N^2\langle |A|^2 \rangle^{1/2}$ and a dyadic argument with $x := 2^k\log N$, $y := 2^{2k}$ for $k \in [-100\log N, 100 \log N]$, we deduce from \eqref{eq:mart_ineq} with $\tau := t$ that
\begin{equation} \label{eq:mart+_est}
	\sup_{0\le s \le t} \biggl\lvert \int_0^{s} \sum_{j,k}\partial_{jk}\bigl\langle G_{r}A\bigr\rangle\sqrt{S_{jk}}\mathrm{d}\Brwn_{jk,r} \biggr\rvert 
	\prec \biggl( \int_0^t \sum_{j,k} S_{jk} \bigl\lvert \partial_{jk}\bigl\langle G_{s}A\bigr\rangle\bigr\rvert^2\mathrm{d}s\biggl)^{1/2} \log N + \frac{\langle |A|^2\rangle^{1/2}}{N},
\end{equation}
The integrand of the quadratic variation on the right-hand side of \eqref{eq:mart+_est} satisfies
\begin{equation} \label{eq:1av_mart}
		\sum_{j,k} S_{jk} \bigl\lvert \partial_{jk}\bigl\langle G_{s}A\bigr\rangle\bigr\rvert^2 \lesssim \frac{1}{N^2}\sum_{j,k} S_{jk} (G_{s}AG_{s})_{kj}(G_{s}^*A^*G_{s}^*)_{jk}
		\lesssim \frac{1}{N^2\eta_s^2}\bigl\langle (\im G_{s}) A (\im G_{s}) A^*\bigr\rangle,
\end{equation}
where we used the upper bound in \eqref{as:S_flat}, and the operator inequality
\begin{equation} \label{eq:ImG_ineq}
	G_s G_s^* \le \norm{(\im \vect{z}_s)^{-1}}_\infty \im G_s \lesssim (\im z)\eta_s^{-1} \im G_s,
\end{equation}
for a generalized resolvent $G_s$ defined in \eqref{eq:Gt}. We use \eqref{eq:ImG_ineq} to replace the usual \textit{Ward identity} $GG^* = (\im z)^{-1} \im G$, available only for standard resolvents $G := (H-z)^{-1}$.
The inequality \eqref{eq:ImG_ineq} follows from \eqref{eq:imz_flat} and the generalized resolvent identity 
\begin{equation} \label{eq:res_id}
	G(X,\bm{\zeta}_1) - G(X,\bm{\zeta}_2) = G(X,\bm{\zeta}_1) \diag{\bm{\zeta}_1 - \bm{\zeta}_2} G(X,\bm{\zeta}_2), \quad \bm{\zeta}_j \in \mathbb{H}^N\cup (\mathbb{H}^*)^N,
\end{equation}
with $\bm{\zeta}_1 = \overline{\bm{\zeta}_2} = \vect{z}_s$.
Integrating \eqref{eq:1av_mart} in time yields
\begin{equation}\label{eq:1av_mart_int}
	\int_0^t\sum_{j,k} S_{jk} \bigl\lvert \partial_{jk}\bigl\langle G_{s}A\bigr\rangle\bigr\rvert^2 \mathrm{d}s
	\prec \int_0^t\frac{\langle |A|^2\rangle}{N^2\eta_{s}^2}\biggl(1 + \frac{\phi_{2}^{\mathrm{hs}}}{\sqrt{N\eta_{s}}}\biggr) \mathrm{d}s \prec \frac{\langle |A|^2\rangle}{N^2\eta_t}\biggl(1 + \frac{\phi_{2}^{\mathrm{hs}}}{\sqrt{N\eta_t}}\biggr),
\end{equation}
where in the first step we used the assumption \eqref{eq:phi_assume} together with the second bound \eqref{eq:regM_bound}, and in the second step we used the integration rule \eqref{eq:int_rules}. Note that the $\log N$ factors from \eqref{eq:int_rules} can be absorbed into $\prec$ by Definition \ref{def:prec}.  Therefore, \eqref{eq:mart+_est}, \eqref{eq:1av_mart_int} imply uniformly in $0\le t \le T$,
\begin{equation} \label{eq:1Gmart_term}
	\sup_{0\le s \le t}\biggl\lvert \int_0^s\frac{1}{2}\sum_{j,k}\partial_{jk}\bigl\langle G_{r}A\bigr\rangle\sqrt{S_{jk}}\mathrm{d}\Brwn_{jk,r} \biggr\rvert \prec \frac{\langle |A|^2\rangle^{1/2}}{N\sqrt{\eta_t}}\biggl(1+\frac{\sqrt{\phi_{2}^{\mathrm{hs}}}}{(N\eta_t)^{1/4}}\biggr).
\end{equation}
%Note that the second term on the right-hand side of \eqref{eq:(G-m)A_evol} can be removed by differentiating $\mathrm{e}^{-t/2}\langle (G_t-\m_t)A \rangle$, and the exponential factor $\mathrm{e}^{-t/2} \sim 1$ since $0\le t \le T\lesssim 1$.

Finally, we estimate the quadratic term, i.e., the second term on the right-hand side of \eqref{eq:(G-m)A_evol}. Defining the matrices $S^{(p)} := N\mathrm{diag}((S_{pj})_{j=1}^N)$ for all $p \in \{1,\dots, N\}$, we obtain
\begin{equation} \label{eq:1av_quad}
		\bigl\langle (G_{s}-M_{s}) \mathscr{S}[ G_{s}AG_{s}]\bigr\rangle = \frac{1}{N}\sum_p \bigl\langle S^{(p)}(G_{s}AG_{s}-M_{[1,2],s})\bigr\rangle (G_{s}-M_{s})_{pp} 
		+ \bigl\langle (G_{s}-M_{s})\mathscr{S}[M_{[1,2],s}] \bigr\rangle,
\end{equation}
where $M_{[1,2],s} := \M{\vect{z}_s,A,\vect{z}_s}$ is the deterministic approximation to the chain $G_sAG_s$, defined in \eqref{eq:M_t_def}. Hence, using the upper bound in \eqref{as:S_flat}, the local laws of Lemma \ref{lemma:1Flaw}, and the assumption \eqref{eq:phi_assume}, we obtain for all $s\in [0,T]$,
\begin{equation} \label{eq:1av_Sterm}
		\bigl\lvert \bigl\langle (G_{s}-M_{s}) \mathscr{S}[ G_{s}AG_{s}]\bigr\rangle \bigr\rvert
		\prec  \frac{1}{\eta_s}\frac{\langle|A|^2\rangle^{1/2}}{N\sqrt{\eta_s}}\phi_{(1,1)} + \frac{\langle |A|^2\rangle^{1/2}}{N\eta_s},
\end{equation}
%\begin{equation} 
%	\begin{split}
%		\bigl\lvert \bigl\langle (G_{s}-\m_{s}) \mathscr{S}[ G_{s}AG_{s}]\bigr\rangle \bigr\rvert
%		\le&~ \frac{\langle |A|^2\rangle^{1/2}}{\sqrt{N\eta_s}}\frac{\phi_{2}}{N\eta_s} + \frac{\langle |A|^2\rangle^{1/2}}{N\eta_s}\frac{\phi_{(1,1)}}{N\eta_s^{3/2}} + \frac{\langle |A|^2\rangle^{1/2}}{N\eta_s}.
%	\end{split}
%\end{equation}
Here, we additionally used the upper bound in Assumption \eqref{as:S_flat} and the second estimate in \eqref{eq:regM_bound} to deduce that $\norm{\mathscr{S}[M_{[1,2],s}]} \lesssim \langle |M_{[1,2],s}|^2\rangle^{1/2} \lesssim \langle |A|^2\rangle^{1/2}$.
Therefore, integrating \eqref{eq:1av_Sterm} in time, using $\eta_t \lesssim 1$ by \eqref{eq:eta_asymp}, and the integration rule \eqref{eq:int_rules}, yields uniformly in  $0\le t\le T$,
\begin{equation} \label{eq:1Gquad_term} 
	 \int_0^t\bigl\lvert\bigl\langle (G_{s} - M_{s}) \mathscr{S}[ G_{s}AG_{s}]\bigr\rangle\bigr\rvert \mathrm{d}s \prec \frac{\langle |A|^2\rangle^{1/2}}{N\sqrt{\eta_t}}\bigl(1 +  \phi_{(1,1)}\bigr).
\end{equation}
Note that since $0\le t \le T \lesssim 1$, the factor $\mathrm{e}^{t/2} \sim 1$ can be ignored in \eqref{eq:(G-m)A_evol}. Hence, integrating \eqref{eq:(G-m)A_evol}, using $\Phi_{1}(0)\prec 1$ from \eqref{eq:Phi_goal_at0}, and the estimates \eqref{eq:1Gmart_term}, \eqref{eq:1Gquad_term} yields the desired \eqref{eq:Master1av} by definition of $\Phi_{1}(t)$ in \eqref{eq:Phi_def}. This concludes the proof of the master inequality \eqref{eq:Master1av}.
\end{proof}

\begin{proof}[Proof of Master Inequalities \eqref{eq:Master2av_HS} and \eqref{eq:Master2av}]
We prove the next two master inequalities, \eqref{eq:Master2av_HS} and \eqref{eq:Master2av}, simultaneously. To this end, we index the quantities $\Phi_{2}^{\n}$ and $\phi_{2}^{\n}$ with a label $\n\in\{\mathrm{hs},\mathrm{op}\}$, and introduce the following notation 
\begin{equation} \label{eq:norm_parameters}
	\norm{A}_{\mathrm{hs}} := \langle |A|^2 \rangle^{1/2}, \quad \norm{A}_{\mathrm{op}} := \norm{A}, \quad \alpha(\mathrm{hs}) := \tfrac{1}{2}, \quad \alpha(\mathrm{op}) := 1.
\end{equation} 
Under this notation, the master inequalities \eqref{eq:Master2av_HS} and \eqref{eq:Master2av} have the form
\begin{equation}
	\Phi_{2}^{\n}(s) \prec 1+ \frac{\phi_{2}^{\mathrm{hs}}}{\sqrt{N\eta_s}} + \frac{\phi_{2}^{\mathrm{op}}}{(N\eta_s)^{\alpha(\n)-1/2}} + \frac{\phi_{(1,1)}^2}{(N\eta_s)^{3/2}}.
\end{equation}
%and the normalized trace of a chain containing two generalized resolvents satisfies
%\begin{equation}
%	\bigl\lvert \bigl\langle (G_{1,t}A_1 G_{2,t} - \M{1,t}) A_2 \bigr\rangle \bigr \rvert \le \Phi_{2}^{\n}(t) \frac{\norm{A_1}_{\mathrm{hs}}\norm{A_2}_{\n}}{\sqrt{N\eta_t}}, \quad t\in[0,T].
%\end{equation}
Furthermore, to condense the presentation, we denote the length two resolvent chains by 
\begin{equation}
	G_{[1,2],t} := G_{1,t}A_1 G_{2,t}, \quad G_{[2,1],t} := G_{2,t}A_2 G_{1,t}. 
\end{equation}

The proof follows the same outline as that of \eqref{eq:Master1av} above, hence we present only the key new steps in full detail.
Starting with \eqref{eq:GBGBevol} and \eqref{eq:M_evol}, we obtain
\begin{equation} \label{eq:(GAG-M)A_evol}
	\begin{split}
		\mathrm{d}\bigl\langle (G_{[1,2],t} - M_{[1,2],t}) A_2 \bigr\rangle =&~ \frac{1}{2}\sum_{j,k}\partial_{jk}\bigl\langle G_{1,t}A_1 G_{2,t} A_2 \bigr\rangle \sqrt{S_{jk}}\mathrm{d}\Brwn_{jk,t} +	\biggl( \bigl\langle (G_{[1,2],t} - M_{[1,2],t}) A_2 \bigr\rangle\\
		&+ \bigl\langle  (G_{[1,2],t} - M_{[1,2],t})\mathscr{S}[M_{[2,1],t}]\bigr\rangle  
		+ \bigl\langle (G_{[2,1],t} - M_{[2,1],t})\mathscr{S}[M_{[1,2],t}]\bigr\rangle\\
		&+\bigl\langle \mathscr{S}[G_{[1,2],t} - M_{[1,2],t}] (G_{[2,1],t} - M_{[2,1],t})\bigr\rangle\\
		&+\bigl\langle(G_{1,t}-M_{1,t})\mathscr{S}[G_{1,t}A_1 G_{2,t} A_2 G_{1,t}]\bigr\rangle\\ 
		&+\bigl\langle(G_{2,t}-M_{2,t})\mathscr{S}[G_{2,t}A_2 G_{1,t} A_1G_{2,t}]\bigr\rangle 
		\biggr)\mathrm{d}t
		,
	\end{split}
\end{equation}
where we denote $M_{1,t} := \M{\vect{z}_{1,t},A_1,\vect{z}_{2,t}}$, and $M_{2,t} := \M{\vect{z}_{2,t},A_2,\vect{z}_{1,t}}$. Here we used the definition of $\mathscr{S}$ in \eqref{eq:superS_def} to assert that for all $X,Y \in \mathbb{C}^{N\times N}$,  $\langle X\mathscr{S}[Y] \rangle = \langle \mathscr{S}[X]Y \rangle$.

Note that similarly to the proof of the master inequality \eqref{eq:Master(1,1)av}, the second term on the right-hand side of \eqref{eq:(GAG-M)A_evol} can be removed by differentiating $\mathrm{e}^{-t}\langle (G_{[1,2],t} - M_{1,t}) A_2 \rangle$ with $\mathrm{e}^{-t} \sim 1$, and therefore we omit it from the analysis.

Fix a label $\n \in \{\mathrm{hs}, \mathrm{op}\}$. First, we expand the quadratic variation of the martingale term on the right-hand side of \eqref{eq:(GAG-M)A_evol}. Similarly to \eqref{eq:1av_mart}--\eqref{eq:1av_mart_int}, we obtain
\begin{equation} \label{eq:2av_mart}
	\begin{split}
		\sum_{j,k}S_{jk}\bigl\lvert\partial_{jk}\bigl\langle G_{1,s}A_1 G_{2,s} A_2 \bigr\rangle \bigr\rvert^2\lesssim&~
		\frac{1}{N^2}\sum_{j,k}S_{jk}\bigl\lvert(G_{1,s}A_1 G_{2,s} A_2G_{1,s})_{kj} + (G_{2,s}A_2 G_{1,s} A_1G_{2,s})_{kj}\bigr\rvert^2 \\ \lesssim&~
		 \frac{1}{N^2\eta_{1,s}^2} \bigl\langle \im G_{1,s}A_1 G_{2,s} A_2 \im G_{1,s} A_2^*G_{2,s}^*A_1^* \bigr\rangle\\
		&+\frac{1}{N^2\eta_{2,s}^2} \bigl\langle \im G_{2,s} A_2 G_{1,s} A_1 \im G_{2,s} A_1^* G_{1,s}^*A_2^* \bigr\rangle\\
		%\prec&~ \frac{\langle |A_1|^2\rangle\langle |A_2|^2\rangle}{N^2\eta^2}\biggl(\frac{1}{\eta} + \frac{\phi_{4,\mathrm{av}}}{N\eta^2}\biggr)
		\prec&~ \frac{1}{\eta_s}\frac{1}{N\eta_s}\biggl(1 + \frac{\phi_{2}^{\mathrm{hs}}}{\sqrt{N\eta_s}}\biggr)^2\langle |A_1|^2\rangle\norm{A_2}_{\n}^2,
	\end{split}
\end{equation}
for all $0\le s \le T$, where we used $\norm{\cdot}_{\mathrm{hs}} \le \norm{\cdot}_{\mathrm{op}}$, and we recall $\eta_s := \min\{\eta_{1,s},\eta_{2,s}\}$.
Here, the first two steps follow by the upper bound in Assumption \eqref{as:S_flat} and the generalized resolvent inequality \eqref{eq:ImG_ineq}, and in the last step we used \eqref{eq:phi_assume}, and the reduction inequality \eqref{eq:Red4av_HS}.
Integrating \eqref{eq:2av_mart}, using \eqref{eq:int_rules}, and the martingale inequality \eqref{eq:mart_ineq}, yields uniformly in $t\in[0,T]$,
\begin{equation} \label{eq:2av_mart_final}
	\sup_{0\le s\le t}\biggl\lvert \int_0^s\sum_{j,k}\partial_{jk}\bigl\langle G_{1,r}A_1 G_{2,r} A_2 \bigr\rangle \sqrt{S_{jk}}\mathrm{d}\Brwn_{jk,r} \mathrm{d}r\biggr\rvert 
	\prec \frac{\langle |A_1|^2\rangle^{1/2}\norm{A_2}_{\n}}{\sqrt{N\eta_t}}\biggl(1 + \frac{\phi_{2}^{\mathrm{hs}}}{\sqrt{N\eta_t}}\biggr).
\end{equation}

Next, we observe that the first term in the second line of \eqref{eq:(GAG-M)A_evol} can be estimated using the local law \eqref{eq:Weak2av}. Indeed, by \eqref{eq:int_rules}, we have uniformly in $t \in[0,T]$,
\begin{equation} \label{eq:2av_tr_term}
	 \int_0^t \bigl\lvert\bigl\langle (G_{[1,2],t} - M_{1,s}) \mathscr{S}[M_{2,s}] \bigr\rangle \bigr\rvert  \mathrm{d}s \prec \int_0^t \frac{\langle |A_1|^2\rangle^{1/2}\norm{\mathscr{S}[M_{2,s}]}}{N\eta_{1,s}\eta_{2,s}}\mathrm{d}s \prec \frac{\langle |A_1|^2\rangle^{1/2}\norm{A_2}_{\n}}{N\eta_t}.
\end{equation}
Here we used the regularity of $A_2$, the second bound in \eqref{eq:regM_bound}, and the estimate
\begin{equation} \label{eq:super-S-norm}
	\norm{\mathscr{S}[X]} \lesssim \langle |X|^2\rangle^{1/2},
\end{equation} 
that follows from the upper bound in Assumption \eqref{as:S_flat}, to assert that $\norm{\mathscr{S}[M_{2,s}]} \lesssim \langle |M_{2,s}|^2\rangle^{1/2} \lesssim \norm{A_2}_{\n}$ for both labels $\n \in \{\mathrm{hs}, \mathrm{op}\}$. The other term in the second line of \eqref{eq:(GAG-M)A_evol} is estimated analogously.

We now turn to estimating the quadratic terms, starting with the term in the third line of \eqref{eq:(GAG-M)A_evol}. Decomposing the matrix $S$ according to Lemma \ref{lemma:S_decomp}, and using the definition of $\mathscr{S}$ in \eqref{eq:superS_def}, we obtain, for $0\le s \le T$,
\begin{equation}
	\begin{split}
		\bigl\langle \mathscr{S}[G_{[1,2],s}-M_{[1,2],s}] (G_{[2,1],s}-M_{[2,1],s})\bigr\rangle  =&~ \frac{1}{N}\sum_p  \bigl\langle \mathring{S}^{(p)}(G_{[1,2],s}-M_{[1,2],s})\bigr\rangle (G_{[2,1],s}-M_{[2,1],s})_{pp}\\  &+  \bigl\langle G_{[1,2],s}-M_{[1,2],s}\bigr\rangle \bigl\langle \diag{\vect{s}}(G_{[2,1],s}-M_{[2,1],s})\bigr\rangle,
	\end{split}
\end{equation}
where $\mathring{S}^{(p)}$, $\vect{s}$ are defined in \eqref{eq:S_decomp}. Moreover, it follows from the isotropic local laws \eqref{eq:Weak1iso_HS}, \eqref{eq:Weak1iso} that
\begin{equation} \label{eq:iso_bound}
	\bigl\lvert (G_{[2,1],s}-M_{[2,1],s})_{pp}  \bigr\rvert  \prec \frac{\norm{A_2}_{\n}}{(N\eta_s)^{\alpha(\n)-1/2}\eta_s}, \quad s\in[0,T].
\end{equation}
Using the regularity of $\mathring{S}^{(p)}(z_1,z_2)$ with respect to $(z_2,z_1)$, %the bound $\norm{\cdot}_{\mathrm{hs}} \le \norm{\cdot}_{\mathrm{op}}$, 
the estimates \eqref{eq:phi_assume}, \eqref{eq:S_decomp_bounds}, \eqref{eq:iso_bound}, and the integration rule \eqref{eq:int_rules}, we obtain uniformly in $t\in[0,T]$,
\begin{equation} \label{eq:2av_S_term}
	\int_0^t \bigl\lvert \bigl\langle \mathscr{S}[G_{[1,2],s}-M_{[1,2],s}] (G_{[2,1],s}-M_{[2,1],s})\bigr\rangle\bigr\rvert  \mathrm{d}s
	\prec \frac{\langle |A_1|^2\rangle^{1/2}\norm{A_2}_{\n}}{\sqrt{N\eta_t}}\biggl(\frac{\phi_{2}^{\mathrm{op}}}{(N\eta_t)^{\alpha(\n)-1/2}} + \frac{\phi_{(1,1)}^2}{\sqrt{N\eta_t}}\biggr).
\end{equation}
Finally, we estimate the quadratic term in the fourth line of \eqref{eq:(GAG-M)A_evol}. For $S^{(p)}$ as in \eqref{eq:1av_quad} and $s\in [0,T]$,
\begin{equation} \label{eq:2a_quad}
	\bigl\langle(G_{1,s}-M_{1,s})\mathscr{S}[G_{1,s} A _1 G_{2,s} A_2 G_{1,s}]\bigr\rangle = \frac{1}{N}\sum_{p} (G_{1,s}-M_{1,s})_{pp} \bigl\langle S^{(p)}(G_{1,s} A_1 G_{2,s} A_2 G_{1,s})\bigr\rangle.
\end{equation}
For a fixed index $p \in \{1,\dots,N\}$, we estimate the trace on right-hand side of \eqref{eq:2a_quad} using the reduction inequality \eqref{eq:Red(2,1)av_HS} with $B := S^{(p)}$, and bound each of the factors $(G_{1,s}-M_{1,s})_{pp}$ using the isotropic local of Lemma \ref{lemma:1Flaw}. Integrating in time and applying \eqref{eq:int_rules}, we obtain for all $0\le t \le T$,
\begin{equation} \label{eq:2av_3-1term}
		\int_0^t \bigl\lvert \bigl\langle(G_{1,s}-M_{1,s})\mathscr{S}[G_{1,s} A _1 G_{2,s} A_2 G_{1,s}]\bigr\rangle \bigr\rvert   \mathrm{d}s 
		\prec \frac{ \langle |A_1|^2\rangle^{1/2}\norm{A_2}_{\n}}{\sqrt{N\eta_t}}\biggl(1+ \frac{\phi_{2}^{\mathrm{hs}}}{\sqrt{N\eta_t}}\biggr).
\end{equation}
The bound on the other remaining term on the right-hand side of \eqref{eq:(GAG-M)A_evol} is completely analogous.

Therefore, evoking \eqref{eq:Phi_goal_at0} and summing the bounds \eqref{eq:2av_mart_final}, \eqref{eq:2av_tr_term}, \eqref{eq:2av_S_term}, \eqref{eq:2av_3-1term}, we conclude the proof of the master inequalities \eqref{eq:Master2av_HS} and \eqref{eq:Master2av}.
\end{proof}

In preparation for proving the master inequality \eqref{eq:Master(1,1)av}, we assert the following norm bound on a generalized resolvent $G(X, \vect{z}_t)$ for any Hermitian $N\times N$ matrix $X$,
\begin{equation} \label{eq:G_norm}
	\norm{G(X,\vect{z}_t)} \lesssim \eta_t^{-1}, \quad z\in\mathcal{D},
\end{equation}
where $\vect{z}_t$ solves \eqref{eq:z_evol} with $\vect{z}_T = z$ at the terminal time $T$ provided by Proposition \ref{prop:global_laws}.
To prove \eqref{eq:G_norm}, we define $\bm{\nu} := (\eta_t^{-1}|\im\vect{z}_t|)^{-1/2} \sim 1$ by \eqref{eq:imz_flat}, and, starting with the definition \eqref{eq:Gt}, we obtain
\begin{equation} \label{eq:G_rescale}
	G(X,\vect{z}_t) = \bigl(X-\vect{z}_t\bigr)^{-1} = \bm{\nu}\bigl(\bm{\nu}(X - \re\vect{z}_t)\bm{\nu}-\I (\sign z) \eta_t\bigr)^{-1} \bm{\nu},
\end{equation}
where we identify the vector $\bm{\nu}$ with the diagonal matrix $\diag{\bm\nu}$. Therefore, using the standard norm bound for the resolvent of a Hermitian matrix $\bm{\nu}(X - \re\vect{z}_t)\bm{\nu}$ at a spectral parameter $\I(\sign z) \eta_t$, we deduce that $\lVert G(X,\vect{z}_t)\rVert \le \eta_t^{-1}\lVert{\bm{\nu}}\rVert_\infty^2 \sim \eta_t^{-1}$ by \eqref{eq:imz_flat}.

\begin{proof}[Proof of Master Inequality \eqref{eq:Master(1,1)av}]
	For this proof, we redefine $G_{[1,2],t} := G_{1,t}A G_{2,t}$, $G_{[2,1],t} := G_{2,t} B G_{1,t}$, and, respectively, $M_{[1,2],t} := \M{\vect{z}_{1,t},A,\vect{z}_{2,t}}$, $M_{[2,1],t} := \M{\vect{z}_{1,t},B,\vect{z}_{2,t}}$.
	Starting with \eqref{eq:GBGBevol} and \eqref{eq:M_evol}, exactly as in \eqref{eq:(GAG-M)A_evol}, we obtain
	\begin{equation} \label{eq:(GAGB-M)_evol}
		\begin{split}
			\mathrm{d}\bigl\langle (G_{[1,2],t} - M_{[1,2],t}) B \bigr\rangle = &~ \bigl\langle (G_{[1,2],t}  - M_{[1,2],t}) \mathscr{S}[M_{[2,1],t}] \bigr\rangle\mathrm{d}t 
			+ \frac{1}{2}\sum_{j,k}\partial_{jk}\bigl\langle G_{[1,2],t} B \bigr\rangle \sqrt{S_{jk}}\mathrm{d}\Brwn_{jk,t}\\
			&+\biggl(\bigl\langle (G_{[1,2],t} - M_{[1,2],t}) B \bigr\rangle + \bigl\langle (G_{[2,1],t} - M_{[2,1],t})\mathscr{S}[M_{[1,2],t}]\bigr\rangle\\
			&+\bigl\langle \mathscr{S}[G_{[1,2],t} - M_{[1,2],t}] (G_{[2,1],t} - M_{[2,1],t})\bigr\rangle\\
			&+\bigl\langle \mathscr{S}[G_{1,t}-M_{1,t}] G_{1,t} A G_{2,t} B G_{1,t}\bigr\rangle\\ &+\bigl\langle \mathscr{S}[G_{2,t}-M_{2,t}] G_{2,t} B G_{1,t} A G_{2,t} \bigr\rangle\biggr)\mathrm{d}t.
		\end{split}
	\end{equation}
	
	We start by analyzing  the time integral of the first term on the right-hand side of \eqref{eq:(GAGB-M)_evol}.
	We distinguish two cases. 
	
	\textbf{Case 1.} First, we consider the easier regime $(\im z_1) (\im z_2) > 0$.
	Owing to \eqref{eq:M_bound} and \eqref{eq:super-S-norm}, we have  $\norm{\mathscr{S}[\M{2,s}]}\lesssim \norm{B}$. Hence, for such $z_1,z_2$, uniformly in $t\in [0,T]$,
	\begin{equation} \label{eq:(1,1)_tr_case1}
		\int_0^t\bigl\langle (G_{1,s} A G_{2,s} - M_{[1,2],s}) \mathscr{S}[M_{[2,1],s}] \bigr\rangle \bigr\rvert  \mathrm{d}s 
		\prec \int_0^t \frac{\langle |A|^2\rangle^{1/2}\norm{B}}{N\eta_s^2}\mathrm{d}s \prec \frac{\langle |A|^2\rangle^{1/2}\norm{B}}{\sqrt{N\eta_t}\sqrt{\eta_t}}\frac{1}{\sqrt{N}},
	\end{equation} 
	where we used the averaged local law \eqref{eq:Weak2av} without any improvement from the regularity of $A$. 
	
	\textbf{Case 2.} Next, we consider the case $(\im z_1) (\im z_2) < 0$.
	We stress that since the bound $\norm{\mathscr{S}[M_{[2,1],t}]} \lesssim \eta_t^{-1}\norm{B}$ is saturated whenever $z_1,z_2$ lie in opposite half-planes, it is not affordable to ignore the regularity of $A_1$ and  simply use \eqref{eq:Weak2av}, as we did in \eqref{eq:(1,1)_tr_case1} for $(\im z_1)(\im z_2) > 0$. 	
	Instead, we employ the generalized resolvent identity \eqref{eq:res_id} and the following lemma that we prove in Section \ref{sec:A_reg_proofs}.
	\begin{lemma}[Observable Regularization] \label{lemma:B_renorm}
		There exists a threshold $1 \lesssim T_* \le T$ such that for all times $T-T_* \le t \le T$, the following holds true. Let $z_1,z_2 \in \mathcal{D}$, and let $\vect{z}_{j,t}$ solve \eqref{eq:z_evol} with $\vect{z}_{j,T} = z_j\vect{1}$, then for any $B\in \mathbb{C}^{N\times N}$, there exists a matrix $\mathring{B}_t := \mathring{B}_t(z_1,z_2)$ and a complex number $b_t := b_t(z_1,z_2)$ such that %for all $t\in[0,T]$
		\begin{equation} \label{eq:B_decomp}
			B = \mathring{B}_t + b_t \diag{\,\widehat{\vect{z}}_{1,t}-\vect{z}_{2,t}}, \quad \widehat{\vect{z}}_{1,t} := \re \vect{z}_{1,t} - \I \frac{\sign(\im z_2)}{\sign (\im z_1)} \im \vect{z}_{1,t},
		\end{equation}
		and the observable $\mathring{B}_t$ is regular with respect to $(z_2,z_1)$ in the sense of Definition \ref{def:reg_A}. Moreover, $\mathring{B}_t$ and $b_t$ satisfy
		\begin{equation} \label{eq:B_decomp_bounds}
			\bigl\lVert\mathring{B}_t\bigr\rVert \lesssim \norm{B},\quad \langle |\mathring{B}_t|^2\rangle^{1/2} \lesssim \langle |B|^2\rangle^{1/2}, \quad |b_t| \lesssim \frac{\langle |B|^2\rangle^{1/2}}{\eta_{1,t}+\eta_{2,t}+|z_1-z_2|},
		\end{equation}
		where $\eta_{j,t} := |\langle \im\vect{z}_{j,t}\rangle|$.
		
		Assuming additionally that the observable $B$ is regular with respect to $(z_3, z_4)$ with some $z_3,z_4\in\mathcal{D}$ satisfying $(\im z_1)(\im z_4) >0$ and $(\im z_2)(\im z_3) >0$, the third estimate in \eqref{eq:B_decomp_bounds} is improved to
		\begin{equation} \label{eq:A_decomp_bound}
			|b_t| \lesssim \langle |B|^2\rangle^{1/2}\frac{|z_1-z_4|+|z_2-z_3|}{\eta_{1,t}+\eta_{2,t}+|z_1-z_2|}.
		\end{equation} 
	\end{lemma}	
	Note that Lemma \ref{lemma:B_renorm} is only applicable for $t \in [T-T_*, T]$, however, in the complementary regime $t \in [0, T-T_*]$, the asymptotic \eqref{eq:eta_asymp} implies that $\eta_{t} \sim 1$. Hence, it follows from the local law \eqref{eq:Weak2av} and \eqref{eq:int_rules}, that, uniformly in $t \in [0,T]$,
	\begin{equation} \label{eq:(1,1)_short_time}
		\int_0^{t\wedge(T-T_*) } \bigl\lvert\bigl\langle (G_{[1,2],s} - M_{[1,2],s}) \mathscr{S}[M_{[2,1],s}] \bigr\rangle \bigr\rvert \mathrm{d}s \prec \int_0^{t\wedge(T-T_*) } \frac{\langle |A|^2\rangle^{1/2}\norm{B}}{N\eta_{1,s}\eta_{2,s}\eta_s} \prec \frac{\langle |A|^2\rangle^{1/2}\norm{B}}{N}. %TODO::clean this up a bit
	\end{equation}
	Hence, our goal is to prove that for $z_1,z_2 \in \mathcal{D}$ satisfying $(\im z_1)(\im z_2) < 0$, the bound
	\begin{equation} \label{eq:(1,1)_tr_goal}
		\int_{T-T_*}^t\bigl\langle (G_{[1,2],s} - M_{[1,2],s}) \mathscr{S}[M_{[2,1],s}] \bigr\rangle \bigr\rvert  \mathrm{d}s 
		\prec \frac{\langle |A|^2\rangle^{1/2}\norm{B}}{\sqrt{N\eta_t}\sqrt{\eta_t}}\biggl(\frac{1 + \phi_1}{\sqrt{N\eta_t}} + \frac{\sqrt{\phi_{2}^{\mathrm{hs}}}}{(N\eta_t)^{1/4}}\biggr),
	\end{equation}
	holds uniformly in $t \in [T-T_*, T_*]$, for which we can apply Lemma \ref{lemma:B_renorm}. Let $T' := T - T_*$.
	%Equipped with Lemma \ref{lemma:B_renorm}, we are ready to estimate the first term on the right-hand side of \eqref{eq:(GAGB-M)_evol}. 
	Since $(\im z_1)(\im z_2) < 0$, we have $\widehat{\vect{z}}_{1,s} = \vect{z}_{1,s}$. Decomposing the matrix $Y_s:=\mathscr{S}[M_{[2,1],s}]$ into $Y_s = \mathring{Y}_s + y_s \diag{\vect{z}_{1,s}-\vect{z}_{2,s}}$ according to Lemma \ref{lemma:B_renorm}, we obtain for all $T' \le s \le T$,
	\begin{equation} \label{eq:(1,1)av_tr_decomp}
		\bigl\langle (G_{[1,2],s} - M_{[1,2],s}) \mathscr{S}[M_{[2,1],s}] \bigr\rangle = \bigl\langle (G_{[1,2],s} - M_{[1,2],s}) \mathring{Y}_s \bigr\rangle + y_s \bigl\langle (G_{[1,2],s} - M_{[1,2],s}) (\vect{z}_{1,s}-\vect{z}_{2,s}) \bigr\rangle.
	\end{equation}
	It follows from \eqref{eq:M_bound}, \eqref{eq:super-S-norm}, and the estimates in \eqref{eq:B_decomp_bounds}, that $\mathring{Y}_s$ and $y_s$ satisfy, for all $T' \le s \le T$,
	\begin{equation} \label{eq:Y_bounds}
		\bigl\lVert\mathring{Y}_s\bigr\rVert \lesssim \norm{Y_s} = \norm{\mathscr{S}[M_{[2,1],s}]} \lesssim \eta_s^{-1}\norm{B}, \quad 
		|y_s| \lesssim (\eta_{1,s}+\eta_{2,s}+|z_1-z_2|)^{-1}\eta_s^{-1}\norm{B}.
	\end{equation}
	Since $\mathring{Y}_s$ is $(z_2,z_1)$-regular  for all $T'\le s \le T$, the integral of the first term on the right-hand side of \eqref{eq:(1,1)av_tr_decomp} admits the bound
	\begin{equation}
		\begin{split}
			\int_{T'}^t \bigl\lvert  \bigl\langle (G_{[1,2],s} - M_{[1,2],s}) \mathring{Y}_s \bigr\rangle \bigr\rvert  \mathrm{d}s 
			&\prec \int_{T'}^t \biggl(\frac{\langle |A|^2\rangle^{1/2}\langle |\mathring{Y}_s|^2\rangle^{1/2}}{\sqrt{N\eta_s}}\phi_{2}^\mathrm{hs}\biggr)^{1/2} \biggl(\frac{\langle |A|^2\rangle^{1/2}\bigl\lVert\mathring{Y}_s\bigr\rVert}{N\eta_s^2}\biggr)^{1/2}  \mathrm{d}s \\
			&\prec \int_{T'}^t \frac{\sqrt{\phi_{2}^\mathrm{hs}}}{(N\eta_s)^{3/4}} \frac{\langle |A|^2\rangle^{1/2}}{\sqrt{\eta_s}} \bigl\lVert\mathring{Y}_s\bigr\rVert \mathrm{d}s 
			%\\	&\prec \int_0^t \frac{\sqrt{\phi_{2}^{\mathrm{op}}}}{N\eta_s^{3/2}} \langle |A|^2\rangle^{1/2} \bigl\lVert\mathring{Y}_s\bigr\rVert \mathrm{d}s 
			\prec \frac{\langle |A|^2\rangle^{1/2}\norm{B}}{\sqrt{N\eta_t}\sqrt{\eta_t}}\frac{\sqrt{\phi_{2}^{\mathrm{hs}}}}{(N\eta_t)^{1/4}},
		\end{split}
	\end{equation}
	uniformly in $t\in[T',T]$, where we used $\Phi_2^\mathrm{hs} \prec \phi_2^\mathrm{hs}$ from \eqref{eq:phi_assume} to bound $|\langle (G_{[1,2],s} - M_{[1,2],s}) \mathring{Y}_s\rangle|^{1/2}$, and \eqref{eq:Weak2av} to bound the other $1/2$ power ignoring the regularity of $A$ and $\mathring{Y}_s$. We additionally used the integration rule \eqref{eq:int_rules}, and the first bound in \eqref{eq:Y_bounds}.
	
	We turn to bound the contribution of the second term on the right-hand side of \eqref{eq:(1,1)av_tr_decomp}. It follows from the vector Dyson equation \eqref{eq:VDE} and the definition \eqref{eq:M_def} that
	\begin{equation}
		\bigl\langle  M_{[1,2],s} (\vect{z}_{1,s}-\vect{z}_{2,s})  \bigr\rangle = \bigl\langle A(M_{1,s} - M_{2,s}) \bigr\rangle, \quad 
		s\in [T-T_*,T].
	\end{equation}
	Hence, using the generalized \eqref{eq:res_id}, we express the second term on the right-hand side of \eqref{eq:(1,1)av_tr_decomp} as
	\begin{equation} \label{eq:G-G_term}
		 y_s \bigl\langle (G_{[1,2],s} - M_{[1,2],s}) (\vect{z}_{1,s}-\vect{z}_{2,s}) \bigr\rangle =  y_s \bigl\langle (G_{1,s} - M_{1,s}) A \bigr\rangle - y_s \bigl\langle (G_{2,s} - M_{2,s}) A \bigr\rangle,
	\end{equation}
	and estimate each term separately.
	Evoking Lemma \ref{lemma:B_renorm} and \eqref{eq:A_decomp_bound} with $\{z_1,z_2,z_3,z_4\} := \{\bar{z}_1, z_1, z_1, z_2\}$, we decompose $A$ for all $ s\in[T',T]$,
	\begin{equation} \label{eq:Az_1z_1}
		A = \mathring{A}_s(\bar{z}_1,z_1) - 2 a_s(\bar{z}_1,z_1) \im\vect{z}_{1,s}, \quad |a_s(\bar{z}_1,z_1)|\lesssim \langle |A|^2\rangle^{1/2}\frac{|\bar{z}_1 -z_2|}{\eta_{1,s}},
	\end{equation}
	with $(z_1,\bar{z}_1)$-regular $\mathring{A}:=\mathring{A}_s(\bar{z}_1,z_1)$ satisfying $\langle |\mathring{A}_s(\bar{z}_1,z_1)|^2\rangle^{1/2}\lesssim \langle |A|^2\rangle^{1/2}$ by the second estimate in \eqref{eq:B_decomp_bounds}. Plugging \eqref{eq:Az_1z_1} into the first term on the right-hand side of \eqref{eq:G-G_term} and integrating in time, we obtain uniformly in $t\in[T',T]$,
	\begin{equation}
		\begin{split}
			\int_{T'}^t \bigl\lvert y_s \bigl\langle (G_{1,s} - M_{1,s}) A \bigr\rangle\bigr\rvert  \mathrm{d}s &\le \int_{T'}^t \frac{\bigl\lvert \bigl\langle (G_{1,s} - M_{1,s}) \mathring{A} \bigr\rangle \bigr\rvert + |a_s|\bigl\lvert \bigl\langle (G_{1,s} - M_{1,s}) \im\vect{z}_{1,s} \bigr\rangle \bigr\rvert}{\eta_{1,s}+\eta_{2,s} + |\bar{z}_1-z_2|}\frac{\norm{B}}{\eta_s}\mathrm{d}s\\
			&\prec \langle |A|^2\rangle^{1/2}\norm{B}\int_{T'}^t \biggl(\frac{\phi_{1}}{N\eta_s^{5/2}} + \frac{1}{N\eta_s^2}\biggr)\mathrm{d}s \prec \frac{\langle |A|^2\rangle^{1/2}\norm{B}}{\sqrt{N\eta_t}\sqrt{\eta_t}}\frac{1+\phi_{1}}{\sqrt{N\eta_t}}.
		\end{split} 
	\end{equation}
	where $a_s := a_s(\bar{z}_1,z_1)$, and we used the second estimate in \eqref{eq:Y_bounds} for $|y_s|$, the bound $|z_1 - z_2| \ge |\bar{z}_1 - z_2|$ which holds for all $z_1,z_2$ in opposite complex half-planes, and the integration rule \eqref{eq:int_rules}. The second term on the right-hand side of \eqref{eq:G-G_term} can be estimated using the same procedure. Therefore, \eqref{eq:(1,1)_tr_goal} is established and hence, together with \eqref{eq:(1,1)_short_time} we get, for all $z_1,z_2\in\mathcal{D}$ with $(\im z_1)(\im z_2) <0$,
	\begin{equation} \label{eq:(1,1)_tr_case2}
		\int_{0}^t\bigl\langle (G_{[1,2],s} - M_{[1,2],s}) \mathscr{S}[M_{[2,1],s}] \bigr\rangle \bigr\rvert  \mathrm{d}s 
		\prec \frac{\langle |A|^2\rangle^{1/2}\norm{B}}{\sqrt{N\eta_t}\sqrt{\eta_t}}\biggl(\frac{1 + \phi_1}{\sqrt{N\eta_t}} + \frac{\sqrt{\phi_{2}^{\mathrm{hs}}}}{(N\eta_t)^{1/4}}\biggr).
	\end{equation}
	
	The contribution of the remaining terms on the right-hand side of \eqref{eq:(GAGB-M)_evol} is estimated similarly to their counterparts in the proof of the master inequality \eqref{eq:Master2av}, and hence we  provide only a brief record.
	For the quadratic variation of the martingale term, using \eqref{eq:phi_assume}, \eqref{eq:imz_flat}, \eqref{eq:int_rules}, the norm bound \eqref{eq:G_norm}, and the generalized resolvent inequality \eqref{eq:ImG_ineq}, we obtain, uniformly in $ t \in [0,T]$,
	\begin{equation} \label{eq:(1,1)av_mart}
			\int_0^t\sum_{j,k}S_{jk}\bigl\lvert\partial_{jk}\bigl\langle G_{[1,2],s} B \bigr\rangle \bigr\rvert^2\mathrm{d}s
			%\lesssim \int\limits_0^s\frac{\norm{B}^2}{N^2\eta_{s}^4} \bigl\langle \im G_{1,r} A \im G_{2,r}A^* \bigr\rangle \mathrm{d}r
			\prec \frac{\langle |A|^2\rangle\norm{B}^2}{N^2\eta_t^3}\biggl(1 + \frac{\phi_{2}^{\mathrm{hs}}}{\sqrt{N\eta_t}}\biggr).
	\end{equation}
	This estimate is cruder than its analog in \eqref{eq:2av_mart}, since $B$ is a general observable and hence the simple estimate $\langle \im G_{1,s}AG_{2,s}B \im G_{1,s} B^* G_{2,s}^* A^* \rangle \lesssim \norm{B}^2(\eta_{1,s}\eta_{2,s})^{-1}\langle \im G_{1,s}A\im G_{2,s} A^* \rangle$ that follows from \eqref{eq:ImG_ineq} is affordable.  
	Using the martingale inequality \eqref{eq:mart_ineq}, we deduce from \eqref{eq:(1,1)av_mart} that, uniformly in $t\in[0,T]$,
	\begin{equation} \label{eq:(1,1)_mart}
		\sup_{0\le s\le t}\biggl\lvert \int_0^s\frac{1}{2}\sum_{j,k}\partial_{jk}\bigl\langle G_{[1,2],r} B \bigr\rangle \sqrt{S_{jk}}\mathrm{d}\Brwn_{jk,r} \biggr\rvert \prec \frac{\langle |A|^2\rangle^{1/2}\norm{B}}{\sqrt{N\eta_t}\sqrt{\eta_t}}\frac{1}{\sqrt{N\eta_t}}\biggl(1 + \frac{\sqrt{\phi_{2}^{\mathrm{hs}}}}{(N\eta_t)^{1/4}}\biggr).
	\end{equation}	

	The two terms in the second line of \eqref{eq:(GAGB-M)_evol} are bounded similarly to \eqref{eq:(1,1)_tr_case1}, using the second bound in \eqref{eq:regM_bound}.
	Next, uniformly in $t \in [0,T]$, we bound the quadratic term by using \eqref{eq:phi_assume}, \eqref{eq:Weak1iso} and \eqref{eq:int_rules}, 
	\begin{equation} \label{eq:(1,1)_quad}
			\int_0^t \bigl\lvert  \bigl\langle\mathscr{S}[G_{[1,2],s} - M_{[1,2],s}] (G_{[2,1],s} - M_{[2,1],s})\bigr\rangle \bigr\rvert   \mathrm{d}s
			\prec \frac{\langle |A|^2\rangle^{1/2}\norm{B}}{\sqrt{N\eta_t}\sqrt{\eta_t}}\frac{\phi_{(1,1)}}{\sqrt{N\eta_t} },
	\end{equation}
	again, since the observable $B$ is general (recall that $M_{[2,1],s} := G_{2,s}B G_{1,s}$), we can afford a simpler estimate than \eqref{eq:2av_S_term} and do not employ the decomposition of $S$ from Lemma \ref{lemma:S_decomp}.
	Finally, to estimate the last two terms on the right-hand side of \eqref{eq:(GAGB-M)_evol}, we observe that by definition of the super-operator $\mathscr{S}$ in \eqref{eq:superS_def}, the averaged local law in \eqref{eq:1G_laws} implies  that 
	\begin{equation} \label{eq:bound(1,2)av}
			\bigl\lvert \bigl\langle \mathscr{S}[G_{1,s} - M_{1,s}] G_{1,s} A G_{2,s} B G_{1,s} \bigr\rangle \bigr\rvert \prec \frac{1}{N\eta_{1,s}}\frac{1}{N}\sum_{j} \bigl\lvert \bigl\langle \vect{e}_j,  G_{1,s} A G_{2,s} B G_{1,s}\vect{e}_j\bigr\rangle \bigr\rvert.
	\end{equation}
	Next, using  the inequality \eqref{eq:ImG_ineq} for $\im G_{j,s}$, we obtain the following series of estimates
	\begin{equation} \label{eq:|3G|tr}
		\begin{split}
			\frac{1}{N}\sum_{j} \bigl\lvert \bigl\langle \vect{e}_j,  G_{1,s} A G_{2,s} B G_{1,s}\vect{e}_j\bigr\rangle \bigr\rvert
			&\lesssim \frac{\norm{B}}{N}\sum_{j} \frac{\bigl\langle \vect{e}_j, G_{1,s} A \im G_{2,s} A^* G_{1,s}^*\vect{e}_j\bigr\rangle^{1/2}}{\sqrt{\eta_{2,s}}} \frac{\bigl\langle \vect{e}_j, \im G_{1,s}\vect{e}_j\bigr\rangle^{1/2}}{\sqrt{\eta_{1,s}}}\\
			&\lesssim \frac{\norm{B}}{\eta_{1,s}\sqrt{\eta_{2,s}}} \bigl\langle \im G_{1,s} A \im G_{2,s} A^* \bigr\rangle^{1/2}\bigl\langle \im G_{1,s}\bigr\rangle^{1/2}.
		\end{split}
	\end{equation}
	We conclude the proof, by using \eqref{eq:int_rules} and inequalities \eqref{eq:bound(1,2)av}, \eqref{eq:|3G|tr} together with \eqref{eq:phi_assume}, Assumption \eqref{as:m_bound}, the averaged local law in \eqref{eq:1G_laws}, and \eqref{eq:M_bound}, to deduce that, uniformly in $t\in[0,T]$,
	\begin{equation} \label{eq:(1,1)_3-1-term}
			\int_0^t \bigl\lvert 	\bigl\langle\mathscr{S}[G_{1,s} - M_{1,s}] G_{1,s} A G_{2,s} B G_{1,s} \bigr\rangle\bigr\rvert   \mathrm{d}s 
			%\prec&~ \bigl\lvert	\bigl\langle(F-\m)\mathring{\mathscr{S}}[F A \other{F}BF]\bigr\rangle \bigr\rvert + \frac{1}{N\eta} \bigl\lvert \bigl\langle \vect{s} F A \other{F} B F \bigr\rangle \bigr\rvert \\
			\prec \frac{\langle |A|^2\rangle^{1/2}\norm{B}}{\sqrt{N\eta_t}\sqrt{\eta_t}}\frac{1}{\sqrt{N\eta_t}}\biggl(1+ \frac{\sqrt{\phi_{2}^\mathrm{hs}}}{(N\eta_t)^{1/4}}\biggr).
	\end{equation}
	The estimate for the remaining term is analogous. Collecting the estimates \eqref{eq:(1,1)_tr_case1}, \eqref{eq:(1,1)_tr_case2}, \eqref{eq:(1,1)_mart}, \eqref{eq:(1,1)_quad} and \eqref{eq:(1,1)_3-1-term} concludes the proof of \eqref{eq:Master(1,1)av}.
\end{proof}

\subsection{Reduction Inequalities. Proof of Lemma \ref{lemma:reds}} \label{sec:reds_proof}
For the remainder of this section, we consider the terminal time fixed and assume that \eqref{eq:phi_assume} holds uniformly in $t\in [0,T]$ and in $z_1,z_2\in\mathcal{D}$. 

Since the generalized resolvents at different vector-valued spectral parameters do not share the same spectral decomposition, as is the case of classical resolvent, the strategy for proving the reduction inequalities laid out in Appendix A.3 of \cite{Cipolloni2023Edge} no longer applies. Instead, we prove the inequalities \eqref{eq:Red4av_HS} and \eqref{eq:Red(2,1)av_HS} using submultiplicativity of trace $\langle X Y \rangle \le N \langle X\rangle \langle Y\rangle$ for positive semidefinite matrices $X,Y \ge 0$ in tandem with the novel integral representation for the generalized resolvent contained in Lemma \ref{lemma:imF_rep} below. Before presenting the proof of Lemma \ref{lemma:reds}, we state the main inputs.

To condense the presentation, for a fixed terminal time $T$, we define the \textit{flow map} $\mathfrak{f}^t \equiv \mathfrak{f}^t_T : \mathbb{H}\cup\mathbb{H}^* \to \mathbb{H}^N\cup (\mathbb{H}^*)^N$ for $t\in[0,T]$ by
\begin{equation} \label{eq:flow_map}
	\mathfrak{f}^t(z) := \vect{z}_{t},\quad \text{where} \quad \vect{z}_t \text{ solves \eqref{eq:z_evol} with } \vect{z}_T = z\,\vect{1}.
\end{equation}
It follows from \eqref{eq:z_evol} and \eqref{eq:m_evol} that the flow map $\mathfrak{f}^t$ admits the explicit expression 
\begin{equation} \label{eq:flow_explicit}
	\mathfrak{f}^t(z) = \mathrm{e}^{(T-t)/2} z\,\vect{1} + \bigl(1-\mathrm{e}^{(T-t)/2} \bigr) \bm{\mathfrak{a}} + 2\sinh\bigl(\tfrac{T-t}{2}\bigr) S[\m(z)].
\end{equation}

In view of Theorem 2.6 in \cite{Ajanki2019QVE} on the structure of the self-consistent density of states $\rho$, the set $\mathcal{D}$, defined in \eqref{eq:calD_def}, consists of a disjoint union of $K\sim 1$ rectangles of order one width. Therefore, the spectral domain $\mathcal{D}$ satisfies the following \textit{cone property} (see Figure \ref{fig:calD_figure}).
\begin{Def}[Cone Property]  \label{def:cone}
 We say that a domain $\other{\mathcal{D}} \subset \mathbb{C}$ satisfies the cone property if and only if there exists a positive constant $1 \lesssim \gamma \le \tfrac{1}{4}$ and an angle function $\omega : \mathcal{\other{D}} \to [-\tfrac{\pi}{2}\gamma, \tfrac{\pi}{2}\gamma]$, such that for any $z \in \mathcal{\other{D}}$, the (half) cone $\mathbb{V}_z \equiv \mathbb{V}_{z,\gamma,\omega(z)}$ with vertex $z$, aperture angle $\gamma\pi$, and tilt angle $\omega(z)$ between the axis and the positive imaginary direction, defined by
	\begin{equation} \label{eq:cone}
		\mathbb{V}_z \equiv \mathbb{V}_{z,\gamma,\omega(z)} := \bigl\{\zeta\in\mathbb{C} : \sign(\im z)\im[\mathrm{e}^{-\I\omega(z)}(\zeta - z)] \ge \cos(\tfrac{\pi}{2}\gamma) |\zeta - z| \bigr\},
	\end{equation}
	is contained entirely in $\mathcal{\other{D}}\cup\{\zeta\in\mathbb{C} : |\im \zeta| \ge \eta_*\}$ for some positive threshold $\eta_*$, i.e.,
	\begin{equation} \label{eq:cone_prop}
		\mathbb{V}_z \subset \mathcal{\other{D}}\cup\{\zeta\in\mathbb{C} : |\im \zeta| \ge \eta_*\}, \quad z \in \mathcal{\other{D}}.
	\end{equation}
\end{Def}
\begin{figure}[h]
	\centering
	\includegraphics[width=.8\textwidth]{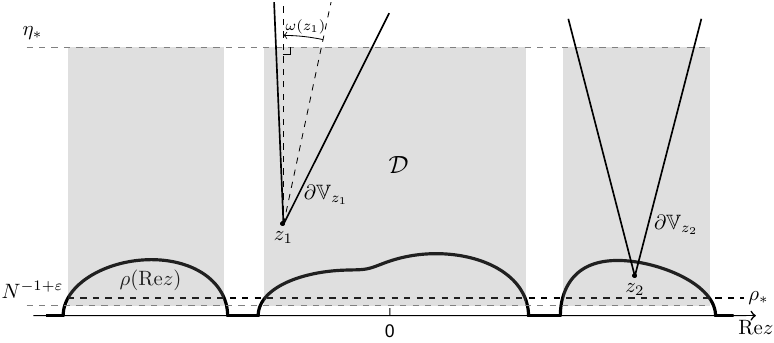}
	\caption{Shaded in gray is the part of a domain $\mathcal{D} = \mathcal{D}_{\rho_*,\eta_*,\varepsilon}$ lying in $\mathbb{H}$. The dashed gray lines indicate the horizontal cut-offs of $\mathcal{D}$, $\im z = N^{-1+\varepsilon}$ and $\im z = \eta_*$. Superimposed in black is the graph of the self-consistent density of states $\rho(\re z)$ supported on multiple intervals and the threshold $\rho_*$ that defines the vertical cut-offs of the domain $\mathcal{D}$. Furthermore, depicted are two points $z_1, z_2 \in \mathcal{D}$ and the boundaries of the corresponding cones $\mathbb{V}_{z_j}$ as in \eqref{eq:cone}. For the cone $\mathbb{V}_{z_1}$, we additionally include the tilt angle $\omega(z_1)$ between the axis on the cone and the vertical line (black, dashed). }
	\label{fig:calD_figure}
\end{figure}

%\begin{lemma}[Integral Representation via $\im G$] \label{lemma:imF_rep}
%	Let $X = X^*$ be an $N\times N$ matrix, let $z\in \mathcal{D}\cap\mathbb{H}$, then for all $t\in[0,T]$,
%	%	Fix $t \le T_*$, assume that $\vect{z} \in \mathcal{D}^t\cap\mathbb{H}^N$. Then there exists a small constant $c\sim 1$, such that for any $0 < \xi \le c\min\{1, \im \mathfrak{f}^t(\vect{z}) \}$, there exists a curve $\Gamma \equiv \Gamma_{\vect{z},\xi} : \mathbb{R} \to \mathcal{D}^t\cap\mathbb{H}^N$ satisfying
%	\begin{equation} \label{eq:ImFint}
%		G(X,\mathfrak{f}^t(z)) = \frac{1}{\pi} \int_{\mathbb{R}} \frac{\im G\bigl(X, (\mathfrak{f}^t\circ\psi)(x)\bigr)}{x - \I\xi^{1/\gamma}}\mathrm{d}x,
%	\end{equation}
%	where $\xi \in (0, \im z)$, $1 \lesssim \gamma \le \tfrac{1}{4}$ is the constant defined in \eqref{eq:cone}. Here $\psi : \overline{\mathbb{H}} \to \mathbb{H}$ is defined as
%	\begin{equation} \label{eq:psi_map}
%		\psi(u) \equiv \psi_{z,\xi}(u) := z + \mathrm{e}^{\I\omega(z)}\bigl(-\I\xi + \mathrm{e}^{\I\frac{\pi}{2}(1-\gamma)} u^\gamma\bigr),
%	\end{equation}
%	$\omega$ is the angle function from \eqref{eq:cone}, and  $u^\gamma := \exp(\gamma \log u)$, where we choose the branch of $\log$ cut along $[0,-\I\infty]$.
%	In particular, the function $\psi$ is continuous in $\overline{\mathbb{H}}$ and conformally maps $\mathbb{H}$ onto the interior of $\{-\I\mathrm{e}^{\I\omega(z)}\xi\} + \mathbb{V}_z$, interpreted as a Minkowski sum (see Figure \ref{fig:psi_map}).
%\end{lemma}
\begin{lemma}[ Conformal Integral Representation % via $\im G$
	] \label{lemma:imF_rep}
	Let $\other{\mathcal{D}}\subset{H}$ be a domain satisfying the cone property with aperture $\gamma$ and tilt angle function $\omega$, as in Definition \eqref{def:cone} above, and let $g : \mathbb{H} \to \mathbb{C}^{N\times N}$ be an analytic operator-valued function satisfying the decay condition $\norm{g(\zeta)} \lesssim |\im \zeta|^{-1}$ for $\zeta \in \other{\mathcal{D}}\cup\{|\im \zeta| \ge \eta_*\}$. Then for all $z \in \other{\mathcal{D}}$,
	\begin{equation} \label{eq:abstract_img}
		g(z) = \frac{1}{\pi}\int_\mathbb{R}\frac{\im g(\psi(x))}{x -\psi^{-1}(z)}\mathrm{d}x,
	\end{equation}
	where $\im g := \tfrac{1}{2\I}(g-g^*)$, and $\xi \in (0, \im z)$. Here $\psi : \overline{\mathbb{H}} \to \mathbb{H}$ is defined as
	\begin{equation} \label{eq:psi_map}
		\psi(u) \equiv \psi_{z,\xi}(u) := z + \mathrm{e}^{\I\omega(z)}\bigl(-\I\xi + \mathrm{e}^{\I\frac{\pi}{2}(1-\gamma)} u^\gamma\bigr),
	\end{equation}
	where $u^\gamma := \exp(\gamma \log u)$, and we choose the branch of $\log$ cut along $[0,-\I\infty]$.
	Moreover, the function $\psi$ is continuous in $\overline{\mathbb{H}}$ and conformally maps $\mathbb{H}$ onto the interior of $\{-\I\mathrm{e}^{\I\omega(z)}\xi\} + \mathbb{V}_z$, interpreted as a Minkowski sum (see Figure \ref{fig:psi_map}), with $\mathbb{V}_z$ defined in \eqref{eq:cone}.
	
	In particular, for any $N\times N$ matrix $X=X^*$ and any  $t\in[0,T]$, Lemma \ref{lemma:imF_rep} implies that for all $z \in \mathcal{D}\cap\mathbb{H}$, the generalized resolvent $G(X,\mathfrak{f}^t(z))$, defined in \eqref{eq:Gvect}, admits the integral representation
	\begin{equation} \label{eq:ImFint}
		G(X,\mathfrak{f}^t(z)) = \frac{1}{\pi} \int_{\mathbb{R}} \frac{\im G\bigl(X, (\mathfrak{f}^t\circ\psi)(x)\bigr)}{x - \I\xi^{1/\gamma}}\mathrm{d}x.
	\end{equation} 
\end{lemma} 
Note that we state Lemma \ref{lemma:imF_rep} only for $z\in\mathcal{D}\cap\mathbb{H}$ for simplicity, and the result can easily be extended to $z \in \mathcal{D}\cap\mathbb{H}^*$ by complex conjugation.
\begin{figure}[h] 
	\centering
	\includegraphics[width=.8\textwidth]{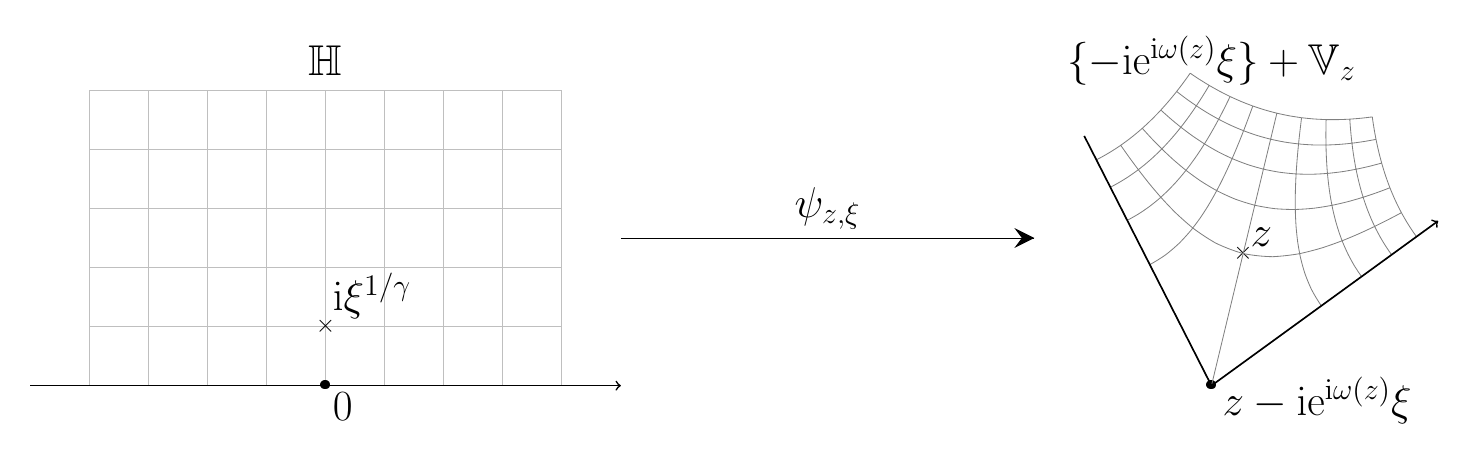}
	\caption{Depicted is the action of the conformal map $\psi \equiv \psi_{z,\xi}$, defined in \eqref{eq:psi_map}, from the complex upper half-plane $\mathbb{H}$ to the cone $\{-\I\mathrm{e}^{\I\omega(z)}\xi\} + \mathbb{V}_z$. Marked are the point $z = \psi(\I\xi^{1/\gamma})$ and the vertex $z-\I\mathrm{e}^{\omega(z)}\xi = \psi(0)$, together with their pre-images under $\psi$.}
	\label{fig:psi_map}
\end{figure}

\begin{remark} \label{rem:im_g_int}
	The main advantage of the integral representation \eqref{eq:ImFint} compared to the standard Stieltjes integral along a horizontal line, 
	\begin{equation} \label{eq:inG_old}
		G(z) = \frac{1}{\pi}\int_\mathbb{R}\frac{\im G\bigl(x+\I\im z-\I \xi\bigr)}{x-\I \xi}\mathrm{d}x, \quad z\in\mathbb{H},
	\end{equation}
	for $0  < \xi < \im z$, is that the argument of $\im G$ in \eqref{eq:ImFint} is restricted to the union of the bulk domain $\mathfrak{f}^t(\mathcal{D})$ (up to a tiny distance of order $\xi$) and the global scale domain $\mathfrak{f}^t(\{|\im \zeta| \gtrsim 1\})$. Hence, the representation \eqref{eq:ImFint} prevents the spectral parameter of $\im G$ from getting too close to the real line, falling below the critical scale $\eta(E)$ defined implicitly by $N\eta(E)\rho(E+\I\eta(E)) = N^{\varepsilon'}$. The key point is that the local laws are effective only above the critical curve, i.e., for $z$ satisfying $\im z \ge \eta(\re z)$. Since $\eta(E)$ increases as $\rho(E)$ becomes small, e.g., when $E$ approaches the spectral edges (it grows from $N^{-1}$ to $N^{-2/3}$), the integral contour in \eqref{eq:inG_old} can go below the critical curve, causing major complications (see Lemma 4.6 in \cite{Cipolloni2023Edge} and its proof).
	
	 Furthermore, the particular shape of the domain $\mathbb{V}_z$ in \eqref{eq:cone} was chosen for concreteness and because the corresponding conformal map $\psi_{z,\xi}$ admits a simple explicit expression \eqref{eq:psi_map}. However, the integral representation \eqref{eq:ImFint} can be extended to a wide class of unbounded domains $\mathbb{V}$ for which the conformal map $\psi : \mathbb{H} \to \mathbb{V}$ grows polynomially at infinity and, e.g., $|\im \psi(x)| \gtrsim |x|$ for $|x| \ge C \sim 1$. 
\end{remark}

\begin{proof} [Proof of Lemma \ref{lemma:imF_rep}]
	We recall the standard Schwarz integral formula in the upper-half plane. Let the matrix-valued function $f:\mathbb{H} \to \mathbb{C}^{N\times N}$ be analytic in $\mathbb{H}$, continuous in the closure $\overline{\mathbb{H}}$, and satisfy $|u|^\alpha\norm{f(u)} \le C$ for some positive constants $\alpha, C >0$, then
	\begin{equation} \label{eq:Schwarz_UPH}
		f(w) = \frac{1}{\pi} \int_\mathbb{R} \frac{\im f(x)}{x - w}\mathrm{d}x, \quad w \in \mathbb{H}.
	\end{equation}
	Note that substituting $f(u) := (g\circ \psi)(u)$ and $w := \psi^{-1}(z) = \I \xi^{1/\gamma}$ into \eqref{eq:Schwarz_UPH} yields \eqref{eq:ImFint} immediately. Since $\psi : \mathbb{H} \to \{-\I\mathrm{e}^{\I\omega(z)}\xi\} + \mathbb{V}_z$ is analytic and $\{-\I\mathrm{e}^{\I\omega(z)}\xi\} + \overline{\mathbb{V}_z} \subset \mathbb{H}$ for $0 < \xi < \im z$, the composite map $f := g\circ \psi$ is analytic in $\mathbb{H}$ and continuous in $\overline{\mathbb{H}}$. Therefore, it suffices to establish analyticity, continuity and polynomial decay of the map $f$ at infinity. By definition of $f$ and the decay of $g$, we have
	\begin{equation}
		\norm{f(u)} = \norm{g(\psi(u))} \lesssim |\im \psi(u)|^{-1} \lesssim \bigl(\im z - \xi + |u|^{\gamma}\bigr)^{-1},
	\end{equation} 
	where we used that $\psi(u) \in (\other{\mathcal{D}}\cup\{\im \zeta \ge \eta_*\})\cap\mathbb{H}$ by the assumed cone property. Here the last inequality follows immediately from \eqref{eq:psi_map} with $0 \le \gamma \le \tfrac{1}{4}$ and $|\omega(z)| \le \frac{\pi}{2}\gamma$, that for all $z\in\mathbb{H}$ and $0 < \xi < \im z$. Therefore, $|u|^\gamma \norm{f(u)} \lesssim 1$, and $f$ satisfies the assumptions of \eqref{eq:Schwarz_UPH}, hence \eqref{eq:abstract_img} is established.
	
	We now prove \eqref{eq:ImFint}.	It follows from \eqref{eq:flow_explicit} and the analyticity of the solution $\m(z)$ to \eqref{eq:VDEz} in $\mathbb{H}$, that the map $\zeta \mapsto \mathfrak{f}^t(\zeta)$ is entry-wise analytic in $\mathbb{H}$. Moreover, the generalized resolvent $\vect{w} \mapsto G(X,\vect{w})$ is analytic in the entries of $\vect{w} \in \mathbb{H}^N$.  Therefore, the composite map $g :\zeta \mapsto G(X,\mathfrak{f}^t(\zeta))$ is analytic in $\mathbb{H}$. Recall that by assumption $z \in \mathcal{D}\cap\mathbb{H}$. Starting with the bound \eqref{eq:G_norm}, we obtain for all $\zeta \in \mathbb{H}$,
	\begin{equation} \label{eq:G_decay}
		\norm{G\bigl(X,\mathfrak{f}^t(\zeta)\bigr)} \lesssim \bigl\langle \im \mathfrak{f}^t(\zeta) \bigr\rangle^{-1} \lesssim (\im \zeta)^{-1},
	\end{equation}
	where in the last inequality we used \eqref{eq:flow_explicit}. Hence $g := G(X,\mathfrak{f}^t(\cdot))$ satisfies the assumptions of \eqref{eq:abstract_img}, and therefore \eqref{eq:ImFint} holds. This concludes the proof of Lemma \ref{lemma:imF_rep}.
\end{proof}

 Representation \eqref{eq:ImFint} allows us to focus on studying chains containing only imaginary parts of resolvents. In particular, to estimate chains of length two containing only $\im G$'s and regular observables inside the integral, we rely on the following technical lemma. 
\begin{lemma}[$\im G$ Chain Integral Bound] \label{lemma:integral2G}
	Let $z_1\in\mathcal{D}$, let $z_2\in\mathcal{D}\cap\mathbb{H}$, and let $A$ be a $(z_2,z_1)$-regular observable. Fix $\xi:=N^{-100}$, and let $\Gamma := \mathfrak{f}^t\circ\psi_{z_2,\xi}$, where $\psi_{z_2,\xi}$ is the function defined in \eqref{eq:psi_map}, then
	\begin{equation} \label{eq:integral2G}
		\int_{\mathbb{R}} \frac{\bigl\lvert\bigl\langle \im G(H_t, \Gamma(x)) A \im G_{1,t} A^*\bigr\rangle \bigr\rvert}{|x-\I\xi^{1/\gamma}|}\mathrm{d}x \prec \langle |A|^2\rangle \biggl(1+\frac{\phi_{2}^{\mathrm{hs}}}{\sqrt{N\eta_t}}\biggr),
	\end{equation}
	uniformly in $t\in[0,T]$, where $G_{1,t}:= G(H_t, \mathfrak{f}^t(z_1))$, we recall $\eta_t := \min\{\eta_{1,t}, \eta_{2,t}\}$ and $\eta_{j,t} := |\langle \im \mathfrak{f}^t(z_j) \rangle|$.
%	Moreover, for all bounded deterministic vectors $\vect{x}\in\mathbb{C}^N$,
%	\begin{equation} \label{eq:integral1G}
%		\int_{\mathbb{R}} \frac{\bigl\langle \vect{x},\im G(H_s, \Gamma(x)) \,\vect{x}\bigr\rangle}{|x-\I\xi^{1/\gamma}|}\mathrm{d}x \prec \norm{\vect{x}}_2^2,
%	\end{equation}
%	\begin{equation} \label{eq:integral1G}
%		\int_{\mathbb{R}} \frac{\bigl\langle \vect{x},\im F(H, \Gamma(x))\,\vect{x}\bigr\rangle}{|x-\psi^{-1}(\I\xi)|}\mathrm{d}x \prec \norm{\vect{x}}_2^2.
%	\end{equation}
\end{lemma}

We defer the proof of Lemma \ref{lemma:integral2G} to Section \ref{sec:int_rep}. We are ready to prove the reduction inequalities of Lemma \ref{lemma:reds}. For brevity, we drop the dependence of $G_{j,t}$, $\vect{z}_{j,t}$ and $\eta_t$ on the time $t$, as it is fixed throughout the proof. 
\begin{proof}[Proof of Lemma \ref{lemma:reds}]
	 First, we prove \eqref{eq:Red4av_HS}.
	Without loss of generality, we can assume that $z_2 \in \mathcal{D}\cap\mathbb{H}$, otherwise we use the fact that $G(X,\overline{\vect{z}}) = G(X,\vect{z})^*$. 
	Define $\mathcal{G}(x) := \im G(H_t, \Gamma(x))$, where $\Gamma := \mathfrak{f}^t\circ\psi_{z_2,\xi}$ with $\xi:=N^{-100}$ and $\psi_{z_2,\xi}$ is given by \eqref{eq:psi_map}, then
	\begin{equation} \label{eq:4av_int}
		\begin{split}
			\bigl\langle \im G_1 A_2 G_2 A_1 \im G_1 A_1^* G_2^* A_2^* \bigr\rangle
			&= \frac{1}{\pi^2}\iint_{\mathbb{R}^2} \frac{\bigl\lvert\bigl\langle \im G_1 A_2 \mathcal{G}(x) A_1 \im G_1 A_1^* \mathcal{G}(y) A_2^* \bigr\rangle \bigr\rvert}{(x-\I\xi^{1/\gamma})(y+\I\xi^{1/\gamma})}\mathrm{d}x\mathrm{d}y\\
			&\lesssim N\biggl(\int_{\mathbb{R}} \frac{\bigl\lvert\bigl\langle \mathcal{G}(x) A_1 \im G_1 A_1^*\bigr\rangle\bigr\rvert^{1/2} \bigl\lvert\bigl\langle \im G_1 A_2 \mathcal{G}(x) A_2^* \bigr\rangle\bigr\rvert^{1/2}}{|x-\I\xi^{1/\gamma}|}\mathrm{d}x\biggr)^2\\
			&\lesssim N\int_{\mathbb{R}} \frac{\bigl\lvert\bigl\langle \mathcal{G}(x) A_1 \im G_1 A_1^*\bigr\rangle\bigr\rvert}{|x-\I\xi^{1/\gamma}|}\mathrm{d}x \int_{\mathbb{R}} \frac{ \bigl\lvert\bigl\langle \im G_1 A_2 \mathcal{G}(x) A_2^* \bigr\rangle\bigr\rvert}{|x-\I\xi^{1/\gamma}|}\mathrm{d}x.
		\end{split}
	\end{equation}
	Here, in the first step, we used the Cauchy-Schwarz inequality and submultiplicativity of trace. The reduction inequality \eqref{eq:Red4av_HS} follow immediately from \eqref{eq:4av_int} and \eqref{eq:integral2G}.
	
	Now, we prove \eqref{eq:Red(2,1)av_HS}. Similarly to the proof of \eqref{eq:Red4av_HS} above, we write
	\begin{equation}\label{eq:(2,1)av_int}
		\begin{split}
			\bigl\lvert \bigl\langle G_1 A_1 G_2 A_2 &G_1 B  \bigr\rangle \bigr\rvert
			\le \frac{1}{\pi}\int_\mathbb{R} \frac{\bigl\lvert \bigl\langle G_1 A_1 \mathcal{G}(x) A_2 G_1 B  \bigr\rangle \bigr\rvert}{|x - \I\xi^{1/\gamma}|} \mathrm{d}x\\
			\lesssim&~ \frac{\norm{B}}{\eta}\int_\mathbb{R} \frac{\bigl\lvert\bigl\langle \im G_1 A_1 \mathcal{G}(x) A_1^* \bigr\rangle\bigr\rvert^{1/2}\bigl\lvert
			\bigl\langle \mathcal{G}(x) A_2 \im G_1 A_2^*  \bigr\rangle\bigr\rvert ^{1/2} }{|x - \I\xi^{1/\gamma}|} \mathrm{d}x\\
			\lesssim&~ \frac{\norm{B}}{\eta}\biggl(\int_\mathbb{R} \frac{\bigl\lvert\bigl\langle \im G_1 A_1 \mathcal{G}(x) A_1^* \bigr\rangle \bigr\rvert}{|x - \I\xi^{1/\gamma}|} \mathrm{d}x\biggr)^{1/2}\biggl(\int_\mathbb{R} \frac{\bigl\lvert\bigl\langle \mathcal{G}(x) A_2 \im G_1 A_2^*  \bigr\rangle \bigr\rvert}{|x - \I\xi^{1/\gamma}|} \mathrm{d}x\biggr)^{1/2},
		\end{split}
	\end{equation}
	where in the second step we used the inequality \eqref{eq:ImG_ineq} for $\im G_1$.
	Hence the reduction inequality \eqref{eq:Red(2,1)av_HS} follows immediately from \eqref{eq:(2,1)av_int} and \eqref{eq:integral2G}.
\end{proof}

\section{Proof of Auxiliary Results} 
We record the preliminary properties on the solution to the vector Dyson equation \eqref{eq:VDEz} that were obtained in \cite{Ajanki2019QVE} under the Assumptions \eqref{as:S_flat} and \eqref{as:m_bound}. We state the properties for the spectral parameter $z \in \mathbb{H}$, but using the definition $\m(\bar{z}) := \overline{\m(z)}$, they can be extended to the lower half-plane $\mathbb{H}^*$. 
\begin{lemma} (Properties of $\m$)
	Let $\m(z)$ be the solution to the VDE \eqref{eq:VDEz} (with scalar $z$).
	Provided that Assumptions \eqref{as:S_flat} and \eqref{as:m_bound} hold for the data pair $(\bm{\mathfrak{a}},S) \in \mathbb{R}^N\times\mathbb{R}_+^{N\times N}$, the solution vector $\m$ satisfies the following properties:
	\begin{itemize}
		\item[(i)] (Theorems 2.1, 2.6, and Proposition 7.1 in \cite{Ajanki2019QVE} The map $\m$ is analytic in $\mathbb{H}$ and uniformly $1/3$-H\"{o}lder continuous in $\overline{\mathbb{H}}$, that is
		\begin{equation} \label{eq:m_holder}
			\norm{\m(z_1) - \m(z_2)}_\infty \lesssim |z_1-z_2|^{1/3}, \quad z_1,z_2\in\overline{\mathbb{H}}.
		\end{equation}
		Moreover, the self-consistent density $\rho$ defined in \eqref{eq:rho_def} is compactly supported, that is, there exists a constant $C\sim 1$, such that 
		\begin{equation} \label{eq:bulk_bounded}
			|E| \le C, \quad E \in \supp{\rho}.
		\end{equation}
		\item[(ii)] (Proposition 5.4 in \cite{Ajanki2019QVE}) The solution $\m$ satisfies the bound
		\begin{equation} \label{eq:m_lower}
			|\m(z)| \sim (1 + |z|)^{-1}\vect{1}.
		\end{equation}
		Moreover, the components of $\im \m$ are comparable in size,
		\begin{equation} \label{eq:im_m_flat}
			\im \m(z) \sim \langle \im \m(z)\rangle\vect{1}, \quad z \in \mathbb{H}.
		\end{equation}
	\end{itemize} 
\end{lemma}

Finally, we note that the flow map  $\mathfrak{f}^t$, defined in \eqref{eq:flow_map} is uniformly Lipschitz-continuous,
\begin{equation} \label{eq:flow_continuity}
	\norm{\mathfrak{f}^t(\zeta_1) - \mathfrak{f}^t(\zeta_2)}_\infty \lesssim |\zeta_1-\zeta_2|, \quad \zeta_1,\zeta_2 \in \mathcal{D}'.
\end{equation}
in  the domain $\mathcal{D}' \supset \mathcal{D}\cap\mathbb{H}$,
\begin{equation} \label{eq:calD'_def}
	\mathcal{D}' \equiv \mathcal{D}'_{\rho_*,\eta_*,\varepsilon} := \{\zeta \in \mathbb{H}: \rho(\re\zeta) \ge \tfrac{1}{2}\rho_*, N^{-1+\varepsilon/2} \le \im \zeta \le \eta_* \},
\end{equation} 
Indeed, subtracting two copies of \eqref{eq:VDEz} yields the identity
\begin{equation} \label{eq:m_diff_identity}
	\m(\zeta_1) - \m(\zeta_2) = (\zeta_1-\zeta_2)\stab_{\zeta_1,\zeta_2}^{-1}[\m(\zeta_1) \m(\zeta_2)].
\end{equation}
Hence, it follows from Assumption \eqref{as:m_bound}, and the first estimate \eqref{eq:scalar_stab_bounds}, that
\begin{equation} \label{eq:m_continuity}
	\norm{\m(\zeta_1) - \m(\zeta_2)}_\infty \lesssim |\zeta_1-\zeta_2|, \quad \zeta_1,\zeta_2 \in \mathcal{D}',
\end{equation}
which implies \eqref{eq:flow_continuity} by \eqref{eq:flow_explicit}.

\subsection{Resolvent Integral Bound. Proof of Lemma \ref{lemma:integral2G}} 
\label{sec:int_rep}

\begin{proof}[Proof of Lemma \ref{lemma:integral2G}] 
	Denote $\mathcal{G}(x) := \im G(H_t, \Gamma(x))$, $G_1 := G_{1,t}$, $\eta_1:=\eta_{1,t}$, and $\eta := \eta_t$. Since $\im G(X,\mathfrak{f}^t(z)) = - \im G(X,\mathfrak{f}^t(\bar z))$, we can assume without loss of generality that $\im z_1 < 0$ and $\im z_2 > 0$.
	
	We note that for $0 \le t \le T-T^*$, where $T_*\sim 1$ is the threshold from Lemma \ref{lemma:B_renorm}, \eqref{eq:eta_asymp} implies that $\eta_t \sim 1$, hence the estimate \eqref{eq:integral2G} follows trivially from the second bound in \eqref{eq:regM_bound}, the local law \eqref{eq:Weak2av}, and the bound $\norm{A} \le \sqrt{N}\langle |A|^2\rangle^{1/2}$. Therefore, we assume that $0 \le T-t \le T^*$ for the remainder of the proof.
	To bound the integral on the left-hand side of \eqref{eq:integral2G},  we split the $x$ integration into two regimes. We define the set $I$ as
	\begin{equation} \label{eq:int_I}
		I \equiv I_{z_2,\xi} := \{x\in\mathbb{R}:|\im \psi(x)| \le \eta_*\} \subset [-(2\eta_*)^{1/\gamma}, (2\eta_*)^{1/\gamma}],
	\end{equation}
	where $\psi := \psi_{z_2,\xi}$ is defined in \eqref{eq:psi_map}, 
	 and the inclusion follows immediately from \eqref{eq:psi_map}. 

	First, in the regime $x\notin I$, we  use the norm-bound \eqref{eq:G_norm} for $\mathcal{G}(x)$ to obtain, 
	\begin{equation} \label{eq:red_far}
		\int_{\mathbb{R}\backslash I} \frac{\bigl\lvert\bigl\langle \mathcal{G}(x) A \im G_{1} A^*\bigr\rangle\bigr\rvert}{|x-\I\xi^{1/\gamma}|}\mathrm{d}x \lesssim \int_{\mathbb{R}\backslash I} \frac{\bigl\lvert\bigl\langle \im G_{1} A^*A\bigr\rangle\bigr\rvert}{(\eta_* + |x|^\gamma)(|x|+\xi^{1/\gamma})}\mathrm{d}x \prec \langle |A|^2\rangle\log N \prec \langle |A|^2\rangle.
	\end{equation}
	Here, in the second step we employed the spectral decomposition of the hermitian matrix $|A|^2 = A^*A \ge 0$ and the isotropic local law \eqref{eq:1G_laws} to deduce the bound
	\begin{equation} \label{eq:1G_<A^2>_bound}
		\bigl\lvert\bigl\langle \im G_1 |A|^2 \bigr\rangle\bigr\rvert = N^{-1}\sum_p \sigma_p(|A|^2) \bigl\lvert\bigl\langle \vect{u}_p^{A}, \im G_1 \vect{u}_p^{A}\bigr\rangle\bigr\rvert \prec \langle |A|^2\rangle,
	\end{equation}
	where $\sigma_p(|A|^2)$, $\vect{u}_p^{A}$ are the eigenvalues and the corresponding eigenvectors of $|A|^2$.  Note that we need to use the isotropic local law since using the corresponding averaged law in \eqref{eq:1G_laws} would yield an error term controlled in terms of the operator norm $\norm{A^*A}$ of the observable.
	
	Next, we consider $x\in I$. Define the map $w:\mathbb{R} \to \mathbb{H}$ by
	\begin{equation} \label{eq:other_psi}
		w(x) := \psi(x) + \mathrm{e}^{\I\omega(z_2)}\I\xi = z_2 + \mathrm{e}^{\I\frac{\pi}{2}(1-\gamma)+\I\omega(z_2)} u^\gamma.
	\end{equation}
	It follows from the cone property \eqref{eq:cone_prop} of the domain $\mathcal{D}$  and the fact that $|\omega(z)| \le \tfrac{\pi}{2}\gamma$, $0 \le \gamma \le \tfrac{1}{4}$, that the map $w:\mathbb{R} \to \mathbb{H}$, defined in \eqref{eq:other_psi}, satisfies
	\begin{equation} \label{eq:other_psi_prop}
		w(x) \in \mathcal{D},\quad 	\bigl\lvert w(x) - \psi(x)\bigr\rvert \le \xi, \quad \im [w(x)-z_2] \gtrsim |w(x)-z_2| \quad x \in I.
	\end{equation}
	Recall that $\xi = N^{-100}$, so the distance between $w(x)$ and $\psi(x)$ is practically negligible. The only reason $\xi$ was introduced in \eqref{eq:ImFint} was to regularize a logarithmically divergent $1/|x|$ singularity in the integral.
	We rewrite the integral on the left-hand side of \eqref{eq:integral2G} over $x\in I$ as
	\begin{equation} \label{eq:2Gint_Ipart}
		\int_{I} \frac{\bigl\lvert\bigl\langle \mathcal{G}(x) A \im G_{1} A^*\bigr\rangle\bigr\rvert}{|x-\I\xi^{1/\gamma}|}\mathrm{d}x 
		\le \int_{I} \frac{\bigl\lvert \bigl\langle (\mathcal{G}-\other{\mathcal{G}})(x) A \im G_{1} A^*\bigr\rangle \bigr\rvert}{|x-\I\xi^{1/\gamma}|}\mathrm{d}x
		+ \int_{I} \frac{\bigl\lvert\bigl\langle \other{\mathcal{G}}(x) A \im G_{1} A^*\bigr\rangle\bigr\rvert}{|x-\I\xi^{1/\gamma}|}\mathrm{d}x,
	\end{equation}
	where $\other{\mathcal{G}}(x) := \im G(H_t, (\mathfrak{f}^t\circ w)(x))$.  Note that $\other{\mathcal{G}}(x)$ differs from $\mathcal{G}(x)$ only in replacing $\psi(x)$ with $w(x)$ in the argument of the flow map, and hence the first term on the right-hand side of \eqref{eq:2Gint_Ipart} is negligible. 
	
	To bound the first term on the right-hand side of \eqref{eq:2Gint_Ipart} rigorously, we observe that the map $\zeta \mapsto G(X, \mathfrak{f}^t(\zeta))$  satisfies the Lipschitz continuity property,
	\begin{equation}
		\norm{G(X, \mathfrak{f}^t(\zeta_1)) - G(X, \mathfrak{f}^t(\zeta_2))} \lesssim N^{2-\varepsilon} |\zeta_1 - \zeta_2|, \quad \zeta_1,\zeta_2 \in \mathcal{D}',
	\end{equation}
	where $\mathcal{D}'$ is defined in \eqref{eq:calD'_def}.
	Indeed, using the generalized resolvent identity \eqref{eq:res_id}, the norm-bound \eqref{eq:G_norm}, and \eqref{eq:flow_continuity}, we deduce that 
	\begin{equation} \label{eq:G_Lip}
		\begin{split}
			\norm{G\bigl(X, \mathfrak{f}^t(\zeta_1)\bigr) - G(X, \mathfrak{f}^t(\zeta_2))} &= \norm{G(X, \mathfrak{f}^t(\zeta_1))} \norm{\mathfrak{f}^s(\zeta_1) - \mathfrak{f}^t(\zeta_2)}_\infty \norm{G(X, \mathfrak{f}^t(\zeta_2))}\\
			&\lesssim  (\im \zeta_1)^{-1} (\im \zeta_2)^{-1} |\zeta_1-\zeta_2| \le N^{2-\varepsilon}|\zeta_1-\zeta_2|.
		\end{split} 
	\end{equation}
	Recalling that $\xi = N^{-100}$, we conclude from \eqref{eq:1G_<A^2>_bound}, \eqref{eq:other_psi_prop} and \eqref{eq:G_Lip}, that
	\begin{equation} \label{eq:2Gint_diff_term}
		\int_{I} \frac{\bigl\lvert \bigl\langle (\mathcal{G}-\other{\mathcal{G}})(x) A\im G_{1} A^*\bigr\rangle \bigr\rvert}{|x-\I\xi^{1/\gamma}|}\mathrm{d}x \le \int_{I} N^{2-\varepsilon}\xi\frac{\bigl\lvert\bigl\langle \im G_{1} |A|^2\bigr\rangle\bigr\rvert }{|x-\I\xi^{1/\gamma}|}\mathrm{d}x \prec \langle |A|^2\rangle  N^{-98-\varepsilon}\log N.
	\end{equation}
	
	We turn to bound the second integral on the right-hand side of \eqref{eq:2Gint_Ipart}.  The key idea is to decompose the $(z_2,z_1)$-regular observable $A$ into an ($x$-dependent) part regular with respect to $(w(x),z_1)$ and a small correction parallel to $\diag{\mathfrak{f}^t(z_1) - (\mathfrak{f}^t\circ w)(x)}$, that can be dealt with using the resolvent identity \eqref{eq:res_id}.  Denote $\vect{z}_1 := \mathfrak{f}^t(z_1)$, $\vect{w}(x) := (\mathfrak{f}^t\circ w)(x)$. Applying the observable regularization Lemma \ref{lemma:B_renorm} with $\{z_1,z_2,z_3,z_4, B\} := \{z_1, w(x), z_2, z_1, A\}$, we decompose
	\begin{equation} \label{eq:A_decomp}
		A = \mathring{A}(x) + a(x)\Delta\vect{z}(x),\quad \Delta\vect{z}(x) := \diag{\vect{z}_1 - \vect{w}(x)},
	\end{equation}
	where $a(x) := a_t(z_1,w(x))$, and the matrix $\mathring{A}(x) := \mathring{A}_t(z_1,w(x))$ is regular with respect to $(w(x),z_1)$. Since $A$ is $(z_2, z_1)$-regular by assumption, the estimate \eqref{eq:A_decomp_bound} implies
	\begin{equation} \label{eq:small_a_bound}
		|a(x)| \lesssim \frac{|w(x) - z_2|}{\eta_1 + \other{\eta}(x) + |w(x) - z_1|},
	\end{equation} 
	where $\other{\eta}(x) := \langle \im \vect{w}(x)\rangle$.
	Therefore, using the Schwarz inequality, we deduce that 
	\begin{equation} \label{eq:red_local}
		\begin{split}
			\int_{I} \frac{\bigl\lvert\bigl\langle \other{\mathcal{G}}(x) A \im G_1 A^*\bigr\rangle\bigr\rvert}{|x-\I\xi^{1/\gamma}|}\mathrm{d}x 
			\lesssim&~ \int_{I} \frac{\bigl\lvert\bigl\langle \other{\mathcal{G}}(x) \mathring{A}(x) \im G_1 \mathring{A}(x)^*\bigr\rangle \bigr\rvert}{|x-\I\xi^{1/\gamma}|}\mathrm{d}x\\
			&+ \int_{I} \frac{\bigl\lvert\bigl\langle \other{\mathcal{G}}(x) \Delta \vect{z}(x) \im G_1 \Delta \vect{z}^*(x)\bigr\rangle\bigr\rvert}{|x-\I\xi^{1/\gamma}|}|a(x)|^2\mathrm{d}x.
		\end{split}
	\end{equation}
	Since $\mathring{A}(x)$ is $(w(x),z_1)$-regular, it follows from \eqref{eq:phi_assume}, the second bound in  \eqref{eq:regM_bound},  and the second bound in \eqref{eq:B_decomp_bounds}, that the first integral on the right-hand side of \eqref{eq:red_local} is stochastically dominated by $(1+(N\eta)^{-1/2}\phi_{2}^{\mathrm{hs}})\langle |A|^2\rangle$.
	
	Therefore, it remains to bound the second integral on the right-hand side of \eqref{eq:red_local}.
	A direct calculation shows that for any $\bm{\zeta}_1, \bm{\zeta}_2\in\mathbb{H}^N \cup (\mathbb{H}^*)^N$, the resolvent identity \eqref{eq:res_id} implies that, with $\Delta\bm{\zeta} := \mathrm{diag}(\bm{\zeta}_1 - \bm{\zeta}_2)$,
	\begin{equation}\label{eq:imG_Dz_imG_Dz}
			\bigl\langle \im G(\bm{\zeta}_2) \Delta\bm{\zeta} \im G(\bm{\zeta}_1) \Delta\bm{\zeta}  \bigr\rangle 
			%=& -\frac{1}{4}\bigl\langle \bigl(G(\bm{\zeta}_1)-G(\bm{\zeta}_1)^*\bigr)\Delta\bm{\zeta} \bigl(G(\bm{\zeta}_2)-G(\bm{\zeta}_2)^*\bigr) \Delta\bm{\zeta}^*  \bigr\rangle\\
			= \bigl\langle \im G(\bm{\zeta}_2) (\im\Delta\bm{\zeta}^*)  \bigr\rangle + \im\bigl\langle G(\bm{\zeta}_2)^* \diag{\im \bm{\zeta}_2} G(\bm{\zeta}_1) \Delta\bm{\zeta}^*  \bigr\rangle.
	\end{equation}	
%	\color{red} [ I am certain this is finally correct. Here's the full derivation: $X := G(\bm{\zeta}_1)$, $Y := G(\bm{\zeta}_2)$, $x := \diag{\bm{\zeta}_1}$, $y := \diag{\bm{\zeta}_2}$, $\Delta := x-y$ satisfy $X(\Delta)Y = X-Y = Y(\Delta)X$. ]
%	\begin{equation}
%		\begin{split}
%			\langle \im Y (\Delta) \im X (\Delta)^* \rangle =&~ (2\I)^{-2} \biggl( \langle Y (x-y) X (x^*-y^*) \rangle - \langle Y (x-y) X^* (x^*-y+y-y^*) \rangle\\
%			&- \langle Y^* (x-y^*+y^*-y) X (x-y)^* \rangle + \langle Y^* (x-y) X^* (x^*-y^*) \rangle \biggr)\\
%			=&~ (2\I)^{-2} \biggl( \langle (X-Y) (x-y)^* \rangle - \langle (x-y) (X^* - Y)\rangle - \langle Y (y-y^*) X^* (x-y)^* \rangle\\
%			&- \langle (X-Y^*) (x^*-y^*) \rangle - \langle Y^* (y^*-y) X (x^*-y^*) \rangle + \langle  (x-y) (X^* - Y^* ) \rangle \biggr)\\
%			=&~ (2\I)^{-2} \biggl( \langle (Y^*-Y) (x^*-y^*-x+y) \rangle 
%			- \bigl\langle (y-y^*) \bigl(X^* (x-y)^*Y - X (x^*-y^*)Y^*\bigr) \bigr\rangle \biggr)\\
%			=&~ \langle \im Y \im\Delta \rangle 
%			+\im\bigl\langle (\im y) X (x^*-y^*)Y^* \bigr\rangle 
%		\end{split}
%	\end{equation}
%	\color{black}
%	\begin{equation}
%		\Delta\bm{\zeta}^* := \mathrm{diag}(\overline{\bm{\zeta}_1} - \overline{\bm{\zeta}_2}) = \diag{\overline{\bm{\zeta}_1} - \bm{\zeta}_2} + 2\I \diag{\im \bm{\zeta}_2}
%	\end{equation}
	Therefore, applying the bounds \eqref{eq:m_upper}, \eqref{eq:M_bound}, and  the averaged local laws in \eqref{eq:1G_laws} and \eqref{eq:Weak2av} to the right-hand side of \eqref{eq:imG_Dz_imG_Dz} with $\bm{\zeta}_1 := \vect{z}_1$, $\bm{\zeta}_2 := \vect{w}(x) = (\mathfrak{f}^t\circ w)(x)$, we obtain
	\begin{equation}
		\begin{split}
			\bigl\lvert \bigl\langle \other{\mathcal{G}}(x)\Delta \vect{z}(x) \im G_1 \Delta\vect{z}^*(x)\bigr\rangle \bigr\rvert 
			\prec&~ \norm{\Delta\vect{z}(x)}\biggl(1 + \frac{1}{N\other{\eta}(x)}\biggr) + \frac{\norm{\Delta\vect{z}(x)}\norm{\im\vect{w}(x)}_\infty}{\other{\eta}(x)}\biggl(1 + \frac{1}{N\eta_1}\biggr),
		\end{split}
	\end{equation}
	where $\other{\eta}(x) := \langle \im \vect{w}(x) \rangle \ge \im z_2 \ge N^{-1+\varepsilon}$. It follows from \eqref{eq:imz_flat} that $\lVert \im\vect{w}(x) \rVert_\infty \sim \other{\eta}(x)$, hence
	\begin{equation} \label{eq:renorm_leftovers}
		\bigl\lvert \bigl\langle \other{\mathcal{G}}(x)\Delta \vect{z}(x) \im G_1 \Delta\vect{z}^*(x)\bigr\rangle \bigr\rvert  \prec \norm{\Delta\vect{z}(x)}.
	\end{equation}
	The comparison \eqref{eq:imz_flat} implies that $\norm{\im \vect{z}_1}_\infty \sim \eta_1$ and $\norm{\im \vect{w}(x)}_\infty \sim \other{\eta}(x)$, hence using \eqref{eq:flow_continuity}
	%, \eqref{eq:other_psi},
	we conclude  that
	\begin{equation}\label{eq:Dz_bound}
		\norm{\Delta \vect{z}(x)} 
		\le \norm{\im \vect{z}_1}_\infty + \norm{\im \vect{w}(x)}_\infty + \norm{\mathfrak{f}^t(\bar z_1) - (\mathfrak{f}^t\circ w)(x)}_\infty
		\lesssim \eta_{1} + \other{\eta}(x) + |\bar{z}_1 - w(x)|.
	\end{equation}
	Therefore, combining estimates \eqref{eq:small_a_bound} for $a(x)$, \eqref{eq:renorm_leftovers} and \eqref{eq:Dz_bound}, we obtain
	\begin{equation} \label{eq:2G_leftovers}
		\int_{I} \frac{\bigl\lvert\bigl\langle \other{\mathcal{G}}(x) \Delta \vect{z}(x) \im G_1 \Delta \vect{z}^*(x)\bigr\rangle\bigr\rvert}{|x-\I\xi^{1/\gamma}|}|a(x)|^2\mathrm{d}x \prec  \int_I  \frac{\langle |A|^2\rangle |w(x)-z_2|^2}{\eta_{1}+\other{\eta}(x)+|z_1-w(x)|}\frac{\mathrm{d}x}{|x| + \xi^{1/\gamma}} \prec \langle |A|^2\rangle,
	\end{equation}
	where in the last inequality we used $|w(x)-z_2| = |x|^\gamma$ and $\other{\eta}(x) \ge \im w(x) \gtrsim |w(x)-z_2|$ that follow from \eqref{eq:other_psi} and \eqref{eq:other_psi_prop}, respectively. Combining the bounds \eqref{eq:red_far},  \eqref{eq:2Gint_Ipart}, \eqref{eq:2Gint_diff_term}, \eqref{eq:red_local}, and \eqref{eq:2G_leftovers} yields \eqref{eq:integral2G}.
	This concludes the proof of Lemma \ref{lemma:integral2G}.
\end{proof}

\subsection{Observable Regularization. Proof of Lemmas \ref{lemma:isotropic}, \ref{lemma:S_decomp} and \ref{lemma:B_renorm}} \label{sec:A_reg_proofs}
We close this section by proving the observable decomposition lemmas. We record the following asymptotic expansion for the smallest eigenvalue of the stability operator $\stab$, defined in \eqref{eq:stab_def}, that we prove in Appendix \ref{sec:stab_section}.
\begin{lemma} \label{lemma:Pi}
	Let $z_1,z_2 \in \mathcal{D}$, defined in \eqref{eq:calD_def}, satisfy $(\im z_1)(\im z_2) < 0$ and  $|z_1-z_2| \le \tfrac{1}{2}\delta$, where $\delta$ is the threshold in Lemma \ref{lemma:stab_lemma}. Let $\eigB_{z_1,z_2}$ be the smallest eigenvalue of $\stab_{z_1, z_2}$, then
		\begin{equation}\label{eq:beta_expan}
			\eigB_{z_1,z_2} = \I\frac{z_1 - z_2}{\kappa(z_1)} + \mathcal{O}(|z_1 - z_2|^2), \quad \kappa(z):=\frac{2}{\langle\im \M{z}\rangle }\biggl\langle \frac{(\im \M{z})^2}{|\M{z}|^2} \biggr\rangle \sim 1, \quad z \in \mathcal{D}.
		\end{equation}
\end{lemma}

\begin{proof} [Proof of Lemma \ref{lemma:B_renorm}]
	By Definition \ref{def:reg_A}, in the regime $\min\{|z_1 - z_2|, |\bar{z}_1- z_2|\} > \tfrac{1}{2}\delta$, all observables are $(z_1,z_2)$-regular, hence the conclusion of Lemma \ref{lemma:B_renorm} is satisfied with $\mathring{B}_t := B$ and $b_t := 0$.	
	Therefore, it remains to consider the regime $\min\{|z_1 - z_2|, |\bar{z}_1- z_2|\} \le \tfrac{1}{2}\delta$.
	
	Let $\Pi_{z_2^-, z_1^+}$ be the eigenprojector corresponding to the smallest eigenvalue of $\stab_{z_2^-, z_1^+}$ as in Lemma \ref{lemma:stab_lemma}, with the spectral parameters $z_j^\pm = \re z_j \pm \I|\im z_j|$. Recall that by construction, $z_1^+ \in \mathbb{H}$ and $z_2^- \in \mathbb{H}^*$,  and $|z_1^+ - z_2^-| \le \tfrac{1}{2}\delta$.
	
	Denote the vector $\Delta\vect{z}_t := \widehat{\vect{z}}_{1,t} - \vect{z}_{2,t}$, where $\widehat{\vect{z}}_{1,t}$ is defined in \eqref{eq:B_decomp}, and $\m_j^\pm := \m(z_j^\pm)$. Since by Lemma \ref{lemma:stab_lemma}  $\mathrm{rank}\,\Pi_{z_2^-, z_1^+} = 1$,  with a slight abuse of notation (i.e., interpreting the ratio of two parallel vectors as a scalar), we define the complex number $b_t$ and the matrix $\mathring{B}_t$ as
	\begin{equation} \label{eq:b_decomp_def}
		b_t := \frac{\Pi_{z_2^-, z_1^+}[\m_1^+\m_2^-\vect{b}^{\mathrm{diag}}]}{\Pi_{z_2^-, z_1^+}[\m_1^+\m_2^-\Delta\vect{z}_t]}, \quad \mathring{B}_t := B - b_t \diag{\Delta\vect{z}_t},
	\end{equation}
	where we recall the notation $\vect{b}^\mathrm{diag} := (B_{jj})_{j=1}^N$. % and $B^\mathrm{od} := B - \mathrm{diag}(\vect{b}^\mathrm{diag})$.
	First, we show that  there exists a threshold $1 \lesssim T_* \le T$ such that for all times $T-T_* \le t \le T$,
	\begin{equation} \label{eq:a_denom_comp}
		\norm{\Pi_{z_2^-, z_1^+}[\m_1^+\m_2^-\Delta\vect{z}_t]}_\infty \sim \eta_{1,t}+ \eta_{2,t} + |z_1 - z_2|, %\gtrsim |\re[z-w]| + \im z +\im w + t$.%\norm{\Delta\vect{z}}.
	\end{equation}
	where we recall that $\eta_{j,t} := |\langle \im \vect{z}_{j,t} \rangle|$. It follows from \eqref{eq:flow_map} and \eqref{eq:flow_explicit} that 
	\begin{equation} \label{eq:z_diff_expr}
		\Delta\vect{z}_t = \mathrm{e}^{(T-t)/2} (\widehat{z}_1 - z_2)\vect{1} + 2\sinh\bigl((T-t)/2\bigr)S[\m(\widehat{z}_1) - \m(z_2)],
	\end{equation}
	where $\widehat{z}_1 := \re z_1 - \I \sign(\im z_2) |\im z_1|$. 
	It follows from \eqref{eq:m_continuity} that 
	\begin{equation}
		\m(\widehat{z}_1) - \m(z_2) = -\sign(\im z_2) \bigl(\m_1^+ - \m_2^-\bigr) + \mathcal{O}\bigl(\min\{|z_1 - z_2|, |\bar{z}_1- z_2|\}\bigr).
	\end{equation}
	Hence, using identity \eqref{eq:m_diff_identity} with $\zeta_1 := z_1^+$, $\zeta_2 := z_2^-$, the definition of the projector $\Pi_{z_2^-, z_1^+}$ in Lemma \ref{lemma:stab_lemma}, and the asymptotic for $\eigB_{z_2^-, z_1^+}$ from \eqref{eq:beta_expan}, we obtain
	\begin{equation} \label{eq:PI_of_m_diff}
		\Pi_{z_2^-, z_1^+}\bigl[\m_1^+\m_2^-S[\m(\widehat{z}_1) - \m(z_2)]\bigr] 
		%= \frac{\other{z}_1-\other{z}_2}{\eigB_{\other{z}_1,\other{z}_2}}\Pi_{\other{z}_1, \other{z}_2}[\other{\m}_1\other{\m}_2] + \mathcal{O}\bigl(|\other{z}_1-\other{z}_2|\bigr) 
		= -\I\sign(\im z_2)\kappa(z_2)\Pi_{z_2^-, z_1^+}[\m_1^+\m_2^-] + \mathcal{O}\bigl(|z_1^+-z_2^-|\bigr).
	\end{equation}
	Therefore, \eqref{eq:z_diff_expr} and \eqref{eq:PI_of_m_diff} imply
	\begin{equation} \label{eq:a_denom_expand}
		\Pi_{z_2^-, z_1^+}[\m_1^+\m_2^-\Delta\vect{z}_t] = \mathrm{e}^{(T-t)/2}K(z_1,z_2)\Pi_{z_2^-, z_1^+}[\m_1^+\m_2^-] + \mathcal{O}\bigl((T-t)|z_1^+ - z_2^-|\bigr),
	\end{equation}
	where the function $K(z_1,z_2)$ is defined as
	\begin{equation}
		K(z_1,z_2) := \widehat{z}_1 - z_2 - \I\sign(\im z_2)\kappa(z_2)\bigl(1 - \mathrm{e}^{t-T}\bigr).
	\end{equation}
	Since  $\kappa(z_2) \sim 1$ by \eqref{eq:beta_expan}, we deduce that
	\begin{equation}
		|\re K(z_1,z_2) | = |\re z_1 - \re z_2|, \quad |\im K(z_1,z_2)| \sim |\im z_1| + |\im z_2| + (T-t).
	\end{equation}
	It follows from \eqref{eq:eta_asymp} that $\eta_{j,t} \sim |\im z_j| + (T-t)$, hence $|K(z_1,z_2)| \sim \eta_{1,t}+\eta_{2,t} + |z_1-z_2|$. Therefore, there exists a threshold $1 \lesssim T_* \le T$ such that for all times $T-T_* \le t \le T$, the bound \eqref{eq:a_denom_comp} follows from \eqref{eq:a_denom_expand} and \eqref{eq:Pi_separation}.
	
	Next,we use the bound  $\lVert\vect{b}^{\mathrm{diag}}\rVert_1 \le N\langle |B|^2 \rangle^{1/2}$ together with Assumption \eqref{as:m_bound} and the estimate \eqref{eq:Pi_bound} to deduce that 
	\begin{equation} \label{eq:Pi_bdiag_bound}
		\norm{\Pi_{z_2^-, z_1^+}[\m_1^+\m_2^-\vect{b}^{\mathrm{diag}}]}_\infty \lesssim \langle |B|^2 \rangle^{1/2}.
	\end{equation}
	Combining \eqref{eq:Pi_bdiag_bound} with the relation \eqref{eq:a_denom_comp}, we immediately obtain the upper bound on $|b_t|$ in \eqref{eq:B_decomp_bounds} using the definition of $b_t$ in \eqref{eq:b_decomp_def}.
	Next, using \eqref{eq:eta_asymp}, \eqref{eq:m_continuity}, and \eqref{eq:z_diff_expr}, we obtain
	\begin{equation}
		\norm{\Delta\vect{z}_t}_\infty \lesssim \eta_{1,t} + \eta_{2,t} + |z_1 - z_2|.
	\end{equation}
	Hence the other two bounds in \eqref{eq:B_decomp_bounds} on the norms of $\mathring{B}_t$ follow from its definition in \eqref{eq:b_decomp_def} and the inequalities $\norm{\diag{\Delta\vect{z}_t}}_\mathrm{hs} \le \norm{\diag{\Delta\vect{z}_t}} \le \norm{\Delta\vect{z}_t}_\infty$.
	
	Finally, we show that the improved  bound \eqref{eq:A_decomp_bound} holds under the assumption that $B$ is $(z_3,z_4)$-regular for $z_3,z_4 \in \mathcal{D}$ lying in the same complex half-plane as $z_1$ and $z_2$, respectively. Once again, it suffices to consider the regime $\min\{|z_1 - z_3|, |z_2-z_4|\} \le \tfrac{1}{2}\delta$.
	To estimate the numerator in \eqref{eq:b_decomp_def}, we observe that %for $\other{z}_3 := \re z_3 - \I|\im z_3|$ and $\other{z}_4 := \re z_4 + \I|\im z_4|$, we have
	\begin{equation} \label{eq:a_bound_num}
		\norm{\Pi_{z_2^-, z_1^+}[\m_1^+\m_2^-\vect{b}^{\mathrm{diag}}] - \Pi_{z_3^-, z_4^+}[\m_3^-\m_4^+\vect{b}^{\mathrm{diag}}]}_\infty \lesssim \langle |B|^2\rangle^{1/2}\bigl(|z_1-z_3|+|z_2-z_4|\bigr),
	\end{equation}
	where we used the bounds \eqref{eq:m_continuity} and \eqref{eq:Pi_continuity}. Since $B$ is $(z_3,z_4)$-regular, $\Pi_{z_3^-, z_4^+}[\m_3^-\m_4^+\vect{b}^{\mathrm{diag}}] = 0$ by Definition \ref{def:reg_A}. Here we used the inequality $\norm{\vect{b}^\mathrm{diag}}_2 \le N^{1/2}\langle |B|^2 \rangle^{1/2}$, and $|z_1^+ - z_4^+| \le |z_1-z_4|$, $|z_2^- - z_3^-| \le |z_2-z_3|$ that follow the definition of $z_j^\pm$ and the conditions $(\im z_1)(\im z_4) > 0$, $(\im z_2)(\im z_3) >0$. 
	This concludes the proof of Lemma \ref{lemma:B_renorm}.
\end{proof}

The proof of Lemma \ref{lemma:S_decomp} is contained in Section 6.3 of \cite{R2023bulk}. However, for the sake of completeness and consistency, we present the argument with the notation of the present paper.
\begin{proof}[Proof of Lemma \ref{lemma:S_decomp}]
	Similarly to the proof of Lemma \ref{lemma:B_renorm}, it suffices to consider the regime $\min\{|z_1-z_2|, |\bar{z}_1-z_2|\} \le \tfrac{1}{2}\delta$, where $\delta$ is the threshold in Lemma \ref{lemma:stab_lemma}.
	Let $\vect{s}^{(p)} := (NS_{pj})_{j=1}^N$ for all $p \in \{1,\dots, N\}$ be the rows of the matrix $S$ multiplied by $N$ so that $\lVert\vect{s}^{(p)}\rVert_\infty \sim 1$ (by Assumption \eqref{as:S_flat}). Using the fact that $\mathrm{rank}\,\Pi_{z_2^-, z_1^+} = 1$ and \eqref{eq:Pi_separation}, we define, for all $p \in \{1,\dots, N\}$,
	\begin{equation} \label{eq:S_decomp_def}
		s_p := \frac{\Pi_{z_2^-, z_1^+}[\m(z_1^+)\m(z_2^-)\vect{s}^{(p)}]}{\Pi_{z_2^-, z_1^+}[\m(z_1^+)\m(z_2^-)]}, \quad \mathring{\vect{s}}^{(p)} := \vect{s}^{(p)} - s_p\vect{1},
	\end{equation}
	hence $\mathrm{diag}(\vect{s}^{(p)})$ are $(z_2,z_1)$-regular by Definition \ref{def:reg_A}.
	Defining $\mathring{S}(z_1,z_2)$ to be the matrix with rows $\mathring{\vect{s}}^{(p)}$, and setting $\vect{s} := (s_p)_{p=1}^N$, we obtain the decomposition \eqref{eq:S_decomp}.
	It follows immediately from the upper bound in Assumptions \eqref{as:S_flat}, \eqref{as:m_bound}, and  bounds \eqref{eq:Pi_bound}, \eqref{eq:Pi_separation}, that $|s_p| \lesssim 1$, and hence the second estimate in \eqref{eq:S_decomp_bounds} holds. The first estimate in \eqref{eq:S_decomp_bounds} follows trivially from the definition of $\vect{s}^{(p)}$ in \eqref{eq:S_decomp_def} and $\norm{\vect{s}} \lesssim 1$. This concludes the proof of Lemma \ref{lemma:S_decomp}.
\end{proof}

\begin{proof}[Proof of Lemma \ref{lemma:isotropic}]
	First, we observe that for $B:=N\vect{y}\vect{x}^*$, we have the identity
	\begin{equation}
		\langle \vect{x}, (G_{1,t}A_1G_{2,t}-M_t)\vect{y} \rangle = \langle (G_{1,t}A_1G_{2,t}-M_t)B \rangle.
	\end{equation}
	In the regime $\min\{|z_1-z_2|, |\bar{z}_1-z_2|\} > \tfrac{1}{2}\delta$, the observable $B$ is $(z_1,z_2)$-regular, hence the statement of Lemma \ref{lemma:isotropic} holds with the choice $A_2 := N^{-1/2}B$ and $a := 0$. 
	In the complementary regime  $\min\{|z_1-z_2|, |\bar{z}_1-z_2|\} \le \tfrac{1}{2}\delta$, we use the fact that $\mathrm{rank}\,\Pi_{z_2^-, z_1^+} = 1$ and the lower bound \eqref{eq:Pi_separation} to define $a$ and $A_2$ by
	\begin{equation} \label{eq:xy_decomp}
		a := N \frac{\Pi_{z_2^-, z_1^+}[\m(z_1^+)\m(z_2^-)\,\vect{y}\,\overline{\vect{x}}]}{\Pi_{z_2^-, z_1^+}[\m(z_1^+)\m(z_2^-)]}, \quad A_2 := \sqrt{N}\bigl(\vect{y}\vect{x}^* - N^{-1}a \, I\bigr).
	\end{equation}
	In particular, the bounds \eqref{eq:Pi_bound} and \eqref{eq:Pi_separation} imply that
	\begin{equation}
		|a| \lesssim  \norm{\overline{\vect{x}}\vect{y}}_1 \lesssim \norm{\vect{x}}_2\norm{\vect{y}}_2.
	\end{equation}
	Finally, by construction \eqref{eq:xy_decomp}, the matrix $A_2$ is $(z_1,z_2)$-regular and satisfies
	\begin{equation}
		\langle |A_2|^2 \rangle^{1/2} \lesssim \sqrt{N}\langle|\vect{y}\vect{x}^*|^2\rangle^{1/2} + |a| \lesssim \norm{\vect{x}}_2\norm{\vect{y}}_2.
	\end{equation}
	This concludes the proof of Lemma \ref{lemma:isotropic}.
\end{proof}

\section{Green Function Comparison. Proof of Proposition \ref{prop:GFT}} \label{sec:GFT}
The inductive argument laid out is Section 5 of \cite{Cipolloni2023Edge} is a robust approach to Green Function Comparison that does not depend on the specific random matrix ensemble, and requires only that the desired local laws hold for a random matrix with some entry distribution with the first three matching moments (that we accomplished in Proposition \ref{prop:flow_local_laws}). While in \cite{Cipolloni2023Edge} the procedure is performed for chains of arbitrary length, our Theorem \ref{th:locallaws} only involves chains of length up to two. 
%The Green Function Comparison argument necessary to prove Proposition \ref{prop:GFT} can be directly imported from Section 5 of \cite{Cipolloni2023Edge}. Indeed, the proof in Section 5 of \cite{Cipolloni2023Edge} does not depend on the specific random matrix ensemble, and requires only that the desired local laws hold for a random matrix with some entry distribution with the first three matching moments (that we accomplished in Proposition \ref{prop:flow_local_laws}). 
However, to close the argument for the averaged two-resolvent local law \eqref{eq:2G_av_reg}, we need an isotropic local law for symmetric chains of length three, which is the content of Proposition \ref{prop:2iso} below. The remainder of the Green Function Comparison argument for isotropic chains of length up to three and averaged chains of length up to two can be imported directly from Sections 5.1 and 5.2 of \cite{Cipolloni2023Edge}, respectively.
\begin{prop}[Isotropic Local Law for Three Resolvents] \label{prop:2iso}
	Fix $\varepsilon > 0$, then under the notation and notation of Proposition \ref{prop:global_laws}, the isotropic local law
	\begin{equation} \label{eq:3G_iso_t}
		\bigl\lvert \bigl\langle \vect{x}, (G_{1,t} A_1 G_{2,t} A_2 G_{1,t} - \M{\vect{z}_{1,t}, A_1, \vect{z}_{2,t}, A_2, \vect{z}_{1,t}})\vect{y} \bigr\rangle \bigr\rvert  \prec \frac{N\langle |A_1|^2\rangle^{1/2}\langle |A_2|^2\rangle^{1/2}\norm{\vect{x}}_2\norm{\vect{y}}_2}{\sqrt{N\eta_t}}
	\end{equation}
	holds uniformly in time $0\le t \le T$, in regular observables $A_1$, $A_2$, in spectral parameters $z_1,z_2\in\mathcal{D}$, and in deterministic vectors $\vect{x}$, $\vect{y}$.  
	Here, the deterministic approximation $\M{\vect{z}_{1,t}, A_1, \vect{z}_{2,t}, A_2, \vect{z}_{1,t}}$ is defined as
	\begin{equation} \label{eq:M2iso_def}
		\M{\vect{z}_{1,t}, A_1, \vect{z}_{2,t}, A_2, \vect{z}_{1,t}} := \bigl(1 - M_{1,t}M_{2,t}\mathscr{S} \bigr)^{-1}\biggl[M_{1,t} \bigl(A_1 + \mathscr{S}[M_{[1,2],t}]\bigr)M_{[2,1],t}\biggr],
	\end{equation}
	where $M_{[1,2],t} := \M{\vect{z}_{1,t},A_1,\vect{z}_{2,t}}$, $M_{[2,1],t} := \M{\vect{z}_{2,t},A_2,\vect{z}_{1,t}}$ are defined in \eqref{eq:M_def}, and recall that $M_{j,t} := \diag{\m(\vect{z}_{j,t})}$ from \eqref{eq:Mt_def}.
\end{prop}
We postpone the proof of Proposition \ref{prop:2iso} until the end of this section. Furthermore, we record the following bounds on the deterministic approximation to a chain containing three resolvents interlaced with regular observables $A_1$, $A_2$, defined in \eqref{eq:M2iso_def},
\begin{equation} \label{eq:M_3_bounds}
	\begin{split}
		\norm{\M{\vect{z}_{1,t}, A_1, \vect{z}_{2,t}, A_2, \vect{z}_{1,t}}} &\lesssim N\langle |A_1|^2\rangle^{1/2}\langle |A_2|^2\rangle^{1/2},\\
		\langle |\M{\vect{z}_{1,t}, A_1, \vect{z}_{2,t}, A_2, \vect{z}_{1,t}}|^2\rangle &\lesssim N\langle |A_1|^2\rangle \langle |A_2|^2\rangle,
	\end{split}
\end{equation}
that follow from the identity \eqref{eq:d-od_stab_identity}, the first estimate in \eqref{eq:scalar_stab_bounds}, the bounds \eqref{eq:regM_bound}, \eqref{eq:super-S-norm}, submultiplicativity of trace, and the trivial inequality $\norm{A_j} \le \sqrt{N}\langle |A_j|^2 \rangle^{1/2}$.
\begin{proof}[Proof of Proposition \ref{prop:GFT}]
	For $1\le j \le k\le 3$, we use the notation $G_{[j,k],t} := \prod_{p=j}^{k-1}(G_{p,t}A_p)G_{k,t}$ introduced in Eq. (4.5) of \cite{Cipolloni2023Edge}, where we define $G_{3,t} := G_{1,t}$ for convenience of indexing, and let  $M_{[j,k],t}$ denote the corresponding deterministic approximations, as defined in \eqref{eq:M_def} for $k=j+1$, in \eqref{eq:M2iso_def} for $k=j+2$, and in \eqref{eq:Mt_def} for $j=k$. Under this convention, for $k \in \{1,2\}$,  we define the isotropic and averaged control quantities
	\begin{equation} 
		\begin{split}
			\Psi_k^\mathrm{iso}(\vect{x},\vect{y}) &:= \frac{\sqrt{N\eta_T}}{N^{k/2}}\bigl\lvert \bigl\langle \vect{x},(G_{[1,k+1],T}-M_{[1,k+1],T})\vect{y} \bigr\rangle \bigr\rvert,\\
			\Psi_k^\mathrm{av} &:= \frac{\sqrt{N\eta_T}}{N^{k/2-1}}\bigl\lvert \bigl\langle (G_{[1,k],T}-M_{[1,k],T})A_{4-k} \bigr\rangle \bigr\rvert,
		\end{split}
	\end{equation}
	where $\vect{x}, \vect{y}$ are deterministic unit vectors in $\mathbb{C}^N$, and the observables $A_1$, $A_2$, $A_3$ with $\langle |A_j|^2\rangle^{1/2} = 1$, are $(z_1,z_2)$, $(z_2,z_1)$ and $(z_1,z_1)$-regular, respectively, in the sense of Definition \ref{def:reg_A}.
	Using the standard local laws for a single resolvent of a Wigner-type matrix (Theorem 2.5 in \cite{Erdos2018CuspUF}),
	\begin{equation} \label{eq:1G_laws_old}
		\bigr\lvert \langle G(z) - M(z)  \rangle \bigr\rvert \prec \frac{1}{N|\im z|}, \quad \bigr\lvert \langle \vect{x}, \bigl(G(z) - M(z)\bigr)\vect{y}  \rangle \bigr\rvert \prec \frac{\norm{\vect{x}}_2 \norm{\vect{y}}_2}{\sqrt{N|\im z|}},
	\end{equation}
	as the base, together with the isotropic local laws \eqref{eq:2G_iso_t} and \eqref{eq:3G_iso_t} at time $t=T$, the bounds \eqref{eq:regM_bound} and \eqref{eq:M_3_bounds}, the induction argument laid out in Section 5.2 of \cite{Cipolloni2023Edge} yields
	\begin{equation} \label{eq:Psi_isos}
		\Psi_k^\mathrm{iso}(\vect{x},\vect{y}) \prec 1, \quad k \in \{1,2\}.
	\end{equation}
	Therefore, using the averaged local laws from \eqref{eq:2G_av_t}, \eqref{eq:1G_av_t}, \eqref{eq:1G_laws_old}, the bounds \eqref{eq:regM_bound}, \eqref{eq:M_3_bounds}, and the induction argument for the averaged quantities $\Psi_k^\mathrm{av}$ in Section 5.3 of \cite{Cipolloni2023Edge}, we obtain
	\begin{equation}
		\Psi_k^\mathrm{av} \prec 1, \quad k \in \{1,2\}.
	\end{equation}
	This concludes the proof of Proposition \ref{prop:GFT}.
\end{proof}
The remainder of the section is dedicated to proving Proposition \ref{prop:2iso}.
\begin{proof}[Proof of Proposition \ref{prop:2iso}]		
	We prove that the global-law version of \eqref{eq:3G_iso_t} holds at time $t=0$ in Appendix \ref{sec:global_laws}.	
	Here we show how to propagate it to the local scale.
	Without loss of generality, we assume that $\norm{\vect{x}}_2=\norm{\vect{y}}_2 = 1$.
	Differentiating the definition \eqref{eq:M2iso_def} in time and using \eqref{eq:m_evol}, \eqref{eq:M_evol}  yields
	\begin{equation}
			\partial_t M_{[1,3],t} = \frac{3}{2}M_{[1,3],t} + \other{M}_{1,t} + \other{M}_{2,t} + \other{M}_{3,t},
	\end{equation}
	where we denote $\other{M}_{1,t} := \M{\vect{z}_{1,t}, \mathscr{S}[M_{[1,2],t}], \vect{z}_{2,t}, A_2, \vect{z}_{1,t}}$, $\other{M}_{2,t} := \M{\vect{z}_{1,t}, A_1, \vect{z}_{2,t}, \mathscr{S}[M_{[2,3],t}], \vect{z}_{1,t}}$, and  $\other{M}_{3,t} := \M{\vect{z}_{1,t}, \mathscr{S}[M_{[1,3],t}],\vect{z}_{2,t}}$.
	Using It\^{o}'s formula together with \eqref{eq:OUflow}, we obtain (dropping all subscripts $t$) 
	\begin{equation} \label{eq:2iso_evol}
		\begin{split}
			\mathrm{d}\bigl\langle \vect{x}, (G_{[1,3]}-M_{[1,3]}) \vect{y}\bigr\rangle = 
			&~  \frac{1}{2}\sum_{j,k}\partial_{jk}\bigl\langle \vect{x}, G_{[1,3]}\vect{y} \bigr\rangle \sqrt{S_{jk}}\mathrm{d}\Brwn_{jk} +
			\frac{3}{2}\bigl\langle \vect{x} (G_{[1,3]}-M_{[1,3]}) \vect{y}\bigr\rangle\mathrm{d}t\\ 
			&+\sum_{1\le j\le k \le 3} \bigl\langle\vect{x},G_{[1,j]}\mathscr{S}[G_{[j,k]}-M_{[j,k]}]G_{[k,3]}  \vect{y}\bigr\rangle\mathrm{d}t \\
			&+ \bigl\langle \vect{x}, (G_{1} \mathscr{S}[M_{[1,3]}] G_{1}- \other{M}_{3}) \vect{y}\bigr\rangle\mathrm{d}t\\
			&+ \bigl\langle \vect{x}, (G_{[1,2]} \mathscr{S}[M_{[2,3]}] G_{1}- \other{M}_{2}) \vect{y}\bigr\rangle\mathrm{d}t \\
			& +\bigl\langle \vect{x}, (G_{1} \mathscr{S}[M_{[1,2]}] G_{[2,3]}- \other{M}_{1}) \vect{y}\bigr\rangle 
			\mathrm{d}t.
		\end{split}
	\end{equation}
%	\begin{equation}
%		\begin{split}
%			\mathrm{d}\bigl\langle \vect{x} (G_{[3],t}-M_{[3],t}) \vect{y}\bigr\rangle = 
%			&~  \frac{1}{2}\sum_{j,k}\partial_{jk}\bigl\langle \vect{x}, G_{[3],t}\vect{y} \bigr\rangle \sqrt{S_{jk}}\mathrm{d}\Brwn_{jk,t} +
%			\frac{3}{2}\bigl\langle \vect{x} (G_{[3],t}-M_{[3],t}) \vect{y}\bigr\rangle\\ 
%			&+\bigl\langle \vect{x}, G_{1,t} \mathscr{S}[G_{1,t}-\m_{1,t}] G_{[3],t} \vect{y}\bigr\rangle
%			+ \bigl\langle \vect{x}, G_{[3],t} \mathscr{S}[G_{1,t}-\m_{1,t}] G_{1,t} \vect{y}\bigr\rangle\\
%			&+ \bigl\langle \vect{x}, G_{1,t} A_1 G_{2,t} \mathscr{S}[G_{2,t}-\m_{2,t}] G_{2,t} A_2 G_{1,t} \vect{y}\bigr\rangle\\
%			&+ \bigl\langle \vect{x}, G_{1,t} \mathscr{S}[G_{1,t} A_1 G_{2,t} -\M{1,t}] G_{2,t} A_2 G_{1,t} \vect{y}\bigr\rangle \\
%			&+ \bigl\langle \vect{x}, G_{1,t} A_1 G_{2,t} \mathscr{S}[G_{2,t} A_2 G_{1,t} -\M{2,t}] G_{1,t} \vect{y}\bigr\rangle \\
%			&+ \bigl\langle \vect{x}, G_{1,t} \mathscr{S}[G_{[3],t}-M_{[3],t}] G_{1,t} \vect{y}\bigr\rangle \\
%			&+ \bigl\langle \vect{x}, (G_{1,t} \mathscr{S}[\M{1,t}] G_{2,t} A_2 G_{1,t}- \M{\vect{z}_{1,t}, \mathscr{S}[\M{1,t}], \vect{z}_{2,t}, A_2, \vect{z}_{1,t}}) \vect{y}\bigr\rangle\\
%			&+ \bigl\langle \vect{x}, (G_{1,t} A_1 G_{2,t} \mathscr{S}[\M{2,t}] G_{1,t}- \M{\vect{z}_{1,t}, A_1, \vect{z}_{2,t}, \mathscr{S}[\M{2,t}], \vect{z}_{1,t}}) \vect{y}\bigr\rangle\\
%			&+ \bigl\langle \vect{x}, (G_{1,t} \mathscr{S}[M_{[3],t}] G_{1,t}- \M{\vect{z}_{1,t}, \mathscr{S}[M_{[3],t}],\vect{z}_{2,t}}) \vect{y}\bigr\rangle.
%		\end{split}
%	\end{equation}
	We can now estimate the integral of each of the terms on the right-hand side of \eqref{eq:2iso_evol}.
	Note that the second term on the right-hand side of \eqref{eq:2iso_evol} can removed by differentiating $\mathrm{e}^{-3t/2}\bigl\langle \vect{x}, (G_{[1,3]}-M_{[1,3]}) \vect{y}\bigr\rangle$ with a harmless exponential factor $\mathrm{e}^{-3t/2} \sim 1$, therefore we omit this term from the analysis.
	
	First, we bound the quadratic variation of martingale term, using the Schwarz inequality and the operator inequality \eqref{eq:ImG_ineq},
	\begin{equation} \label{eq:2iso_mart}
		\begin{split}
			\sum_{j,k}S_{jk}\bigl\lvert\partial_{jk}\bigl\langle \vect{x}, G_{[1,3],s}\vect{y} \bigr\rangle\bigr\rvert^2 
			\lesssim&~ \frac{1}{N\eta_s^2} \bigl\langle \vect{x}, \im G_{1,s}\vect{x} \bigr\rangle \bigl\langle \vect{y},G_{[2,3],s}^*A_1^* \im G_{1,s} A_1G_{[2,3],s}\vect{y} \bigr\rangle\\
			&+\frac{1}{N\eta_s^2} \bigl\langle \vect{y}, \im G_{1,s}\vect{y} \bigr\rangle \bigl\langle \vect{x},G_{[1,2],s}A_2 \im G_{1,s} A_2^*G_{[1,2],s}^*\vect{x} \bigr\rangle\\
			&+ \frac{1}{N\eta_s^2} \bigl\langle \vect{x},G_{1,s}A_1 \im G_{2,s} A_1^* G_{1,s}^* \vect{x} \bigr\rangle \bigl\langle \vect{y},G_{1,s}^*A_2^* \im G_{2,s} A_2 G_{1,s} \vect{y} \bigr\rangle.
		\end{split}
	\end{equation}
To estimate the average trace of a chain containing five generalized resolvents, we apply the integral representation \eqref{eq:ImFint} of Lemma \ref{lemma:imF_rep} to the first and last resolvents, use the H\"{o}lder inequality and the submultiplicativity for trace, to deduce that 
\begin{equation} \label{eq:5G_iso_int}
	\begin{split}
		\bigl\lvert\bigl\langle \vect{y},G_{[2,3]}^*A_1^* \im G_{1} A_1G_{[2,3]}\vect{y} \bigr\rangle \bigr\rvert
		&\le \iint_{\mathbb{R}^2}\frac{\bigl\lvert \bigl\langle \vect{y},\mathcal{G}_1(x)A_2^*G_2^*A_1^* \im G_{1} A_1G_2A_2 \mathcal{G}_1(y)\vect{y} \bigr\rangle \bigr\rvert}{|x-\I\xi^{1/\gamma}||y-\I\xi^{1/\gamma}|}\mathrm{d}x\mathrm{d}y\\
		%
		%\le&~  \biggl\lvert \int_{\mathbb{R}}\frac{\bigl\langle \vect{y},\mathcal{G}_1(x)A_2^*G_2^*A_1^* \im G_{1} A_1G_2A_2 \mathcal{G}_1(x)\vect{y} \bigr\rangle^{1/2}}{|x-\I\xi^{1/\gamma}|}	\mathrm{d}x \biggr\rvert^2\\
		%
		&\le  N\int_{\mathbb{R}}\frac{\bigl\lvert\bigl\langle \vect{y},\mathcal{G}_1(x) \vect{y} \bigr\rangle  \bigr\rvert}{|x-\I\xi^{1/\gamma}|}	\mathrm{d}x\int_{\mathbb{R}}\frac{\bigl \langle A_2^*G_2^*A_1^* \im G_{1} A_1G_2A_2 \mathcal{G}_1(y) \bigr\rangle}{|y-\I\xi^{1/\gamma}|}	\mathrm{d}y,
		%
		%\prec&~  N^2\iint_{\mathbb{R}^2}\frac{\bigl \langle A_2^* \mathcal{G}_2(y)A_2 \mathcal{G}_1(x) \bigr\rangle}{|x-\I\xi^{1/\gamma}||y-\I\xi^{1/\gamma}|}	\mathrm{d}x\mathrm{d}y 		\int_{\mathbb{R}}\frac{\bigl \langle  \mathcal{G}_2(y) A_1^* \im G_{1} A_1\bigr\rangle}{|y-\I\xi^{1/\gamma}|}	\mathrm{d}y,
	\end{split}
\end{equation}
where $\mathcal{G}_j(x) \equiv \mathcal{G}_{j,s}(x):= \im G\bigl(H_s, (\mathfrak{f}^s\circ\psi_{z_j,\xi})(x) \bigr)$. Similarly to the proof of Lemma \ref{lemma:integral2G}, we bound the first integral in the last line of \eqref{eq:5G_iso_int} by considering the regimes $x\in I$ and $x\notin I$ separately, with $I := I_{z_1,\xi}$ defined in \eqref{eq:int_I}. The contribution from the integral over $\mathbb{R}\backslash I$ is stochastically dominated by $\norm{\vect{y}}_2^2$ due to the norm bound \eqref{eq:G_norm}. The contribution from $x\in I$ is stochastically dominated by $\norm{\vect{x}}_2^2(1+ (N\eta_{1,s})^{-1/2})$ owing the isotropic local law in \eqref{eq:1G_laws}.  

Next, to estimate the other integral in last line of \eqref{eq:5G_iso_int}, we apply the integral representation \eqref{eq:ImFint} of Lemma \ref{lemma:imF_rep} to the resolvents $G_2$ and $G_2^*$, use H\"{o}lder inequality and the submultiplicativity of trace again, to obtain
\begin{equation} \label{eq:4Gtrace_int}
	\bigl \langle A_2^*G_2^*A_1^* \im G_{1} A_1G_2A_2 \mathcal{G}_1(x) \bigr\rangle
	\le N \int_{\mathbb{R}}\frac{\bigl\lvert \bigl \langle A_2^* \mathcal{G}_2(y)A_2 \mathcal{G}_1(x) \bigr\rangle \bigr\rvert}{|y-\I\xi^{1/\gamma}|}	\mathrm{d}y \int_{\mathbb{R}}\frac{\bigl\lvert\bigl \langle  \mathcal{G}_2(y) A_1^* \im G_{1} A_1\bigr\rangle\bigr\rvert}{|y-\I\xi^{1/\gamma}|}	\mathrm{d}y.
\end{equation}
The second factor on the right-hand side of \eqref{eq:4Gtrace_int} is stochastically dominated by $\langle |A_1|^2\rangle$ by Lemma \ref{lemma:integral2G} and \eqref{eq:Phi_final}. For the first integral in \eqref{eq:4Gtrace_int}, owing to \eqref{eq:G_norm}, the contribution from the regime $y\notin I_{z_2,\xi}$ admits the bound
\begin{equation} \label{eq:y_notin_I}
	\int_{\mathbb{R}\backslash I_{z_2,\xi}}\frac{\bigl\lvert\bigl \langle A_2^* \mathcal{G}_2(y)A_2 \mathcal{G}_1(x) \bigr\rangle\bigr\rvert}{|y-\I\xi^{1/\gamma}|}	\mathrm{d}y \lesssim \bigl\lvert\bigl\langle \mathcal{G}_{1}(x)|A_2|^2\bigr\rangle\bigr\rvert \log N, \quad x \in \mathbb{R}
\end{equation}
In the complementary regime $y\in I_{z_2,\xi}$, we obtain the bound
\begin{equation} \label{eq:y_in_I}
	\int_{I_{z_2,\xi}}\frac{\bigl\lvert\bigl \langle A_2^* \mathcal{G}_2(y)A_2 \mathcal{G}_1(x) \bigr\rangle\bigr\rvert}{|y-\I\xi^{1/\gamma}|}	\mathrm{d}y \prec \biggl(1 + \frac{\mathds{1}_{x\notin I_{z_1,\xi}}}{\eta_* + |x|^\gamma}\biggr) \bigl\langle |A_2|^2\bigr\rangle\log N, \quad x \in \mathbb{R},
\end{equation}
where, in the regime  $x\notin I_{z_1,\xi}$, we used $\norm{\mathcal{G}_1(x)} \lesssim (\eta_* + |x|^\gamma)^{-1}$ and the estimate \eqref{eq:1G_<A^2>_bound}, while in the regime $x\in I_{z_1,\xi}$, we use \eqref{eq:Phi_final}, the $(z_2,z_1)$-regularity of $A_2$ and the argument as in the proof of Lemma \ref{lemma:integral2G}.
Hence, combining the bounds \eqref{eq:y_notin_I}--\eqref{eq:y_in_I}, we deduce that
\begin{equation}
	\iint_{\mathbb{R}^2}\frac{\bigl\lvert\bigl \langle A_2^* \mathcal{G}_2(y)A_2 \mathcal{G}_1(x) \bigr\rangle\bigr\rvert}{|x-\I\xi^{1/\gamma}||y-\I\xi^{1/\gamma}|}	\mathrm{d}y \mathrm{d}x \prec N\langle |A_2|^2 \rangle,
\end{equation}
and therefore
\begin{equation} \label{eq:5G_iso}
	\bigl\lvert\bigl\langle \vect{y},G_{[2,3]}^*A_1^* \im G_{1} A_1G_{[2,3]}\vect{y} \bigr\rangle\bigr\rvert \prec N^2\langle |A_1|^2\rangle \langle |A_2|^2\rangle.
\end{equation}
The term in the second line of \eqref{eq:2iso_mart} admits an analogous estimate.

For the term in the third line of \eqref{eq:2iso_mart}, we use \eqref{eq:Phi_final} and Lemmas \ref{lemma:imF_rep}--\ref{lemma:integral2G} to obtain
\begin{equation} \label{eq:3Giso_crude}
	\begin{split}
		\bigl\lvert\bigl\langle \vect{x},G_{1}A_1 \im G_2 A_1^* G_{1}^* \vect{x} \bigr\rangle\bigr\rvert \le&~ \int_{\mathbb{R}}\frac{\bigl\lvert\bigl\langle \vect{x}, \mathcal{G}_{1}(x)A_1 \im G_2 A_1^*  \mathcal{G}_{1}(x) \vect{x} \bigr\rangle\bigr\rvert^{1/2}}{|x-\I\xi^{1/\gamma}|}\mathrm{d}x\\
		\lesssim&~ N\int_{\mathbb{R}}\frac{\bigl\lvert\bigl\langle \vect{x}, \mathcal{G}_{1}(x)\vect{x} \bigr\rangle\bigr\rvert}{|x-\I\xi^{1/\gamma}|}\mathrm{d}x 
		\int_{\mathbb{R}}\frac{\bigl\lvert\bigl\langle  \mathcal{G}_{1}(x)A_1 \im G_2 A_1^* \bigr\rangle \bigr\rvert}{|x-\I\xi^{1/\gamma}|}\mathrm{d}x \prec N\langle |A_1|^2 \rangle,
	\end{split}
\end{equation}
and the other factor is estimated similarly.
Therefore, collecting the estimates \eqref{eq:5G_iso}, \eqref{eq:3Giso_crude} for the quadratic variation in \eqref{eq:2iso_mart}, using the martingale inequality \eqref{eq:mart_ineq} and the integration rule \eqref{eq:int_rules}, we obtain for all $0\le t \le T$,
\begin{equation} \label{eq:3iso_mart}
	\sup_{0\le s \le t}\biggl\lvert \int_0^s\sum_{j,k}\partial_{jk}\bigl\langle \vect{x}, G_{[1,3],r}\vect{y} \bigr\rangle \sqrt{S_{jk}}\mathrm{d}\Brwn_{jk,r} \biggr\rvert \prec N\frac{\langle |A_1|^2  \rangle^{1/2}\langle |A_2|^2 \rangle^{1/2} }{\sqrt{N\eta_t}}.
\end{equation}

Next, we bound the terms in the second line of \eqref{eq:2iso_evol}. To this end, observe that
\begin{equation}\label{eq:j-k-terms}
	\bigl\lvert \bigl\langle\vect{x},G_{[1,j]}\mathscr{S}[G_{[j,k]}-M_{[j,k]}]G_{[k,3]}  \vect{y}\bigr\rangle \bigr\rvert \le \max_p \bigl\lvert \bigl\langle S^{(p)}(G_{[j,k]}-M_{[j,k]}) \bigr\rangle \bigr\rvert \cdot \bigl\lVert G_{[1,j]}^*\vect{x}\bigr\rVert_2  \cdot\bigl\lVert G_{[k,3]}\vect{y}\bigr\rVert_2.
\end{equation}

Observe that the  bound in \eqref{eq:super-S-norm}, the stability estimate \eqref{eq:scalar_stab_bounds}, the second bound in \eqref{eq:regM_bound}, and the definition of $M_{[1,3],s}$ in \eqref{eq:M2iso_def} imply that 
\begin{equation} \label{eq:M3S}
	\bigl\lvert \bigl\langle S^{(p)} M_{[1,3],s} \bigr\rangle \bigr\rvert \lesssim \langle |A_1|^2 \rangle^{1/2} \langle |A_2|^2\rangle^{1/2}\bigl\lVert S^{(p)}\bigr\rVert.
\end{equation}
Hence, for all integers $1\le j \le k \le 3$, using the bound $\norm{S^{p}} \lesssim 1$, we deduce that
\begin{equation} \label{eq:j-k-tr}
	\max_p \bigl\lvert \bigl\langle S^{(p)}(G_{[j,k],s}-M_{[j,k],s}) \bigr\rangle \bigr\rvert \prec \frac{N^{(k-j)/2}}{N\eta_s}\prod_{i=j}^{k-1}\langle |A_i|^2 \rangle^{1/2}, \quad s \in [0,T],
\end{equation}
where for $k=j$ we use the averaged local law from \eqref{eq:1G_laws}, for $k=j+1$ we use the definition of $\Phi_{(1,1)}$ in \eqref{eq:Phi(1,1)_def} and the bound \eqref{eq:Phi_final}, while for $k = j+2$ we used the reduction inequality \eqref{eq:Red(2,1)av_HS} and the bound \eqref{eq:M3S} to estimate the resolvent chain and the deterministic approximation separately.
%It follows from Lemma \ref{lemma:stab_lemma}, the bound \eqref{eq:regM_bound}, the definition \eqref{eq:M2iso_def}, the reduction inequality \eqref{eq:Red(2,1)av_HS} together with \eqref{eq:Phi_final}, and the averaged local law in \eqref{eq:1G_laws}, that f
Using the inequality \eqref{eq:ImG_ineq}, we obtain for all $1\le j \le k \le 3$,
\begin{equation} \label{eq:j-k-norm}
	\bigl\lVert G_{[j,k],s}\vect{x}\bigr\rVert_2^2 \lesssim  \eta_{j,s}^{-1}\bigl\lvert\bigl\langle \vect{x}, G_{[j+1,k],s}^*A_j^*\im G_{j,s}A_jG_{[j+1,k],s}   \vect{x}\bigr\rangle\bigr\rvert^{1/2} \prec \frac{N^{k-j}}{\eta_{j,s}}\prod_{i=j}^{k-1}\langle |A_i|^2\rangle, \quad s\in[0,T]
\end{equation}
where we used \eqref{eq:1G_laws} for $k=j$, \eqref{eq:3Giso_crude} for $k=j+1$, and \eqref{eq:5G_iso} for $k=j+2$. Therefore,  using the integration rule \eqref{eq:int_rules} and the bounds \eqref{eq:j-k-terms}--\eqref{eq:j-k-norm}, we obtain for all $0\le t\le T$,
\begin{equation}\label{eq:3iso_jk}
	\frac{1}{\langle |A_1|^2\rangle^{1/2} \langle |A_2|^2\rangle^{1/2} }\int_{0}^t\biggl\lvert\sum_{1\le j\le k \le 3} \bigl\langle\vect{x},G_{[1,j],s}\mathscr{S}[G_{[j,k],s}-M_{[j,k],s}]G_{[k,3],s}  \vect{y}\bigr\rangle\biggr\rvert  \mathrm{d}s
	\prec \int_0^t \frac{\mathrm{d}s}{\eta_s^2}\prec \frac{N}{N\eta_t}.
\end{equation}

To estimate the remaining terms on the right-hand side of \eqref{eq:2iso_evol}, we observe that bound \eqref{eq:super-S-norm}, \eqref{eq:scalar_stab_bounds} from Lemma \ref{lemma:stab_lemma}, \eqref{eq:regM_bound} and the definition \eqref{eq:M2iso_def} imply
\begin{equation} \label{eq:3Mbounds}
	\bigl\lVert\mathscr{S}[M_{[1,3],s}] \bigr\rVert \lesssim N \langle |A_1|^2\rangle^{1/2}\langle |A_2|^2\rangle^{1/2}, \quad \bigl\lVert\other{M}_{p,s}\bigr\rVert \lesssim \sqrt{N}\eta_s^{-1}\langle |A_1|^2\rangle^{1/2}\langle |A_2|^2\rangle^{1/2}, \quad p\in{1,2}.
\end{equation}
For the first term in the third line of \eqref{eq:2iso_evol}, we use \eqref{eq:Weak1iso}, \eqref{eq:int_rules}, and \eqref{eq:3Mbounds}, to obtain
\begin{equation} \label{eq:3iso_2}
	\int_0^t \bigl\lvert\bigl\langle \vect{x}, (G_{1,s} \mathscr{S}[M_{[1,3],s}] G_{1,s}- \other{M}_{3,s}) \vect{y}\bigr\rangle \bigr\rvert  \mathrm{d}s  
	\prec \int_0^t \frac{\bigl\lVert\mathscr{S}[M_{[1,3],s}] \bigr\rVert}{\sqrt{N\eta_s}}\frac{\mathrm{d}s}{\eta_s} \prec \frac{N\langle |A_1|^2\rangle^{1/2} \langle |A_2|^2\rangle^{1/2} }{\sqrt{N\eta_t}}.
\end{equation}
For the second term in the third line of \eqref{eq:2iso_evol}, we use the upper bound in Assumption \eqref{as:S_flat}, \eqref{eq:regM_bound} and \eqref{eq:j-k-norm} to estimate the resolvent chain, and \eqref{eq:3Mbounds} to bound the deterministic term, yielding
\begin{equation} \label{eq:3iso_3}
	\frac{1}{\langle |A_1|^2\rangle^{1/2} \langle |A_2|^2\rangle^{1/2}} \int_0^t \bigl\lvert \bigl\langle \vect{x}, (G_{[1,2],s} \mathscr{S}[M_{[2,3],s}] G_{1,s}- \other{M}_{2,s}) \vect{y}\bigr\rangle \bigr\rvert \mathrm{d}s 
	\prec \int_0^t \frac{\sqrt{N}}{\eta_s}\mathrm{d}s \prec \sqrt{N}.
\end{equation}
The term in the last line of \eqref{eq:2iso_evol} is estimated similarly. Summing the bounds \eqref{eq:3iso_mart}, \eqref{eq:3iso_jk}, \eqref{eq:3iso_2}, \eqref{eq:3iso_3} concludes the proof of Proposition \ref{prop:2iso}.
\end{proof}

\section{A-Priori Bounds. Proof of Lemmas \ref{lemma:1Flaw} and \ref{lemma:2G_weak}} \label{seq:prior_laws}

We now prove the local laws involving one and two resolvents without exploiting the regularity of the observables in Lemmas \ref{lemma:1Flaw} and \ref{lemma:2G_weak} simultaneously.
The main obstacle to proving Lemma \ref{lemma:2G_weak} using the approach laid out in Section \ref{sec:locallaws_proof} is the presence of the \textit{linear} terms of the form $\langle (G_{1,t} B_1 G_{2,t} - \M{\vect{z}_{1,t},B_1,\vect{z}_{2,t}}) \mathscr{S}[\M{\vect{z}_{2,t},B_2,\vect{z}_{1,t}}]  \rangle$ in the time differential of $\langle (G_{1,t} B_1 G_{2,t} - \M{\vect{z}_{1,t},B_1,\vect{z}_{2,t}}) B_2\rangle$ and its isotropic analog, c.f. \eqref{eq:(GAG-M)A_evol}. However, the new observables $\mathscr{S}[\M{\dots}]$ lie in the range of the super-operator $\mathscr{S}$. Therefore, we can construct a self-consistent system of time-evolution equations for observables in the range of $\mathscr{S}$, and solve it using the following variant of the Gronwall estimate.
\begin{lemma} [$\ell^\infty$ Stochastic Gronwall's Inequality] (c.f. Lemma 5.6 in \cite{Cipolloni23Gumbel}) \label{lemma:Gronwall}
	Fix $k \in \{1,2\}$, and let $\mathcal{X}_t \in \mathbb{C}^{N^k}$ be a solution to the stochastic differential equation
	\begin{equation} \label{eq:Xk_evol}
		\mathrm{d}\mathcal{X}_t = \mathcal{A}_t^{\oplus k}\bigl[\mathcal{X}_t\bigr]\mathrm{d}t + \mathcal{F}_t\mathrm{d}t + \mathrm{d}\mathcal{E}_t,
	\end{equation}
	where the forcing term $\mathcal{F}_t := (\mathcal{F}_{\vect{j},t})_{\vect{j}\in \{1,\dots,N\}^k}\in \mathbb{C}^{N^k}$ is adapted to a continuous family of $\sigma$-algebras associated with the martingale $\mathcal{E}_t := (\mathcal{E}_{\vect{j},t})_{\vect{j}\in \{1,\dots,N\}^k}\in\mathbb{C}^{N^k}$, and $\mathcal{A}_t\in\mathbb{C}^{N\times N}$ is a family of operators in $\mathbb{C}^{N\times N}$. Here, $\mathcal{A}^{\oplus k}$ denotes the $k$-fold direct sum of $\mathcal{A}$ with itself \footnote{
		In particular, for $k=2$, the action of  $\mathcal{A}^{\oplus 2} = \mathcal{A}\oplus \mathcal{A}$ on matrices $X \in \mathbb{C}^{N\times N}$ is given by $\mathcal{A}^{\oplus 2}[X] = \mathcal{A} X + X \mathcal{A}^\mathfrak{t}$.
	}. 
	
	Assume additionally that $|\mathcal{F}_{\vect{j},t}| + |\mathcal{E}_{\vect{j},t}| \le N^D$ for some $D>0$, and that  there exists a time-independent eigenprojector $\mathcal{P} \in \mathbb{C}^{N\times N}$ with $\mathrm{rank}\,\mathcal{P}\le 1$ and bounded norm $\norm{\mathcal{P}}_{\ell^\infty\to\ell^\infty} \le C_1$, and a complex function $f_t$ that satisfy
	\begin{equation} \label{eq:stabPpm}
		\mathcal{P}\mathcal{A}_t = \mathcal{A}_t\mathcal{P} = f_t\mathcal{P}, \quad \norm{\mathcal{A}_t(1-\mathcal{P})}_{\ell^\infty\to\ell^\infty} \le C_2.
	\end{equation}
	Assume additionally that $\tau$ is a random stopping time such that 
	\begin{equation} \label{eq:forcing_mart_assume}
		\max_{\vect{j}\in \{1,\dots,N\}^k}\biggl( \int_0^{t\wedge \tau} |\mathcal{F}_{\vect{j},s}|\mathrm{d}s\biggr)^2 + \max_{\vect{j}\in \{1,\dots,N\}^k}\biggl[ \int_0^{\cdot} \mathrm{d}\mathcal{E}_{\vect{j},s} \biggr]_{t\wedge\tau} \le h_{t\wedge\tau}^2,
	\end{equation}
	with very high probability for some positive deterministic function $h_t \ge N^{-D}$, where $[\cdot]_t$ denotes the quadratic variation process. Then the random variable $\mathcal{Z}_t := \max_{\vect{j}\in \{1,\dots,N\}^k} |\mathcal{X}_{\vect{j},t}|^2$ satisfies the bound
	\begin{equation} \label{eq:Z_improved_bound}
		\sup_{0\le s\le t\wedge\tau} \mathcal{Z}_s \lesssim \mathcal{Z}_0 + N^{3\theta}h_t^2 + \int_0^{t\wedge\tau} \bigl(\mathcal{Z}_0 + N^{3\theta}h_s^2\bigr)\bigl(1+|\re f_s|\bigr)\exp\biggl\{2k(1+N^{-\theta})\int_s^{t\wedge\tau} |\re f_r|\mathrm{d}r\biggr\} \mathrm{d}s
	\end{equation}
	with very high probability, for any time $0 \le t \le C_3$ and any arbitrary small parameter $\theta > 0$, with the implicit constant depending on the constants $C_1,C_2,C_3, D, \theta$ and $k$.
\end{lemma}
We defer the proof of Lemma \ref{lemma:Gronwall} to the end of the Section.
In the sequel, we apply Lemma \ref{lemma:Gronwall} with the operator $\mathcal{A}_t := (\mathrm{const} + \stab_{\vect{z}_{1,t},\vect{z}_{2,t}}^{-1})^\mathfrak{t}$, where $(\cdot)^\mathfrak{t}$ denotes the transpose. Therefore, we collect the necessary properties of the stability operator $\stab_{\vect{z}_{1,t},\vect{z}_{2,t}}$ in the following lemma, that we prove  in Appendix \ref{sec:stab_section}.
\begin{lemma} [Stability along the Flow] \label{lemma:stab_flow}
	Let $z_1,z_2 \in \mathcal{D}$ and let $\vect{z}_{j,t} := \mathfrak{f}^t(z_j)$, where $\mathfrak{f}^t$ is the flow map defined in \eqref{eq:flow_map}. Let $ M_{j,t} := \diag{\m(\vect{z}_{j,t})}$, then the stability operator $\stab_{\vect{z}_{1,t},\vect{z}_{2,t}} := 1 - M_{1,t}M_{2,t}S$ satisfies, for all $t\in [0,T]$,
	\begin{equation} \label{eq:stab_t_bounds}
		\bigl\lVert\stab^{-1}_{\vect{z}_{1,t},\vect{z}_{2,t}}\bigr\rVert_* \lesssim 1, \quad \text{if}\quad (\im z_1)(\im z_2) >0 \text{ or } \min\{|z_1 - z_2|, |\bar{z}_1-z_2|\} > \tfrac{1}{2}\delta,
	\end{equation}
	where $\delta$ is the threshold from Lemma \ref{lemma:stab_lemma}.
	On the other hand, for $z_1,z_2$ satisfying $(\im z_1)(\im z_2) < 0$ and $|\bar{z}_1-z_2| \le \delta$, we have 
	\begin{equation} \label{eq:Pi_stab_t_commute}
		\stab_{\vect{z}_{1,t},\vect{z}_{2,t}}\Pi_{z_1,z_2} = \Pi_{z_1,z_2}\stab_{\vect{z}_{1,t},\vect{z}_{2,t}} = \eigB_{\vect{z}_{1,t},\vect{z}_{2,t}}\Pi_{z_1,z_2},\quad	\bigl\lVert\stab^{-1}_{\vect{z}_{1,t},\vect{z}_{2,t}}(1-\Pi_{z_1,z_2})\bigr\rVert_* \lesssim 1,
	\end{equation}
	where $\Pi_{z_1,z_2}$ is the time-independent eigenprojector defined in Lemma \ref{lemma:stab_lemma}. The eigenvalue $\eigB_{\vect{z}_{1,t},\vect{z}_{2,t}}$ satisfies
	\begin{equation} \label{eq:eigB_t}
		\eigB_{\vect{z}_{1,t},\vect{z}_{2,t}} = 1 - \mathrm{e}^{t-T}(1-\eigB_{z_1,z_2}), \quad |\eigB_{\vect{z}_{1,t},\vect{z}_{2,t}}| \sim \eta_{1,t} + \eta_{2,t} + |z_1 - z_2|.
	\end{equation}
\end{lemma}

\begin{proof} [Proof of Lemmas \ref{lemma:1Flaw} and \ref{lemma:2G_weak}]
	We note that without loss of generality, we can assume that $\norm{\vect{x}}_2=\norm{\vect{y}}_2=1$, and that  the observables $B,B_1,B_2$ are Hermitian. Indeed, this follows immediately from the multi-linearity of the resolvent chains and the deterministic approximations involved in \eqref{eq:1G_laws}--\eqref{eq:Weak1iso} in the observables, and the bounds $\norm{\re B}_\n, \norm{\im B}_\n \le \norm{B}_\n$  for any choice of label $\n \in \{\mathrm{hs},\mathrm{op}\}$. %, where $\re B := \tfrac{1}{2}(B+B^*)$ and $\im B := \tfrac{1}{2\I}(B-B^*)$ are both Hermitian.

	We prove the local laws \eqref{eq:1G_laws}--\eqref{eq:Weak1iso} in three steps: first, with both observables replaced by diagonal matrices $S^{(p)}$ containing the entries of the $p$-th row of the matrix $NS$, then - with only $B_2$ replaced by $S^{(p)}$, and finally for arbitrary Hermitian $B_1,B_2$.
	
	For a fixed pair of deterministic Hermitian observables $B_1, B_2$ and  deterministic vectors $\vect{x}, \vect{y}$, we define the set of matrices $\mathcal{M}$ and the set of vectors $\mathcal{V}$ as
	\begin{equation} \label{eq:V_set}
		\mathcal{M} := \{B_1, B_2\}\cup \{S^{{(p)}}\}_{p=1}^N, \quad \mathcal{V}:=\{\vect{x},\vect{y}\}\cup \{\vect{e}_j\}_{j=1}^N \cup \{\vect{u}_j^{B_1},\vect{u}_j^{B_2}\}_{j=1}^N,
	\end{equation}
	where $\vect{e}_j$ are the coordinate vectors in $\mathbb{R}^N$, and $\vect{u}_j^{B_k}$ are the normalized eigenvectors of $|B_k|^2$.
	
	\textbf{Step 1.} First, we prove the local laws \eqref{eq:1G_laws}--\eqref{eq:Weak1iso} with $B_1$, $B_2$ replaced by $S^{(p)}$ and $S^{(q)}$ for any $p,q \in \{1, \dots, N\}$. For all $z_j \in \mathcal{D}$ and $t\in [0,T]$, we define the (inverse) target size parameters
	\begin{equation}
		\begin{split}
			\size_t^{0,\mathrm{iso}}(z_1)\equiv \size_t^{0,\mathrm{iso}}(z_1) := \sqrt{N\eta_{1,t}},& \quad \size_t^{1,\mathrm{av}}\equiv \size_t^{1,\mathrm{av}}(z_1) := N\eta_{1,t},\\
			\size_t^{1,\mathrm{iso}}\equiv \size_t^{1,\mathrm{iso}}(z_1, z_2) := \sqrt{N\eta_{1,t}\eta_{2,t}\eta_t},& \quad \size_t^{2,\mathrm{av}} \equiv \size_t^{2,\mathrm{av}}(z_1, z_2) := N\eta_{1,t}\eta_{2,t}.
		\end{split}
	\end{equation} 
	Observe that the uniform bounds in Assumption \eqref{as:S_flat} imply that for all $p \in \{1,\dots,N\}$,
	\begin{equation} \label{eq:S_norms}
		\langle |S^{(p)}|^2 \rangle^{1/2} = \biggl(N\sum_j (S_{pj})^2\biggr)^{1/2} \ge \sum_j S_{pj} %= \bigl(S[\vect{1}]\bigr)_p 
		\ge C_{\sup}^{1-L}\sum_{j} \bigl(S^{L}\bigr)_{pj}
		\ge C_{\sup}^{1-L}c_{\inf} \gtrsim 1.
	\end{equation}
	In particular, for all $p \in \{1,\dots,N\}$, the Hilbert-Schmidt and the operator norms of $S^{(p)}$ are comparable
	\begin{equation}
		\bigl\lVert S^{(p)} \bigr\rVert \sim \langle |S^{(p)}|^2 \rangle^{1/2} \sim 1.
	\end{equation}
	Therefore, since $\sqrt{N\eta_t} \gtrsim 1$, in the special case $B_1 := S^{(p)}$ the size parameter $\size_t^{1,\mathrm{iso}}$ reflects both the local laws \eqref{eq:Weak1iso_HS} and \eqref{eq:Weak1iso}.
	We introduce the following sets of auxiliary variables indexed by $p,q \in \{1,\dots, N\}$,
	\begin{equation}
		\begin{split}
			\chain^{0,\mathrm{iso}}(t) &\equiv \chain^{0,\mathrm{iso}}(z_1,\vect{x}_1,\vect{x}_2,t) := \bigl\langle \vect{x}_1,(G_{1,t} - M_{1,t})\vect{x}_2  \bigr\rangle,\quad \vect{x}_1,\vect{x}_2 \in \mathcal{V},\\
			\chain^{1,\mathrm{av}}_{p}(t) &\equiv \chain^{1,\mathrm{av}}_{p}(z_1,t) := \bigl\langle (G_{1,t} - M_{1,t} )S^{(p)}  \bigr\rangle,\\
			\chain^{1,\mathrm{iso}}_{p}(t) &\equiv\chain^{1,\mathrm{iso}}_{p}(z_1,z_2,\vect{x}_1,\vect{x}_2,t) := \bigl\langle \vect{x}_1, \bigl(G_{1,t} S^{(p)} G_{2,t} - M^{(p)}_{t} \bigr)\vect{x}_2 \bigr\rangle,\quad \vect{x}_1,\vect{x}_2 \in \mathcal{V},\\
			\chain^{2,\mathrm{av}}_{pq}(t) &\equiv \chain^{2,\mathrm{av}}_{pq}(z_1,z_2,t) := \bigl\langle \bigl(G_{1,t} S^{(p)} G_{2,t} - M^{(p)}_{t} \bigr)S^{(q)} \bigr\rangle,
		\end{split}
	\end{equation}
	where $\mathcal{V}$ is defined in \eqref{eq:V_set}\footnote{
		Since the matrices $S^{(p)}$ are diagonal, in Step 1 of the proof we only need to consider $\vect{x}, \vect{y}, \{\vect{e}_j\}_{j=1}^N$.
	}, and $M^{(p)}_{t} := \M{\vect{z}_{1,t},S^{(p)},\vect{z}_{2,t}}$ as in \eqref{eq:M_def}.
	Here the superscript $k$ is equal to the number of matrices $S^{(p)}$ that appear in the corresponding resolvent chain, and coincides with the dimension of the quantity $\mathcal{X}^{k,\dots}$.
	We fix $0 < \varepsilon' \le \tfrac{1}{10}\varepsilon$, and define a stopping time $\tau_1$ as
	\begin{equation}\label{eq:tau}
		\begin{split}
			\tau_1 := \inf\biggl\{ t\in [0,T]: \sup_{z_1,z_2\in\mathcal{D}}\max_{\vect{x}_1,\vect{x}_2 \in \mathcal{V}}\max_{p,q}\biggl(&\size_t^{0,\mathrm{iso}}\bigl\lvert\chain^{0,\mathrm{iso}}(t)\bigr\rvert + N^{\varepsilon'}\size_t^{1,\mathrm{av}} \bigl\lvert\chain^{1,\mathrm{av}}_{p}(t)\bigr\rvert \\
			& + \size_t^{2,\mathrm{av}}\bigl\lvert\chain^{2,\mathrm{av}}_{pq}(t)\bigr\rvert + \size_t^{1,\mathrm{iso}} \bigl\lvert\chain^{1,\mathrm{iso}}_{p}(t)\bigr\rvert \biggr) = N^{2\varepsilon'}\biggr\},
		\end{split}
	\end{equation}
	where we made the dependence of the $\chain$ quantities on all arguments except the time $t$ implicit for brevity. Note the additional factor $N^{\varepsilon'}$ in front of $\size_t^{1,\mathrm{av}}$. Since the set $\mathcal{V}$ contains at most $3N+2$ vectors, and the indices $p,q \in \{1,\dots, N\}$, a simple grid argument in $\mathcal{D}$ together with Proposition \ref{prop:global_laws} shows that $\tau_1 > 0$ with very high probability. Our first goal is to show that $\tau_1 = T$ with very high probability. For the remainder of the proof, we consider the implicit arguments of the $\chain$ quantities fixed.
	
	To this end, we use the evolution equations \eqref{eq:GBevol} and \eqref{eq:GBGBevol}, together with \eqref{eq:m_evol}, \eqref{eq:M_evol}, and the definition of $\mathscr{S}$ in \eqref{eq:superS_def}, to deduce that the quantity $\chain^{0,\mathrm{iso}}(t)% \equiv \chain^{0,\mathrm{iso}}_1(t)
	$, the vectors $\chain^{1,\mathrm{av}}(t) := (\chain^{1,\mathrm{av}}_{p}(t))_{p=1}^N$ and $\chain^{2,\mathrm{iso}}(t) := (\chain^{2,\mathrm{iso}}_{p}(t))_{p=1}^N$, and the matrix $\chain^{2,\mathrm{av}}(t) := (\chain^{2,\mathrm{av}}_{pq}(t))_{p,q=1}^N$ satisfy the following stochastic differential equations
	\begin{equation} \label{eq:0isoX_evol}
		\mathrm{d}\chain^{0,\mathrm{iso}}(t) = \frac{1}{2}\chain^{0,\mathrm{iso}}(t) \mathrm{d}t + \mathcal{F}_t^{0,\mathrm{iso}}\mathrm{d}t + \mathrm{d}\mathcal{E}_t^{0,\mathrm{iso}},
	\end{equation}
	\begin{equation} \label{eq:1avX_evol}
		\mathrm{d}\chain^{1,\mathrm{av}}(t) = -\frac{1}{2}\chain^{1,\mathrm{av}}(t) \mathrm{d}t + (\stab^{-1}_{\vect{z}_{1,t}, \vect{z}_{1,t}})^\mathfrak{t}\chain^{1,\mathrm{av}}(t) \mathrm{d}t + \mathcal{F}_t^{1,\mathrm{av}}\mathrm{d}t + \mathrm{d}\mathcal{E}_t^{1,\mathrm{av}},
	\end{equation}
	\begin{equation} \label{eq:1isoX_evol}
		\mathrm{d}\chain^{1,\mathrm{iso}}(t) = (\stab^{-1}_{\vect{z}_{1,t}, \vect{z}_{2,t}} )^\mathfrak{t}\chain^{1,\mathrm{iso}}(t)\mathrm{d}t + \mathcal{F}_t^{1,\mathrm{iso}}\mathrm{d}t + \mathrm{d}\mathcal{E}_t^{1,\mathrm{iso}},
	\end{equation}
	\begin{equation} \label{eq:2avX_evol}
		\mathrm{d}\chain^{2,\mathrm{av}}(t) = \bigl(-\chain^{2,\mathrm{av}}(t) + (\stab^{-1}_{\vect{z}_{1,t}, \vect{z}_{2,t}})^\mathfrak{t}\chain^{2,\mathrm{av}}(t)  + \chain^{2,\mathrm{av}}(t) \stab^{-1}_{\vect{z}_{1,t}, \vect{z}_{2,t}}\bigr)\mathrm{d}t  + \mathcal{F}_t^{2,\mathrm{av}}\mathrm{d}t + \mathrm{d}\mathcal{E}_t^{2,\mathrm{av}},
	\end{equation}
	where $(\cdot)^\mathfrak{t}$ denotes the transpose. 
	Here for a label $(k,\mu) \in \bigl\{(0,\mathrm{iso}), (1,\mathrm{av}), (1,\mathrm{iso}), (2,\mathrm{av})\bigr\}$, the martingale terms $\mathrm{d}\mathcal{E}^{k,\mu}_t$ are given by
	\begin{equation} \label{eq:X_marts}
		\mathrm{d}\mathcal{E}^{k,\mu}_{\vect{j},t}  := \frac{1}{2}\sum_{j,k}\partial_{jk}\chain_{\vect{j}}^{k,\mu}(t) \sqrt{S_{jk}}\mathrm{d}\Brwn_{jk,t}, \quad 		\vect{j} \in \{1,\dots, N\}^k.
	\end{equation}
	For $k=0$, we identify the index set $\{1,\dots, N\}^0$ with the singleton $\{1\}$, e.g, $\mathrm{d}\mathcal{E}^{0,\mathrm{iso}}_{t} \equiv (\mathrm{d}\mathcal{E}^{0,\mathrm{iso}}_{1,t})$, and drop the subscript one in the sequel.
	The forcing $\mathcal{F}^{k,\mu}_t$ terms in \eqref{eq:0isoX_evol}--\eqref{eq:2avX_evol} are defined as
	\begin{equation} \label{eq:0iso_forcing}
		\mathcal{F}_t^{0,\mathrm{iso}} %\equiv \mathcal{F}_{1,t}^{0,\mathrm{iso}} 
		:= \bigl\langle \vect{x}_1, G_{1,t} \mathscr{S}[G_{1,t}-M_{1,t}] G_{1,t}\vect{x}_2\bigr\rangle,
	\end{equation}
	\begin{equation} \label{eq:1av_forcing}
		\mathcal{F}_{p,t}^{1,\mathrm{av}}  := \bigl\langle \mathscr{S}[G_{1,t}-M_{1,t}] (G_{1,t} S^{(p)}G_{1,t} - M_{t}^{(p)})\bigr\rangle,
	\end{equation}
	\begin{equation} \label{eq:1iso_forcing}
		\begin{split}
			\mathcal{F}^{1,\mathrm{iso}}_{p,t} :=&~ \bigl\langle \vect{x}_1, G_{1,t}\mathscr{S}[G_{1,t}S^{(p)}G_{2,t}-M_{t}^{(p)}]G_{2,t}  \vect{x}_2 \bigr\rangle 
			+ \bigl\langle \vect{x}_1, G_{1,t}\mathscr{S}[G_{1,t}-M_{1,t}]G_{1,t}S^{(p)}G_{2,t} \vect{x}_2 \bigr\rangle\\
			&+\bigl\langle \vect{x}_1, G_{1,t}S^{(p)}G_{2,t}\mathscr{S}[G_{2,t}-M_{2,t}]G_{2,t} \vect{x}_2 \bigr\rangle,
		\end{split}
	\end{equation}
	\begin{equation} \label{eq:2av_forcing}
		\begin{split}
			\mathcal{F}^{2,\mathrm{av}}_{pq,t} := &~ \bigl\langle \mathscr{S}\bigl[G_{1,t} S^{(p)} G_{2,t} - M^{(p)}_{t}\bigr] \bigl(G_{2,t} S^{(q)} G_{1,t} - M^{(q)}_{t}\bigr)\bigr\rangle\\
			&+\bigl\langle\mathscr{S}[G_{1,t}-M_{1,t}]G_{1,t} S^{(p)} G_{2,t} S^{(q)} G_{1,t}\bigr\rangle 
			+\bigl\langle\mathscr{S}[G_{2,t}-M_{2,t}]G_{2,t} S^{(q)} G_{1,t} S^{(p)} G_{2,t}\bigr\rangle.
		\end{split}
	\end{equation}
	We claim that the martingale and forcing terms defined in \eqref{eq:X_marts}--\eqref{eq:2av_forcing} satisfy
	\begin{equation} \label{eq:X_mart_forcing_bounds}
		\max_{\vect{j}\in \{1,\dots,N\}^k}\biggl(\int_0^{t\wedge\tau_1} \bigl\lvert \mathcal{F}^{k,\mu}_{\vect{j},s} \bigr\rvert \mathrm{d}s\biggr)^2
		+ \max_{\vect{j}\in \{1,\dots,N\}^k}\biggl[\int_0^\cdot \mathrm{d} \mathcal{E}^{k,\mu}_{\vect{j},s}\biggr]_{t\wedge\tau_1} \lesssim \frac{1 + \mathds{1}_{k=2}N^{2\varepsilon'}}{\bigl(\size_{t\wedge\tau_1}^{k,\mu}\bigr)^2},
	\end{equation}
	for all labels $(k,\mu) \in \bigl\{(0,\mathrm{iso}), (1,\mathrm{av}), (1,\mathrm{iso}), (2,\mathrm{av})\bigr\}$. 
	 
	Therefore, it follows from \eqref{eq:Pi_bound} and Lemma \ref{lemma:stab_flow} that the vector $\chain^{1,\mathrm{iso}}(t)$ and the matrix $\chain^{2,\mathrm{av}}$ satisfy the assumptions of Lemma \ref{lemma:Gronwall} with $k = 1$ and $k=2$, respectively, $h_t := (\size_t^{k,\mu})^{-1} \sqrt{1 + \mathds{1}_{k=2}N^{2\varepsilon'}}$, $\mathcal{A}_t := (\stab_{\vect{z}_{1,t},\vect{z}_{2,t}}^{-1} - \tfrac{k-1}{2})^\mathfrak{t}$, $C_1, C_2 \lesssim 1$ (owing to \eqref{eq:Pi_bound} and Lemma \ref{lemma:stab_flow}), $\mathcal{P} := \chi(z_1,z_2)\Pi_{z_1,z_2}^\mathfrak{t}$ and $f_t := f_{k,t} = \chi(z_1,z_2)(\eigB_{\vect{z}_{1,t},\vect{z}_{2,t}}^{-1} -\tfrac{k-1}{2})$, where $\chi(z_1,z_2)$ is defined as
	\begin{equation}
		\chi(z_1,z_2) := \mathds{1}_{(\im z_1)(\im z_2) <0}  \cdot \mathds{1}_{|\bar{z}_1-z_2| \le \delta}.
	\end{equation}
	For $r \in [0,T]$, the expansion \eqref{eq:beta_expan} and \eqref{eq:eigB_t} imply the upper bound
	\begin{equation} 
		\bigl\lvert\re[f_{k,r}]\bigr\rvert \le \mathrm{e}^{T-r}\bigl(\mathrm{e}^{T-r} - 1+\kappa(z_1)^{-1}(|\im  z_1|+|\im  z_2|)\bigr)^{-1} + \mathcal{O}\bigl(1+|\eigB_{\vect{z}_{1,r},\vect{z}_{2,r}}|^{-2}|\bar{z}_1 - z_2|^2\bigr).
	\end{equation}
	Hence, using the asymptotic for $\eigB_{\vect{z}_{1,r},\vect{z}_{2,r}}$ in \eqref{eq:eigB_t} and a simple convexity estimate, we deduce that for all $z_1,z_2\in\mathcal{D}$ and all $r \in [0,T]$,
	\begin{equation} \label{eq:re_beta_bound}
		2\bigl\lvert\re[f_{k,r}]\bigr\rvert \le  -\partial_r \log\bigl(\mathrm{e}^{T-r} - 1+2\kappa(z_1)^{-1}|\im  z_1|\bigr) -\partial_r \log\bigl(\mathrm{e}^{T-r} - 1+2\kappa(z_1)^{-1}|\im  z_2|\bigr) + \mathcal{O}(1).
	\end{equation}
	Integrating the bound \eqref{eq:re_beta_bound}, and using \eqref{eq:eta_asymp}, $\kappa(z_1)\sim 1$ from \eqref{eq:beta_expan}, we deduce that
	\begin{equation} \label{eq:exp_fs_bound}
		\exp\biggl\{2(1+N^{-\varepsilon})\int_s^{t\wedge\tau_1} |\re f_{k,r}|\mathrm{d}r\biggr\} \lesssim  \frac{\eta_{1,s}\eta_{2,s}}{\eta_{1,t\wedge\tau_1}\eta_{2,t\wedge\tau_1}},\quad 0\le s \le t \wedge \tau_1.
	\end{equation}

	On the other hand, the vector $\chain^{1,\mathrm{av}}(t)$ satisfies the assumptions of Lemma \ref{lemma:Gronwall} with $k=1$, $h_t := \size_t^{1,\mathrm{av}}$, $\mathcal{A}_t := -\tfrac{1}{2} + (\stab_{\vect{z}_{1,t},\vect{z}_{1,t}}^{-1})^\mathfrak{t}$, $\mathcal{P} = 0$, $f_t = 0$ and $C_1, C_2 \lesssim 1$ (owing to Lemma \ref{lemma:stab_flow}).
	Therefore, using Lemma \ref{lemma:Gronwall} with $\theta = \tfrac{1}{3}\varepsilon'$, the evolution equations \eqref{eq:1avX_evol}--\eqref{eq:2avX_evol}, the bounds \eqref{eq:X_mart_forcing_bounds}, $f_t = 0$ for $(k,\mu) = (1,\mathrm{av})$, and  the estimate \eqref{eq:exp_fs_bound} for $(k,\mu) = (1,\mathrm{iso}), (2,\mathrm{av})$, we conclude that
	\begin{equation} \label{eq:X_imporved}
		\sup_{0\le s \le t\wedge \tau_1} \max_{\vect{j}\in\{1,\dots,N\}^k} \bigl\lvert \chain_{\vect{j}}^{k,\mu}(s) \bigr\rvert^2  \lesssim \frac{N^{\varepsilon'} + \mathds{1}N^{3\varepsilon'}}{\bigl(\size_{t\wedge\tau_1}^{k,\mu}\bigr)^2} \quad 
		\text{ with very high probability},
	\end{equation}
	for all labels $(k,\mu) \in \bigl\{(1,\mathrm{av}), (1,\mathrm{iso}), (2,\mathrm{av})\bigr\}$. Here we additionally used that $|\chain_{\vect{j}}^{k,\mathrm{av}}(0)| \lesssim N^{-1}$ and $|\chain_{\vect{j}}^{k,\mathrm{iso}}(0)| \lesssim N^{-1/2}$ since $\tau_1 > 0$ with very high probability. The corresponding bound for $(k,\mu) = (0,\mathrm{iso})$ follows immediately from \eqref{eq:0isoX_evol}, the bound \eqref{eq:X_mart_forcing_bounds} and the martingale inequality \eqref{eq:mart_ineq}. Therefore, assuming \eqref{eq:X_mart_forcing_bounds} holds, the stopping time $\tau_1 = T$ with very high probability, and hence the local laws \eqref{eq:1G_laws}--\eqref{eq:Weak1iso} hold with $B_1 = S^{(p)}$ and $B_2 = S^{(q)}$ for any $p,q \in \{1,\dots, N\}$.
	
	We now prove the estimate \eqref{eq:X_mart_forcing_bounds}. Note that by the upper bound in \eqref{as:S_flat}, $\norm{S^{(p)}} \lesssim 1$.
	First, we consider $(k,\mu) = (0,\mathrm{iso})$. Computing the quadratic variation of $\mathrm{d}\mathcal{E}^{0,\mathrm{iso}}$ using the definition of $G_{1,t}$ in \eqref{eq:Gt},  the integration rules \eqref{eq:int_rules}, the upper bound in \eqref{as:S_flat}, and the definition of $\tau_1$ in \eqref{eq:tau}, we obtain
	\begin{equation}
		\biggl[\int_0^\cdot \mathrm{d} \mathcal{E}^{0,\mathrm{iso}}_{s}\biggr]_{t\wedge\tau_1} \lesssim \int_{0}^{t\wedge\tau_1}\frac{1}{N\eta_{1,s}^2} \bigl\langle \vect{x}_1, \im G_{1,s} \vect{x}_1 \bigr\rangle\bigl\langle \vect{x}_2, \im G_{1,s} \vect{x}_2 \bigr\rangle \mathrm{d}s \lesssim \frac{1}{N\eta_{1,{t\wedge\tau_1}}}\biggl(1 + \frac{N^{4\varepsilon'}}{N\eta_{1,{t\wedge\tau_1}}}\biggr).
	\end{equation}
	To estimate the integral of \eqref{eq:0iso_forcing}, we compute, using \eqref{eq:int_rules} and \eqref{eq:tau},
	\begin{equation}
		\int_0^{t\wedge\tau_1} \bigl\lvert\mathcal{F}_s^{0,\mathrm{iso}}\bigr\rvert \mathrm{d}s \lesssim \int_0^{t\wedge\tau_1}\frac{N^{\varepsilon'}}{N\eta_{1,s}} \biggl(1+ \frac{N^{2\varepsilon'}}{\sqrt{N\eta_{1,s}}}\biggr)\frac{\mathrm{d}s}{\eta_{1,s}} \lesssim \frac{N^{\varepsilon'}}{N\eta_{1,t\wedge\tau_1}} \lesssim \frac{1}{\sqrt{N\eta_{1,t\wedge\tau_1}}}.
	\end{equation}
	Therefore, \eqref{eq:X_mart_forcing_bounds} is established for $(k,\mu) = (0,\mathrm{iso})$.	
	
	We proceed to prove \eqref{eq:X_mart_forcing_bounds} with $(k,\mu) = (1,\mathrm{av})$.
	In fact, we will prove a stronger statement with a general observable $B' \in \mathcal{M}$ in place of $S^{(p)}$ and $M^{(p)}$ replaced with the corresponding deterministic approximation $M^{B'}_s := \M{\vect{z}_{1,s},B',\vect{z}_{1,s}}$.
	For an observable $B' \in \mathcal{M}$, defined in \eqref{eq:V_set}, using the same approach as in \eqref{eq:1av_mart} but estimating one of the resolvents by its norm via \eqref{eq:G_norm}, we obtain
	\begin{equation} \label{eq:1av_mart_qv}
		\int_0^{t\wedge\tau_1} \sum_{j,k} S_{jk}\bigl\lvert \partial_{jk}\bigl\langle G_{1,s} B'\bigr\rangle \bigr\rvert^2\mathrm{d}s \lesssim 
		\int_0^{t\wedge\tau_1} \frac{\bigl\lvert\bigl\langle \im G_{1,s} |B'|^2 \bigr\rangle\bigr\rvert}{N^2\eta_{1,s}^3}\mathrm{d}s \lesssim \frac{\langle |B'|^2\rangle}{(N\eta_{1,t\wedge\tau_1})^2}\biggl(1 + \frac{N^{2\varepsilon'}}{\sqrt{N\eta_{t\wedge\tau_1}}}\biggr).
	\end{equation}
	Here, in the last step we used the equality from \eqref{eq:1G_<A^2>_bound}, the definition of $\tau_1$ in \eqref{eq:tau}, and the integration rules \eqref{eq:int_rules}.
	
	Next, to bound the integral of the forcing term, using the definition of the stopping time $\tau_1$ in \eqref{eq:tau}, we deduce that for all $0 \le s \le t\wedge\tau_1$,
	\begin{equation} \label{eq:1av_forcing_bound}
		\begin{split}
			\bigl\lvert \bigl\langle \mathscr{S}[G_{1,s} - M_{1,s}]\bigl(G_{1,s} B' G_{1,s} - M^{B'}_{s}\bigr)\bigr\rangle \bigr\rvert &=
			\frac{1}{N}\sum_j \bigl\lvert \bigl\langle(G_{1,s} - M_{1,s}) S^{(j)}\bigr\rangle (G_{1,s}B'G_{1,s}-M^{B'}_s)_{jj} \bigr\rvert\\
			& \le \frac{N^{\varepsilon'}}{N\eta_{1,s}}\max_j \bigl\lvert \bigl\langle \vect{e}_j (G_{1,s}B'G_{1,s}-M^{B'}_s)\vect{e}_j \bigr\rangle \bigr\rvert.
		\end{split}
	\end{equation}
	%	\begin{equation}
		%		\begin{split}
			%			\int_0^{t\wedge\tau} \bigl\lvert \mathcal{F}^{1,\mathrm{av}}_{p,s} \bigr\rvert
			%			&\lesssim \int_0^{t\wedge\tau}\max_q \bigl\lvert \chain_q^{1,\mathrm{av}}(z_1,s)  \chain_p^{1,\mathrm{iso}}(z_1,z_1,\vect{e}_q,\vect{e}_q,s)\bigr\rvert \mathrm{d}s  \lesssim \frac{N^{4\varepsilon'}}{(N\eta_{1,t\wedge\tau_1})^{3/2}} \lesssim \frac{1}{N\eta_{1,t\wedge\tau_1}}.
			%		\end{split}
		%	\end{equation}
	In particular, setting $B' := S^{(p)}$ in \eqref{eq:1av_forcing_bound}, using \eqref{eq:m_upper}, the definition of $\tau_1$ in \eqref{eq:tau}, $\varepsilon' \le \tfrac{\varepsilon}{10}$, and the integration rules \eqref{eq:int_rules}, together with \eqref{eq:1av_mart_qv}, we deduce \eqref{eq:X_mart_forcing_bounds} for $(k,\mu) = (1,\mathrm{av})$.
	
	Next, we prove \eqref{eq:X_mart_forcing_bounds} for $(k,\mu)= (1,\mathrm{iso})$. Again, we analyze a more general quantity with an arbitrary $B'\in\mathcal{M}$ in place of $S^{(p)}$, and obtain the bounds in terms of both the Hilbert-Schmidt and the operator norms of $B'$ corresponding to the right-hand sides of \eqref{eq:Weak1iso_HS} and \eqref{eq:Weak1iso}, respectively. 
	
	Observe that by the definition of $\tau_1$, for any observable $B' \in \mathcal{M}$, the quadratic variation process of $Q_t:=[\int_0^\cdot \sum_{j,k}\partial_{jk} \bigl\langle \vect{x}_1, G_{1,s}B'G_{2,s}\vect{x}_2 \bigr\rangle \sqrt{S_{jk}}\mathrm\Brwn_{jk,s}]_t$ satisfies 
	\begin{equation} \label{eq:iso_qvterm}
		\begin{split}
			Q_{t\wedge\tau_1}
			&\lesssim \int_0^{t\wedge\tau_1}\biggl(\frac{\bigl\lvert\bigl\langle \vect{x}_2, G_{2,s}^* B'^* \im G_{1,s} B' G_{2,s}\vect{x}_2 \bigr\rangle\bigr\rvert}{N\eta_{1,s}^2} + \frac{\bigl\lvert\bigl\langle \vect{x}_1, G_{1,s} B' \im G_{2,s} B'^* G_{1,s}^*\vect{x}_1 \bigr\rangle\bigr\rvert }{N\eta_{2,s}^2}\biggr)\mathrm{d}s\\
			&\lesssim \int_0^{t\wedge\tau_1} \frac{\norm{B'}^2}{N\eta_{1,s}\eta_{2,s}\eta_s}\biggl(1 + \frac{N^{2\varepsilon'}}{\sqrt{N\eta_{s}}}\biggr)\frac{\mathrm{d}s}{\eta_s} \lesssim  \biggl(\frac{\norm{B'}}{\size^{1,\mathrm{iso}}_{t\wedge\tau_1}(z_1,z_2)}\biggr)^2
		\end{split}
	\end{equation}
	where we used \eqref{eq:tau}, the generalized resolvent inequality \eqref{eq:ImG_ineq}.
	On the other hand, using the submultiplicativity of trace $\Tr[AB] \le \Tr[A]\Tr[B]$ for $A,B \ge 0$ and the equality in \eqref{eq:1G_<A^2>_bound}, we also obtain
	\begin{equation} \label{eq:iso_qvterm_hs}
		Q_{t\wedge\tau_1}\lesssim \int_0^{t\wedge\tau_1} \frac{\langle|B'|^2\rangle}{\eta_{1,s}\eta_{2,s}}\biggl(1 + \frac{N^{2\varepsilon'}}{\sqrt{N\eta_{s}}}\biggr)^2 \frac{\mathrm{d}s}{\eta_s} \prec \frac{\langle|B'|^2\rangle}{\eta_{1,t\wedge\tau_1}\eta_{2,t\wedge\tau_1}}.
	\end{equation}
	Similarly to the estimate \eqref{eq:j-k-norm} with $j=1, k=2$ for regular $A_1$, we observe that for any $B'\in \mathcal{M}$, 
	\begin{equation}
		\norm{G_{1,s}B_1'G_{2,s} \vect{x}_2}_2^2 = \frac{\bigl\langle \vect{x}_2, G_{2,s}^* (B')^*\im G_{1,s}B'G_{2,s} \vect{x}_2\bigr\rangle}{\eta_{1,s}} \lesssim N\frac{\bigl\langle \im G_{1,s}|B'|^2\rangle \langle\vect{x}_2, \im G_{2,s} \vect{x}_2\bigr\rangle}{\eta_{1,s}\eta_{2,s}}.
	\end{equation}
	On the other hand, using  the inequality \eqref{eq:ImG_ineq} and the operator norm bound \eqref{eq:G_norm}, we deduce that
	\begin{equation}
		\norm{G_{1,s}B_1'G_{2,s} \vect{x}_2}_2^2 = \eta_{1,s}^{-1}\bigl\langle \vect{x}_2, G_{2,s}^* (B')^*\im G_{1,s}B'G_{2,s} \vect{x}_2\bigr\rangle \lesssim \eta_{1,s}^{-2} \eta_{2,s}^{-1}\norm{B'}^2\langle\vect{x}_2, \im G_{2,s}. \vect{x}_2\bigr\rangle
	\end{equation}
	Therefore, the integration rules \eqref{eq:int_rules}, the definition of $\tau_1$ in \eqref{eq:tau}, the equality in \eqref{eq:1G_<A^2>_bound}, and the bound $\norm{G_{1,s}\vect{x}_1}_2 \prec \eta_{1,s}^{-1/2}\langle \vect{x}_1, \im G_{1,s} \vect{x}_1 \rangle^{1/2}$ imply
	\begin{equation} \label{eq:iso1-3term}
		\begin{split}
			\int_0^{t\wedge\tau_1}\bigl\lvert \bigl\langle \vect{x}_1, G_{1,s}\mathscr{S}[G_{1,s}-M_{1,s}]G_{1,s}B'G_{2,s} \vect{x}_2 \bigr\rangle \bigr\rvert \mathrm{d}s
			%&\lesssim \int_0^{t\wedge\tau_1}\frac{N^{\varepsilon'}\langle \vect{x}_2 G_{2,s}^* B_1'^* \im G_{1,s}B_1' G_{2,s}\vect{x}_2 \rangle^{1/2}}{N\eta_{1,s}^2}\mathrm{d}s\\
			&\lesssim \biggl( \frac{\norm{B'}}{%\sqrt{N\eta_{t\wedge\tau_1}}
				\size^{1,\mathrm{iso}}_{t\wedge\tau_1}(z_1,z_2)} \wedge \frac{\langle |B'|^2 \rangle^{1/2}}{\sqrt{\eta_{1,t\wedge\tau_1}\eta_{2,t\wedge\tau_1}}}\biggr)\frac{N^{\varepsilon'}}{\sqrt{N\eta_{t}}}.
		\end{split}
	\end{equation}
	Similarly, using $\norm{G_{j,s}\vect{x}_a}_2 \prec \eta_{j,s}^{-1/2}\langle \vect{x}_a, \im G_{j,s} \vect{x}_a \rangle^{1/2}$ and \eqref{eq:tau}, we deduce that
	\begin{equation} \label{eq:iso2-2term}
		\bigl\lvert  \bigl\langle \vect{x}_1, G_{1,s}\mathscr{S}[G_{1,s}B'G_{2,s}-M^{B'}_{s}]G_{2,s}  \vect{x}_2 \bigr\rangle \bigr\rvert  \prec (\eta_{1,s}\eta_{2,s})^{-1/2}\max_j \bigl\lvert \bigl \langle(G_{1,s}B'G_{2,s}-M^{B'}_{s})S^{(j)}\bigr\rangle \bigr\vert.
	\end{equation}
	Therefore, combining \eqref{eq:iso_qvterm}, \eqref{eq:iso_qvterm_hs}, \eqref{eq:iso1-3term} and \eqref{eq:iso2-2term} with $B' := S^{(p)}$, and using the definition of the stopping time $\tau_1$ in \eqref{eq:tau} to bound the averaged chains in \eqref{eq:iso2-2term}, we deduce that \eqref{eq:X_mart_forcing_bounds} holds for $(k,\mu) = (1,\mathrm{iso})$.
	
	Finally, we prove \eqref{eq:X_mart_forcing_bounds} for $(k,\mu) = (2,\mathrm{av})$. For two observables $B_1', B_2' \in \mathcal{M}$\footnote{
		Once again, in Step 1 we only use $B_j' = S^{(p)}$, but in the future steps of the proof, we use the same estimates for $B_j' = B_1, B_2$ as well. In preparation for that we obtain the estimates for all $B_1',B_2' \in \mathcal{M}$ simultaneously.
	}, similarly to \eqref{eq:2av_mart}, using \eqref{eq:int_rules}, the equality in \eqref{eq:1G_<A^2>_bound},  and \eqref{eq:tau}, we obtain
	\begin{equation} \label{eq:qvterm}
		\begin{split}
			\int_{0}^{t\wedge\tau_1}\sum_{j,k}S_{jk} \bigl\lvert\partial_{jk}\bigl\langle G_{1,s}B_1' G_{2,s} B_2' \bigr\rangle\bigr\rvert^2\mathrm{d}s &\lesssim \norm{B_2'}^2\int_{0}^{t\wedge\tau_1}\frac{\bigl\langle \bigl(|\im G_{1,s}| +  |\im G_{2,s}|\bigr)|B_1'|^2 \bigr\rangle}{(N\eta_{1,s}\eta_{2,s})^2} \frac{\mathrm{d}s}{\eta_s}\\ &\lesssim \frac{\langle |B_1'|^2 \rangle \norm{B_2}^2}{N\eta_{1,t\wedge\tau_1}\eta_{2,t\wedge\tau_1}}\biggl(1 + \frac{N^{2\varepsilon'}}{\sqrt{\eta_{t\wedge\tau_1}}}\biggr).
		\end{split}
	\end{equation} 
	Using the definition of $\tau_1$ in \eqref{eq:tau} to bound $\langle (G_{1,s}-M_{1,s})S^{(p)} \rangle$, we obtain, for all $0 \le s \le \tau_1$,
	\begin{equation}
		\begin{split}
			\bigl\lvert\bigl\langle\mathscr{S}[G_{1,s}-M_{1,s}]G_{1,s} B_1' G_{2,s} B_2' G_{1,s}\bigr\rangle\bigl\lvert &\le \frac{N^{\varepsilon'}}{N\eta_{1,s}}\frac{1}{N}\sum_j \bigl\lvert\bigl\langle \vect{e}_j, G_{1,s} B_1' G_{2,s} B_2' G_{1,s}\vect{e}_j  \bigr\rangle\bigr\rvert\\
			&\le \frac{N^{\varepsilon'}\norm{G_{2,s}}}{N\eta_{1,s}^2} \bigl\lvert\bigl\langle \im G_{1,s} |B_1'|^2 \bigr\rangle\bigr\rvert^{1/2}\bigl\lvert\bigl\langle \im G_{1,s} |B_2'|^2 \bigr\rangle\bigr\rvert^{1/2}.
		\end{split}
	\end{equation}
	Therefore, the norm bound \eqref{eq:G_norm}, the equality in \eqref{eq:1G_<A^2>_bound}, \eqref{eq:tau}, and the integration rules \eqref{eq:int_rules} imply
	\begin{equation} \label{eq:3-1term}
		\int_{0}^{t\wedge\tau_1}\bigl\lvert \bigl\langle\mathscr{S}[G_{1,s}-M_{1,s}]G_{1,s} B_1' G_{2,s} B_2' G_{1,s}\bigr\rangle  \bigr\rvert \mathrm{d}s 
		%\lesssim \int_{0}^{t\wedge\tau} \frac{N^{\varepsilon'}\langle |B_1'|^2 \rangle^{1/2}\langle |B_2'|^2 \rangle^{1/2}}{N\eta_{1,s}^2\eta_{2,s}}\mathrm{d}s 
		\lesssim \frac{N^{\varepsilon'}\langle |B_1'|^2 \rangle^{1/2}\langle |B_2'|^2 \rangle^{1/2}}{N\eta_{1,t\wedge\tau_1}\eta_{2,t\wedge\tau_1}}\biggl(1 + \frac{N^{2\varepsilon'}}{\sqrt{\eta_{t\wedge\tau_1}}}\biggr).
	\end{equation}
	The other term in the second line of \eqref{eq:2av_forcing} admits an analogous bound.
	To estimate the first term on the right-hand side of \eqref{eq:2av_forcing}, we observe that for chains $G_{[j,k],s} := G_{j,s}B_{j}'G_{k,s}$ and $M_{[j,k],s}$ denoting the corresponding deterministic approximations, we have the identity
	\begin{equation} \label{eq:2-2term}
		\bigl\langle \mathscr{S}[G_{[1,2],s}-M_{[1,2],s}] (G_{[2,1],s}-M_{[2,1],s})\bigr\rangle  = \frac{1}{N}\sum_j  \bigl\langle (G_{[1,2],s}-M_{[1,2],s})S^{(j)}\bigr\rangle (G_{[2,1],s}-M_{[2,1],s})_{jj}.
	\end{equation}
	%	The individual factors on the right-hand side of \eqref{eq:2-2term} for $B_1' = S^{(p)}$, $B_2' = S^{(q)}$ can be estimated using the definition of the stopping time $\tau$ in \eqref{eq:tau} the integration rules \eqref{eq:int_rules}, and the fact that $\eta_{j,t} \ge |\im z_j| \ge N^{-1+\varepsilon}$ by definition of $\mathcal{D}$ in \eqref{eq:calD_def}, to obtain the very-high-probability bound
	%	\begin{equation} 
		%		\int_0^t\bigl\lvert	\bigl\langle \mathscr{S}[G_{1,s} S^{(p)} G_{2,s} - M_s^{(p)}] (G_{2,s} S^{(q)} G_{1,s} - \M{s}^{(q)})\bigr\rangle \bigr\rvert \mathrm{d}s \le \frac{N^{4\varepsilon'-\tfrac{1}{2}\varepsilon}}{N\eta_{1,t}\eta_{2,t}}\int_{0}^t\frac{\mathrm{d}s}{\eta_s}\le \frac{1}{N\eta_{1,t}\eta_{2,t}}.
		%	\end{equation}
	Combining \eqref{eq:qvterm}, \eqref{eq:3-1term}, \eqref{eq:2-2term} and using the definition of the stopping time $\tau_1$ in \eqref{eq:tau} to estimate the individual factors on the right-hand side of \eqref{eq:2-2term} for $B_1' = S^{(p)}$, $B_2' = S^{(q)}$, together with the integration rules \eqref{eq:int_rules}, and the fact that $\eta_{j,t} \ge |\im z_j| \ge N^{-1+\varepsilon}$ by definition of $\mathcal{D}$ in \eqref{eq:calD_def}, yields \eqref{eq:X_mart_forcing_bounds} for $(k,\mu) = (2,\mathrm{av})$. This concludes the proof of \eqref{eq:X_mart_forcing_bounds}.
	
	%
	%
	%
	%==============================================================
	%
	%
 	\textbf{Step 2.}  Next, we prove that the local laws \eqref{eq:1G_laws}--\eqref{eq:Weak1iso} hold with only $B_2$ replaced by $S^{(q)}$ for any $q \in \{1,\dots, N\}$. To this end, we introduce the quantities
 	\begin{equation}
 		\begin{split}
 			\mathcal{Y}^{1,\mathrm{av}}(t) &:= \langle (G_{1,t}-M_{1,t})B_1 \rangle, \quad \mathcal{Y}^{1,\mathrm{iso}}(t) := \langle \vect{x}_1,(G_{1,t}B_1G_{2,t}-M^{B_1}_{[1,2],t})\vect{x}_2 \rangle,\\ \mathcal{Y}^{2,\mathrm{av}}(t) &:= \bigl(\langle (G_{1,t}B_1G_{2,t}-M^{B_1}_{[1,2],t})S^{(q)} \rangle\bigr)_{q=1}^N,
 		\end{split}
 	\end{equation}
 	where $M^{B_1}_{[1,2],t} := \M{\vect{z}_{1,t},B_1,\vect{z}_{2,t}}$ and we suppress the dependence of $\mathcal{Y}$'s on $z_1,z_2,\vect{x}_1, \vect{x}_2$ for brevity.
 	
 	Note that compared to the $\chain$ quantities, the dimension of $\mathcal{Y}^{k,\mu}$ are $k-1$, where the superscript $k$ now equal to one plus the number of observables $S^{(p)}$ in the corresponding chain.
 	For a fixed $0 < \varepsilon' \le \tfrac{1}{10}\varepsilon$ as in \eqref{eq:tau}, define a random stopping time $\tau_2$ by
 	\begin{equation}\label{eq:tau2}
 		\begin{split}
 			\tau_2 := \inf\biggl\{ t\in [0,T]: \sup_{z_1,z_2\in\mathcal{D}}\max_{\vect{x}_1,\vect{x}_2 \in \mathcal{V}}\max_{p}\biggl(&\other{\size}_t^{1,\mathrm{iso}}\bigl\lvert\mathcal{Y}^{1,\mathrm{iso}}(t)\bigr\rvert + \other{\size_t}^{2,\mathrm{av}}\bigl\lvert\mathcal{Y}^{2,\mathrm{av}}_{p}(t)\bigr\rvert \biggr) = N^{4\varepsilon'}\biggr\},
 		\end{split}
 	\end{equation}
 	where, for $k \in \{1,2\}$, we define the size parameters $\other{\size}$ adjusted for a general observable $B_1\neq 0$ as
 	\begin{equation}
 		\other{\size}_t^{1,\mathrm{iso}} \equiv \other{\size}_t^{1,\mathrm{iso}}(z_1,z_2) := \max \biggl\{ \frac{\size_t^{1,\mathrm{iso}}(z_1,z_2)}{\norm{B_1}}, \frac{\sqrt{\eta_{1,t}\eta_{2,t}}}{\langle|B_1|^2\rangle^{1/2}} \biggr\}, \quad \other{\size}_t^{k,\mathrm{av}} \equiv \other{\size}_t^{k,\mathrm{av}}(z_1,z_2) := \frac{\size_t^{k,\mathrm{av}}(z_1,z_2)}{\langle|B_1|^2\rangle^{1/2}}.
 	\end{equation}
 	Note that the definition of $\other{\size}^{1,\mathrm{iso}}$ reflects both isotropic local laws \eqref{eq:Weak1iso_HS} and \eqref{eq:Weak1iso}.
 	Once again, by Proposition \ref{prop:global_laws}, $\tau_2 > 0$ with very high probability.
 	
 	For the remainder of the proof, we consider the implicit arguments of the $\mathcal{Y}$ quantities to be fixed.
 	Similarly to \eqref{eq:1avX_evol}--\eqref{eq:2avX_evol}, we observe that the quantities $\mathcal{Y}^{k,\mu}(t)$ satisfy the evolution equations
 	\begin{equation} \label{eq:Y_evol}
 		\mathrm{d}\mathcal{Y}^{k,\mu}(t) = \mathcal{A}_t^{k,\mu}\mathcal{Y}^{k,\mu}(t)\mathrm{d}t + N^{-1}\bigl(\chain^{k,\mu}(t)\bigr)^\mathfrak{t}\, \vect{b}_{1,t}^{k,\mu} \mathrm{d}t  + \other{\mathcal{F}}^{k,\mu}_t\mathrm{d}t + \mathrm{d}\other{\mathcal{E}}^{k,\mu}_t,
 	\end{equation}
 	for all $(k,\mu) \in \{(1,\mathrm{av}), (1,\mathrm{iso}), (2,\mathrm{av})\}$, where $\mathcal{A}_t^{1,\mathrm{av}} := \tfrac{1}{2}$, $\mathcal{A}_t^{1,\mathrm{iso}} := 1$, $\mathcal{A}_t^{2,\mathrm{av}} := (\stab_{\vect{z}_{1,t},\vect{z}_{2,t}}^{-1})^\mathfrak{t}$, and
 	\begin{equation}
 		\vect{b}_{1,t}^{1,\mathrm{av}} := \stab_{\vect{z}_{1,t},\vect{z}_{1,t}}^{-1}[\m_{1,t}^2\vect{b}_1^\mathrm{diag}], \quad \vect{b}_{1,t}^{1,\mathrm{iso}} = \vect{b}_{1,t}^{2,\mathrm{av}}:= \stab_{\vect{z}_{1,t},\vect{z}_{2,t}}^{-1}[\m_{1,t}\m_{2,t}\vect{b}_1^\mathrm{diag}],
 	\end{equation} 
 	with $\vect{b}_1^\mathrm{diag} := (B_{1,jj})_{j=1}^N$ denoting the main diagonal of $B_1$.
 	The martingale $\other{\mathcal{E}}^{k,\mu}_{t}$ and forcing  $\other{\mathcal{F}}^{2, \mathrm{av}}_{t}$ terms in \eqref{eq:Y_evol} are defined analogously to \eqref{eq:X_marts} and \eqref{eq:2av_forcing}, respectively, but with $S^{(p)}$ replaced with $B_1$, and their respective dimensions reduced from $k$ to $k-1$.
 	
 	Using the bounds \eqref{eq:1av_mart_qv}, \eqref{eq:1av_forcing_bound} for $(k,\mu) = (1,\mathrm{av})$, \eqref{eq:iso_qvterm}--\eqref{eq:iso2-2term} for $(k,\mu) = (1,\mathrm{iso})$, and \eqref{eq:qvterm}--\eqref{eq:2-2term} for $(k,\mu) = (2,\mathrm{av})$ with arbitrary small $\varepsilon' > 0$, by the local laws established in Step 1 above, we deduce that, with very high probability,
	\begin{equation} \label{eq:Y_mart_forcing_bounds}
		\begin{split}
			\max_{\vect{j}\in \{1,\dots,N\}^{k-1}}\biggl(\int_0^{t\wedge\tau_2} \bigl\lvert \other{\mathcal{F}}^{k,\mu}_{\vect{j},s} \bigr\rvert \mathrm{d}s\biggr)^2
			+ \max_{\vect{j}\in \{1,\dots,N\}^{k-1}}\biggl[\int_0^\cdot \mathrm{d} \other{\mathcal{E}}^{k,\mu}_{\vect{j},s}\biggr]_{t\wedge\tau_2} \lesssim \frac{1 + \mathds{1}_{k=2}N^{2\varepsilon'}}{\bigl(\other{\size}_{t\wedge\tau_2}^{k,\mu}\bigr)^2}.
		\end{split}
 	\end{equation}
 	Therefore, to apply Lemma \ref{lemma:Gronwall}, it suffices to bound the additional forcing term in \eqref{eq:Y_evol}, namely $(\chain^{k,\mu}(t))^\mathfrak{t}\, N^{-1}\vect{b}_{1,t}^{k,\mu}$. Using Lemma \ref{lemma:stab_flow} and \eqref{eq:X_imporved}, we obtain a very high probability estimate
 	\begin{equation}
 		\int_0^{t} \norm{\bigl(\chain^{k,\mu}(s)\bigr)^\mathfrak{t}\, \vect{b}_{1,s}^{k,\mu}}_\infty \mathrm{d}s \lesssim \int_0^{t}  \frac{ N^{3\varepsilon'/2}}{\size^{k,\mu}_s}\frac{\langle |B_1|^2\rangle^{1/2}}{\eta_s}\mathrm{d}s \lesssim \frac{N^{3\varepsilon'/2}}{\other{\size}_t^{k,\mu}}.
 	\end{equation}
 	Hence, using Lemma \ref{lemma:Gronwall} and reasoning as in \eqref{eq:X_imporved}, we conclude that
 	\begin{equation} \label{eq:Y_imporved}
 		\sup_{0\le s \le t\wedge \tau_2} \max_{\vect{j}\in\{1,\dots,N\}^{k-1}} \bigl\lvert \mathcal{Y}_{\vect{j}}^{k,\mu}(s) \bigr\rvert^2  \lesssim \frac{N^{2\varepsilon'}}{\bigl(\other{\size}_{t\wedge\tau_2}^{k,\mu}\bigr)^2} \quad 
 		\text{ with very high probability}.
 	\end{equation}
 	Therefore, the stopping time $\tau_2 = T$ with very high probability, and hence the local laws \eqref{eq:1G_laws}, \eqref{eq:Weak1iso_HS}, \eqref{eq:Weak1iso}, as well as \eqref{eq:Weak2av} with $B_2 = S^{(p)}$ for all $p$, are established.
 	
 	\textbf{Step 3.} Finally, we prove \eqref{eq:Weak2av} with a general $B_2$. For this easy step, we do not need to introduce a new stopping time or use Lemma \ref{lemma:Gronwall}. Indeed, by examining the evolution equation \eqref{eq:GBGBevol}, we note that after removing the harmless $\langle G_{1,t}B_1G_{2,t}B_2 \rangle$ term from the right hand-side (see the discussion below \eqref{eq:(GAG-M)A_evol} on why this term can be omitted), the remaining terms can be estimated using the bounds \eqref{eq:qvterm}--\eqref{eq:2-2term} together with \eqref{eq:X_imporved}, \eqref{eq:Y_imporved}, and the martingale inequality \eqref{eq:mart+_est}.
 	 This concludes the proof of Lemmas \ref{lemma:1Flaw} and \ref{lemma:2G_weak}.
\end{proof}

\begin{proof}[Proof of Lemma \ref{lemma:Gronwall}]
	We focus on proving the case $k=2$, as the proof for $k=1$ differs only in using the projector $\mathcal{P}$ on one side of the variable $\mathcal{X}$ in the following proof.
	
	For $\alpha,\alpha' \in \{-,+\}$, we define the processes $ \mathcal{Z}_{\alpha\alpha',t}:= \max_{p,q}\sup_{0\le s \le t}\bigl\lvert \bigl(\mathcal{P}_{\alpha}\mathcal{X}(s) \mathcal{P}_{\alpha'}^\mathfrak{t}\bigr)_{pq} \bigr\rvert^2$, 
	where the projectors $\mathcal{P}_\pm$ are defined as $\mathcal{P}_+ := \mathcal{P}$, and $\mathcal{P}_-  := 1-\mathcal{P}$.
	Using the It\^{o}'s formula for $|(\mathcal{P}_{\alpha}\mathcal{X}(t)\mathcal{P}_{\alpha'}^\mathfrak{t})_{pq}|^2$, \eqref{eq:stabPpm}, we deduce that for all $0\le t\le T$,
	\begin{equation} \label{eq:Z_t_est}
		\begin{split}
			\mathcal{Z}_{\alpha\alpha',t\wedge\tau} \le&~ \mathcal{Z}_{\alpha\alpha',0} + 2k\int_0^{t\wedge\tau} \bigl(C + 2\bigl\lvert\re[\eigB_s^{-1}]\bigr\rvert\bigr)\mathcal{Z}_{\alpha\alpha',s}  \mathrm{d}s
			+ \other{C}\mathcal{Z}_{\alpha\alpha',t\wedge\tau}^{1/2}\max_{p,q}\int_0^{t\wedge\tau} |\mathcal{F}_{pq,s}|\mathrm{d}s\\ 
			&+  2\max_{p,q}\sup_{0\le s \le t}\biggl\lvert\int_0^{s\wedge\tau} (\mathcal{P}_{\alpha}\mathcal{X}(r)\mathcal{P}_{\alpha'}^\mathfrak{t})_{pq}\mathrm{d}(\mathcal{P}_{\alpha}\mathcal{E}_{r}\mathcal{P}_{\alpha'}^\mathfrak{t})_{pq}\biggr\rvert 
			+ \max_{p,q}\biggl[ \int_0^{\cdot} \mathrm{d}(\mathcal{P}_{\alpha}\mathcal{E}_{s}\mathcal{P}_{\alpha'}^\mathfrak{t})_{pq} \biggr]_{t\wedge\tau},
		\end{split}
	\end{equation}
	with $C:=(1+C_2)^k$, $\other{C} := (1+C_1)^k$ where we recall that $[\cdot]_{t}$ denotes the quadratic variation process.
	
	To bound the stochastic term in \eqref{eq:Z_t_est}, we use the martingale inequality \eqref{eq:mart_ineq}
	the bounds $|\mathcal{X}_{pq,t}| + |\mathcal{E}_{pq,t}| \le N^{D}$ and a dyadic argument, to deduce that
	\begin{equation} \label{eq:Z_mart_est}
		\sup_{0\le s \le t} \biggl\lvert \int\limits_0^{s\wedge \tau} (\mathcal{P}_{\alpha}\mathcal{X}(r)\mathcal{P}_{\alpha'}^\mathfrak{t})_{pq}\mathrm{d}(\mathcal{P}_{\alpha}\mathcal{E}_{r}\mathcal{P}_{\alpha'}^\mathfrak{t})_{pq} \biggr\rvert \le \biggl[ \int\limits_0^{\cdot} (\mathcal{P}_{\alpha}\mathcal{X}(r)\mathcal{P}_{\alpha'}^\mathfrak{t})_{pq}\mathrm{d}(\mathcal{P}_{\alpha}\mathcal{E}_{r}\mathcal{P}_{\alpha'}^\mathfrak{t})_{pq}\biggr]_{t\wedge \tau}^{1/2} \log N + \frac{1}{N^D},
	\end{equation}
	with very high probability.
	Furthermore, by definition of $\mathcal{Z}_{\alpha\alpha',s}$, we obtain
	\begin{equation} \label{eq:mart_qv_est}
		\biggl[ \int_0^{\cdot} (\mathcal{P}_{\alpha}\mathcal{X}(r)\mathcal{P}_{\alpha'}^\mathfrak{t})_{pq}\mathrm{d}(\mathcal{P}_{\alpha}\mathcal{E}_{r}\mathcal{P}_{\alpha'}^\mathfrak{t})_{pq} \biggr]_{t\wedge \tau} \le \mathcal{Z}_{\alpha\alpha',t\wedge\tau} \biggl[ \int_0^{\cdot}\mathrm{d}(\mathcal{P}_{\alpha}\mathcal{E}_{r}\mathcal{P}_{\alpha'}^\mathfrak{t})_{pq} \biggr]_{t\wedge \tau}.
	\end{equation}
	To bound the quadratic variation processes in \eqref{eq:Z_t_est} and \eqref{eq:mart_qv_est}, we observe that by \eqref{eq:forcing_mart_assume} and the assumption  that $\norm{\mathcal{P}}_{\ell^\infty \to \ell^\infty} \le C_1$,
	\begin{equation} \label{eq:qv_bound}
		\biggl[ \int_0^{\cdot}\mathrm{d}(\mathcal{P}_{\alpha}\mathcal{E}_{r}\mathcal{P}_{\alpha'}^\mathfrak{t})_{pq} \biggr]_{t\wedge \tau} \le \frac{C_1^4}{2} \max_{p',q'}\biggl[ \int_0^{\cdot}\mathrm{d}\mathcal{E}_{p'q',r}\biggr]_{t\wedge \tau} \le \frac{C_1^4}{2} h_{t\wedge\tau}^2.
	\end{equation}
	Therefore, combining  \eqref{eq:forcing_mart_assume}, \eqref{eq:Z_t_est}, \eqref{eq:Z_mart_est}--\eqref{eq:qv_bound}, we obtain for all $\alpha,\alpha'\in\{+,-\}$,
	\begin{equation} \label{eq:Z_alpha_gron}
		\mathcal{Z}_{\alpha\alpha',t\wedge\tau} \le (2k + N^{-\theta})\int_0^{t\wedge\tau} \bigl(C + \bigl\lvert\re[\eigB_s^{-1}]\bigr\rvert\bigr)\mathcal{Z}_{\alpha\alpha',s}  \mathrm{d}s
		+  N^{3\theta}h_{t\wedge\tau}^2,\quad \text{w.v.h.p.}
	\end{equation}
	Note that $\mathcal{Z}_t \le \sum_{\alpha,\alpha' =\pm}\mathcal{Z}_{\alpha\alpha',t} \le 4(1+C_1)^k \mathcal{Z}_t$, hence, using the standard Gronwall inequality, we conclude \eqref{eq:Z_improved_bound} from \eqref{eq:Z_alpha_gron}.
\end{proof}

\appendix
\section{Additional Technical Results}
\subsection{Properties of the Characteristic Flow. Proof of Lemmas \ref{lemma:flow_exists} and \ref{lemma:imz}} \label{sec:eta_sec}
%\begin{proof}[Proof of Lemma \ref{lemma:eta_t}]
%	Taking the imaginary part of \eqref{eq:flow_explicit}, and using the fact $S_{jk} \ge 0$ by definition, the second bound in \eqref{as:S_flat}, Assumption \eqref{as:m_bound}, and that $(\im z )(\im \m(z)) > 0$, we obtain, for all $z\in\mathcal{D}$ and all $t\in[0,T]$,
%	\begin{equation} \label{eq:eta_ineq}
	%		 1 \gtrsim \mathrm{e}^{(T-t)/2}\eta_* + 2\sinh\bigl(\tfrac{T-t}{2}\bigr) C_\mathrm{sup}C_{\mathrm{m}} \ge |\langle \im \mathfrak{f}^t(z) \rangle| \ge \mathrm{e}^{(T-t)/2}|\im z| \ge |\im z| \ge N^{-1+\varepsilon}.
	%	\end{equation}
%	Here we used $T\sim 1$, and the definition of $\mathcal{D}$ in \eqref{eq:calD_def}.  Since $\eta_t = |\langle \im \vect{z}_t \rangle| = |\langle \im \mathfrak{f}^t(z) \rangle|$, this implies \eqref{eq:eta_t_lower} and concludes the proof of Lemma \ref{lemma:eta_t}.
%\end{proof}
 
\begin{proof}[Proof of Lemma \ref{lemma:flow_exists}]
	For addressing the existence of the solutions, it is convenient to study the time-reversed version of the characteristic flow \eqref{eq:z_evol}, namely
	\begin{equation} \label{eq:z_flow-t}
		\partial_t\vect{z}_t =  \frac{1}{2}\vect{z}_t + S[\m(\vect{z_t})],
	\end{equation}
	with the initial condition $\vect{z}_0 = z \vect{1}$. It suffices to show that any $z\in \mathbb{H}$, the solution $\vect{z}_t$ exists for $t\in[0,T]$ for any $T > 0$.
	
	We begin by obtaining crude bounds for the solution $\m(\vect{z})$ of the generalized vector Dyson equation \eqref{eq:VDE} for $\vect{z} \in \mathbb{H}^N$. Recall that for such $\vect{z}$, $\m(\vect{z}) \in \mathbb{H}^N$. 
	Let $h(\vect{z}) := \min\{ h \ge 0 : \im\vect{z} \ge h\,\vect{1}\}$ and note that for any $\vect{z} \in \mathbb{H}^N$, $h(\vect{z}) > 0$. 
	By taking the imaginary part of \eqref{eq:VDE}, using the fact that the entries of $S$ are non-negative and $\im\m(\vect{z}) > 0$, we obtain the trivial bound $\im\bigl[-\m(\vect{z})^{-1}\bigr] \ge h(\vect{z})\vect{1}$, that implies
	\begin{equation} \label{eq:m_trivial_upper}
		\norm{\m(\vect{z})}_\infty \le h(\vect{z})^{-1}.
	\end{equation}
	Taking the $\ell^\infty$-norm of both sides of \eqref{eq:VDE} and using the upper bounds in \eqref{eq:vect_a}, Assumption \eqref{as:S_flat}, and \eqref{eq:m_trivial_upper}, to obtain the estimate for the reciprocal of $\m$, we conclude
	\begin{equation} \label{eq:m_trivial}
		\bigl(1 + \norm{\vect{z}}_\infty + h(\vect{z})^{-1}\bigr)^{-1}\vect{1} \lesssim |\m(\vect{z})| \le h(\vect{z})^{-1} \vect{1}.
	\end{equation}
	Define the Hermitian operator $F(\vect{z}) := |\m(\vect{z})|S|\m(\vect{z})|$, and let $\vv(\vect{z})$ be the unique positive eigenvector of $F(\vect{z})$ (that exists by Lemma 5.6 in \cite{Ajanki2019QVE}). 
%	Using \eqref{eq:m_trivial} and Lemma 5.6 in \cite{Ajanki2019QVE}, we deduce that the \textit{unique} $\ell^2$-normalized eigenvector $\vv(\vect{z}) \ge 0$, corresponding to the maximal eigenvalue $\norm{F(\vect{z})}$ of $F(\vect{z})$, satisfies
%	\begin{equation}
%		\other{h}(\vect{z})^{-2L}\vect{1} \lesssim \frac{\vv(\vect{z})}{\langle \vv(\vect{z}) \rangle} \lesssim \other{h}(\vect{z})^4\vect{1}, \quad \other{h}(\vect{z}) :=  h(\vect{z})^{-1} + h(\vect{z})^{-1}\norm{\vect{z}}_\infty + h(\vect{z})^{-2}.
%	\end{equation}
	Using \eqref{eq:m_trivial}, Lemma 5.5 in \cite{Ajanki2019QVE}, we conclude that
	\begin{equation}
		1 -\norm{F(\vect{z})} \ge \frac{\bigl\langle \vv(\vect{z}), |\m(\vect{z})| \im\vect{z}  \bigr\rangle}{\langle \vv(\vect{z}), |\m(\vect{z})|^{-1}\im\m(\vect{z}) \rangle} \gtrsim \frac{h(\vect{z})}{1 + \norm{\vect{z}}_\infty + h(\vect{z})^{-1}},
	\end{equation}
	where we used that $|\m(\vect{z})|^{-1}\im\m(\vect{z}) \le \vect{1}$. Therefore, the derivative of $\m(\vect{z})$ in the direction $\vect{w} \in \mathbb{C}^N$ admits the bound
	\begin{equation} \label{eq:F_crude_norm}
		\norm{(\partial_{\vect{w}}\m)(\vect{z})}_2 = \norm{(1-\m(\vect{z})^2S)^{-1}[\m(\vect{z})^2\vect{w}]}_2 \lesssim \frac{1 + \norm{\vect{z}}_\infty + h(\vect{z})^{-1}}{h(\vect{z})^{3}}\norm{\vect{w}}_2,
	\end{equation}
	hence the map $\m : \mathbb{H}^N \to \mathbb{H}^N$ is uniformly Lipschitz-continuous in the domain 
	\begin{equation}
		\Omega_{h_0} := \{\vect{z}\in\mathbb{H} : h(\vect{z}) \ge h_0, \norm{\vect{z}}_\infty \le h_0^{-1}\}, \quad h_0 > 0.
	\end{equation}
	Hence, by standard ODE arguments, for any $z \in \mathbb{H}$, there exists a unique solution $\vect{z}_t$, $t \in [0, T(z)]$ to \eqref{eq:z_flow-t} with $\vect{z}_0 = z\vect{1}$ that extends to the boundary of the domain $\Omega_{h_0}$ with $0 < h_0 
	 < \im z$. Moreover, since $S$ is positivity-preserving, it follows from \eqref{eq:z_flow-t} that $\im \vect{z}_t \ge \im\vect{z}_0$ and hence $\norm{\vect{z}_{T(z)}}_\infty = h_0^{-1}$.
	 
	 Differentiating \eqref{eq:VDE} with respect to time along the solution $\vect{z} = \vect{z}_t$, we obtain
	 \begin{equation} \label{eq:m_flow-t}
	 	\bigl(1-\m(\vect{z}_t)^2S\bigr)\bigl[\partial_t\m(\vect{z}_t)\bigr] = -\frac{1}{2}\bigl(1-\m(\vect{z}_t)^2S\bigr)\bigl[\m(\vect{z}_t)\bigr], \quad 
	 	\text{hence} \quad \partial_t\m(\vect{z}_t) = -\frac{1}{2}\m(\vect{z}_t),
	 \end{equation}
 	 where we used that the operator $1-\m^2S$ in invertible in $\Omega_{h_0}$ by \eqref{eq:F_crude_norm}. In particular, $\m(\vect{z}_t) = \mathrm{e}^{-t/2}\m(z)$, and a simple Gronwall estimate implies that 
 	 \begin{equation}
 	 	\norm{\vect{z}_t}_\infty \le C\mathrm{e}^{t}\bigl(|z| + \sqrt{1-\mathrm{e}^{-t}}\norm{\m(z)}_\infty\bigr) \le C\mathrm{e}^{t}\bigl(C_\mathfrak{m}+ |z|\bigr),
 	 \end{equation}
  	 for some $C\sim 1$, where we used the upper bound on $\norm{\m}_\infty$ from Assumption \eqref{as:m_bound}.
 	 Therefore, for any $T>0$, setting $h_0 := \min\{\im z/2, \mathrm{e}^{-2T}C^{-1}(C_\mathfrak{m}+|z|)^{-1}\}$ implies that $T(z) \ge T$.
 	 
 	 The sign preservation and the lower bound in \eqref{eta_t_lower} follow immediately from \eqref{eq:z_flow-t} and the fact that $\im\m(\vect{z}) > 0$ for $\vect{z} \in \mathbb{H}^N$. 
 	 The evolution equation \eqref{eq:m_evol} follows immediately from \eqref{eq:m_flow-t} by reversing the time. This concludes the proof of Lemma \ref{lemma:flow_exists}.
\end{proof}

\begin{proof}[Proof of Lemma \ref{lemma:imz}]
	Recall that for $t\in[0,T]$ and $z \in \mathcal{D}\cap \mathbb{H}$, we $\vect{z}_t = \mathfrak{f}^t(z)$, where the flow map $\mathfrak{f}^t$ is defined in \eqref{eq:flow_map}. 
	First, similarly to \eqref{eq:S_norms}, Assumption \eqref{as:S_flat} and \eqref{eq:im_m_flat} imply that for all $z\in\mathbb{H}$,
	\begin{equation}
		C_{\sup}^{1-L}c_{\inf} \rho(z) \vect{1} \le  \rho(z) S[\vect{1}]  \lesssim  S[\im\m(z)]  \lesssim \rho(z) S[\vect{1}] \le \rho(z) C_{\sup} \vect{1},
	\end{equation}
	where $\rho(z) := |\langle \im \m(z) \rangle|$. Moreover, the equation \eqref{eq:VDEz} and the lower bound in \eqref{eq:m_lower} imply that $\rho(z) \gtrsim |\im z|/(1+|z|^2)$. Hence by definition of $\mathcal{D}$ in \eqref{eq:calD_def}, \eqref{eq:bulk_bounded}, and the H\"{o}lder-regularity of $\m$ in \eqref{eq:m_holder}, we have the bound $\rho(z) \sim 1$  for all $z \in \mathcal{D}$.
	
	Therefore, taking the imaginary part of \eqref{eq:flow_explicit}  and using $(\im z )(\im \m(z)) > 0$,  we conclude that
	\begin{equation} \label{eq:im_z_comp}
		(\sign \im z) \im \mathfrak{f}^t(z) \sim \bigl(\mathrm{e}^{(T-t)/2}|\im z| + \sinh\bigl(\tfrac{T-t}{2}\bigr)\rho(z)\bigr) \vect{1} \sim \bigl(|\im z| + (T-t)\bigr)\vect{1},
	\end{equation}
	for all $z \in \mathcal{D}$ and all $t\in [0,T]$. Taking the average of \eqref{eq:im_z_comp} implies both \eqref{eq:imz_flat} and \eqref{eq:eta_asymp}. The integration rules in \eqref{eq:int_rules} follow immediately from \eqref{eq:eta_asymp}. This concludes the proof of Lemma \ref{lemma:imz}.
\end{proof}

\subsection{Stability Operator. Proof of Lemmas \ref{lemma:stab_lemma}, \ref{lemma:Mbounds}, \ref{lemma:Pi}, and \ref{lemma:stab_flow}} \label{sec:stab_section}
Recall the definition of the stability operator $\stab_{z_1,z_2} := 1 - M(z_1)M(z_2)S$ from \eqref{eq:stab_def}, and observe that
\begin{equation} \label{eq:stab_identity}
	\stab_{z_1,z_2} =  |M(z_1)M(z_2)|^{1/2}U_{z_1,z_2}\bigl(U_{z_1,z_2}^* - F_{z_1,z_2}\bigr)|M(z_1)M(z_2)|^{-1/2},
\end{equation}
where $U_{z_1,z_2} := M(z_1)M(z_2)/|M(z_1)M(z_2)|$, and the \textit{saturated self-energy operator} $F_{z_1,z_2}$ is defined as
\begin{equation} \label{eq:F_op_def}
	F_{z_1,z_2} := |M(z_1)M(z_2)|^{1/2}S|M(z_1)M(z_2)|^{1/2}.
\end{equation}
In the following lemma, we summarize several important properties of the operator $F_{z_1,z_2}$ that were proved in \cite{Ajanki2019QVE} and \cite{Landon2021Wignertype}.
\begin{lemma}[Properties of the Saturated Self-Energy Operator] \label{lemma:F_op}
	Provided that Assumptions \eqref{as:S_flat} and \eqref{as:m_bound} hold, the operator $F_{z_1,z_2}$ defined in \eqref{eq:F_op_def} satisfies
	\begin{itemize}
		\item [(i)] (Lemma 4.4 in \cite{Landon2021Wignertype}) For all $z_1,z_2 \in \mathbb{C}\backslash\mathbb{R}$,
		\begin{equation} \label{eq:2bodyF_bound}
			\norm{F_{z_1,z_2}} \le \frac{1}{2}\norm{F_{z_1,z_1}} + \frac{1}{2}\norm{F_{z_2,z_2}}.
		\end{equation}
		\item [(ii)] (Lemma 5.6 in \cite{Ajanki2019QVE}) For all $z_1, z_2\in \mathbb{H}$ satisfying $|z_1|,|z_2| \lesssim 1$, the operator $F_{z_1,z_2}$ has a simple eigenvalue $\norm{F_{z_1,z_2}}$ and an order one \textit{spectral gap}, that is
		\begin{equation} \label{eq:gapF_bound}
			\Gap{F_{z_1,z_2}} \sim 1,
		\end{equation}
		where $\Gap{T}$ is the difference between the two largest eigenvalues of $|T| = \sqrt{T^*T}$ ($\Gap{T} = 0$ if the largest eigenvalue of $|T|$ is degenerate). Moreover, there exists a unique $\ell^2$-normalized eigenvector $\vect{v}(z_1,z_2)$ satisfying
		\begin{equation} \label{eq:v_vect}
			F_{z_1,z_2}\vect{v}(z_1,z_2) = \norm{F_{z_1,z_2}}\vect{v}(z_1,z_2), \quad \vect{v}(z_1,z_2) \sim N^{-1/2}\vect{1}.
		\end{equation}
		%		Finally, for all $z_1,z_2 \in \mathbb{C}\backslash\mathbb{R}$, and all $\zeta \notin\spec(F_{z_1,z_2})\cup\{0\}$,
		%		\begin{equation}
			%			\norm{(\zeta - F_{z_1,z_2})^{-1}}_{\ell^\infty \to \ell^\infty} \lesssim |\zeta|^{-1}\bigl(1 + \norm{(\zeta - F_{z_1,z_2})^{-1}}\bigr).
			%		\end{equation}
		\item [(iii)] (Lemma 5.5 in \cite{Ajanki2019QVE}) For all $z \in \mathbb{C}\backslash\mathbb{R}$, the norm of the one-body saturated self-energy operator $F_{z,z}$ satisfies
		\begin{equation} \label{eq:F_bound}
			\norm{F_{z,z}} = 1 - |\im z|\frac{\bigl\langle \vect{v}(z,z), |\m(z)| \bigr\rangle}{\bigl\langle \vect{v}(z,z), |\m(z)|^{-1}|\im \m(z)| \bigr\rangle},
		\end{equation}
		where $\vect{v}(z,z)$ is the principal eigenvector of $F_{z,z}$ defined in \eqref{eq:v_vect}.
	\end{itemize}
\end{lemma}
Equipped with Lemma \ref{lemma:F_op}, we are ready to prove Lemma \ref{lemma:stab_lemma}.
\begin{proof}[Proof of Lemma \ref{lemma:stab_lemma}]
	We begin by proving the estimate \eqref{eq:far_stability} in the long-range regime $|z_1 - z_2| \ge c$ for some $c\sim 1$ and $z_1, z_2 \in \bulk\cap\mathbb{H}$.  Note that it suffices to prove the $\ell^2\to\ell^2$ bound. Indeed, the $\ell^\infty\to\ell^\infty$ bound then follows by from the upper bound in Assumptions \eqref{as:S_flat}, \eqref{as:m_bound} and the identity $(1-\stab)^{-1} = 1 + \stab + \stab(1-\stab)^{-1}\stab$.	
	We abbreviate $\stab := \stab_{\bar z_1,z_2}$, $F := F_{\bar z_1,z_2}$, and $U := U_{\bar z_1,z_2}$. Using the identity \eqref{eq:stab_identity}, the bounds \eqref{eq:m_lower}, and the rotatation-inversion lemma (Lemma 5.10 in \cite{Ajanki2019QVE}), we conclude that
	\begin{equation} \label{eq:rotation_inversion}
		\norm{\stab^{-1}} \lesssim \frac{1}{\Gap {F} \bigl\lvert 1 - \norm{F }\langle \vect{v}, U \vect{v} \rangle \bigr\rvert} \lesssim \frac{1}{1-\norm{F} + \bigl\lvert \langle \vect{v}, (1-U)\vect{v} \rangle\bigr\rvert} \lesssim \frac{1}{1-\norm{F} + \norm{(1-U)\vect{v}}_2^2},
	\end{equation}
	where $\vect{v} := \vect{v}(z_1,z_2)$ is the principal eigenvector of $F$ defined in \eqref{eq:v_vect}. Here, to obtain the second estimate we used \eqref{eq:gapF_bound}, and the bound $\norm{F}\le 1$ that follows from \eqref{eq:2bodyF_bound}, \eqref{eq:F_bound}. Next, recall the identity \eqref{eq:m_diff_identity} (with $\zeta_j := z_j$) that implies via \eqref{eq:stab_identity}
	\begin{equation} \label{eq:m-m_F_identity}
		(U^*-F)\biggl[\frac{\Delta\m}{|\m_1\m_2|^{1/2}}\biggr] = U^*|M_1M_2|^{-1/2}\stab [\Delta\m] = (\bar z_1 - z_2) |\m_1\m_2|^{1/2}, \quad \Delta\m := \overline \m_1-\m_2.
	\end{equation}
	Multiplying the identity \eqref{eq:m-m_F_identity} by the eigenvector $\vect{v}$ of $F$, we obtain
	\begin{equation} \label{eq:m-m_v_idenity}
		\bigl\langle (U-1)\vect{v}, |\m_1\m_2|^{-1/2}\Delta\m \bigr\rangle + (1-\norm{F})\bigl\langle \vect{v}, |\m_1\m_2|^{-1/2}\Delta\m \bigr\rangle = (\bar z_1 - z_2) \langle \vect{v}, |\m_1\m_2|^{1/2}\rangle.
	\end{equation}
	Therefore, it follows from \eqref{eq:m_lower}, \eqref{eq:v_vect} and \eqref{eq:m-m_v_idenity} that
	Since $\norm{\vect{\Delta\m}}_2 \lesssim 1$ by \eqref{eq:m_lower}, the left-hand side of \eqref{eq:m-m_v_idenity} satisfies
	\begin{equation} \label{eq:m-m_lower}
		\norm{(I-U)\vect{v}}_2 + 1 - \norm{F} \gtrsim |\bar{z}_1-z_2|\langle \vect{v}, |\m_1\m_2|^{1/2}\rangle \gtrsim |\bar{z}_1 - z_2|,
	\end{equation}
	for all $z_1, z_2 \in \mathbb{H}$ with $|z_1|, |z_2| \lesssim 1$.
	The estimate \eqref{eq:rotation_inversion} together with \eqref{eq:m-m_lower} yields \eqref{eq:far_stability} for $z_1, z_2 \in \bulk\cap\mathbb{H}$ satisfying $|z_1-z_2| \ge c \sim 1$.
	
	For the remainder of the proof, we consider the complementary regime, $z_1,z_2 \in \bulk \cap\mathbb{H}$ with $|z_1 - z_2| \le c$ for some $c\sim 1$. Claims 6.4 -- 6.7 in \cite{R2023bulk} guarantee that there exists a threshold $\other{\delta} \sim 1$ such that \eqref{eq:stab_gap}--\eqref{eq:Pi_separation} hold for $z_1,z_2 \in \bulk \cap\mathbb{H}$ with $|z_1-z_2|\le \other{\delta}$. Therefore, using analytic perturbation theory for a simple isolated eigenvalue of the operator $\stab_{\bar{z}_1, z_1}$, we conclude that
	\begin{equation}
		\eigB_{\bar{z}_1,z_2} = \eigB_{\bar{z}_1,z_1} + (z_2-z_1)\bigl\langle |\m_1|^{-1}\vv, \bigl(\partial_{z_2}\stab_{\bar{z}_1, z_2}\bigr)\lvert_{z_2 = z_1}|\m_1|\vv  \bigr\rangle + \mathcal{O}(|z_1-z_2|^2).
	\end{equation}
	Here, we used the identity $\partial_z^2\m(z) = 2\stab_{z,z}^{-1}[\partial_z\m(z)/\m(z)]$ and the first bound in \eqref{eq:scalar_stab_bounds} to estimate the second derivative of $\eigB$. Similar computations are performed in \cite{Landon2021Wignertype} for the operator $F_{z_1,z_2}$, see, e.g., Proposition 6.5 in \cite{Landon2021Wignertype}.
	Moreover, taking the imaginary part of \eqref{eq:VDEz} yields
	\begin{equation}
		|\m(z)|^{-1}\im \m(z) = (\im z) (1-F_{z,z})^{-1} \bigl[|\m(z)|\bigr].
	\end{equation}
	Hence, it follows from \eqref{eq:m_lower}, \eqref{eq:im_m_flat}, and \eqref{eq:gapF_bound} that for all $z \in\bulk \cap\mathbb{H}$,
	\begin{equation} \label{eq:v_f_bound}
		\norm{\vect{v}(z,z)- \vect{f}(z)}_2 \lesssim |\im z|, \quad \vect{f}(z) := |\m(z)|^{-1}\im\m(z) \norm{|\m(z)|^{-1}\im\m(z)}_2^{-1}.
	\end{equation}
	Using the identities $\partial_z \m(z) = \m(z)^2(1-S\m(z)^2)^{-1}[\vect{1}]$, $|\m(z)|^2S[\im \m(z)] = \im\m(z) -  (\im z) |\m(z)|^2$, and $2\m \im\m = \I(|\m|^2 - \m^2)$, we compute
	\begin{equation} \label{eq:deriv_iden}
		\begin{split}
			\bigl\langle |\m_1|^{-1}\vv, \bigl(\partial_{z_2}\stab_{\bar{z}_1, z_2}\bigr)\lvert_{z_2 = z_1}|\m_1|\vv  \bigr\rangle &=  -\frac{\bigl\langle S[\im\m_1], \im\m_1  \m_1\bigl(1-S\m_1^2\bigr)^{-1}[\vect{1}]\bigr\rangle}{\norm{|\m_1|^{-1}\im\m_1}_2^2} + \mathcal{O}\bigl(|\im z_1|\bigr) \\
			&= -\I\frac{\langle \im\m_1, \vect{1}\rangle}{2\norm{|\m_1|^{-1}\im\m_1}_2^2}  + \mathcal{O}\bigl(|\im z_1|\bigr).
		\end{split}
	\end{equation}
	Therefore, using \eqref{eq:v_vect}, \eqref{eq:F_bound}, \eqref{eq:v_f_bound} and \eqref{eq:deriv_iden}, we find the asymptotic expansion
	\begin{equation}\label{eq:beta_expan_proof}
		\eigB_{\bar z_1,z_2} = \I\frac{\bar z_1 - z_2}{\kappa(z_1)} + \mathcal{O}(|\bar{z}_1 - z_2|^2), \quad \kappa(z):=\frac{2}{\langle\im \m(z), \vect{1}\rangle }\norm{ \frac{\im\m(z)}{|\m(z)|}}_2^{2} \sim 1, \quad z \in \mathcal{D}.
	\end{equation}
	In particular, in the regime $z_1,z_2 \in \bulk\cap\mathbb{H}$ with $|z_1 - z_2| \le \other{\delta}$, the estimate \eqref{eq:far_stability} follows from \eqref{eq:beta_expan_proof}, the decomposition $\stab_{\bar{z}_1,z_2}^{-1} = \eigB_{\bar{z}_1,z_2}^{-1}\Pi_{\bar{z}_1,z_2} + \stab_{\bar{z}_1,z_2}^{-1}(1-\Pi_{\bar{z}_1,z_2})$, and \eqref{eq:stab_gap}, \eqref{eq:Pi_int}. This concludes the proof of \eqref{eq:far_stability}.
	
	Next, we prove \eqref{eq:Pi_bound}. Fix  $z_1,z_2 \in \bulk \cap\mathbb{H}$ with $|z_1 - z_2| \le \other{\delta}$, and let $\Pi := \Pi_{\bar z_1,z_2}$. First, we observe that the contour-integral representation \eqref{eq:Pi_int} and the second estimate in \eqref{eq:Pi_separation} imply that
	\begin{equation}
		\norm{\Pi}_{\ell^\infty \to \ell^\infty} + \norm{\Pi^\mathfrak{t}}_{\ell^\infty \to \ell^\infty} \lesssim 1. 
	\end{equation}
	Furthermore, it follows from \eqref{eq:Pi_int}, \eqref{eq:Pi_separation} and \eqref{eq:m_continuity}, that for all $z_1,z_2,z_3 \in \bulk\cap\mathbb{H}$ satisfying $|z_1 - z_2|, |z_1-z_3| \le \other{\delta}$
	\begin{equation} \label{eq:Pi_inf_continuity}
		\norm{\Pi_{\bar z_1, z_2}-\Pi_{\bar z_1, z_3}}_{\ell^\infty \to \ell^\infty} + \norm{\Pi_{\bar z_1, z_2}^\mathfrak{t}-\Pi_{\bar z_1, z_3}^\mathfrak{t}}_{\ell^\infty \to \ell^\infty} \lesssim |z_2-z_3|. 
	\end{equation}
	In particular, since $\Pi_{\bar{z}_1, z_1} = |\m_1|\vv_{\bar z_1, z_1} (|\m_1|^{-1}\vv_{\bar z_1, z_1})^*$ and $\vv_{\bar z_1,z_1} \sim N^{-1/2}$ by \eqref{eq:v_vect}, there exists a threshold $1 \lesssim \delta \le \other{\delta}$, such that if $|z_1-z_2| \le \delta$, the vectors $\vect{r} := \vect{r}_{\bar{z}_1,z_2} := \Pi[|\m_1|\vv_{\bar{z}_1,z_1}]$ and $\bm{\ell} := \bm{\ell}_{\bar{z}_1,z_2} :=\Pi^* [|\m_1|^{-1}\vv_{\bar{z}_1,z_1} ]$ satisfy
	\begin{equation} \label{eq:rl_props}
		\lvert \vect{r} \rvert \sim N^{-1/2}\vect{1}, \quad \lvert \bm{\ell} \rvert \sim N^{-1/2}\vect{1}, \quad 
		\lvert \langle\vect{r},\bm{\ell}\rangle \rvert \sim 1,
	\end{equation}
	where we used \eqref{eq:Pi_inf_continuity} with $z_3 := z_1$.
	It follows from \eqref{eq:rl_props} and the fact that $\mathrm{rank}\,\Pi = 1$ that
	\begin{equation}
		\Pi = \frac{\vect{r}\,\bm{\ell}^*}{\langle \bm{\ell}, \vect{r} \rangle}, \quad \norm{\Pi}_{\ell^\infty\to\ell^\infty}  = \norm{\Pi^\mathfrak{t}}_{\ell^\infty\to\ell^\infty} = \frac{\norm{\vect{r}}_\infty \norm{\bm{\ell}}_\infty}{|\langle \bm{\ell},\vect{r}\rangle|} \lesssim N^{-1},
	\end{equation}
	hence \eqref{eq:Pi_bound} is established.
%
%	The identity \eqref{eq:Pi_barzz} follows from \eqref{eq:v_f_bound} and H\"{o}lder-regularity of the solution $\m$ in \eqref{eq:m_holder}.
	
	Finally, we prove \eqref{eq:Pi_continuity}. It follows form \eqref{eq:Pi_inf_continuity}, that for all $z_1,z_2,z_3 \in \bulk\cap\mathbb{H}$ satisfying $|z_1 - z_2|, |z_1-z_3| \le \delta$,
	\begin{equation} \label{eq:rl_continuity}
		\norm{\vect{r}_{\bar{z}_1,z_2}-\vect{r}_{\bar{z}_1,z_3}}_\infty + \norm{\bm{\ell}_{\bar{z}_1,z_2} - \bm{\ell}_{\bar{z}_1,z_3}}_\infty \lesssim |z_2-z_3|.
	\end{equation}
	By \eqref{eq:rl_props}, the estimates \eqref{eq:rl_continuity} imply \eqref{eq:Pi_continuity} immediately. This concludes the proof of Lemma \ref{lemma:stab_lemma}.
\end{proof}
\begin{proof}[Proof of Lemma \ref{lemma:Pi}]
	The asymptotic expansion \eqref{eq:beta_expan} is proved in \eqref{eq:beta_expan_proof} above.
\end{proof}

\begin{proof}[Proof of Lemma \ref{lemma:stab_flow}]
By analogy with \eqref{eq:stab_def}, let $\stab_{\vect{z}_{1,t}, \vect{z}_{2,t}}$ denote the time-dependent stability operator
\begin{equation}
	\stab_{\vect{z}_{1,t}, \vect{z}_{2,t}} := 1 - M_{1,t}M_{2,t}S,
\end{equation}
where we recall  $M_{j,t} := \diag{\m(\vect{z}_{j,t})}$, $\vect{z}_{j,t} := \mathfrak{f}^t(z_j)$, where the flow map is defined in \eqref{eq:flow_map}.
It follows from \eqref{eq:m_evol} that for all $t\in[0,T]$, the stability operator $\stab_{\vect{z}_{1,t}, \vect{z}_{2,t}}$ satisfies the identity
\begin{equation} \label{eq:stab_t_identity}
	\stab_{\vect{z}_{1,t}, \vect{z}_{2,t}} =  1 - \mathrm{e}^{t-T}\M{z_1}\M{z_2}S = \bigl(1 - \mathrm{e}^{t-T}\bigr) + \mathrm{e}^{t-T}\stab_{z_1,z_2}.
\end{equation}
The estimates \eqref{eq:2bodyF_bound}, \eqref{eq:v_vect}, \eqref{eq:F_bound} imply the operator inequality for all $z_1,z_2 \in \bulk$,
\begin{equation} \label{eq:1-UF_pos_def}
	1 - \re\bigl[U_{z_1,z_2}F_{z_1,z_2}\bigr] \ge 1 - \norm{F_{z_1,z_2}} \gtrsim \bigl(|\im z_1| + |\im z_2|\bigr),
\end{equation}
hence, using the identity \eqref{eq:stab_identity}, we find $\re \stab_{z_1,z_2} \gtrsim |\im z_1| + |\im z_2|$. Therefore, since the operator $\stab_{z_1,z_2}$ is invertible and $1- \mathrm{e}^{t-T} \sim T-t > 0$ for all $t \in [0,T]$ with $T\lesssim 1$, we deduce from \eqref{eq:stab_t_identity} that
\begin{equation}
	\norm{\stab_{\vect{z}_{1,t}, \vect{z}_{2,t}}^{-1}} \lesssim \frac{1}{|\im z_1| + |\im z_2| + T-t} + \norm{\stab_{\bar z_1,z_2}}^{-1}.
\end{equation} 
Therefore, \eqref{eq:stab_t_bounds} follows from the first estimate in \eqref{eq:scalar_stab_bounds} and \eqref{eq:far_stability} immediately.

Using the definition of the projector $\Pi_{z_1,z_2}$ in \eqref{eq:Pi_int} together with the identity \eqref{eq:stab_t_identity} and the bound \eqref{eq:stab_gap} yields \eqref{eq:Pi_stab_t_commute} and the first equality in \eqref{eq:eigB_t}. The comparison relation in \eqref{eq:eigB_t} follows from \eqref{eq:beta_expan} and the fact that $\re \eigB_{\bar{z}_1,z_2} > 0$ by \eqref{eq:1-UF_pos_def}. This concludes the proof of Lemma \ref{lemma:stab_flow}
\end{proof}

\begin{proof}[Proof of Lemma \ref{lemma:Mbounds}] 
	The bound \eqref{eq:m_upper} follows trivially from Assumption \eqref{as:m_bound} and the evolution equation for $\m(\vect{z}_t) $ in \eqref{eq:m_evol}.
	Next, we prove the evolution equation \eqref{eq:M_evol}. It follows from \eqref{eq:m_evol} that
	\begin{equation}
		\partial_t(1-M_{1,t}M_{2,t}\mathscr{S})^{-1} = (1-M_{1,t}M_{2,t}\mathscr{S})^{-1}(M_{1,t}M_{2,t}\mathscr{S})(1-M_{1,t}M_{2,t}\mathscr{S})^{-1},
	\end{equation}
	where $M_{j,t}$ are defined in \eqref{eq:Mt_def}. Therefore, by definition of $M_{[1,2],t} := \M{\vect{z}_{1,t},B_1,\vect{z}_{2,t}}$ in \eqref{eq:M_t_def}, we have the identity
	\begin{equation}
		\partial_t M_{[1,2],t}  = M_{[1,2],t}  +  M_{1,t}M_{2,t}(1-\mathscr{S}M_{1,t}M_{2,t})^{-1}\mathscr{S}(1-M_{1,t}M_{2,t}\mathscr{S})^{-1}\bigl[M_{1,t}B_1M_{2,t}\bigr],
	\end{equation}
	which yields \eqref{eq:M_evol} when tested against a matrix $B_2$. 
	
	Next, we prove the estimate \eqref{eq:regM_bound} for regular observable $A$. In view of the expression \eqref{eq:M_offdiag} and the trivial estimates $\norm{A^{\mathrm{od}}} \le 2\norm{A}$ and $\langle |A^{\mathrm{od}}|^2\rangle \le 2\langle |A|^2 \rangle$, it suffices to consider only the diagonal part $\mathrm{diag}(\vect{a}^\mathrm{diag})$ of the observable $A$. Note that
	\begin{equation} \label{eq:a_diag_bounds}
		\bigl\lVert\vect{a}^\mathrm{diag}\bigr\rVert_1 \le \sqrt{N}\bigl\lVert\vect{a}^\mathrm{diag}\bigr\rVert_2 \le N \langle |A|^2\rangle^{1/2} \le N \norm{A}, \quad \bigl\lVert \vect{a}^\mathrm{diag} \bigr\rVert_\infty \lesssim \norm{A}.
	\end{equation}
	On the other hand, for any $\vect{x} \in \mathbb{C}^N$,
	\begin{equation} \label{eq:diag_norm_bounds}
		\norm{\diag{\vect{x}}} \lesssim \norm{\vect{x}}_\infty,\quad \langle |\diag{\vect{x}}|^2\rangle^{1/2} \lesssim N^{-1/2}\norm{\vect{x}}_2 \lesssim \norm{\vect{x}}_\infty.
	\end{equation}
	Hence both estimates in \eqref{eq:regM_bound} follow from \eqref{eq:stab_t_bounds} immediately in the regime $(\im z_1)(\im z_2) > 0$ or $\min\{|z_1-z_2|, |\bar z_1 -z_2|\} > \tfrac{1}{2}\delta$. In the complementary regime, we employ the decomposition
	\begin{equation} \label{eq:stab_a_decomp}
		\stab_t^{-1}\bigl[\m_{1,t}\m_{2,t}\vect{a}^\mathrm{diag}\bigr] = \eigB_t^{-1} \Pi\bigl[\m_{1,t}\m_{2,t}\vect{a}^\mathrm{diag}\bigr] + \stab_t^{-1}(1-\Pi)\bigl[\m_{1,t}\m_{2,t}\vect{a}^\mathrm{diag}\bigr],
	\end{equation}
	where we abbreviate $\stab_t:= \stab_{\vect{z}_{1,t}, \vect{z}_{2,t}}$, $\eigB_t := \eigB_{\vect{z}_{1,t},\vect{z}_{2,t}}$, $\Pi := \Pi_{z_1,z_2}$, and used the commutation relations in \eqref{eq:Pi_stab_t_commute}.
	The estimate in \eqref{eq:Pi_stab_t_commute} implies that
	\begin{equation} \label{eq:a_leftovers}
		\norm{\stab_{\vect{z}_{1,t}, \vect{z}_{2,t}}^{-1}(1-\Pi_{z_1,z_2})\bigl[\vect{a}^\mathrm{diag}\bigr]}_\infty \lesssim \norm{A}, \quad N^{-1/2}\norm{\stab_{\vect{z}_{1,t}, \vect{z}_{2,t}}^{-1}(1-\Pi_{z_1,z_2})\bigl[\vect{a}^\mathrm{diag}\bigr]}_2 \lesssim \langle |A|^2 \rangle^{1/2}.
	\end{equation}
	By \eqref{eq:m_evol}, $\m_{j,t} = \mathrm{e}^{(T-t)/2}\m(z_j)$, hence, it remain to estimate the first term on the right-hand side of \eqref{eq:stab_a_decomp}. The estimates \eqref{eq:Pi_bound}, \eqref{eq:Pi_continuity}, \eqref{eq:m_lower}, and \eqref{eq:m_continuity} imply that
	\begin{equation} \label{eq:regularity_Pi_correction}
		\norm{\Pi\bigl[\m_{1,t}\m_{2,t}\vect{a}^\mathrm{diag}\bigr] - \mathrm{e}^{T-t}\Pi_{z_1^-,z_2^+}\bigl[\m(z_1^-)\m(z_2^+)\vect{a}^\mathrm{diag}\bigr]}_\infty \lesssim N^{-1}|\re z_1 - \re z_2|\bigl\lVert\vect{a}^\mathrm{diag}\bigr\rVert_1.
	\end{equation}
	By Definition \ref{def:reg_A}, since $A$ is $(z_1,z_2)$-regular, $\Pi_{z_1^-,z_2^+}\bigl[\m(z_1^-)\m(z_2^+)\vect{a}^\mathrm{diag}\bigr] = 0$. Hence, we conclude from \eqref{eq:eigB_t}, \eqref{eq:a_diag_bounds}, and  \eqref{eq:regularity_Pi_correction}, that
	\begin{equation} \label{eq:a_Pi_bound}
		|\eigB_t|^{-1} \norm{\Pi\bigl[\m_{1,t}\m_{2,t}\vect{a}^\mathrm{diag}\bigr]}_\infty \lesssim \frac{\langle |A|^2 \rangle^{1/2}|\re z_1 -\re z_2|}{\eta_{1,t}+\eta_{2,t} + |z_1 - z_2|}\lesssim \langle |A|^2 \rangle^{1/2}.
	\end{equation}
	Combining \eqref{eq:diag_norm_bounds}, \eqref{eq:stab_a_decomp}, \eqref{eq:a_leftovers}, and \eqref{eq:a_Pi_bound} concludes the proof of \eqref{eq:regM_bound}.
	
	The proof of \eqref{eq:M_bound} follows from analogously, except absent the regularity of $A$, \eqref{eq:Pi_bound} loses the $|\re z_1 - \re z_2|$ factor in the numerator. This concludes the proof of Lemma \ref{lemma:Mbounds}.
\end{proof}

\subsection{Global Laws} \label{sec:global_laws}
We closely follow the procedure established in Appendix B of \cite{Cipolloni2022Optimal} and Appendix A of \cite{Cipolloni2022RankUnif}. However, to initialize the induction in the length of the chain that lies at the core of the argument, we need the following global laws for a single generalized resolvent, that we prove at the end of Section \ref{sec:global_laws}.
\begin{lemma} [A-Priori Single-Resolvent Global Laws] \label{lemma:1G_global_old}
	Let $T \sim 1$ be a fixed terminal time, and let $\mathfrak{f}^t \equiv \mathfrak{f}^t_T$  be the flow map defined in \eqref{eq:flow_map}. Then for $\vect{z}_0 := \mathfrak{f}^0(z)$, the following global laws
	\begin{equation} \label{eq:1G_global_old}
		\bigl\lvert \bigl\langle \bigl(G(H, \vect{z}_{0}) - M(\vect{z}_{0})\bigr)\diag{\vect{w}} \bigr\rangle \bigr\rvert \prec \frac{\norm{\vect{w}}_\infty}{N}, \quad \bigl\lvert \bigl\langle \vect{x}, \bigl(G(H, \vect{z}_{0}) - M(\vect{z}_{0})\bigr)\vect{y}\bigr\rangle \bigr\rvert \prec \frac{\norm{\vect{x}}_2\norm{\vect{y}}_2}{\sqrt{N}},
	\end{equation}
	uniformly in $z\in\mathcal{D}$ and deterministic vectors $\vect{x}, \vect{y}, \vect{w}$.
\end{lemma}
\begin{proof} [Proof of Proposition \ref{prop:global_laws} and the Global Law for Proposition \ref{prop:2iso}] 
	We only present the proof in the case $H$ is complex Hermitian 
	and $\mathscr{T} = 0$, the easy modifications in the case
	 $H$ is real symmetric are left to the reader (see also Section \ref{sec:realT} below). 
	First, we prove the local laws \eqref{eq:Global1av}.
	Using the definition of the generalized resolvent in \eqref{eq:Gt} and the vector Dyson equation \eqref{eq:VDE}, we deduce that for $G:=G_{1,0}$, $M := M_{1,0}$, and $\vect{z} := \vect{z}_{1,0}$, we have
	\begin{equation} \label{eq:G_underline}
		G - M = M \,\mathscr{S}[G - M] - M\, \underline{W_0 G(H_0, \vect{z}) },
	\end{equation}
	where $W_0 := H_0 - \bm{\mathfrak{a}}$, and for a functions $f(W)$, $g(W)$ of a random matrix $W$, we use the notation
	\begin{equation}
		\underline{f(W)W g(W)} := f(W)Wg(W) - \Expv_{\other{W}}\bigl[(\partial_{\other{W}}f)(W) \other{W} g(W) + f(W)\other{W} (\partial_{\other{W}}g)(W) \bigr].
	\end{equation}
	Here $\partial_{\other{W}}$ denotes the derivative in the direction of the random matrix $\other{W} = \other{W}^*$ with independent centered Gaussian entries with matrix of variances $S$.
	
	Note that \eqref{eq:m_upper}, \eqref{eq:1G_global_old} and \eqref{eq:G_underline} imply
	\begin{equation} \label{eq:1Gav_underline}
		\bigl\langle (G  - M ) B \bigr\rangle = - \bigl\langle BM\, \underline{W_0 G(H_0, \vect{z}) }\bigr\rangle + \mathcal{O}_\prec\bigl(N^{-1}\langle |B|^2 \rangle^{1/2}\bigr)
	\end{equation}
	Using \eqref{eq:1Gav_underline}, the cumulant expansion (see Eq. (4.14) in \cite{Cipolloni2022Optimal}), we deduce that, denoting $\mathcal{Q}_1 := \bigl\langle  BM \, \underline{W_0 G(H_0, \vect{z}) } \bigr\rangle$,
	\begin{equation}\label{eq:1G_expan}
		\Expv\bigl[|\mathcal{Q}_1|^{2p}\bigr] \lesssim \Expv\bigl[\other{\Xi}_1|\mathcal{Q}_1|^{2p-2}\bigr] + \sum_{R\ge \vect{l} + \sum J\cup J^* \ge 2}\Expv\bigl[\Xi_1(\vect{l},J,J^*) |\mathcal{Q}_1|^{2p-1-|J\cup J^*|}\bigr] + \mathcal{O}\bigl(N^{-2p}\langle|B|\rangle^p\bigr),
	\end{equation}
	where $R := R(p) \in \mathbb{N}$, the summation runs over $\vect{l} \in \mathbb{Z}_{\ge 0}^2$ and $J,J^* \subset \mathbb{Z}_{\ge 0}^2\backslash \{(0,0)\}$ such that $|J\cup J^*| \le 2p-1$, and $\sum J\cup J^* := \sum_{\vect{j}\in J\cup J^*} |\vect{j}|$. Here, the main Gaussian term $\other{\Xi}_1$ is given by
	\begin{equation} \label{eq:gaussian_bound}
		\other{\Xi}_1 := \frac{1}{N^2}\biggl\lvert\sum_{a,b}S_{ab}(GBM)_{ba}(GBG)_{ba}\biggr\rvert + \frac{1}{N^2}\biggl\lvert \sum_{a,b}S_{ab}(GBM)_{ba}(G^*B^*G^*)_{ba}\biggr\rvert \lesssim \frac{1}{N^2}\langle |B|^2 \rangle,
	\end{equation}
	where to obtain the estimate we used \eqref{eq:m_upper} and the bound $\norm{G}\lesssim 1$ by \eqref{eq:G_norm}.
	The quantities $\Xi_1(\vect{l},J,J^*)$ are defined as
	\begin{equation} \label{eq:Xi_1}
		\Xi_1(\vect{l},J,J^*) := N^{-(|\vect{l}| + \sum J\cup J^* + 3)/2} \sum_{a,b} \bigl\lvert\partial_{ab}^{\vect{l}}(GBM)_{ba}\bigr\rvert \prod_{\vect{j}\in J} \bigl\lvert \partial_{ab}^{\vect{j}}\mathcal{Q}_1 \bigr\rvert \prod_{\vect{j}\in J^*} \bigl\lvert \partial_{ab}^{\vect{j}}\overline{\mathcal{Q}_1}\bigr\rvert,
	\end{equation}
	where for $\vect{j} = (j_1,j_2) \in \mathbb{Z}_{\ge0}^2$, we denote $\partial_{ab}^\vect{j} := \partial_{ab}^{j_1}\partial_{ba}^{j_2}$, and $\partial_{ab}$ denotes the complex partial derivative in the direction of the matrix element $H_{ab,0}$. Using \eqref{eq:m_upper} and \eqref{eq:G_norm}, we deduce the bounds
	\begin{equation} \label{eq:deriv_bounds}
		N^{-3/2}\sum_{a,b} \bigl\lvert(GBM)_{ba}\bigr\rvert+ N^{-2}\sum_{a,b} \bigl\lvert \partial_{ab}^{\vect{j}}(GBM)_{ba}\bigr\rvert  + N^{1/2}
		\bigl\lvert \partial_{ab}^{\vect{j}}\mathcal{Q}_1 \bigr\rvert \lesssim \langle |B|^2\rangle^{1/2}, \quad |\vect{j}| > 0.
	\end{equation}
	Plugging  the estimates \eqref{eq:gaussian_bound} and \eqref{eq:deriv_bounds} into \eqref{eq:1G_expan}, and using Young's inequality, we conclude that $|\mathcal{Q}_1| \prec N^{-1}\langle |B|^2\rangle^{1/2}$, hence establishing \eqref{eq:Global1av}.

	Next, we address the laws \eqref{eq:Global2av}, \eqref{eq:Global1iso}, and the global law version of \eqref{eq:3G_iso_t}.
	We adopt the notation $G_{[j,k]} := \prod_{p=j}^{k-1}(G_{p,0}B_p)G_{k,0}$ introduced in \cite{Cipolloni2023Edge}, and let $M_{[j,k]} := M_{[j,k],0}$ denote the corresponding deterministic approximations, as defined in \eqref{eq:M_def}, \eqref{eq:M2iso_def}, and given by $\diag{\m(\vect{z}_{j,0})}$ for $j=k$.
	Note that it suffices to prove the averaged law \eqref{eq:Global2av} and its three-resolvent analog,
	\begin{equation} \label{eq:Global3G_av}
		\bigl\lvert\bigl\langle (G_{[1,3]} - M_{[1,3]}) B_3 \bigr\rangle\bigr\rvert \prec N^{-1}\norm{B_1}\norm{B_2}\langle |B_3|^3 \rangle^{1/2},
	\end{equation}
	since setting $B_k := N\vect{y}\vect{x}^*$, with $\langle |B_k|^2 \rangle = \sqrt{N}\norm{x}_2\norm{y}_2$, the immediately imply the isotropic laws \eqref{eq:Global1iso} and $\bigl\lvert\bigl\langle \vect{x}, (G_{[1,3]} - M_{[1,3]}) \vect{y} \bigr\rangle\bigr\rvert \prec N^{-1/2}\norm{B_1}\norm{B_2}$.
	For the normalized trace of a chain containing $k$ resolvents, the analog of \eqref{eq:1Gav_underline} reads
	\begin{equation}
		\begin{split}
			\bigl\langle (G_{[1,k]} - M_{[1,k]}) B_k \bigr\rangle =&~ - \bigl\langle \underline{W_0 G_{[1,k]}}B_k' \bigr\rangle 
			+ \biggl\langle \bigl(G_{[2,k]}-M_{[2,k]}\bigr) B_k' \biggl(B_1 + \sum_{j=2}^{k-1} \mathscr{S}\bigl[M_{[1,j]}\bigr]\biggr) \biggr\rangle\\
			&+\sum_{j=1}^{k-1}\bigl\langle \mathscr{S}\bigl[G_{[1,j]}-M_{[1,j]}\bigr]G_{[j,k]}B_k' \bigr\rangle 
			+ \bigl\langle \mathscr{S}\bigl[(G_{[k]}-M_{[k]})B_k'\bigr]G_{[1,k]} \bigr\rangle,
		\end{split}
	\end{equation}
	where $B_k' := M_k^{-1}(1-M_1M_k\mathscr{S})^{-1}[M_k B_kM_1]$. For $k \in \{2,3\}$ assuming that the averaged local law is proved for chains of length $j \in \{1,\dots, k-1\}$, we once again conclude that all the terms except the underline are stochastically dominated by $N^{-1}\prod_{j=1}^{k-1}\norm{B_j}\langle |B_k|^2\rangle^{1/2}$ using the Hilbert-Schmidt bound in  \eqref{eq:scalar_stab_bounds} for $B_k'$  and the norm bound \eqref{eq:G_norm} for $G_j$. Then \eqref{eq:1G_expan} holds for $\mathcal{Q}_k := \langle (G_{[1,k]} - M_{[1,k]}) B_k \rangle$, with $\other{\Xi}_1$ and $\Xi_1$ replaced by $\other{\Xi}_k$ and, respectively, $\Xi_k$. The Gaussian terms $\other{\Xi}_k$ is given by 
	\begin{equation}
		\other{\Xi}_2 := \frac{1}{N^2}\biggl\lvert \sum_{a,b} S_{ab} (G_{[1,k]}B_k')_{ba} (G_{[1,k]}G_1)_{ba} \biggr\rvert  + (\dots)  \lesssim \frac{1}{N^2}  \prod_{j=1}^{k-1}\norm{B_j}^2\langle |B_k|^2\rangle,
	\end{equation}
	where $(\dots)$ are structurally identical terms that differ only in replacing $G_j$ and $B_j$ by their respective adjoints or cyclically permuting the indices $\{1,\dots, k\}$ in the factor $(G_{[1,k]}G_1)_{ba})$. Here the bound follows from \eqref{eq:scalar_stab_bounds}, \eqref{eq:m_upper} and \eqref{eq:G_norm}. The higher-order derivatives $\Xi_k$ are defined analogously to \eqref{eq:Xi_1} with $(GBM)$ replaced with $(G_{[1,k]}B_k')$ and $\mathcal{Q}_1$ with $\mathcal{Q}_k$. Similarly, the bounds \eqref{eq:deriv_bounds} also hold with $(GBM)$ replaced with $(G_{[1,k]}B_k')$, yielding sequentially the global law \eqref{eq:Global2av} and \eqref{eq:Global3G_av}. This concludes the proof of Proposition \ref{prop:global_laws} and the global law for Proposition \ref{prop:2iso}.
\end{proof}
We close the section by proving Lemma \ref{lemma:1G_global_old}.
\begin{proof}[Proof of Lemma \ref{lemma:1G_global_old}]
	To see that \eqref{eq:1G_global_old} indeed hold, we recall the resolvent rescaling \eqref{eq:G_rescale} that implies
	\begin{equation} \label{eq:globabl_G_rescale}
		G_{0} = \bigl(H-\vect{z}_0\bigr)^{-1} = \bm{\nu}\bigl(\bm{\nu}(H - \re\vect{z}_0)\bm{\nu}-\I (\sign z) \eta_0\bigr)^{-1} \bm{\nu} =: \bm{\nu}\other{G}(\I\eta_0 \sign z ) \bm{\nu},
	\end{equation}
	where $\bm{\nu} := \diag{(\eta_0^{-1}|\im\vect{z}_0|)^{-1/2}} \sim 1$ by \eqref{eq:imz_flat}, and $\eta_0 := |\langle \im\vect{z}_0 \rangle| \sim 1$ by \eqref{eq:eta_asymp} with $T\sim 1$. Therefore, the global laws \eqref{eq:1G_global_old} with  follow immediately from the corresponding global laws for the standard resolvent $\other{G}(\zeta)$ of a different Wigner-type random matrix $\other{H} := \bm{\nu} (H - \re\vect{z}_0)\bm{\nu}$. Indeed, since $\bm{\nu} \sim 1$ and $\norm{\re \vect{z}_0}_\infty \lesssim 1$ by \eqref{eq:flow_explicit} and \eqref{eq:m_lower}, the random matrix $\other{H}$ satisfies the assumptions Assumption \eqref{as:S_flat}, \eqref{as:moments}, and the solution $\mathfrak{m}$ to the corresponding vector Dyson equation
	\begin{equation} \label{eq:VDE_rescale}
		-\frac{1}{\mathfrak{m}(\zeta)} = \zeta - \bm{\nu}^2 (\bm{\mathfrak{a}} - \re \vect{z}_0) + \bm{\nu}^2 S \bm{\nu}^2 [\mathfrak{m}(\zeta)], \quad (\im \zeta) (\im \mathfrak{m}(\zeta)) >0
	\end{equation}
	satisfies the trivial bound $\lVert\mathfrak{m}(\zeta)\rVert_\infty \le |\im \zeta|^{-1}$, hence the Assumption \eqref{as:m_bound} with some constant $\other{c}_{\mathrm{m}} \lesssim 1$ also holds in the domain $\{|\im \zeta| \ge c\}$ for any $c \sim 1$. By examining the proof of Proposition 3.1 in \cite{Ajanki2016Univ}, we observe that this is indeed enough to establish the global laws
	\begin{equation} \label{eq:1G_global_rescale}
		\bigl\lvert \bigl\langle \bigl(\other{G}(\zeta_0) - \mathfrak{m}(\zeta_0)\bigr)\bm{\nu}^2\diag{\vect{w}} \bigr\rangle \bigr\rvert \prec \frac{\norm{\vect{w}}_\infty}{N}, \quad \bigl\lvert \bigl\langle \bm\nu\vect{x},\bigl(\other{G}(\zeta_{0}) - \mathfrak{m}(\zeta_{0})\bigr)\bigr\rangle \bigr\rvert \prec \frac{\norm{\vect{x}}_2\norm{\vect{y}}_2}{\sqrt{N}},
	\end{equation}
	uniformly in $\vect{x}, \vect{y}, \vect{w}$, at $\zeta_{0} := \I\eta_0 \sign z$. Here we identified $\mathfrak{m}$ with $\diag{\mathfrak{m}}$ for brevity. Since $\bm{\nu}^2\diag{\mathfrak{m}(\zeta_{0})} = M(\vect{z}_0)$ by \eqref{eq:VDE_rescale} and the uniqueness of the solution, the global laws \eqref{eq:1G_global_old} follow from \eqref{eq:1G_global_rescale}, with uniformity in $z \in \mathcal{D}$ obtained via a simple grid argument.
\end{proof}

\subsection{The Case of Self-Energy Operators with  Off-Diagonal Part} \label{sec:realT}

In this section, we present the minor modifications necessary to lift the condition $\mathscr{T} = 0$, 
where $\mathscr{T}$,  defined in \eqref{eq:superT_def}, is the off-diagonal components of the self-energy operator. 
To preserve $\Expv[(H_{jk,t})^2]$ along the flow, we replace the standard Brownian motion on the right-hand side of \eqref{eq:OUflow} by  a Hermitian random matrix $\widehat{\Brwn}_t$  with independent (up to symmetry) centered Gaussian entries satisfying $\Expv[|\widehat{\Brwn}_{jk,t}|^2] = t$ and  $\Expv[(\widehat{\Brwn}_{jk,t})^2] = t \, \mathcal{T}_{jk}/S_{jk}\mathds{1}_{S_{jk}\neq 0}$, where $\mathcal{T}$ is defined in \eqref{eq:superT_def}. Note that the entries of $\mathcal{T}$ satisfy the trivial bound 
\begin{equation} \label{eq:Tjk_bound}
	|\mathcal{T}_{jk}| \le S_{jk}.
\end{equation}
%For concreteness, we set 
%\begin{equation}
%	\Brwn_{jk,t} := \sqrt{\tfrac{1}{2}S_{jk}} \mathrm{e}^{\I(\alpha_{jk} + \theta_{jk}) } \mathfrak{W}^{(1)}_{jk,t} + \sqrt{\tfrac{1}{2}S_{jk}} \mathrm{e}^{\I(\alpha_{jk} - \theta_{jk}) } \mathfrak{W}^{(2)}_{jk,t},
%\end{equation}
%where $\alpha_{jk}~:=~\tfrac{1}{2}\cos^{-1}(|\mathcal{T}_{jk}|/S_{jk})$ and $\theta_{jk} := \tfrac{1}{2}\arg(\mathcal{T}_{jk})$, where $\mathfrak{W}^{(1)}_t $ and $\mathfrak{W}^{(2)}_t$ are two independent copies of standard Brownian motion in the space of real symmetric $N\times N$ matrices.

Recall the notation $G_{[j,k],t} := \prod_{p=j}^{k-1} (G_{p,t}A_p)G_{k,t}$ for resolvent chains and $M_{[j,k],t}$ for their respective deterministic approximations.
By It\^{o}'s formula, the evolution equation of $\langle (G_{[1,k],t} - M_{[1,k],t})A_k \rangle$ 
 (e.g., \eqref{eq:(G-m)A_evol} for $k=1$ or \eqref{eq:(GAG-M)A_evol} for $k=2$)  under the flow \eqref{eq:OUflow} with the modified stochastic term $\mathrm{d}\widehat{\Brwn}_t$, contains two new terms, 
\begin{equation} \label{eq:realT_new_terms}
	\frac{1}{2}\sum_{a,b} \partial_{ab}\bigl\langle G_{[1,k],t}A_k \bigr\rangle\mathrm{d}\widehat{\Brwn}_{ab,t}, \quad \text{and}\quad \sum_{1\le p \le q \le k} \bigl\langle \mathscr{T}\bigl[G_{[p,q],t}\bigr]G_{[q,k],t} A_k G_{[1,p],t}   \bigr\rangle.
\end{equation}
It suffices to consider $k \in \{1,2\}$ for the proof of Proposition \ref{prop:masters} and $k=3$ for Proposition \ref{prop:2iso}.

Note that by \eqref{eq:Tjk_bound} and the Cauchy-Schwarz inequality, the quadratic variation of the martingale term in \eqref{eq:realT_new_terms} admits the bound
\begin{equation}
	\biggl[\int_0^\cdot \sum_{a,b}\partial_{ab}\bigl\langle G_{[1,k],s}A_k \bigr\rangle\mathrm{d}\widehat{\Brwn}_{ab,s}\biggr]_t \lesssim 
	\int_{0}^t\sum_{a,b} S_{ab} \bigl\lvert \partial_{ab} \bigl\langle G_{[1,k],s}A_k \bigr\rangle \bigr\rvert^2 \mathrm{d}s,
\end{equation}
and therefore admits the same estimate as its counterpart in the case $\mathscr{T} = 0$.

To estimate the second term in \eqref{eq:realT_new_terms}, we observe that for any
 $N\times N$ matrices $X$ and $Y$, the upper bound in \eqref{as:S_flat} and \eqref{eq:Tjk_bound} imply
\begin{equation} \label{eq:T_quadform}
	\bigl\lvert \bigl\langle \mathscr{T}[X] Y\bigr\rangle \bigr\rvert \lesssim \frac{1}{N}\langle |X|^2\rangle^{1/2} \langle |Y|^2\rangle^{1/2}.
\end{equation}
The $N^{-1}$ prefactor makes the simple estimate \eqref{eq:T_quadform} effective.  Namely, the operator inequality \eqref{eq:ImG_ineq} reduces $\langle |G_{[p,q],t}|^2\rangle$ and $\langle |G_{[q,k],t} A_k G_{[1,p],t}|^2\rangle$ to $\eta_{p,t}^{-1}\eta_{q,t}^{-1}$ times an averaged traces of alternating resolvent chains of even lengths between $2$ and $2k$; the only exception being the $p=q$ case: $\langle |G_{[p,q],t}|^2\rangle \lesssim \eta_{p}^{-1} \langle \im G_{p,t} \rangle$. All such terms have been estimated in the $\mathscr{T} = 0$ case.
We exemplify this by considering the case $k = 2$ with  regular observables $A_1$, $A_2$ as in the proof of the master inequality \eqref{eq:Master2av_HS}. 
%Indeed, if the observables $A_i$ are regular or the target error bound is controlled in terms of $\langle |A_i|^2\rangle^{1/2}$, the chains of length $4$ and $6$ can be bounded in terms of chains of length $2$ using the reduction inequalities of Lemma \ref{lemma:reds}. 
For $p=q=1$, using the reduction inequality \eqref{eq:Red4av_HS} and the integration rule \eqref{eq:int_rules}, we obtain 
\begin{equation} \label{eq:T_term1}
	\begin{split}
		\int_0^t \bigl\lvert \bigl\langle \mathscr{T}\bigl[G_{1,s}\bigr]G_{[1,2],s} A_2 G_{1,s}   \bigr\rangle \bigr\rvert \mathrm{d}s 
		&\prec \int_0^t \frac{1}{N\eta_{1,s}^{3/2}}\bigl\lvert \bigl\langle \im G_{1,s}A_1G_{2,s} A_2 \im G_{1,s} A_2^* G_{2,s}^*A_1^*   \bigr\rangle \bigr\rvert^{1/2}\mathrm{d}s \\ 
		&\prec \frac{\langle |A_1|^2 \rangle^{1/2}\langle |A_2|^2 \rangle^{1/2}}{\sqrt{N\eta_t}}\biggl(1 + \frac{\phi_2^{\mathrm{hs}}}{\sqrt{N\eta_t}}\biggr).
	\end{split}
\end{equation}
Analogous estimate holds for $p=q=2$. For $p=1, q=2$, we do not apply the reduction inequality and retain the full $N^{-1}$ prefactor, 
\begin{equation} \label{eq:T_term2}
	\begin{split}
		\int_0^t \bigl\lvert \bigl\langle \mathscr{T}\bigl[G_{[1,2],s}\bigr]G_{[2,1],s}\bigr\rangle \bigr\rvert \mathrm{d}s 
		\prec&~ \int_0^t  \bigl\lvert\bigl\langle \im G_{2,s} A_2 \im G_{1,s} A_2^*   \bigr\rangle \bigr\rvert^{1/2}\bigl\lvert\bigl\langle \im G_{1,s} A_1 \im G_{2,s} A_1^*   \bigr\rangle \bigr\rvert^{1/2}\frac{\mathrm{d}s}{N\eta_{s}^2}\\
		\prec&~ \frac{\langle |A_1|^2 \rangle^{1/2}\langle |A_2|^2 \rangle^{1/2}}{N\eta_t}\biggl(1 + \frac{\phi_2^{\mathrm{hs}}}{\sqrt{N\eta_t}}\biggr).
	\end{split}
\end{equation}
Note that the right-hand sides of both \eqref{eq:T_term1} and \eqref{eq:T_term2} can be incorporated into the bound \eqref{eq:2av_mart_final}.

By similar considerations, the contribution from the second term in \eqref{eq:realT_new_terms} admits the same bound as the quadratic variation of the martingale in the corresponding evolution equation for $\langle (G_{[1,k],t} - M_{[1,k],t})A_k \rangle$.
Therefore, the re-introduction of non-trivial $\mathscr{T}$ does not alter the estimates that appear in the Sections \ref{sec:master_proofs}, \ref{sec:GFT} and \ref{seq:prior_laws}.

\vspace{5pt}
{\bf Acknowledgment.} The authors thank Joscha Henheik and Oleksii Kolupaiev for many helpful discussions.

\vspace{5pt}
{\bf Data availability.} There is no data associated to this work.

\vspace{5pt}
{\bf Competing Interests.} The authors have no  conflict of interest to disclose.

\printbibliography
\end{document}